\documentclass[intlim,righttag,10pt]{amsart}
\usepackage{amscd}
\usepackage[bbgreekl]{mathbbol}
\usepackage{amssymb}
\usepackage[all]{xy}
\usepackage[textsize=tiny]{todonotes}

\RequirePackage{amsfonts,latexsym,amssymb}

\RequirePackage{mathrsfs}
\let\mathcal\mathscr
\RequirePackage[T1]{fontenc}
\let\cal\mathcal
\usepackage{bbm}

\usepackage{url}
\usepackage{bbm}
\usepackage{imakeidx}

\usepackage{hyperref}

\newtheorem{theorem}[equation]{Theorem}
 \newtheorem{lemma}[equation]{Lemma}
 \newtheorem{proposition}[equation]{Proposition}
 \newtheorem{corollary}[equation]{Corollary}

\theoremstyle{definition}

\theoremstyle{remark}
\newtheorem{remark}[equation]{Remark}

\newtheorem*{acknowledgments}{Acknowledgments}
\newtheorem*{structure}{Structure of the paper}

\def\jcdot{\scriptscriptstyle\bullet}

\def\invlim{\mathop{\vtop{\ialign{##\crcr$\hfill{\lim}\hfil$\crcr
\noalign{\kern1pt\nointerlineskip}\leftarrowfill\crcr\noalign
{\kern -3pt}}}}\limits}
\def\dirlim{\mathop{\vtop{\ialign{##\crcr$\hfill{\lim}\hfil$\crcr
\noalign{\kern1pt\nointerlineskip}\rightarrowfill\crcr\noalign
{\kern -3pt}}}}\limits} 
\def\lomapr#1{\smash{\mathop{\relbar\joinrel\longrightarrow}\limits^{#1}}\,}
 \def\verylomapr#1{\smash{\mathop{\relbar\joinrel\relbar\joinrel\relbar\joinrel\longrightarrow}\limits^{#1}}\,}
\def\veryverylomapr#1{\smash{\mathop{\relbar\joinrel\relbar\joinrel\relbar
\joinrel\relbar\joinrel\relbar\joinrel\longrightarrow}\limits^{#1}}\,}
\def\phi{\varphi}
\def\epsilon{\varepsilon}
\def\bwedge{\underset{\widehat{\phantom{\ \ }}}{\phantom{\_}}}

\newcommand{\ovk}{\overline{K} }
\newcommand{\wlim}{\operatorname{lim} }
\newcommand{\dr}{\operatorname{dR} } 
  \newcommand{\hk}{\operatorname{HK} }   
  \newcommand{\LL}{\mathrm {L} }
   \newcommand{\nr}{\operatorname{nr} }   
 \newcommand{\colim}{\operatorname{colim} }

 \newcommand{\proeet}{\operatorname{pro\acute{e}t} } 
 \newcommand{\eet}{\operatorname{\acute{e}t} }

 \newcommand{\conv}{\operatorname{conv} }
 \newcommand{\Spec}{\operatorname{Spec} }

 \newcommand{\Spf}{\operatorname{Spf} }
  
 \newcommand{\Hom}{{\rm{Hom}} }

 \newcommand{\Gal}{\operatorname{Gal} }
 
 \newcommand{\can}{ \operatorname{can} }
 
 \newcommand{\id}{ \operatorname{Id} }
\newcommand{\synt}{ \operatorname{syn} }
 
\newcommand{\Lie}{ \operatorname{Lie} }

 \newcommand{\Cone}{\operatorname{Cone} }
  \newcommand{\End}{\operatorname{End} }

\newcommand{\st}{\operatorname{st} }

 \newcommand{\kker}{\operatorname{Ker} }

 \newcommand{\crr}{\operatorname{cr} }
  \newcommand{\sem}{\operatorname{ss} }
 \newcommand{\gr}{\operatorname{gr} }
 
  \newcommand{\rig}{\operatorname{rig} }
 \newcommand{\im}{\operatorname{Im} }
  \newcommand{\gk}{\rm GK }

 \newcommand{\kr}{^{\jcdot }}

 \newcommand{\sff}{{\mathcal{F}}}  
 
 \newcommand{\sy}{{\mathcal{Y}}}
 
 \newcommand{\sg}{{\mathcal{G}}}
 
 \newcommand{\sv}{{\mathcal{V}}}
 \newcommand{\su}{{\mathcal{U}}}
 \newcommand{\sbb}{{\mathcal{B}}}
 \newcommand{\scc}{{\mathcal{C}}}
 \newcommand{\sk}{{\mathcal{K}}}
 \newcommand{\sll}{{\mathcal{L}}}
 
 \newcommand{\so}{{\mathcal O}}
 \newcommand{\sj}{{\mathcal J}}
 \newcommand{\se}{{\mathcal{E}}}
 \newcommand{\sa}{{\mathcal{A}}}
 
 \newcommand{\sx}{{\mathcal{X}}}

\newcommand{\sd}{{\mathcal{D}}} 
\newcommand{\sm}{{\mathcal{M}}}

 \newcommand{\wt}{\widetilde}
 \newcommand{\wh}{\widehat}

\newcommand{\Z}{ {\mathbf Z} }
\newcommand{\laz}{ {\mathcal  L}\hskip-.05cm{ az} }
   \newcommand{\Q}{ {\mathbf Q}}
   \newcommand{\N}{{\mathbf N}}
   \def\cdot{{\scriptscriptstyle\bullet}}

\def\C{{\bf C}}
\def\R{{\mathrm R}}

\def\epsilon{\varepsilon}

\def\bcris{{\bf B}_{{\rm cris}}}

\def\bbA{{\mathbb A}} 
\def\Ainf{{\mathbb A}_{\rm inf}}
\def\Acris{{\mathbb A}_{\rm cr}}
\def\Bcris{{\mathbb B}_{{\rm cr}}} \def\Bdr{{\mathbb B}_{{\rm dR}}}
\def\Bst{{\mathbb B}_{{\rm st}}}

\def\E{{\bf E}} \def\A{{\bf A}} 
\def\B{{\bf B}}

\let\cal\mathcal

\def\Q{{\bf Q}} \def\Z{{\bf Z}}
\def\C{{\bf C}}
\def\N{{\bf N}}

\def\epsilon{\varepsilon}

\def\bcris{{\bf B}_{{\rm cris}}}

\def\E{{\bf E}} \def\A{{\bf A}} \def\B{{\bf B}}

\def\rg{{\rm R}\Gamma}

\def\OK{{$\so_K^\times$, $\so_K^0$}}
\def\Xn{{$X_0$, $X_1$, $X_n$, $X_1^0$, $X_h$, $X_h^{\rm o}$}}
\def\Smo{{${\rm Sm}$, ${\rm Sm}^\dagger$}}
\def\Perf{{${\rm Perf}_C$}}
\def\CK{{$C_K$, $C_{\Q_p}$}}
\def\Dcal{{$\sd(-)$}}
\def\HHH{{$\wt{H}$}}
\def\PCK{{${\rm PC}_K$, ${\rm PD}_K$}}
\def\LHK{{$LH(C_K)$}}
\def\SSS{{$S$, $S_K$, $S^0$, $\overline{S}$}}
\def\rpd{{$r_F^{\rm PD}$, $r_K^{\rm PD}$}}
\def\Lcal{{${\cal L}$, ${\cal L}_K$, ${\cal L}^0_K$, $\overline{\cal L}$, $\overline{\cal L}^0$, ${\cal L}_{\crr}$, ${\cal L}_\varphi$}}
\def\BST{{$\B^+_{\st}$, $\widehat{\B}^+_{\st}$, $\B_{\st}^{\leq r}$, $\widehat{\B}^+_{l,\st}$, $\widehat{\B}^+_{p,\st}$}}
\def\Nnilp{{$(-)^{N{\text{-nilp}}}$, $(-)^{N_l{\text{-nilp}}}$}}

\def\slambda{{$s_\lambda$}}
\def\iotap{{$\iota$, $\iota_p$}}
\def\iotahk{{$\iota_{\rm HK}$, $\iota^\dagger_{\rm HK}$, $\iota_{\proeet}$, $\iota_{\rm syn}^\dagger$}}

\def\ilst{{$i_l$, $\iota_l$}}
\def\thetalambda{{$\theta_\lambda$}}
\def\basechange{{$\beta$}}
\def\epshk{{$\epsilon_{\st}^{\rm HK}$, $\epsilon_{\crr}^{\rm HK}$, $\epsilon_{\dr}^{\rm HK}$, $\epsilon_{\B_{\dr}^+}^{\rm HK}$}}

\def\rgcris{{$\rg_{\crr}$, $\rg_{\crr}^{\rm rel}$, $\rg_{{\crr},\ovk}$, $\rg_{\crr}(X_1/R)_l$}}
\def\rghk{{$\rg_{\rm HK}$, $\rg^\dagger_{\rm HK}$, $\rg_{\rm HK}^{\rm GK}$, $\rg_{{\rm HK},\breve{F}}$, $\rg_{{\rm HK},\breve{F}}^{\rm GK}$}}
\def\rgproet{{$\rg_{\proeet}$}}
\def\rgrig{{$\rg_{{\rm rig},\ovk}$}}
\def\rgsyn{{$\rg_{\rm syn}^{\rm GK}$}}
\def\rgconv{{$\rg_{{\rm conv},\ovk}$}}
\def\rgdr{{$\rg_{\dr}$, $\rg_{\dr}(X/\B_{\dr}^+)$, $\rg_{\dr}^\dagger(X/\B_{\dr}^+)$, $\rg_{\dr}^{\rm BMS}$, $\rg_{\dr}^{\rm Guo}$}}
\def\rginf{{$\rg_{\rm inf}(X/\B_{\dr}^+)$}}

\def\alphar{{$\alpha_r$, $\widehat\alpha_r$, $\alpha_r^\dagger$}}
\def\alphahk{{$\alpha_{\rm HK}$, $\alpha^\dagger_{\rm HK}$, $\alpha_{{\rm HK},\breve{F}}$}}

\def\Mss{{$\sm^{\rm ss}$, $\sm^{{\rm ss}, b}$}}

\def\acalcris{{${\cal A}_{\crr}$, ${\cal A}_{\crr}^{\rm rel}$, ${\cal A}_{{\crr},\ovk}$, ${\cal A}_{\crr}^{\bwedge}$}}
\def\acaldr{{${\cal A}_{\dr}$}}
\def\acalhk{{${\cal A}_{\rm HK}^c$}}

\def\tensor{{$\widehat\otimes^{\rm R}_{F^{\rm nr}}$}}

\def\thetavar{{$\vartheta$}}
\def\BETAXX{{$\beta_X$}}
\def\iotabk{{$\iota_{\rm BK}$, $\iota^\dagger_{\rm BK}$, $\hat\iota_{\rm BK}$, $\iota_{{\rm BK},r}$}}

\def\hkr{{${\rm HK}(-,-)$, $\widetilde{\rm HK}(-,-)$}}
\def\drr{{${\rm DR}(-,-)$}}

\numberwithin{equation}{section}
\makeindex
\begin{document}
\title[On the cohomology of $p$-adic analytic spaces, I: The basic comparison theorem.]
 {On the cohomology of $p$-adic analytic spaces, I:  The basic comparison theorem.}
 \author{Pierre Colmez} 
\address{CNRS, IMJ-PRG, Sorbonne Universit\'e, 4 place Jussieu, 75005 Paris, France}
\email{pierre.colmez@imj-prg.fr} 
\author{Wies{\l}awa Nizio{\l}}
\address{CNRS, IMJ-PRG, Sorbonne Universit\'e, 4 place Jussieu, 75005 Paris, France}
\email{wieslawa.niziol@imj-prg.fr}
 \date{\today}
\thanks{The authors' research was supported in part by the grant ANR-19-CE40-0015-02 COLOSS  and the NSF grant No. DMS-1440140.}
 \begin{abstract}
The purpose of this paper is to prove a basic $p$-adic comparison theorem 
for smooth rigid analytic and dagger varieties over the algebraic closure $C$ of a $p$-adic field:  
$p$-adic pro-\'etale cohomology, in a stable range, can be expressed as a filtered Frobenius 
eigenspace of de Rham cohomology (over $\B^+_{\dr}$). 
The key computation is the passage from absolute crystalline 
cohomology to Hyodo-Kato cohomology and  the construction of the related Hyodo-Kato isomorphism. 
We also ``geometrize'' our comparison theorem by turning $p$-adic pro-\'etale and syntomic cohomologies 
into sheaves on the category ${\rm Perf}_C$ of perfectoid spaces over $C$ and the period morphisms into maps between  such sheaves (this geometrization
will be crucial in our study of the $C_{\rm st}$-conjecture in the sequel to this paper and in the formulation
of duality for geometric $p$-adic pro-\'etale cohomology).
 \end{abstract}

\maketitle

 \tableofcontents
\section{Introduction}
Let $\so_K$ be a complete discrete valuation ring with fraction field
$K$  of characteristic 0 and with perfect
residue field $k$ of characteristic $p$.  Let $\ovk$ be an algebraic closure of $K$, let $C$ be its $p$-adic completion, and let $\so_{\ovk}$ denote the integral closure of $\so_K$ in $\ovk$. Let
$W(k)$ be the ring of Witt vectors of $k$ with 
 fraction field $F$ (i.e, $W(k)=\so_F$) and 
let $\varphi$ be the absolute
Frobenius on $W(\overline {k})$. Set $\sg_K=\Gal(\overline {K}/K)$.

  In a joint work with Gabriel Dospinescu \cite{CDN1}, \cite{CDN3} we have computed the $p$-adic (pro-)\'etale cohomology of certain $p$-adic symmetric spaces. A key ingredient of these computations was a one-way (de Rham-to-\'etale) comparison theorem for rigid analytic Stein varieties over $K$ with a semistable formal model over $\so_K$. This theorem had two parts: first, it  related  (pro-)\'etale cohomology to rigid analytic  syntomic cohomology and, then,  it expressed  rigid analytic  syntomic cohomology as a filtered Frobenius eigenspace associated to  de Rham cohomology (tensored with  $\B^+_{\dr}$). From these two parts it is the second one  that had much harder proof. 
  
    The current paper 
is the second one in a series  extending such comparison theorems  to smooth rigid analytic varieties over $K$ or $C$ (without any assumption on the existence of a nice integral model).
 While in the first paper \cite{CN3} we have focused on the arithmetic case,  here we focus on  the geometric case.
 Moreover, in comparison with \cite{CDN1} and \cite{CN3}, we  significantly simplify the passage from rigid analytic  syntomic cohomology to  a filtered Frobenius eigenspace associated to  $\B^+_{\dr}$-cohomology\footnote{If the variety is defined over $K$, its $\B^+_{\dr}$-cohomology  is just de Rham cohomology tensored with $\B^+_{\dr}$.}.
 This requires a foundational work on Hyodo-Kato cohomology and Hyodo-Kato morphism, which occupies a good portion of this paper. 

In \cite{CN5}, the third paper in the series, 
we will use the results of this paper to prove the $C_{\rm st}$-conjecture for classes of smooth (dagger)
varieties over $C$ including quasi-compact varieties and some classes of
holomorphically convex varieties (hopefully, this conjecture should hold for general smooth 
partially proper varieties). 
This includes a description of the $\B^+_{\dr}$-cohomology (with its extra-structures, namely
Frobenius and monodromy) in terms of the $p$-adic pro-\'etale cohomology and, conversely,
a description of the $p$-adic pro-\'etale cohomology in terms of differential forms
(the $\B^+_{\dr}$-cohomology and the de Rham complex). To this end, 
the comparison isomorphisms proved here are "geometrized", i.e., we view them as
$C$-points of isomorphisms between Vector Spaces. This geometrization is also  essential in the formulation of duality for  geometric $p$-adic  pro-\'etale cohomology  \cite{CGN2}. 

      \subsection{Main results}
   
\subsubsection{The basic comparison theorem for rigid analytic varieties} 
We start the survey of our main results with 
    the following  comparison theorem: 

  \begin{theorem}{\rm ({\rm Basic comparison theorem})}
  \label{basic}
  Let $X$ be a smooth rigid analytic  variety over $C$. Let $r\geq 0$. There is a natural  strict quasi-isomorphism\footnote{All cohomology complexes live in the bounded below derived $\infty$-category of locally convex topological vector spaces over $\Q_p$. Quasi-isomorphisms in this category we call {\em strict quasi-isomorphisms}. See Section 1.2.1 for details.}
 {\rm (period isomorphism)}:
  \begin{equation}
  \label{cos1}
\tau_{\leq r}\rg_{\proeet}(X,\Q_p(r))\simeq \tau_{\leq r}  
\big[[\rg_{\hk}(X)\wh{\otimes}_{F^{\nr}}\B^+_{\st}]^{N=0,\phi=p^r}\lomapr{\iota_{\hk}}\rg_{\dr}(X/\B^+_{\dr})/F^r\big],
  \end{equation}
 where  the brackets $[...]$ denote the fiber. 
 \end{theorem}

  Most of the paper is devoted to the definition of the objects appearing in (\ref{cos1}) as well as the period morphism itself. 
This can be summed up in the following  theorem-construction from which Theorem \ref{basic} follows immediately. As before in \cite{CDN3}, \cite{CN3}, there are two steps: passage from pro-\'etale cohomology to syntomic cohomology (easier) and a passage from syntomic cohomology to Frobenius eigenspaces of de Rham cohomology over $\B^+_{\dr}$ (more difficult). 
    \begin{theorem}
 \label{main1}
 To any smooth rigid analytic  variety $X$ over $C$ there are  naturally  associated:
  \begin{enumerate}
\item A (rigid analytic) syntomic cohomology $\rg_{\synt}(X,\Q_p(r))$, $r\in \N$,  with a natural period morphism
 \begin{equation}
 \label{period-now}
 \alpha_r: \rg_{\synt}(X,\Q_p(r))\to \rg_{\proeet}(X,\Q_p(r)),
 \end{equation}
 which  is a strict quasi-isomorphism after truncation $\tau_{\leq r}$. 
 \item A Hyodo-Kato cohomology 
$\rg_{\hk}(X)$. This is  a dg $F^{\nr}$-algebra
equipped with a Frobenius $\phi$ and a monodromy operator $N$. We have  natural Hyodo-Kato strict quasi-isomorphisms
$$
  \iota_{\hk}: \rg_{\hk}(X)\wh{\otimes}^{\R}_{F^{\nr}}C\stackrel{\sim}{\to} \rg_{\dr}(X),\quad   \iota_{\hk}: \rg_{\hk}(X)\wh{\otimes}^{\R}_{F^{\nr}}\B^+_{\dr}\stackrel{\sim}{\to} \rg_{\dr}(X/\B^+_{\dr}).
 $$
 \item A  distinguished triangle
  \begin{equation}
  \label{triangle11}
 \rg_{\synt}(X,\Q_p(r))\lomapr{}  [\rg_{\hk}(X)\wh{\otimes}_{F^{\nr}}\B^+_{\st}]^{N=0,\phi=p^r}\lomapr{\iota_{\hk}} \rg_{\dr}(X/\B^+_{\dr})/F^r
 \end{equation}
  that can be lifted to the derived category of Vector Spaces. 
 \end{enumerate}
 \end{theorem}

 \subsubsection{Dagger varieties}
Set
$${\rm HK}_r^i(X):=H^i [\rg_{\hk}(X)\wh{\otimes}_{F^{\nr}}\B^+_{\st}]^{N=0,\phi=p^r},\quad
{\rm DR}_r^i(X):=H^i(\rg_{\dr}(X/\B^+_{\dr})/F^r).$$
 The distinguished triangle (\ref{triangle11}) yields a long exact sequence of cohomology groups
 \begin{equation}
 \label{rigid1}
\cdots\to {\rm DR}_r^{i-1}(X)\to  H^i_{\synt}(X,\Q_p(r))\lomapr{} {\rm HK}_r^i(X)
\lomapr{\iota_{\hk}} {\rm DR}_r^i(X)\to\cdots,
\end{equation}
which,  together with the period isomorphism 
$$
H^i_{\synt}(X,\Q_p(r))\stackrel{\sim}{\to} H^i_{\proeet}(X,\Q_p(r)),\quad i\leq r,
$$
obtained from (\ref{period-now}),   is a starting point for our work on generalizations of the 
$C_{\rm st}$-conjecture to rigid analytic varieties (see the sequel to this paper \cite{CN5}). This sequence is, however,  difficult to use since, locally, the rigid analytic de Rham cohomology and Hyodo-Kato cohomology are, in general,  very ugly:   infinite dimensional and not Hausdorff. But
 we are mainly interested in partially proper  rigid analytic varieties and  these varieties have a canonical overconvergent (or dagger) structure\footnote{Recall that a dagger variety is a rigid analytic variety equipped with an overconvergent structure sheaf. See \cite{GK0} for the basic definitions and properties.}. Moreover, a dagger affinoid has de Rham cohomology that is a finite rank vector space with its natural Hausdorff topology.
 
 Hence we are led to study dagger varieties.  We prove an analog of  Theorem \ref{main1} for smooth dagger varieties. The dagger version of (\ref{rigid1}) is the  long exact sequence:
 $$
 \cdots \to {\rm DR}_r^{i-1}(X)\to  H^i_{\synt}(X,\Q_p(r))\lomapr{} (H^i _{\hk}(X)\wh{\otimes}_{F^{\nr}}\B^+_{\st})^{N=0,\phi=p^r}\lomapr{\iota_{\hk}} {\rm DR}_r^i(X)\to \cdots
$$
But now, if $X$ is a dagger affinoid,  both cohomologies $H^i _{\hk}(X)$ and $ H^i_{\dr}(X/\B^+_{\dr})$ are (free) of finite rank. 
If $X$ is a  dagger variety the overconvergent  constructions are compatible with the rigid analytic constructions for $\wh{X}$, the completion of $X$. If $X$ is partially proper the two sets of constructions are 
  strictly quasi-isomorphic.

\subsubsection{Geometrization}
We show in~\cite{CN5} that the above long exact sequence (\ref{rigid1}), in a stable range,  splits into short exact sequences 
if $X$ is proper or, more generally, dagger quasi-compact or "small", or if $X$ is Stein. In order to do so, we need to put some extra-structure
on the terms of the exact sequence.  In~\cite{CN1}, we treated the proper case (with a semi-stable model)
by using the fact that the terms in the exact sequence outside of the $H^i_{\synt}(X,\Q_p(r))$'s
were naturally $C$-points of Banach-Colmez spaces (called BC's in what follows). That this is also the
case of the $H^i_{\synt}(X,\Q_p(r))$'s, for $i\leq r$,  follows from the comparison with pro-\'etale cohomology
and Scholze's theorem~\cite{Sch} which  states that these cohomology groups are in fact finite dimensional
over $\Q_p$ and independent of the field $C$: hence they are the $C$ points of quite trivial BC's.
Then the basic theory of BC's~\cite{CB,CF} could be used to show that the long exact sequence splits in a stable range.
(Actually, putting a BC structure on syntomic cohomology can be done directly~\cite{NFF}, but to
prove the splitting of (\ref{rigid1}), one still needs Scholze's finiteness theorem, if one is to stick
to the methods of~\cite{CN1}).

In our present situation, the $H^i_{\proeet}(X,\Q_p(r))$'s are very much not finite dimensional over $\Q_p$
and depend on the field $C$.  Hence they are not obviously $C$-points of anything sensible.
But one can turn them into $C$ points of sheaves on ${\rm Perf}_C$, and this is a category
of geometric objects (the category of Vector Spaces, VS's for short)
that contains naturally the category of BC's as was advocated in Le Bras' thesis~\cite{lebras}. 

One turns the $p$-adic pro-\'etale cohomology into a sheaf on ${\rm Perf}_C$ by taking
the sheaf associated to the presheaf $S\mapsto \rg_{\proeet}(X_S,\Q_p(r))$, for perfectoid algebras
$S$ over $C$. 
Likewise, one geometrizes syntomic cohomology by geometrizing the period rings; for example, $\B_{\crr}$ becomes the functor 
$S\mapsto \bcris(S)$.  We extend the proof of Theorem \ref{main1}
to this geometrized context to obtain:
\begin{theorem}\label{cosgeo}
The quasi-isomorphisms from Theorem~\ref{basic} and~(1) of Theorem~\ref{main1}
are the evaluations on ${\rm Spa}(C,\so_C)$
of quasi-isomorphisms of Vector Spaces.
\end{theorem}
This promotes the exact sequence
(\ref{rigid1}) to a sequence of VS's which can be analyzed using the geometric point of view
on BC's developed in~\cite{lebras} (this analysis is quite involved and is postponed to~\cite{CN5}).

 \subsection{Proof of Theorems~\ref{basic} and~\ref{main1}} 
 We will now sketch  how Theorem \ref{basic}  and Theorem~\ref{main1} are proved. 
 
 (i)  {\em Rigid-analytic varieties. }
Recall that \cite[Sec.\,2]{CN3}, using the rigid analytic \'etale local alterations of Hartl and Temkin \cite{Urs}, \cite{Tem},  one can equip the \'etale topology of $X$ with  a (Beilinson) base\footnote{This should be distinguished from a Verdier base; in a Beilinson base the condition on fullness of the base morphisms is dropped.
 See \cite[2.1]{CN3}.} consisting of semistable  formal schemes (always assumed to be of finite type) over  $\so_C$.
 This allows us to define sheaves by specifying them on such integral models and then sheafifying for the $\eta$-\'etale topology\footnote{Here $\eta$-\'etale means topology induced from the \'etale topology of the rigid analytic  generic fiber.}.
 For example, in (1)  the syntomic cohomology $\rg_{\synt}(X,\Q_p(r))$ of a rigid analytic variety $X$ is defined by  $\eta$-\'etale descent  from the crystalline syntomic cohomology of Fontaine-Messing.
 Recall that the latter is defined as   the fiber  ($\sx$ is a semistable formal scheme over $\so_C$ equipped with its canonical log-structure)
  $$\rg_{\synt}(\sx,\Q_p(r)):=[F^r\rg_{\crr}(\sx)\lomapr{\phi-p^r}\rg_{\crr}(\sx)], 
  $$ where the (logarithmic) crystalline cohomology is absolute (i.e., over $\Z_p$).
 By definition, it fits into the distinguished triangle
  \begin{equation}
  \label{triangle22}
  \rg_{\synt}(X,\Q_p(r))\to [\rg_{\crr}(X)]^{\phi=p^r}\to \rg_{\crr}(X)/F^r,
  \end{equation}
  which looks different than the triangle (\ref{triangle11}) that we want in (3).
 However, we easily find\footnote{The easiest way to see it is by interpreting, locally, both sides as derived de Rham cohomology.} that  $\rg_{\crr}(X)/F^r\simeq \rg_{\dr}(X/\B^+_{\dr})/F^r$.
 Here $\rg_{\dr}(X/\B^+_{\dr})$ is the $\B^+_{\dr}$-cohomology as defined by Bhatt-Morrow-Scholze in \cite{BMS1}, which we have redefined in the paper as $\eta$-\'etale descent  of Hodge-completed rational absolute crystalline cohomology of semistable schemes.
 But the construction of an isomorphism between the middle terms in (\ref{triangle22}) and  (\ref{triangle11}) requires a refined version of the Hyodo-Kato morphism. 

   The period map in (1),  is  defined   by  $\eta$-\'etale descent of  Fontaine-Messing period map $$\alpha_r: \rg_{\synt}(\sx,\Q_p(r))\to \rg_{\eet}(\sx_C,\Q_p(r)), $$
   for a semistable formal scheme $\sx$ over $\so_C$.
  The fact that it is a strict quasi-isomorphism in a stable range  follows from the computations of $p$-adic nearby cycles via syntomic complexes done by Tsuji  in \cite{Ts}.
 However, to lift it to the derived category of Vector Spaces we use its reinterpretation via $(\phi,\Gamma)$-modules by  Colmez-Nizio{\l} and Gilles in \cite{CN1}, \cite{SG}.
 This new interpretation of the period morphism is then 
      lifted  from $C$  to perfectoid spaces over $C$ to prove Theorem~\ref{cosgeo}.

  The construction  of the Hyodo-Kato morphism in (2) is quite involved; in fact, a detailed study of Hyodo-Kato cohomology and its relation to $\B^+_{\dr}$- and de Rham cohomologies occupies a large portion of this paper. 
  The original Hyodo-Kato morphism \cite{HK} works for semistable (formal) schemes. It can not be transferred to rigid analytic varieties because, a priori, it is dependent on the choice of the uniformizer of the base field (which varies for local semistable models). Moreover, a key map in the construction\footnote{For experts:  the section of the projection $T\mapsto 0$.} is defined as an element of the classsical derived category. A more careful data keeping allowed Beilinson \cite{Bv2} to make the Hyodo-Kato morphism  independent of choices in the case of proper schemes. We adapt here his technique to formal schemes and along the way lift the morphism to derived $\infty$-category. As a byproduct  we get the identification
  $$
  [\rg_{\crr}(X)]^{\phi=p^r}\simeq [\rg_{\hk}(X)\wh{\otimes}_{F^{\nr}}\B^+_{\st}]^{N=0,\phi=p^r}
  $$
  and an identification of (\ref{triangle22}) with (\ref{triangle11}), as wanted.

   (ii) {\em Dagger varieties.}
The pro-\'etale cohomology in (1) is defined in the most naive way: if $X$ is a smooth dagger affinoid with a \index{Xn@\Xn}presentation $\{X_h\}_{h\in\N}$ by a pro-affinoid rigid analytic variety, we set 
 $$\rg_{\proeet}(X,\Q_p(r)):=\colim_h\rg_{\proeet}(X_h,\Q_p(r));$$
  then, we globalize. From this description it is clear that we have a natural map $$\rg_{\proeet}(X,\Q_p(r))\to
 \rg_{\proeet}(\wh{X},\Q_p(r)),
 $$ where $\wh{X}$ is the completion of $X$ (a rigid analytic variety).   It is easy to see that in the case $X$ is partially proper, this morphism is a strict quasi-isomorphism (see \cite[Prop. 3.17]{CN3}).

  The other overconvergent cohomologies (Hyodo-Kato, de Rham, $\B^+_{\dr}$-, syntomic)  and morphisms between them can be  defined in an analogous way  without difficulties. In some cases though, they do however already have independent definitions: Hyodo-Kato and de Rham cohomologies were  defined by Grosse-Kl\"onne in \cite{GKFr} and we define syntomic cohomology as the fiber giving  the following  distinguished triangle
 \begin{equation}
 \label{triangle33}
   \rg_{\synt}(X,\Q_p(r))\lomapr{}  [\rg_{\hk}(X)\wh{\otimes}_{F^{\nr}}\B^+_{\st}]^{N=0,\phi=p^r}\lomapr{\iota_{\hk}} \rg_{\dr}(X/\B^+_{\dr})/F^r.
\end{equation}
  In these cases, we prove that the two sets of definitions yield  strictly quasi-isomorphic objects. 
  As an illustration of the power of the new definitions of overconvergent  cohomologies, let us look at the simple proof of the  following fact, whose arithmetic analog 
   was  the main technical result of \cite{CN3}:
   \begin{proposition}
   \label{first4}Let $r\geq 0$. Let $X$ be 
a smooth dagger variety over $K$. There is a natural morphism $$\rg_{\synt}(X,\Q_p(r))\to \rg_{\synt}(\wh{X},\Q_p(r)).$$ 
It is a strict quasi-isomorphism if $X$ is partially proper. 
   \end{proposition}This proposition  is proved  by representing, using distinguished triangles (\ref{triangle11}) and (\ref{triangle33}),  both sides of the morphism by means of  the rigid analytic  and the overconvergent Hyodo-Kato cohomology, respectively, then  passing  through the rigid analytic  and the overconvergent Hyodo-Kato quasi-isomorphisms (that are compatible by construction) to the de Rham cohomology, where the result is known. 

\begin{remark}
The approach we have taken here to deal with dagger varieties is very different from the one in \cite{CDN3} or \cite{CN3} (these two approaches also differing  between themselves). That is, we do not use Grosse-Kl\"onne's overconvergent Hyodo-Kato cohomology nor the related Hyodo-Kato morphism (which is difficult to work with and is also very different from the rigid analytic  version making checking the overconvergent-rigid analytic compatibility a bit of a nightmare). Instead, we induce all the overconvergent cohomologies from their rigid analytic analogs; hence, by definition, the two constructions are compatible. This was only possible because we have constructed a functorial, $\infty$-category version of the Hyodo-Kato morphism. 
\end{remark}
\begin{structure} Sections 2 and 4 are devoted to a definition of a functorial, $\infty$-categorical Hyodo-Kato quasi-isomorphism. In Section 3 we present our definition of $\B^+_{\dr}$-cohomology. Section 5 puts the above things together and introduces overconvergent geometric syntomic cohomology. In Section 6 we define comparison morphisms and in Section 7 we put a geometric structure on them. 
\end{structure}
 \begin{acknowledgments}W.N. would like to thank MSRI, Berkeley, and the Isaac Newton Institute, Cambridge, for hospitality during Spring 2019 and Spring 2020 semesters, respectively,  when parts of this paper were written. We would like to thank  
 Piotr Achinger, Guido Bosco, Gabriel Dospinescu,  Ofer Gabber, Sally Gilles, Veronika Ertl,   
Matthew Morrow, Michael Temkin, and Peter Scholze for helpful discussions related to the content of this paper and  Shane Kelly for patiently explaining to us $\infty$-categorical constructions described in Section \ref{sketchy}. 

 Special thanks go to the referee for a very careful reading of the manuscript and many suggestions that have improved the presentation of the material.
  \end{acknowledgments}
   \subsubsection*{Notation and conventions.}\label{Notation}
 Let $\so_K$ be a complete discrete valuation ring with fraction field
$K$  of characteristic 0 and with perfect
residue field $k$ of characteristic $p$. Let $\ovk$ be an algebraic closure of $K$ and let $\so_{\ovk}$ denote the integral closure of $\so_K$ in $\ovk$. Let $C=\wh{\ovk}$ be the $p$-adic completion of $\ovk$.  Let
$W(k)$ be the ring of Witt vectors of $k$ with 
 fraction field $F$ (i.e., $W(k)=\so_F$); let $e=e_K$ be the ramification index of $K$ over $F$.   Set $\sg_K=\Gal(\overline {K}/K)$ and 
let $\phi$ be the absolute
Frobenius on $W(\overline {k})$. 
We will denote by $\A_{\crr}, \B_{\crr}, \B_{\st},\B_{\dr}$ the crystalline, semistable, and  de Rham period rings of Fontaine \cite{bures}. 

 We will denote by $\so_K$,
\index{OK@\OK}$\so_K^{\times}$, and $\so_K^0$, depending on the context,  the scheme $\Spec ({\so_K})$ or the formal scheme $\Spf (\so_K)$ with the trivial, the canonical (i.e., associated to the closed point), and the induced by $\N\to \so_K, 1\mapsto 0$,
log-structure, respectively.  Unless otherwise stated all   formal schemes are $p$-adic, locally of finite type, and equidimensional. For a ($p$-adic formal) scheme $X$ over $\so_K$, let 
\index{Xn@\Xn}$X_0$ denote
the special fiber of $X$; let $X_n$ denote its reduction modulo $p^n$. For an $\so_K$-module $M$, we set $M_n:=M\otimes^{\rm L}_{\so_K}\so_{K}/p^n$. 

All rigid analytic spaces considered will be over $K$ or $C$. 
We assume that they are separated, taut, and countable at infinity. 
If $L=K,C$, we \index{Smo@\Smo}let ${\rm Sm}_L$ (resp.~${\rm Sm}_L^\dagger$)
be the category of smooth rigid analytic (resp.~dagger) varieties over $L$, and 
we denote \index{Perf@\Perf}by ${\rm Perf}_C$ the category of perfectoid spaces over $C$.

    Unless otherwise stated, we work in the derived (stable) $\infty$-category 
\index{Dcal@\Dcal}$\sd(A)$ of left-bounded complexes of a quasi-abelian category $A$ (the latter will be clear from the context).
  Many of our constructions will involve (pre)sheaves of objects from $\sd(A)$. 
  We will use a shorthand for certain homotopy limits:
 if $f:C\to C'$ is a map  in the derived $\infty$-category of a quasi-abelian  category, we set
$$[\xymatrix{C\ar[r]^f&C'}]:=\lim(C\to C^{\prime}\leftarrow 0).$$ 
We also set
$$
\left[\begin{aligned}
\xymatrix{C_1\ar[d]\ar[r]^f & C_2\ar[d]\\
C_3\ar[r]^g & C_4
}\end{aligned}\right]
:=[[C_1\stackrel{f}{\to} C_2]\to [C_3\stackrel{g}{\to} C_4]],
$$ 
 where the diagram in the brackets is a commutative diagram in the same $\infty$-category.
For an operator $F$ acting on $C$, we will use the brackets $[C ]^F$ to denote the derived eigenspaces and,  if $C$ is a concrete complex and $F$ an operator acting on $C$, the brackets $(C)^F$ or simply $C^F$,   to denote the non-derived ones.

  Our cohomology groups will be equipped with a canonical topology. To talk about it in a systematic way,  we will work  rationally in the category of locally convex $K$-vector spaces and integrally in the category of pro-discrete $\so_K$-modules. For details  the reader may consult \cite[Sec.\,2.1, 2.2]{CDN3}. To summarize quickly:
   \begin{enumerate}
   \item \index{CK@\CK}$C_K$ is the category of convex $K$-vector spaces; 
  it  is a quasi-abelian category.
   We will denote the left-bounded derived $\infty$-category of $C_K$ by 
\index{Dcal@\Dcal}$\sd(C_K)$.
  A morphism of complexes that is a quasi-isomorphism in $\sd(C_K)$, i.e., its cone is strictly exact,  will be called a {\em strict quasi-isomorphism}.
 The associated cohomology objects are denoted \index{LHK@\LHK}by\footnote{$LH$ stands for ``left heart''.} 
\index{HHH@\HHH}$\wt{H}^n(E)\in {LH}(C_K)$; they are called {\em classical} if the natural map $\wt{H}^n(E)\to {H}^n(E)$ is an isomorphism\footnote{In our situations this is usually equivalent to $H^n(E)$ being separated.}.
\item We will often work in a slightly more general setting. Let $A_K:={\rm LH}(C_K)$. It is an abelian category and we have $\sd(C_K)\stackrel{\sim}{\to} \sd(A_K)$. Let $B\in C_K$ be a topological algebra over $K$. We will denote by $A_B$ the  abelian subcategory of $A_K$ of $B$-modules.  We set $\sd(C_B):= \sd(A_B)$. 
\item For the  default tensor product (over $K$) in $C_K$ we have chosen the projective tensor product (which commutes with projective limits). It is left exact. 
\item Objects in the category \index{PCK@\PCK}$PD_K$ of pro-discrete $\so_K$-modules are topological $\so_K$-modules that are countable inverse limits, as topological $\so_K$-modules, of discrete $\so_K$-modules $M^i$, $i\in \N$. 
It is a quasi-abelian category.
  Inside  $PD_K$ we distinguish the category \index{PCK@\PCK}$PC_K$ of pseudocompact $\so_K$-modules, i.e., pro-discrete modules $M\simeq\wlim_iM_i$ such that each $M_i$ is of finite length (we note that if $K$ is a finite extension of $\Q_p$ this is equivalent to $M$ being profinite).
 It is an abelian category.
\item   There is a tensor product functor from the category of pro-discrete $\so_K$-modules to convex $K$-vector spaces:
  $$(-){\otimes}K: PD_K\to C_{K}, \quad M\mapsto M\otimes _{\so_K}K.
  $$
 Since $C_K$ admits filtered inductive limits, the functor $(-){\otimes}K$ extends 
to a functor $(-){\otimes}K: {\rm Ind}(PD_K)\to C_{K}$.   The  functor $(-){\otimes}K$ is right exact but not, in general, left exact.
  We will consider its (compatible) left derived functors
  $$
  (-){\otimes}^{\LL}K: \sd^{-}(PD_K)\to {\rm Pro}(\sd^{-}(C_K)),\quad  (-){\otimes}^{\LL}K: \sd^{-}({\rm Ind}(PD_K))\to {\rm Pro}(\sd^{-}(C_K)).
  $$
 If $E$ is a complex of torsion free and $p$-complete (i.e., $E\simeq \wlim_n E/p^n$) modules from $PD_K$ then the natural map
  $$
   E{\otimes}^{\LL}K\to  E{\otimes}K
  $$ is a strict quasi-isomorphism \cite[Prop. 2.6]{CDN3}. 
 \end{enumerate}

    Finally, we will use freely the notation and results from \cite{CN3}.

\section{Hyodo-Kato rigidity revisited} 
The original Hyodo-Kato morphism \cite{HK} works for semistable (formal) schemes.
 It can not be transferred to rigid analytic varieties because, a priori, it is dependent on the choice of the uniformizer of the base field (which varies for local semistable models).
 A more careful data keeping allowed Beilinson \cite{Bv2} to make it independent of choices in the case of proper schemes.
 We adapt here his technique to semistable formal schemes and add some extra functoriality by
 lifting the morphism to the derived $\infty$-category.  

  This gives us local
Hyodo-Kato morphisms for rigid analytic varieties; the extra functoriality will be crucial
for the globalization of these maps for
rigid analytic and dagger varieties discussed in Chapter \ref{hihi1} (it makes it possible to glue
local maps from an hypercover by semistable formal schemes).

\subsection{Preliminaries} We gather in this section basic properties of period rings, isogenies, and $\phi$-modules that we will need in the paper. 
\subsubsection{Period rings} \label{period-rings}  
We will review first  the definitions of  the  rings of periods that we will need. We follow here Beilinson \cite[1.14, 1.19]{Bv2}, where the reader can find more details. Beilinson's definitions are a slight modification of the classical ones; they stress the dependence on choices in a better way. 

 (i) {\em Arithmetic setting.} Let \index{SSS@\SSS}$S=S_K:=\Spf \so_K^{\times}$, $S^0:=\Spf \so_F^0$.  We denote the corresponding log-structures by \index{Lcal@\Lcal}$\sll=\sll_K$ and $\sll^0_K$, respectively. Note that the second log-structure can be conveniently described by the pre-log structure
$\so_K\setminus \{0\}\to \so_F, a\mapsto [\overline{a}]$, where $\overline{a}:=a \mod {\mathfrak m}_K$ and $[-]$ denotes the Teichm\"uller lift.

Consider the algebra $\so_F[T]$
with the log-structure associated to $T$.
  We denote by $r_F^{\rm PD}$ the associated $p$-adic divided powers polynomial algebra.
 In a more natural in $K$ way, we can write
\index{rpd@\rpd}$r_F^{\rm PD}$ as $r_K^{\rm PD,0}$ -- the $p$-adic completion of $\so_F<t_a>$, the divided powers polynomial algebra generated by  $t_a, a\in ({\mathfrak m}_K/{\mathfrak m}^2_K)\setminus \{0\}$, with $t_{a^{\prime}}=[a^{\prime}/a]t_a$.
   We denote by  \index{rpd@\rpd}$r_K^{\rm PD}$ the $p$-adic completion of the subalgebra of the PD algebra $\so_F<t_a>$ generated by $t_a$ and $t^{ne}_a/n!, n\geq 1$.
 The log-structure is induced by the $t_a$'s, Frobenius action by $t_a\mapsto t_a^p$, and monodromy by the derivation sending $t_a\mapsto t_a$.
 Set $E=E_K:= \Spf r^{\rm PD}_K$, $E^0=E^0_K:= \Spf r^{\rm PD,0}_K$.
 We have  canonical exact embeddings $i_0: S^0\hookrightarrow E, i^*_0(t_a)=[a]\in \sll^0_K$, 
$i^0_0: S^0\hookrightarrow E^0, i^{0,*}_0(t_a)=[a]\in \sll^0_K$. 

We have an exact closed embedding $S^0_1\hookrightarrow S_1$.  
Retractions $\pi_l$
are given by maps $\pi^*_l:\sll^0_1\to\sll_1, a\mapsto l_a,$ with $l_{a^{\prime}}=[a^{\prime}/a]l_a$.
Every  retraction $\pi_l: S_1\to S^0_1$ yields a $k^0$-structure on $S_1$,  hence an 
exact closed embedding $i_l:S_1\hookrightarrow E_1$, $i^*_l(t_a)=l_a$.  

    (ii) {\em Geometric setting.} Let \index{SSS@\SSS}$\overline{S}:=\Spf\so^{\times}_{C}$. We denote its log-structure by \index{Lcal@\Lcal}$\overline{\sll}$.  We normalize the valuation on $C$ by $v(p)=1$. 
 Let \index{Lcal@\Lcal}$\overline{\sll}^0$ be the log-structure on $\overline{S}^0:=\Spf W(\overline{k})$ generated by the pre-log structure $\so_{C}\setminus \{0\}\to W(\overline{k}), a\mapsto [\overline{a}]$, 
$\overline{a}:= a \mod {\mathfrak m}_{\ovk}$.
 Then $\overline{\sll}^0$ has  a natural Frobenius action compatible with the Frobenius: $\phi([a])=[a^p]$.
 There is an exact embedding $\overline{S}^0_1\hookrightarrow \overline{S}_1$.

    We will denote by  $\A^{\times}_{\crr}$ the period ring  $\A_{\crr}$ equipped with the unique log-structure \index{Lcal@\Lcal}$\sll_{\crr}$ extending the one on $\so^{\times}_{C,1}$. Let $J_{\crr}$ be the PD-ideal, $\A_{\crr}/J_{\crr}\simeq \so_{C,1}$. 
    Set $\se_{\crr}:=\Spf \A_{\crr}^{\times}$.
 The exact embedding  $\Spec \so_{C,1}^{\times}\hookrightarrow \se_{\crr,n}$ given by    the Fontaine map $\theta: \A_{\crr}\to \so_C$ is
    a PD-thickening in the crystalline site of
   $\so_{{F},1}$.
    
    Recall  the definition of the period ring \index{BST@\BST}$\B^+_{\st}$.
 Let $\log: \A_{\crr}^{*}/\overline{k}^{*}\to \B^+_{\crr}$ be the logarithm: the unique homomorphism which extends the logarithm on $(1+J_{\crr})^{*}$, 
where $J_{\crr}=(p,{\rm Ker}\,\theta)$.
  Then $\B^+_{\st}$ is defined as the universal $\B^+_{\crr}$-algebra equipped with a homomorphism of monoids $\log: \sll_{\crr}/\overline{k}^{*}\to\B^+_{\st}$ extending  the above $\log$ on $\A^{*}_{\crr}$.
  Since $v:\sll_{\crr}/\A^{*}_{\crr}\stackrel{\sim}{\to}\Q_{\geq 0}$, it is clear  that, for any $\lambda\in\sll_{\crr}/\overline{k}^{*}$ with $v(\lambda)\neq 0$, the element $\log(\lambda)$ freely generates $\B^+_{\st}$ over $\B^+_{\crr}$, i.e., $\B^+_{\crr}[\log(\lambda)]\stackrel{\sim}{\to}\B^+_{\st}$.
 The Frobenius action extends to $\B^+_{\st}$ via universality. 
   The monodromy $N$ is the $\B^+_{\crr}$-derivation on $\B^+_{\st}$ such that $N(\log(\lambda))=-v(\lambda)$.
 We have  $N\phi=p\phi N$.
 Moreover, any $\lambda$ as above yields a retraction 
\index{slambda@\slambda}$s^*_{\lambda}:\B^+_{\st}\to \B^+_{\crr}, s^*_{\lambda}(\log(\lambda))=0$.
 If \index{Lcal@\Lcal}$\lambda\in\sll_{\phi}:=\{\lambda\in\sll_{\crr}:\phi(\lambda)=\lambda^p\}$ then $s^*_{\lambda}$ is compatible with Frobenius action. 
   
Now, recall  the definition of  the   period ring  \index{BST@\BST}$\wh{\B}^+_{l,\st}$.
Let $r\in\Q_{>0}$.  Denote by ${\mathbb A}^{(r)}_{W(\overline{k})}$, the log affine space, i.e.,  the formal scheme $\Spf W(\overline{k})\{t_a\}, a\in\tau_r$, with $t_{a^{\prime}}=[a^{\prime}/a]t_a$.
 Here $\tau_{r}:=\{a\in \overline{\sll}_1^0: v(a)=r\}$.
 The log-structure is generated by the $t_a$'s.
 The map $i_r: \overline{S}^0_1\to {\mathbb A}^{(r)}_{W(\overline{k})}$, $i_r^*(t_a)=a,$ can be extended  to a 
\index{ilst@\ilst}map $i_l:\overline{S}_1\to {\mathbb A}^{(r)}_{W(\overline{k})}$ by choosing  $l_a:=i_l^*(t_a)\in\overline{\sll}_1$ that lifts $a$.
  
  We have the commutative diagram 
  $$\xymatrix{& {\mathbb A}^{(r)}_{W(\overline{k})}\times_{W(\overline{k})}\se_{\crr,n}\ar[d]\\
  \overline{S}_1\ar@{^(->}[r]\ar@{^(->}[ru]^{(i_l,\theta)}   & \se_{\crr,n}
  }
  $$
 Let $i_{l,\st}: \overline{S}_1\hookrightarrow \se_{l,\st,n}$ be the  PD-envelope of $(i_l,\theta)$ over $\se_{\crr,n}$. We write $\se_{l,\st,n}=\Spec \wh{\A}_{l,\st,n}$ and set 
 $$\wh{\A}_{l,\st}:=\wlim_n\wh{\A}_{l,\st,n},\quad   \wh{\B}^+_{l,\st}:=\wh{\A}_{l,\st}[\tfrac{1}{p}], \quad \se_{l,st}:=\Spf \wh{\A}_{l,\st}.
 $$ We note that  $\wh{\B}^+_{l,\st}$ is a Banach space over $F$ (which makes it easier to handle topologically than $\B^+_{\st}$). Frobenius action is given by $t_a\mapsto t_a^p$ and the  monodromy
operator by $N_l:=t_a\partial_{t_a}$.  We have the exact sequence 
\begin{equation}
\label{Breuil}
0\to \A_{\crr,n}\to \wh{\A}_{l,\st,n}\stackrel{N_l}{\to} \wh{\A}_{l,\st,n}\to  0.
\end{equation}
Every lifting of $l$ to $\lambda\in\sll_{\phi}$ yields a map $s^*_{\lambda}: \wh{\A}_{l,\st,n}\to \A_{\crr,n}, s^*_{\lambda}(t_a):=\lambda_a$, which is compatible with Frobenius action, and an identification $\wh{\A}_{l,\st,n}\stackrel{\sim}{\to}\A_{\crr,n}<t_a\lambda^{-1}_a-1>$.

    Let $\wh{\A}_{l,\st}^{N_l{\text{-nilp}}}$ 
\index{Nnilp@\Nnilp}be the $\A_{\crr}$-subalgebra of $\wh{\A}_{l,\st}$ 
formed by the elements killed by a power of $t_a\partial_{t_a}$. It is the  divided powers polynomial algebra $\A_{\crr}<\log(t_a\lambda_a^{-1})>$. There is a   $\B^+_{\crr}$-linear isomorphism 
$$
\kappa_l: \B^+_{\st}\stackrel{\sim}{\to}\wh{\A}_{l,\st,\Q_p}^{N_l{\text{-nilp}}}
$$
which sends a generator $\log(\lambda)$ of $\B^+_{\st}$, where $\lambda$ lifts $l$, to $-\log(t_a\lambda^{-1}_a)\in \wh{\A}_{l,\st}^{N_l{\text{-nilp}}}$. 
It is compatible  with the action of $\sg_K$, Frobenius, and it identifies $N$ on $ \B^+_{\st}$ with the action of $rt_a\partial_{t_a}$.

    Finally we have maps to $\B^+_{\dr}$. We will normalize them for the rest of the paper at $p$. 
That is, we fix a lift $[\tilde{p}]\in \sll_{\phi}$ of $p$ and define the 
\index{iotap@\iotap}maps: 
 $$
 \iota=\iota_p: \wh{\B}^+_{p,\st}\to \B^+_{\dr},\quad \iota=\iota_p:=\iota_p\kappa_p: \B^+_{\st}\to\B^+_{\dr}.
 $$
 The first map is obtained by sending $t_p$ to $p$;  the second map, by sending $\log([\tilde{p}])$ to $-\log(p/[\tilde{p}])$.  Otherwise saying, we can \index{BST@\BST}set 
 \begin{align*}
&  \wh{\B}^+_{\st}:=\wh{\B}^+_{p,\st}:=\A_{\crr}<t_p[\wt{p}]^{-1}-1>^{\bwedge}[\tfrac{1}{p}],
\quad \B^+_{\st}:=\B^+_{p,\st}:=\B^+_{\crr}[\log([\wt{p}])],\\
& \kappa: \B^+_{\st}\to\wh{\B}^+_{\st}, \,\log([\wt{p}])\mapsto -\log(t_p[\wt{p}]^{-1}),\quad  \iota: \B^+_{\st}\to\B^+_{\dr},\, \log([\wt{p}])\mapsto -\log(p[\wt{p}]^{-1}),\\
&  \iota: \wh{\B}^+_{\st}\to\B^+_{\dr},\, t_p[\wt{p}]^{-1}\mapsto p[\wt{p}]^{-1}.
\end{align*}

\subsubsection{Tensoring with period rings}

    (1)  Let $M$ be a bounded complex of Banach spaces, which are topological $\B^+_{\crr}$-modules. We define the topological tensor product $M\otimes_{\B^+_{\crr}}C$ as the algebraic tensor product equipped with the quotient topology induced from $M$ via the map $\theta$.
     This product tends to be  compatible with strict quasi-isomorphisms: 
     \begin{lemma}\label{trick1}
     Let $M, M^{\prime}$ be bounded complexes of Banach spaces, which are flat $\B^+_{\crr}$-modules. Let $\alpha: M\to M^{\prime}$ be a  strict quasi-isomorphism. Then the induced morphism
     $$
     \alpha\otimes\id: M\otimes_{\B^+_{\crr}}C\to M^{\prime}\otimes_{\B^+_{\crr}}C
     $$
     is a strict quasi-isomorphism as well.  
          \end{lemma}
     \begin{proof}
     Let $C(\alpha)$ denote the mapping fiber of $\alpha$. It is a bounded complex of Banach spaces. We claim that the complex
     $$
     C(\alpha\otimes\id)=[M\otimes_{\B^+_{\crr}}C\lomapr{ \alpha\otimes\id} M^{\prime}\otimes_{\B^+_{\crr}}C]\simeq C(\alpha)\otimes_{\B^+_{\crr}}C
     $$
     is strictly acyclic. Indeed, since $M, M^{\prime}$ are bounded and built from flat $\B^+_{\crr}$-modules, 
this is  so  algebraically.  
%
Now the terms of $C(\alpha\otimes\id)$ are Banach spaces as quotients of Banach spaces by closed subspaces 
and the Open Mapping Theorem implies that a complex of Banach spaces is strictly acyclic if and only if
it is acyclic (apply the OMT to the isomorphism
${\rm Im}(d_i)\to{\rm Ker}(d_{i+1})$ which are both Banach spaces since $d_i$
and $d_{i+1}$ are continuous).
     \end{proof}
     
      (2) Similarly, for a bounded complex $M$ of Banach spaces, which are topological $\B^+_{\crr}$-modules, we define the topological tensor product $M\otimes_{\B^+_{\crr}}(\B^+_{\crr}/F^i)$, $i\geq 0$,  as the algebraic tensor product equipped with the quotient topology induced from $M$. We have analog of the Lemma \ref{trick1} in this setting. 
      
     We will denote this tensor product by 
     $$
     M\wh{\otimes}^{\LL}_{\B^+_{\crr}}(\B^+_{\crr}/F^i),\quad i\geq 0. 
     $$
     
      (3) For a bounded complex $M$ of Banach spaces, which are topological $\B^+_{\crr}$-modules, 
      we define 
     $$
      M\wh{\otimes}^{\R}_{\B^+_{\crr}}\B^+_{\dr}:=\R\wlim_i(M\wh{\otimes}^{\LL}_{\B^+_{\crr}}(\B^+_{\crr}/F^i)).
     $$
 We have analog of the Lemma \ref{trick1} in this setting as well.

\subsubsection{Isogenies}
We recall now  some terminology from \cite[Sec.\,1.1]{Beco} (see also \cite[Sec.\,2.3]{AMMN}).

Let $\scc$ be  an additive category (or $\infty$-category). A map $f: P\to Q$ is  an {\em isogeny} if there exists $g: Q\to P$ and an integer $N >0$ such that $gf=N\id_P$ and $fg=N\id_Q$ (in the homotopy category.) 
 An object $X\in\scc$ is {\em bounded torsion}  if it is killed by some $N$, i.e., if  $N\id_X=0$ (also in the homotopy category). If  $\scc$ is  an additive category, we denote by $\scc\otimes\Q$ the category 
 with the same objects as $\scc$, with a functor $\scc\to \scc\otimes\Q, X \mapsto X_{\Q}$, and 
 with $\Hom(X_{\Q},Y_{\Q})=\Hom(X,Y)\otimes\Q$.  
Then $\scc\otimes\Q$ is the localization of $\scc$ with respect to isogenies; 
for $X\in\scc$, we have $X_{\Q}=0$, i.e., $X$ is isogenous to $0$, 
if and only if $X$ is a bounded torsion object. If $\scc$ is abelian then $\scc\otimes \Q$ is abelian as well and it equal to the quotient $\scc_{\Q}$ of $\scc$ modulo the Serre subcategory of bounded torsion objects. 

   Let $\scc$ be a stable $\infty$-category equipped with a $t$-structure.  If a  map  is an isogeny then it induces isogenies on all cohomology groups $H^n$, $n\in\Z$, in the heart $\scc^{\heartsuit}$. For maps between bounded object  the opposite is true as well: the map $f: P\to Q$ of  bounded objects is an isogeny, if, for each $n$, the map $H^nP\to H^n Q$  is an isogeny. 
In particular,  $X\in\scc$ is isogeneous to $0$ if each $H^n(X)$ is a bounded torsion group.  
\begin{remark}\label{westin}Consider the tensor product functor in the top row of the diagram:
$$
\xymatrix{
\sd({\rm PD}_K)\ar[r]^{(-)\otimes^{\LL}_K}  \ar[d]_{\can}& \sd(C_K)\\
 \sd({\rm PD}_K)_{\Q}\ar@{-->}[ur]_{(-)_{K}}
}
$$
It factors naturally through the isogeny category; we will denote so obtained functor from $ \sd({\rm PD}_K)_{\Q}$ to $\sd(C_K)$ by $(-)_{K}$. 
\end{remark}
\subsubsection{$\phi$-modules}\label{Fr-modules}
A Frobenius on  an $\so_F$-module is a $\phi_F$-linear endomorphism. Let $R$ be an $\so_F$-algebra equipped with a Frobenius $\phi_R$. For an $R$-module $M$,  a Frobenius on $M$ compatible with the $R$-module structure is an $R$-linear map $\phi_M: \phi_R^*M\to M$. Pairs $(M,\phi_M)$ form an abelian tensor $\Z_p$-category $R_{\phi}$-${\rm Mod}$. Let $\sd_{\phi}(R)$ be its bounded derived $\infty$-category. 
\begin{remark}\label{E1}By \cite[Rem. 7.1.1.16]{Lur1}, we may (and often will) identify the $\infty$-category $\sd_{\phi}(R)$ with the $\infty$-category of left modules from $\sd(R)$ over the $E_1$-ring $R\{\phi\}$ defined as the abelian group 
$R[\phi]$  with multiplication rules  $\phi a=\phi(a)\phi,$ for $ a\in R$. We will do the same for other categories of left modules over associative rings that appear in this paper. This sometimes may entail unbounding the derived category to the left but this will not cause problems. 
\end{remark}
  Consider the bounded derived $\infty$-categories $\sd_{\phi}(R), \sd_{\phi}(R)_{\Q}$ of bounded  complexes of $\phi$-modules over $R$. Then $\sd_{\phi}(R)_{\Q}$ is the quotient of $\sd_{\phi}(R)$ modulo the full subcategory of complexes with bounded torsion cohomology.

We need to discuss projective resolutions.  For an $R$-module $M$, set $M_{\phi}:=\oplus_{n\geq 0}\phi^{n*}_RM$ and equip it with  the evident Frobenius.  The functor $R$-${\rm Mod}\to R_{\phi}$-${\rm Mod}$, $M\mapsto M_{\phi}$, is left adjoint to the forgetful functor $R_{\phi}$-${\rm Mod}\to R$-${\rm Mod}, $ $(M,\phi_M)\mapsto M$. If follows that, for a projective $R$-module $M$, the $\phi$-module $M_{\phi}$ is a projective object of 
$R_{\phi}$-${\rm Mod}$. 

 For every $M=(M,\phi_M)\in R_{\phi}$-${\rm Mod}$, there is a natural short exact sequence
 $$
 0\to (\phi^*_RM)_{\phi}\lomapr{\delta} M_{\phi}\lomapr{\epsilon} M \to 0
 $$
 in $R_{\phi}$-${\rm Mod}$. The maps  $\epsilon$ and $\delta$ are  induced, respectively,  by adjunction from $\id_M$ and the map $\phi^*_RM\to M_{\phi}$ that sends $r\otimes m$ to $\phi_M(r\otimes m)-r\otimes m\in M\oplus \phi^*_RM\subset M_{\phi}$. Set $\wt{M}:=\Cone(\delta)$, so we have a resolution $\epsilon: \wt{M}\to M$. If $M$ is a projective $R$-module, this is a projective resolution in $R_{\phi}$-${\rm Mod}$. 
  
  We will need a version of the above constructions for derived $p$-complete  modules. Recall \cite[4.1]{BMS2} that, for a ring $R$ over $\Z_p$, $M\in\sd(R)$ is called $p$-completely flat if $M\otimes^{\LL}_RR/p\in\sd(R/p)$ is concentrated in degree $0$, where it is a flat $R/p$-module. If $R$ has 
 bounded $p^{\infty}$-torsion and an $R$-module $M$ is derived $p$-complete and $p$-completely flat then $M$  is a classically $p$-complete
$R$-module concentrated in degree $0$, with bounded $p^{\infty}$-torsion, such that $M/p^nM$ is flat over
$R/p^nR$, for all $n\geq 1$ (see \cite[Lemma 4.7]{BMS2}). And, conversely, if $M$  is a classically $p$-complete
$R$-module concentrated in degree $0$, with bounded $p^{\infty}$-torsion, such that $M/p^nM$ is flat over
$R/p^nR$, for all $n\geq 1$, then $M$ is $p$-completely flat.

 Now we specialize to $R=W(k)$. Then, a $W(k)$-module $M$ is derived $p$-complete and $p$-completely flat if and only if  $M$  is a classically $p$-complete
$W(k)$-module concentrated in degree $0$,  such that $M/p^nM$ is flat over $W_n(k)$, for all $n\geq 1$ (equivalently $M$ is $p$-torsion-free and   classically $p$-complete over 
$W(k)$).   We note that such a module $M$ can always be written as the $p$-adic completion of a free $W(k)$-module $M_0$ hence it is a projective object in the category of derived $p$-complete modules\footnote{We use here that, if a $W(k)$-module $N$ is derived $p$-complete then $\Hom_{W(k)}(M,N)\stackrel{\sim}{\to} \Hom_{W(k)}(M_0,N)$.}. 

   It follows that the above algebraic construction goes through once we derived  $p$-adically complete the objects. That is, if we denote by $\wh{\sd}(W(k))$ the full $\infty$-subcategory of $\sd(W(k))$ spanned by complexes built from derived $p$-complete modules\footnote{By \cite[Tag 09IU]{Sta}, this is the same as the full subcategory of $\sd(W(k))$ od derived $p$-complete complexes.} and, similarly,  if we denote by $\wh{\sd}_{\phi}(W(k))$ the full $\infty$-subcategory of $\sd_{\phi}(W(k))$ spanned by complexes built from derived $p$-complete $\phi$-modules
then, for a $\phi$-modules $M=M[0]\in \wh{\sd}_{\phi}(W(k))$, which is $p$-completely flat, the $p$-adic completion $\widehat{M}_{\phi}\in \wh{\sd}_{\phi}(W(k))$ of $M_{\phi}$ is a projective derived $p$-complete $\phi$-module and we have a projective resolution of derived $p$-complete $\phi$-modules
\begin{equation}\label{referee2}
 0\to \wh{(\phi^*_R{M})_{\phi}}\lomapr{\delta} \wh{M_{\phi}}\lomapr{\epsilon} M \to 0.
\end{equation}
We set $\wt{M}:={\rm Cone}(\delta)$. We note that, for a complex $M^{\jcdot}$ of $p$-completely flat, derived $p$-complete $\phi$-modules over $W(k)$,  the functor $M^{\jcdot}\mapsto \wt{M}^{\jcdot}$ preserves quasi-isomorphisms (since so does the functor $M^{\jcdot}\mapsto \wh{M_{\phi}}^{\jcdot}$ because 
$\wh{M_{\phi}}^{\jcdot}$ is equal to $\wh{M_{\phi}^{\jcdot}}$).

  We will often use the fact that the canonical functor $\sd^c_{\phi}(W(k))\to \wh{\sd}_{\phi}(W(k))$, where $\sd^c_{\phi}(W(k))$ is the full subcategory of $\wh{\sd}_{\phi}(W(k))$ spanned by complexes built from $p$-torsion free modules is an equivalence. Indeed, 
if  $M$ is a derived $p$-complete $W(k)$-module then we can find an exact sequence (see \cite[Tag 09AT]{Sta})
$$
0\to M_1\to M_0\to M\to 0,
$$
where $M_0, M_1$ are derived $p$-complete and torsion free. Moreover, if $M$ is a $\phi$-module, we can lift its Frobenius to this exact sequence. This and the existence of the projective resolutions \eqref{referee2} yield the desired equivalence of $\infty$-categories. 

\vskip.2cm
  Let $R:=r_K^{\rm PD}$.
Let  $P=(P,\phi_P)\in \wh{\sd}_{\phi}(R),$ $Q=(Q,\phi_Q)\in \sd^c_{\phi}(W(k))$. We assume that the complex $P$ is built from $p$-completely flat modules. Denote by $\overline{P}$ the  cofiber of  $
P\to P\wh{\otimes}^{\LL}_RW(k)$ viewed as an object of  $\sd^c_{\phi}(W(k))$. 
 \begin{lemma}\label{Newton1}If the Frobenius on $Q_\Q$ is invertible\footnote{This means that Frobenius map $\phi_{Q_{\Q}}:\phi^*_{W(k)}Q_{\Q}\to Q_{\Q}$ is a quasi-isomorphism in $\sd(W(k))_{\Q}$.} 
 then $$\R\Hom(Q,\overline{P})_\Q=\R\Hom(Q_\Q,\overline{P}_\Q)=0,$$
 where the  $\R\Hom$ is taken in $\sd^c_{\phi}(W(k))_\Q$.
 \end{lemma}
 \begin{proof}We claim that we have the short exact sequence  of complexes of  $\phi$-modules over $W(k)$ ($p$-torsion-free and $p$-complete)
 \begin{equation}
 \label{seq11}
 0\to IP\to P \to P\wh{\otimes}^{\LL}_RW(k)\to 0,
 \end{equation}
 where $I\subset R$ is the kernel of the projection $R\to W(k)$ and we set
  $IP:=I\wh{\otimes}^{\rm L}_RP\stackrel{\sim}{\to} I\wh{\otimes}_RP$. Indeed, because $P_n$ is a complex of flat  $R_n$-modules we have a compatible family of exact sequences
 $$
 0\to P_n\otimes_{R}I_n\to P\otimes_RR_n\to P_n\otimes_RW_n(k)\to 0
 $$
 Passing to the  limit we get the short exact sequence
 $$
0\to  \wlim_n(P_n\otimes_{R}I_n)\to \wlim_n(P\otimes_RR_n)\to \wlim_n(P_n\otimes_RW_n(k))\to 0
 $$
 Since $P$ is derived $p$-complete and its terms are $p$-completely flat modules, the natural morphism $P\to  \wlim_n(P\otimes_RR_n)$ is a quasi-isomorphism (in fact, an isomorphism) \cite[Tag 091Z]{Sta} and the above exact sequence  yields the exact sequence (\ref{seq11}). 
 
      From the exact sequence (\ref{seq11}) we get a distinguished triangle
 $$
 \R\Hom(Q,IP)\to \R\Hom(Q,P)\to \R\Hom(Q,P\wh{\otimes}^{\LL}_RW(k)),
 $$
 where
 $\R\Hom$ is computed in $\sd^c_{\phi}(W(k))$.   To prove the lemma, it suffices to show that 
  $\R\Hom(Q,IP)_\Q=0$.

  Assume, for a moment,   that $Q$ is  concentrated in degree $0$. 
    For  any  derived $p$-complete $\phi$-module $M$ over $W(k)$, we can  compute $\R\Hom(Q,M)$ using the 
 projective resolution $\wt{Q}$ of $Q$ from \eqref{referee2}.
     We get a two-term complex $C(Q,M)$ with 
    \begin{align*}C^0(Q,M) & =\Hom(\wh{Q_{\phi}},M)\simeq \Hom_{W(k)_{\phi}}({Q_{\phi}},M)\simeq\Hom_{W(k)}(Q,M),\\
     C^1(Q,M) & =\Hom(\wh{\phi^*(Q)_{\phi}},M)\simeq \Hom_{W(k)_{\phi}}({\phi^*(Q)_{\phi}},M)\simeq \Hom_{W(k)}(\phi^*(Q),M),
     \end{align*}
      and the differential $d=d_1-d_2: C^0(Q,M)\to C^1(Q,M)$, where $d_1(A)=A\phi_Q$, $d_2(A)=\phi_M\phi^*(A)$.
 Let $C_{*}(Q,M)$ be the complex with the same terms as $C(Q,M)$ but the differential simply $d_1$.
 Since we assumed that the Frobenius action on $Q_{\Q}$ is invertible, the complex $C_*(Q,M)_\Q$ is acyclic.

    We go back now to general $Q$. To show that $\R\Hom(Q,IP)_\Q=0$ we     we may  assume  that $P$ is  concentrated in degree $0$.  Indeed, since $P$ is bounded, we can detach one  term of the complex after another using the fact that $IP=I\wh{\otimes}^{\rm L}_RP$. We will denote by $C(Q,M)$, for $M$ as above,  the total complex of the double complex obtained by applying $C(-,M)$ to all the terms of $Q$.    We will  prove that $C(Q,IP)_\Q$ is acyclic by defining a finite filtration on $IP$, by derived $p$-complete $W(k)$-submodules,  such that $C(Q,\gr^j IP)\simeq C_*(Q,\gr^j IP)$.  (Note that then the gradings $\gr^i IP$ are also derived $p$-complete and, since $Q$ is built from projective modules,  the functor $C(Q,-)$ is exact.)  The latter complex is acyclic by the argument presented above. 
    
   Let $I^{[j]}$, $j\geq 1$, be the ideal of $R$ formed by series $\sum a_it^{[i]}$ with $a_0=\ldots=a_{j-1}=0$. We have  $I=I^{[1]}$.   We set
  $$I^{[j]}P:=I^{[j]}\wh{\otimes}^{\rm L}_RP\stackrel{\sim}{\to} I^{[j]}\wh{\otimes}_RP. $$
  Since $R/I^{[j]}\simeq W(k)^{\oplus j-1}$, $I^{[j]}P\hookrightarrow P$ (argue as in the proof of \eqref{seq11} above), and it is a derived $p$-complete module. 
    Since $\phi(I^{[j]})\subset I^{[pj]}$, one has $C(Q,I^{[j]}P/I^{[j+1]}P)=C_*(Q,I^{[j]}P/I^{[j+1]}P)$. It remains to  show that $C(Q,I^{[n]}P)$ is quasi-isomorphic to $C_*(Q,I^{[n]}P)$ for $n$  sufficiently large (then the soughed-after finite filtration is $I^{[j]}P$, $j\leq n.$)
  
  By assumption, for $m$ sufficiently large, there is $\psi:Q\to \phi^*Q$ such that $\phi_Q\psi=p^m\id_Q, \psi\phi_Q=p^m\id_{\phi^*(Q)}$. For $n$ sufficiently large, we have  $\phi(I^{[n]})\subset p^{m+1}I^{[pn]}$. Hence $d_2$ on $C(Q,I^{[n]}P)$ is divisible by $p^{m+1}$. Set $f:=\psi^{\tau}(p^{-m-1}d_2)\in\End(C^0(Q,I^{[n]}P))$; then $d_2=pd_1f$, i.e., $d=d_1(1-pf)$. We used here that, since $Q$ is built from   projective modules and $I^{[n]}P$ is derived $p$-complete, we have  
  \begin{align*}C^0(Q,I^{[n]}P)=\Hom_{W(k)}(Q,I^{[n]}P) & \simeq\R\Hom_{W(k)}(Q,I^{[n]}P),\\
  C^1(Q,I^{[n]}P)=\Hom_{W(k)}(\phi^*(Q),I^{[n]}P) & \simeq\R\Hom_{W(k)}(\phi^*(Q),I^{[n]}P).
  \end{align*} Moreover, since $\R\Hom_{W(k)}(Q,I^{[n]}P)$ is derived $p$-complete \cite[Tag 0A6E]{Sta}, it follows that $(1-pf)$ is a quasi-isomorphism (use derived Nakayama Lemma \cite[Tag 0G1U]{Sta}). This yields the quasi-isomorphism $C(Q,I^{[n]}P)\stackrel{\sim}{\to} C_*(Q,I^{[n]}P)$, as wanted. 
 \end{proof}

\subsection{Hyodo-Kato rigidity} Now we pass to the  main constructions. 
\subsubsection{The Hyodo-Kato section}  In this section we will prove the existence of the Hyodo-Kato section in the derived $\infty$-category. We follow faithfully the arguments of Beilinson from \cite[Sec.\,1.14]{Bv2} with the following modifications:
\begin{enumerate}
\item Beilinson works in the setting of proper log-smooth log-schemes hence all of his cohomology complexes are perfect; we replace them with a weaker condition of derived $p$-complete and $p$-completely flat,
\item to prove that the Hyodo-Kato section (when linearized) is a quasi-isomorphism Beilinson uses finiteness of Hyodo-Kato cohomology; we replace his argument with the original one due to Hyodo-Kato \cite{HK}.
\end{enumerate}
Since the argument of Beilinson can only be found in a preliminary version of a published paper, for the benefit of the reader and the authors, we supply all the details.

  Let $f: X_1\to S_1$ be a log-smooth map with  Cartier type reduction, with $X_1$ fine. 
\index{Xn@\Xn}Let $f^0: X_1^0\to S^0_1$ be its pullback to $S^0_1$.
 Let $R:=r_K^{\rm PD}$.  Recall the definition and basic properties of the arithmetic Hyodo-Kato cohomology and the associated $r^{\rm PD}$ cohomology (in the terminology\footnote{The notation we use here is a bit different 
than the one we used in \cite{CN3}.
 This is because we have adopted here Beilinson's approach to the Hyodo-Kato morphism and with it his notation.
 The advantage of Beilinson's notation is that it keeps better track of the underlying data.} 
\index{rgcris@\rgcris}\index{rghk@\rghk}from \cite[4.2]{CN3}):
\begin{align}
\label{fifa1}
 & \rg_{\crr}(X_1/R)_{l,n}:=\rg_{\crr}(X_1/(S_1,E_n)),\quad i_l: S_1\hookrightarrow E_1,\quad \mbox{in } \sd_{\phi}(R_n);\\
& \rg_{\hk}(X^0_1)_n:=\rg_{\crr}(X_1^0/(S^0_1,S^0_n)),\quad \mbox{in } \sd_{\phi}(W_n(k));\notag\\
& \rg_{\crr}(X_1/R)_l:=\lim_n\rg_{\crr}(X_1/R)_{l,n},\quad \mbox{in } \wh{\sd}_{\phi}(R); \notag\\
&  \rg_{\hk}(X^0_1):=\lim_n\rg_{\hk}(X^0_1)_n, \quad \mbox{in } \sd^c_{\phi}(W(k)). \notag
\end{align}
The embedding $i:X^0_1\hookrightarrow X_1$  over $i_{l,n}: (S^0_1,S^0_n)\hookrightarrow (S_1,E_n)$ yields compatible morphisms $i^*_{l,n}: \rg_{\crr}(X_1/R)_{l,n}\to \rg_{\hk}(X^0_1)_n$,
$i^*_{l}: \rg_{\crr}(X_1/R)_{l}\to \rg_{\hk}(X^0_1)$ in $\sd_{\phi}(W_n(k))$ and $\sd^c_{\phi}(W(k))$, respectively. These constructions are functorial in $X_1$: this is standard (see \cite[1.6]{Bv2} and use the functorial PD-envelopes from \cite[1.4]{Bv2}).

   Moreover,
\begin{enumerate}
\item  $\rg_{\hk}(X^0_1)$ is a complex of derived $p$-complete, $p$-completely flat modules over $W(k)$  and 
\begin{equation}
\label{leg1}
\rg_{\hk}(X^0_1)_n\simeq \rg_{\hk}(X^0_1)\wh{\otimes}^{\LL}_{W(k)}W_n(k),\quad \mbox{in } \sd_{\phi}(W_n(k));
\end{equation}
\item  $\rg_{\crr}(X_1/R)_{l}$ is a complex of derived $p$-complete, $p$-completely flat modules over $R$ and 
\begin{equation}
\label{leg2}
\rg_{\crr}(X_1/R)_{l,n}\simeq \rg_{\crr}(X_1/R)_{l}\wh{\otimes}^{\LL}_{R}R_n,\quad \mbox{in } \sd_{\phi}(R_n);
\end{equation}
\item we have a quasi-isomorphism 
\begin{equation}
\label{leg3}
i^*_l: \rg_{\crr}(X_1/R)_l\wh{\otimes}^{\LL}_RW(k)\stackrel{\sim}{\to}\rg_{\hk}(X^0_1),\quad \mbox{in } \sd^c_{\phi}(W(k)).
\end{equation}
\end{enumerate}

Now we present the key construction in the Hyodo-Kato theory.
\begin{theorem}\label{section1}
\begin{enumerate}
\item The Frobenius action on $\rg_{\hk}(X^0_1)_\Q$ is invertible in $\sd^c(W(k))_{\Q}$.
\item The map $i^*_{l}: \rg_{\crr}(X_1/R)_{l,\Q}\to \rg_{\hk}(X^0_1)_\Q$ admits a unique natural $W(k)$-linear section $\iota_{l}$ in  $\sd^c_{\phi}(W(k))_{\Q}$. Its $R$-linear extension   is a 
\index{ilst@\ilst}quasi-isomorphism in $\wh{\sd}_{\phi}(R)_{\Q}$:
$$
\iota_{l}: (R\wh{\otimes}^{\LL}_{W(k)}\rg_{\hk}(X^0_1))_\Q\stackrel{\sim}{\to} \rg_{\crr}(X_1/R)_{l,\Q}.
$$
\end{enumerate}
\end{theorem}
\begin{proof}
Claim (1) is proved in \cite[2.24]{HK}. In fact, Hyodo-Kato prove more: they show that there exists a $p^d$-inverse of Frobenius, where $d=\dim X^0_1$. 

For the existence part of claim (2),  recall that Beilinson \cite[1.14]{Bv2} proved it in the case $X_1$ is proper. We will adapt his argument to our  (general) local situation.

  Take $P=\rg_{\crr}(X_1/R)_l$ in Lemma \ref{Newton1}. By claim (1)  the Frobenius action on $(P\wh{\otimes}_RW(k))_\Q$ is invertible.  Moreover,  $P$ is derived $p$-complete and  a complex of completely $p$-adically 
  flat $R$-modules. Lemma \ref{Newton1} implies that the morphism $P_\Q\to (P\wh{\otimes}^{\LL}_RW(k))_\Q$ in $\sd^c_{\phi}(W(k))_{\Q}$ admits a unique right inverse $\iota_l$, as wanted.  We will now show the functoriality of $\iota_l$ with respect to the maps $f: X_1\to S_1, i_l: S_1\to E$. Consider a commutative diagram of such maps
  $$
  \xymatrix{
  X^{\prime}_1\ar[d]^{\pi_X}\ar[r]^{f^{\prime}} & S_1^{\prime}\ar[d]^{\pi_S} \ar@{^(->}[r]^{i_{l^{\prime}}} & E^{\prime}\ar[d]^{\pi_E}\\
   X_1\ar[r]^{f} & S_1 \ar@{^(->}[r]^{i_l} & E
  }
  $$
It yields  the commutative diagram in $\sd^c_{\phi}(W(k))$
   $$
  \xymatrix{
  \rg_{\crr}(X^{\prime}_1/R)_{l^{\prime}} \ar[r]^{i^*_{l^{\prime}}} & \rg_{\hk}(X^{\prime,0}_1)\\
  \rg_{\crr}(X_1/R)_l\ar[u]^{\pi_{\crr}^*} \ar[r]^{i^*_{l}} & \rg_{\hk}(X^0_1)\ar[u]^{\pi_{\hk}^*}
  }
  $$
  and the induced  diagram in $\sd^c_{\phi}(W(k))_{\Q}$
  $$
  \xymatrix{
  \rg_{\crr}(X^{\prime}_1/R)_{l^{\prime},\Q}  & \rg_{\hk}(X^{\prime,0}_1)_{\Q}\ar[l]-_{\iota^*_{l^{\prime}}}\\
  \rg_{\crr}(X_1/R)_{l,\Q}\ar[u]^{\pi_{\crr}^*}  & \rg_{\hk}(X^0_1)_{\Q}\ar[u]^{\pi_{\hk}^*}\ar[l]_{\iota^*_{l}}\ar@/_9pt/[lu]^-{f_{l^{\prime}}}\ar@/^10pt/[lu]^-{f_{l}},
  }
  $$
  where $f_l=\pi_{\crr}^*\iota^*_{l}$ and $f_{l^{\prime}}=\iota^*_{l}\pi_{\hk}^*$.  Hence the left and right corner triangles in the last diagram commute.  It suffice thus to show that $f_l=f_{l^{\prime}}$. 
  But we have 
  \begin{align*}
 i^*_{l^{\prime}} f_l  & =  i^*_{l^{\prime}}\pi_{\crr}^*\iota^*_{l}=\pi_{\hk}^* i^*_{l}\iota^*_{l}=\pi_{\hk}^*,\\
  i^*_{l^{\prime}} f_{l ^{\prime}} & =  i^*_{l^{\prime}}\iota^*_{l^{\prime}}\pi_{\hk}^*=\pi_{\hk}^*.
  \end{align*}
  Hence, if $\overline{P}$ denotes the cofiber of the map $ i^*_{l^{\prime}}$ is suffices to show that $\R\Hom_{W(k)_{\phi}}( \rg_{\hk}(X^0_1),\overline{P})_{\Q}=0$. But this can be done by the same arguments as in the proof of Lemma \ref{Newton1}.
  
  Consider its $R$-linearization
 $$
 \iota_l:(R\wh{\otimes}^{\LL}_{W(k)}P\wh{\otimes}^{\LL}_RW(k))_\Q\to P_\Q.
 $$
We need to show that this is a quasi-isomorphism.  But this was done by  Hyodo-Kato \cite[Lemma 4.16, Prop. 4.8]{HK} using the explicit de Rham-Witt presentation of the Hyodo-Kato complex. We are done.
\end{proof}
\subsubsection{The Hyodo-Kato morphism} Now, as usual, the Hyodo-Kato morphism can be obtained from the section constructed in Theorem \ref{section1}. Let $X$ be a fine logarithmic formal scheme log-smooth over $S$. Assume that $X_1$ has  Cartier type reduction over $S_1$. Let $\varpi$  be a uniformizing parameter of $\so_K$. 
\begin{corollary}
There is a natural  quasi-isomorphism in $\sd^c(\so_K)_{\Q}$
$$
i_{\varpi}: \quad (\rg_{\hk}(X^0_1)\wh{\otimes}^{\LL}_{W(k)}\so_K)_{\Q}\stackrel{\sim}{\to}\rg_{\dr}(X)_{\Q}. 
$$
\end{corollary}
\begin{proof} 
Take  $E$ with $l:=\varpi \mod p\mathfrak{m}_K$. This yields an embedding $i_{\pi}: S\hookrightarrow E, i_{\varpi}^*(t_a)=\varpi, a:=\varpi \mod {\mathfrak m}^2_K$. 
We start with the  quasi-isomorphism from Theorem \ref{section1}
\begin{align*}
 \iota_l: (\rg_{\hk}(X^0_1)\wh{\otimes}^{\LL}_{W(k)}R)_{\Q}\to \rg_{\crr}(X_1/R)_{l,\Q}.
\end{align*}
Tensoring it with $\so_K$ (over $R$) we obtain the  quasi-isomorphisms
$$
( \rg_{\hk}(X^0_1)\wh{\otimes}^{\LL}_{W(k)}\so_K)_{\Q}\stackrel{\sim}{\to} (\rg_{\crr}(X_1/R)_{l}\wh{\otimes}^{\LL}_{R}\so_K)_{\Q}\simeq \rg_{\crr}(X_1/\so^{\times}_K)_{\Q}\simeq \rg_{\dr}(X)_{\Q}.
$$
This is the Hyodo-Kato quasi-isomorphism $i_{\varpi}$ we wanted.
\end{proof}
\subsubsection{Monodromy action revisited} A {\em  $(\phi,N)$-module over $W(k)$} is a triple $(M,\phi,N)$ with $(M,\phi)$ --  a $\phi$-module
over  $W(k)$ and $N: M\to M$ -- a $W(k)$-linear  endomorphism, called {\em monodromy} operator, such that
$N\phi = p\phi N$. The category of $(\phi,N)$-modules over $W(k)$ is abelian. We will denote by
$\sd_{\phi,N}(W(k))$  the corresponding derived $\infty$-category. Using Remark \ref{E1},  we will identify this $\infty$-category with the $\infty$-category of left modules from $\sd(W(k))$ over the associative ring $W(k)\{\phi, N\}$ defined as the abelian group $W(k)[\phi, N]$ with multiplication rules
$\phi a=\phi(a)\phi, Na=aN, N\phi=p\phi N$, for $a\in W(k)$. We denote by $\sd^c_{\phi,N}(W(k))$ the full $\infty$-subcategory of $\sd_{\phi,N}(W(k))$ spanned by complexes of $p$-torsion-free and $p$-complete modules. We have similar structures over $W_n(k)$. 

The constructions in \eqref{fifa1} live in respective $\sd^c_{\phi,N}(-)$ $\infty$-categories and are functorial in $X_1$. The subsequent base changes \eqref{leg1}, \eqref{leg2}, \eqref{leg3} also lift to the $\infty$-categories $\sd^c_{\phi,N}(-)$. One way to see this is to use the description of the monodromy action in the paragraphs that follow.

The purpose of this section is to prove the following:
\begin{proposition}\label{monodromy1}
The section
$$
\iota_{l}:  \rg_{\hk}(X^0_1)_\Q\to \rg_{\crr}(X_1/R)_{l,\Q}
$$
from Theorem \ref{section1} commutes with monodromy, i.e., it can be lifted to a section in $\sd^c_{\phi,N}(W(k))_{\Q}$. 
\end{proposition}
Recall that the monodromy on $\rg_{\crr}(X_1/R)_{l}$ is defined as the Gauss-Manin connection and the one on $ \rg_{\hk}(X^0_1)$ as  its residue at $t=0$. However, to prove Proposition \ref{monodromy1} we will work with the "integration" of the monodromy action. The argument  follows that of Beilinson in  \cite[Sec.\,1.16]{Bv2} with the modifications mentioned earlier.

 (i) {\em Equivariant structures.} Let $A^*$ be a cosimplicial algebra. An  {\em $A^*$-complex} is a complex $M^*$ of cosimplicial $A^*$-modules such that,  for every cosimplicial structure map $M^a\to M^b$, its $A^b$-linearization  
 $A^b\otimes^{\LL}_{A^a}M^a\to M^b$ is a quasi-isomorphism. Denote by $\sd(A^*)$ the  derived $\infty$-category of bounded below $A^*$-complexes. 
 We think of an element of $\sd(A^*)$ as an $A^*$-complex with values in $\sd(A^a)$ in degree $a$. 
   For an endomorphism $T$ of $A^*$, $\sd_{T}(A^*)$ will denote  the derived $\infty$-category of  bounded below $A^*$-complexes equipped with a $T$-action.   We think of an element of $\sd_T(A^*)$ as an $A^*$-complex with values in 
   $\sd_{T^a}(A^a)$ in degree $a$ (the derived $\infty$-category of $A^a\{T^a\}$-modules).

 Fix an  affine scheme $S$ as a base. Let $G$ be an affine group scheme acting  on $X=\Spec A$. Let $[X/G]:=EG\times_GX=\Spec A^*_G$ be the  simplicial quotient. We have $[X/G]_m=X\times G^m$. Set $\sd_G(A):=\sd(A^*_G)$.
 
 Let $\mathfrak{g}^{\vee}=\mathfrak{m}_e/\mathfrak{m}_e^2$ be the Lie coalgebra of $G$. Let $[X/{\mathfrak{g}}]:=\Spec A^*_{\mathfrak{g}}$ be the closed subscheme of $[X/G]$ defined by the simplicial ideal generated by $\sk^2$, where $\sk$ is the ideal of $[X/G]_{0}\subset [X/G]_{1}$, i.e., $\sk=\mathfrak{m}_e\otimes A\subset \so(G\times X)$. We set $\sd_{\mathfrak{g}}(A):=\sd(A^*_{\mathfrak{g}})$, etc. There is a canonical conservative restriction functor 
    $${\rm Lie}: \sd_{G}(A)\to\sd_{\mathfrak{g}}(A).$$
    Moreover:
    \begin{enumerate}
    \item Compatible endomorphisms $T_G$ and $T_X$ of $G$ and $X$ yield an endomorphism  of $[X/G]$. We have 
    $$\sd_{T,G}(A):=\sd_{T}(A^*_G), \quad \sd_{T,\mathfrak{g}}(A):=\sd_{T}(A^*_{\mathfrak{g}}). $$
    \item  For a group scheme $G$, we denote by $G^{\natural}$ its PD-completion at the unit \cite[Sec.\,1.2]{Bv2}; this is a group PD-scheme, i.e., a scheme equipped with a PD-ideal. For example, we have ${\mathbb G}_m^{\natural}((U,T))=\Gamma(T,(1+\sj_T)^*)$. If $G$ is a group PD-scheme with PD-ideal $\mathfrak{m}_e$, then, in the above,  we can also consider the Lie coalgebra in PD-sense $\mathfrak{g}^{\vee}:=\mathfrak{m}_e/\mathfrak{m}_e^{[2]}$. 
    \end{enumerate}
     Objects of $\sd_{T,G}(A)$, $\sd_{T,\mathfrak{g}}(A)$ are called $G$-, resp. {\em $\mathfrak{g}$-equivariant $A$-complexes}. For an $A$-complex $M$, a {\em $G$-equivariant structure} on it is an object $M^*_G\in \sd_G(A)$ together with a quasi-isomorphism $M^0_G\stackrel{\sim}{\to} M$. 

  (ii) {\em Equivariant structures on crystalline cohomology.} Let us go back to  the setting of Proposition \ref{monodromy1}. We note that the objects $(S_1,E_n)\in (S_1/W_n(k))_{\crr}$ and $(S^0_1,S_n^0)\in (S^0_1/W_n(k))_{\crr}$ have natural ${\mathbb G}_m^{\natural}$-actions:  $(S_1,E_n)$ is a coordinate thickening (with coordinate $t_a$),  ${\mathbb G}_m^{\natural}$ acts on it by homotheties, and we equip $S_n^0\subset (S_1,E_n)$ with the induced action.  To see the latter action explicitly, we note that, for $(U,T)\in (S^0_1/W_n)_{\crr}$, a map $f: (U,T)\to (S^0_1,S^0_n)$ amounts to a lifting $f([a])$ of $a\in (\mathfrak{m}_K/\mathfrak{m}^2_K)\setminus \{0\}\subset \sll_1^0$ to $\sll_T^0$; these liftings form a ${\mathbb G}_m^{\natural}((U,T))$-torsor yielding our action.
  This ${\mathbb G}_m^{\natural}$-action is compatible with the Frobenius action  ($\phi$ acts on ${\mathbb G}_m^{\natural}$ as $\phi^*(t):=t^p$).
 
   We will now show that the crystalline cohomology complexes $\rg_{\crr}(X_1/R)_{l,n}$, $\rg_{\hk}(X^0_1)_n$ are naturally equipped with ${\mathbb G}_m^{\natural}$-equivariant structures.  
   Take the simplicial objects $(S_1,E_{n*})$ and $(S^0_1,S^0_{n*})$. Here, for $(U,T)\in (Z/S)_{\crr}$, we wrote  $(U,T_*):=\check{C}((U,T)/Z)$ for the \v{C}ech nerve of the crystalline open $(U,T)\in(Z/S)_{\crr}$; it is a  simplicial object of $(Z/S)_{\crr}$ with terms $(U,T_a):=(U,T)^{a+1}$ (we use the crystalline site product).
 It is easy to see \cite[Exercise 1.7]{Bv2} that $(S_1,E_{n*})=[(S_1,E_n)/{\mathbb G}_m^{\natural}]$ and $(S^0_1,S^0_{n*})=[(S^0_1,S^0_{n})/{\mathbb G}_m^{\natural}]$.  Consider the objects $\R f_{\crr}(\so_{X_1/W_n(k)})$ and $\R f^0_{\crr}(\so_{X^0_1/W_n(k)})$. They are equipped with a Frobenius action. Restricting them to our simplicial objects, we get:
   \begin{align}\label{structures-Paris}
   &\rg_{\crr}(X_1/R)_{l,n}^*:=\R f_{\crr*}(\so_{X_1/W_n(k)})_{(S_1,E_{n*})}\in\wh{\sd}_{\phi,{\mathbb G}_m^{\natural}}(R_n),\\ 
      &\rg_{\hk}(X^0_1)_{n}^*:=\R f^0_{\crr*}(\so_{X^0_1/W_n(k)})_{(S^0_1,S^0_{n*})}\in\sd^c_{\phi,{\mathbb G}_m^{\natural}}(W_n(k^{\prime})).  \notag
      \end{align}
      Since $(\R f_{\crr*}(\so_{X_1/W_n(k)})_{(S_1,E_{n*})})^0\simeq \rg_{\crr}(X_1/R)_{l,n}$ and $(\R f^0_{\crr*}(\so_{X^0_1/W_n(k)})_{(S^0_1,S^0_{n*})})^0\simeq \rg_{\hk}(X^0_1)_n$, these are the ${\mathbb G}_m^{\natural}$-equivariant structures we wanted. 

  We are actually interested in $\mathfrak{n}$-action  that comes from the above ${\mathbb G}_m^{\natural}$-action, where  $\mathfrak{n}$ is the Lie algebra of ${\mathbb G}_m^{\natural}$ in PD-sense (it is a line). The objects from (\ref{structures-Paris}) form projective systems with respect to $n$. Applying $\R\lim_n$, we get natural  $\mathfrak{n}$-structures on $\rg_{\hk}(Z^0_1)$ and $\rg_{\crr}(Z_1/R)_{l}$.
Set $N=e^{-1}t\partial_t$, $e=[K^{\prime}:F^{\prime}]$; it is a generator of $\mathfrak{n}\otimes\Q$.  An $\mathfrak{n}_\Q$-equivariant structure on $W(k^{\prime})_{\Q}$-complex amounts  to an endomorphism $N$. The equality $N\phi=p\phi N$ comes from the compatibility of the ${\mathbb G}_m^{\natural}$-action with Frobenius. 
\begin{proof}({\em of Proposition \ref{monodromy1}}) 
We proceed as in the proof of Theorem \ref{section1} but work in the ${\mathbb G}_m^{\natural}$-equivariant setting. Namely, we start with the  natural map 
$i^*_l: \rg_{\crr}(X_1/R)_{l,\Q}^*\to \rg_{\hk}(X^0_1)_{\Q}^*$, that lifts the map $i^*_l: \rg_{\crr}(X_1/R)_{l,\Q}\to \rg_{\hk}(X^0_1)_{\Q}$, and we look for its ${\mathbb G}_m^{\natural}$-equivariant section  (this will be a ${\mathbb G}_m^{\natural}$-equivariant lift of the section in our proposition). This is supplied by Lemma \ref{section1B} below.  The induced map ${\rm Lie}( \iota_l)$ yields a section between the corresponding $\mathfrak{n}_{\Q}$-equivariant structures. Since  it lifts the original section $\iota_l$ we get  the wanted compatibility of the latter with monodromy. 
\end{proof}
The following lemma was used in the above proof: 
\begin{lemma}\label{section1B}
\begin{enumerate}
\item The Frobenius action on $\rg_{\hk}(X^0_1)^*_\Q$ is invertible in $\sd^c_{{\mathbb G}_m^{\natural}}(W(k^{\prime}))_{\Q}$.
\item The map $i^*_{l}: \rg_{\crr}(X_1/R)^*_{l,\Q}\to \rg_{\hk}(X^0_1)^*_\Q$ admits a  natural $W(k^{\prime})$-linear section $\iota_{l}$ in  $\sd^c_{\phi,{\mathbb G}_m^{\natural}}(W(k^{\prime}))_{\Q}$  
\begin{equation}\label{nie1}
\iota^*_{l}: \rg_{\hk}(X^0_1)^*_\Q{\to} \rg_{\crr}(X_1/R)^*_{l,\Q}.
\end{equation}
\end{enumerate}
\end{lemma}
\begin{proof} In claim (1) we need to proof the invertibility, up to a controlled denominator, of the Hyodo-Kato Frobenius. Since, by  \eqref{structures-Paris},
\begin{equation}
\label{pada1}
\rg_{\hk}(X^0_1)^*_a=\R f^0_{\crr *}(\so_{X^0_1/W(k)})_{(S^0_1,S^0_a)}\simeq \R\Gamma_{\crr}(X^0_1/S^0_a),
\end{equation} where $(S^0_1,S^0_a)$ is the crystalline product $(S^0_1,S^0)^{a+1}$,  we can use again \cite[2.24]{HK}. And, recall  that, Hyodo-Kato prove more: they show that there exists a $p^d$-inverse of Frobenius, where $d=\dim X^0_1$.

  To prove claim (2), take $P=\rg_{\crr}(X_1/R)_{l}^*$ and $Q=\rg_{\hk}(X^0_1)^*$.    We have 
 \begin{align*}
 & Q_a\simeq  \R\Gamma_{\crr}(X^0_1/S^0_a),\\
 & P_a=\rg_{\crr}(X_1/R)_{l,a}^*=\R f_{\crr*}(\so_{X_1/W(k)})_{(S_1,E_{a})}\simeq \R \Gamma_{\crr}(X_1/E_{a})_l,
 \end{align*}
 where $(S_1,E_a)$ is the crystalline product $(S_1,E)^{a+1}$. More explicitly,
 \begin{align*}
 S^0_a & =\Spf W(k)<u_1-1,\ldots, u_a-1>, \\
 E_a & =\Spf W(k)<t,u_1-1,\ldots, u_a-1>; \quad i_{b,0}: S^0_b\hookrightarrow E_a,\, t_s\mapsto [s],  u_i\mapsto u_i,
 \end{align*}
 and the log-structure is induced by $t_s$ (and its reduction). In particular, the Frobenius on the ideal of the embedding  $i_{a,0}$ is highly topologically nilpotent. This implies, by an argument identical to the one used in the proof of Theorem \ref{section1}, that we have a unique section in $\sd^c_{\phi}(W(k^{\prime}))_{\Q}$
 $$
 \iota_{l,a}: Q_{a,\Q}\to P_{a,\Q}
 $$
 of the canonical  projection $$P_{a,\Q} \to Q_{a,\Q}$$ and that this section is functorial with respect to all the cosimplicial maps. Hence it yields the section $\iota^*_l$ from \eqref{nie1}.
 Functoriality of this section follows from the functoriality of the individual sections $\iota_{l,a}$. 
 \end{proof}
\subsection{Geometric absolute crystalline cohomology and Hyodo-Kato cohomology}\label{lambda} 
We are now ready to prove the existence of geometric Hyodo-Kato quasi-isomorphisms.
\subsubsection{The comparison theorem}
 Let now $\overline{f}: X_1\to \overline{S}_1$ be a map of log-schemes with $X_1$  integral and quasi-coherent. Assume that $\overline{f}$ is the base change of a log-scheme $f: Z_1\to S_1$, which is log-smooth and with Cartier type reduction.  Choose $l$, hence $(S_1,E)$, as in Section \ref{period-rings}. 
\index{thetalambda@\thetalambda}Choose a Frobenius compatible map $\theta_{\lambda}: (\overline{S}_1,\se_{\crr})\to (S_1,E)$ of PD-thickenings that extends the map $\theta_1$, where $\theta$ is the  canonical  map $\theta:\overline{S}\to S$. 
This amounts to a choice of $\lambda_a:=\theta_{\lambda}(t_a)\in \sll_{\phi}$ that lifts $l_a\in\sll_1\subset \overline{\sll}_1$.

 The following well-known corollary of Theorem \ref{section1} describes geometric absolute crystalline cohomology $\R\Gamma_{\crr}(X_1):=\R\Gamma_{\crr}(X_1/W(k))$ via Hyodo-Kato cohomology (but losing the Galois action).
\begin{corollary}
\begin{enumerate}
\item There is a functorial system of compatible quasi-isomorphisms in $\sd_{\phi}(\A_{\crr,n})$
$$
\epsilon^R_{\lambda,n}:\rg_{\crr}(Z_1/R)_{l,n}\otimes^{\LL}_{R_n}\A_{\crr,n}\stackrel{\sim}{\to} \rg_{\crr}(X_1)_n.
$$
Here the tensor product is taken with respect to the map
$\theta^*_{\lambda,n}: R_n\to \A_{\crr,n}$.
\item There is a natural quasi-isomorphism in $\sd_{\phi}(\A_{\crr})$
$$
\epsilon^R_{\lambda}:\rg_{\crr}(Z_1/R)_l\wh{\otimes}^{\LL}_R\A_{\crr}\stackrel{\sim}{\to} \rg_{\crr}(X_1).
$$
\item There is a natural strict quasi-isomorphism  in $\sd_{\phi}(C_{\B^+_{\crr}})$
$$
\epsilon^{\hk}_{\lambda}:(\rg_{\hk}(Z_1^0)\wh{\otimes}^{\LL}_{W(k)}\A_{\crr})_{\Q}\stackrel{\sim}{\to}\rg_{\crr}(X_1)_\Q.
$$
\end{enumerate}
\end{corollary}
\begin{remark}
\begin{enumerate}
\item The functor $(-)_\Q: \sd^c(\Z_p)\to \sd({C_{\Q_p}})$ in (3)  is induced from the functor $(-)_{\Q}$ from Remark \ref{westin} via the map  $ \sd^c(\Z_p)\to   \sd({\rm PD}_{\Q_p})$.
\item We set $\sd_{\phi}(C_{\B^+_{\crr}}):={\rm LMod}_{\B^+_{\crr}\{\phi\}} \sd(C_{\B^+_{\crr}})$.
\end{enumerate}
\end{remark}
\begin{proof}
Since $Z_1$ is log-smooth, claim (1) follows from the log-smooth base change\footnote{The proof of which is almost identical to the proof of smooth base change  in the case without log-structures, see \cite[2.3.5]{BBM} or \cite[V.3.5.1]{B}.} (recall that $\rg_{\crr}(X_1)_n\simeq \rg_{\crr}(X_1/\A_{\crr,n})$). Claim (2) follows from claim (1) by taking limits. Claim (3) follows from claim (2) and Theorem \ref{section1}.
\end{proof}

 Let $\overline{f}^0: X_1^0\to \overline{S}_1^0$ be the pullback of $\overline{f}$ to $\overline{S}^0_1$.  We have the {\em completed geometric Hyodo-Kato \index{rghk@\rghk}cohomology}
$$
\rg_{\hk}(X^0_1):=\rg_{\crr}(X_1^0/\overline{S}^0).
$$
It is a $W(\overline{k})$-module. 
It compares with the arithmetic Hyodo-Kato cohomology via 
the log-smooth base change \index{basechange@\basechange}quasi-isomorphism in $\sd_{\phi}(W(\overline{k}))$
\begin{equation} \label{beta}
\beta: \quad \rg_{\hk}(Z^0_1)\wh{\otimes}^{\LL}_{W(k_L)}W(\overline{k})\stackrel{\sim}{\to} \rg_{\hk}(X^0_1).
\end{equation}
\begin{theorem}
\label{HK-crr}
There is a natural strict  \index{epshk@\epshk}quasi-isomorphism in $\sd_{\phi,N}(C_{\B^+_{\st}})$ 
$$
\epsilon^{\hk}_{\st}:\rg_{\hk}(X_1^0)_{\Q_p}\wh{\otimes}_{\breve{F},\iota}\B^+_{\st}\stackrel{\sim}{\to} \rg_{\crr}(X_1)_{\Q_p}\wh{\otimes}_{\B^+_{\crr},\iota}\B^+_{\st} 
$$
such that $\epsilon^{\hk}_{\lambda}=s^*_{\lambda}\epsilon^{\hk}_{\st}\beta$. 
\end{theorem}
\begin{remark}
\begin{enumerate}
\item
Here, for $M=\rg_{\hk}(X_1^0),  \rg_{\crr}(X_1)$, we have defined\footnote{This is the only context in the paper where we use inductive tensor products.}
\begin{equation}
\label{def10}
M_{\Q_p}\wh{\otimes}_{\breve{F},\iota}\B^+_{\st}:=\LL\colim_r(M_{\Q_p}\wh{\otimes}^{\R}_{\breve{F}}\B^{\leq r}_{\st}),\quad 
M_{\Q_p}\wh{\otimes}_{\B^+_{\crr},\iota}\B^+_{\st}:=\LL\colim_r(M_{\Q_p}\wh{\otimes}^{\R}_{\B^+_{\crr}}\B^{\leq r}_{\st}),
\end{equation}
respectively, 
where \index{BST@\BST}$\B^{\leq r}_{\st}:=\oplus_{i=0}^r\B^+_{\crr}u^i_{\lambda}$,   
$u_{\lambda}=\log(\lambda)$, for fixed $\lambda$.
\item We set $\sd_{\phi,N}(C_{\B^+_{\st}}):={\rm LMod}_{\B^+_{\st}\{\phi,N\}}\sd(C_{\B^+_{\st}})$, where the ring $\B^+_{\st}\{\phi,N\}$ is defined as the abelian group $\B^+_{\st}[\phi,N]$ with multiplication rules $\phi a=\phi(a)\phi, Na-aN=N(a), N\phi=p\phi N$, 
for $a\in\B^+_{\st}$. 
\end{enumerate}
\end{remark}
\begin{proof} 
The proof of the theorem runs over sections~\ref{pro1} (construction on the map) 
and~\ref{pro2} (compatibility
with all structures).
\end{proof}
\subsubsection{Construction of the quasi-isomophism}~\label{pro1}\\
\noindent $\bullet$ {\it The index sets}.
Recall that we have assumed that one can find a  finite extension $L/K$ such that
$\overline{f}$ is the base change of 
 a fine log-scheme  ${f}_L:Z_{1}\to \Spec(\so_{L,1})^{\times}$, log-smooth and of Cartier type,  by the natural map $\theta: \overline{S}_1\to\Spec(\so_{L,1})^{\times}$.
 That is, we have a map $\theta_{L}: X_1\to Z_{1}$ such that the square $(\overline{f},{f}_L,\theta,\theta_{L})$ is Cartesian.
Such data $\Sigma:=\{(L,f_L,\theta_{L})\}$ clearly form a filtered set\footnote{In \cite[4.3.1]{CN3}, in the case of a semistable formal scheme $\sx$ over $\so_C$,  we have used a different index set $\Sigma$, call it $\Sigma^{\rm old}$. It is easy to see that we obtain the same theory with both choices of the index set: if $\sx$ is affine then the canonical map $\Sigma^{\rm old}\to\Sigma$ makes $\Sigma^{\rm old}$ cofinal in $\Sigma$. }.  
 
  We have similar data $\Sigma^0:=\{(\theta^0,f^0,\theta^0_L)\}$:
 \begin{enumerate}
 \item $\theta^0: \overline{S}^0\to S^{\prime,0}$ is a map of log-schemes over $S^0_K$ with $S^{\prime,0}=\Spf W(k^{\prime})$, where $k^{\prime}\subset \overline{k}$ is finite over $k$ and the log-structure of $S^{\prime,0}$ is generated by one element; the Frobenius on the log-scheme $S^{\prime,0}$ is induced, via the map $\theta^0$, from the Frobenius on $ \overline{S}^0$;
 \item $f^0: Z^0_1\to S^{\prime,0}_1$ is log-smooth, fine and  integral, of Cartier type;
 \item $\theta^0_X: X^0_1\to Z^0_1$ is such that the square $(\overline{f}^0, f^0, \theta^0_1,\theta^0_{X})$ is Cartesian.
 \end{enumerate}
 Such data again form a filtered set. There is a map of filtered  sets $\Sigma\to\Sigma^0, Z_1/S_L\mapsto Z^0_1/S^0_L$; it is cofinal. These filtered sets are clearly functorial with respect to the maps  $\overline{f}$.

\vskip.2cm
\noindent $\bullet$
{\em Construction of $\epsilon^{\hk}_{\st}$.} 
Let us first construct $\epsilon^{\hk}_{\st}$. For $\xi^0=Z^0_1/S^{\prime,0}_1\in\Sigma^0$, $S^{\prime}=\Spf \so_{K^{\prime}}$,  let $\Psi_{\xi^0}$ be the set of triples\footnote{We like to call them {\em Frobenius-twisted descent data}.} $\pi=(\pi,\pi_S,n_{\pi})$, where $n_{\pi}\in\N$ and $\pi, \pi_S$ are maps such that the diagram
\begin{equation}
\label{correction1}
\xymatrix@R=6mm{
X_1^0\ar@{^(->}[r]\ar[d]^-{\theta^0_X} &  X_1\ar[r]\ar[d]^-{\pi} & \overline{S}_1\ar[d]^-{\pi_S}\\
Z^0_1\ar[r]^{{\rm Fr}^{n_{\pi} }} & Z^0_1\ar[r] & S^{\prime,0}_1
}
\end{equation}
commutes. Here we denoted by ${\rm Fr}$ he absolute Frobenius. The set  $\Psi_{\xi^0}$ is ordered: $\pi_1\leq \pi_2$ means $m=n_{\pi_2}/n_{\pi_1}\in\Z$ and $\pi_2={\rm Fr}^m\pi_1$. We claim that the set $\Psi_{\xi^0}$ is  filtered. For that it suffices to show that, for $n\geq e_{K^{\prime}}$, any two triples 
$\pi_1=(\pi_1,\pi_{S,1},n)$ and $\pi_2=(\pi_2,\pi_{S,2},n)$ are in fact equal, that is, $\pi_1=\pi_2$ and $\pi_{S,1}=\pi_{S,2}$. But, for $n$ as above, we have the diagram ($\pi=\pi_1,\pi_2$)
$$
\xymatrix@R=7mm{
X_1^0\ar@{^(->}[drr]\ar[ddr]_{\theta^0_X} \\
& X_1 \ar@{-->}[ul]^f\ar[r]_{{\rm Fr}^{n}} \ar[d]^-{\pi}&  X_1\ar[r]\ar[d]^-{\pi} & \overline{S}_1\ar[d]^-{\pi_S}\\
& Z^0_1\ar[r]^{{\rm Fr}^{n }} & Z^0_1\ar[r] & S^{\prime,0}_1
}
$$
in which the  two small squares, the square with corner $X_1^0$,  and the top triangle commute. 
This implies that  ${\rm Fr}^{n}\pi={\rm Fr}^{n }\theta^0_Xf$. Since there are no nilpotents in $\so_{X^0_1}$, we get $\pi=\theta^0_Xf$, hence $\pi_1=\pi_2$, as wanted. 

  Similarly, we have the diagram
\begin{equation}
\label{smart}
\xymatrix@R=7mm{\overline{S}_1 \ar[ddr]_{\pi_S}\ar[rrd]^{{\rm Fr}^{n }}\ar@{-->}[dr]^{f_S}\\
& \overline{S}^{0}_1 \ar@{^(->}[r]\ar[d]^-{\theta^0}   & \overline{S}_1\ar[d]^-{\pi_S}\\
& S^{\prime,0}_1 \ar[r]^{{\rm Fr}^{n }}  & S^{\prime,0}_1
}
\end{equation}
where the  square with vertex $\overline{S}_1 $ and  the top triangle commute. The small square commutes as well:  map  the commutative diagram
$$
\xymatrix@R=6mm{
X_1^0\ar@{^(->}[r]\ar[d]^-{\theta^0_X} &  X_1\ar[d]^-{\pi} \\
 Z^0_1\ar[r]^{{\rm Fr}^{n }} & Z^0_1
}
$$
to it using the canonical maps and use the fact that $ \overline{S}^{0}_1$ is a field. Diagram (\ref{smart}) now implies that 
 ${\rm Fr}^{n }\pi_S={\rm Fr}^{n }\theta^0f_S$. Since $S^{\prime,0}_1 $ is a field we get $\pi_S=\theta^0f_S$, hence $\pi_{S,1}=\pi_{S,2}$, as wanted.

 Let now $e$ be the ramification index of $\so_{K^{\prime}}$. Denote by ${\mathbb A}^{(1/e)}_{W({k}^{\prime})}$ the formal scheme $\Spf W({k}^{\prime})\{t_a\}, $ where $a\in\tau_{1/e}$ is such that $[a]$ lies in the image of the embedding $\theta^{0,*}:\sll^{\prime,0}\hookrightarrow\overline{\sll}^0$, with $t_{a^{\prime}}=[a^{\prime}/a]t_a$. The log-structure is generated by $t_a$. We have an embedding 
  $i: {S}^{\prime,0}\hookrightarrow {\mathbb A}^{(1/e)}_{W({k}^{\prime})}$, $i^*(t_a)=[a]$.
 Let $R^{\prime,0}:=r^{{\rm PD},0}_{K^{\prime}}$, $E^0:=\Spf R^{\prime,0}$. We have the  PD-thickenings $i^0_0: (S^{\prime,0}_1, E^0_n)$, $i^{0,*}_0(t_a)=a\in\sll^{\prime,0}_1$. The map $\pi_l:=i\pi_S: \overline{S}_1\to {\mathbb A}^{(1/e)}_{W(k^{\prime})}$ induces a map $i_l: \overline{S}_1\to {\mathbb A}^{(r)}_{W(\overline{k})}$, for $r=p^{n_{\pi}}/e$, which corresponds (see Section \ref{period-rings}) to a class
 $l_{\pi}\in \Lambda$ such that   $v(l_{\pi})=r$. The map  $\pi_S$ extends  canonically to a map of PD-thickenings $\pi_{\st}: (\overline{S}_1,\se_{l_{\pi},\st,n})\to (S^{\prime,0}_1,E^0_{n})$, i.e., we have the following commutative diagram
 $$
 \xymatrix@R=7mm{
 \overline{S}_1\ar@{^(->}[r]^{i_l} \ar[d]^-{\pi_S}& \se_{l_{\pi},\st,n}\ar[d]^-{\pi_E} \\
  S^{\prime,0}_1\ar@{^(->}[r]^{i^0_0} & E^0_{n},
 }
 $$
 where the map $\pi_E$ sends $t_a\mapsto t_{a^{p^{n_{\pi}}}}$.
 
  We have the maps
  \begin{align}
  \label{maps}
  \rg_{\hk}(Z^0_1)_n\stackrel{i^*}{\leftarrow}\rg_{\crr}(Z^0_1/(S^{\prime,0}_1,E^0_n))\verylomapr{(\pi^*,\pi^*_{\st})}\rg_{\crr}(X_1/(\overline{S}_1,\se_{l_{\pi},\st,n}))\stackrel{\sim}{\leftarrow}\rg_{\crr}(X_1)_n\otimes^{\LL}_{\A_{\crr,n}}\wh{\A}_{l_{\pi},\st,n}.
  \end{align}
  By applying $\R\wlim_n$ to these complexes we can remove $n$. Now, $i^*_{\Q}$ has a  section $\iota$ (use Theorem \ref{section1}  for $E=E^0$). Composing it with the rest of the maps from (\ref{maps}) we get a map in  $\sd^c_{\phi,N}(W(k))_{\Q}$
  $$
  \epsilon_{\xi^0,\pi}: \rg_{\hk}(Z^0_1)_\Q\to (\rg_{\crr}(X_1)\wh{\otimes}^{\LL}_{\A_{\crr}}\wh{\A}_{l_{\pi},\st})_{\Q}.
  $$

   Before proceeding let us make the following remark.
\begin{remark}
\label{identification}
Let $M$ be a complex equipped with an $N$-action. 
Let \index{Nnilp@\Nnilp}$M^{N{\text{-nilp}}}:=[M\to M[N^{-1}]]$, where 
$$
M[N^{-1}]:=\LL\colim(M\lomapr{N}M\lomapr{N}\cdots).
$$
For  $M=\rg_{\hk}(X_1^0)$, we have strict quasi-isomorphisms in  $\sd_{\phi,N}(C_{\B^+_{\st}})$
\begin{equation}
\label{Paris-hot}
M_{\Q_p}\wh{\otimes}_{\breve{F},\iota}\B^+_{\st}\stackrel{\sim}{\leftarrow} (M_{\Q_p}\wh{\otimes}_{\breve{F},\iota}\B^+_{\st})^{N{\text{-nilp}}}\stackrel{\sim}{\to} 
(M_{\Q_p}\wh{\otimes}^{\R}_{\breve{F}}\wh{\B}^+_{l,\st})^{N{\text{-nilp}}}\simeq (M\wh{\otimes}^{\LL}_{W(\overline{k})}\wh{\A}_{l,\st})^{N{\text{-nilp}}}_{\Q_p}.
\end{equation}
The last quasi-isomorphism in (\ref{Paris-hot}) holds because $M$ is built from torsion-free and $p$-complete modules.   The previous two quasi-isomorphisms are clear algebraically because we can assume that $N$ is globally nilpotent on $M$ (see \cite[0.1]{Mok}); it is also clear topologically because $M_{\Q_p}$ is built from Banach spaces and $\wh{\B}^+_{l,\st}$ is a Banach space (so the derived tensor product is given by the tensor product itself).

     Similarly, for $M=\rg_{\crr}(X_1)$ (note that now the action of $N$ on $M$ is trivial), we have strict quasi-isomorphisms in  $\sd_{\phi,N}(C_{\B^+_{\st}})$
$$
M_{\Q_p}\wh{\otimes}_{\B^+_{\crr},\iota}\B^+_{\st}\stackrel{\sim}{\leftarrow} (M_{\Q_p}\wh{\otimes}_{\B^+_{\crr},\iota}\B^+_{\st})^{N{\text{-nilp}}}\stackrel{\sim}{\to} (M\wh{\otimes}^{\LL}_{\A_{\crr}}\wh{\A}_{l,\st})_{\Q_p}^{N{\text{-nilp}}}.
$$
\end{remark}

   Extending the map $ \epsilon_{\xi^0,\pi}$ by $\wh{\A}_{l_{\pi},\st}$-linearity and using the quasi-isomorphism (\ref{beta}) we get a map in  $\sd_{\phi,N}(C_{\breve{F}})$
   $$
 \wh {\epsilon}_{\st,\xi^0,\pi}:\ \quad (\rg_{\hk}(X_1^0)\wh{\otimes}^{\LL}_{W(\overline{k})}\wh{\A}_{l_{\pi},\st})_{\Q}\to (\rg_{\crr}(X_1)\wh{\otimes}^{\LL}_{\A_{\crr}}\wh{\A}_{l_{\pi},\st})_{\Q}.
$$

   Now, we define the map   in  $\sd_{\phi,N}(C_{\B^+_{\st}})$
   $$
  \epsilon_{\st,\xi^0,\pi}:=(\wh{\epsilon}_{\st,\xi^0,\pi})^{N_{l_{\pi}}{\text{-nilp}}}: \quad      (\rg_{\hk}(X_1^0)\wh{\otimes}^{\LL}_{W(\overline{k})}\wh{\A}_{l_{\pi},\st})_{\Q_p}^{N_{l_{\pi}}{\text{-nilp}}}
\to (\rg_{\crr}(X_1)\wh{\otimes}_{\A_{\crr}}\wh{\A}_{l_{\pi},\st})_{\Q_p}^{N_{l_{\pi}}{\text{-nilp}}}.$$
We can use the quasi-isomorphisms from Remark \ref{identification}  and get the  map
    $$
    \epsilon_{\st,\xi^0,\pi}: \quad \rg_{\hk}(X_1^0)_{\Q_p}\wh{\otimes}_{\breve{F},\iota}\B^+_{\st}\to \rg_{\crr}(X_1)_{\Q_p}\wh{\otimes}_{\B^+_{\crr},\iota}\B^+_{\st}.
$$
Finally, since the Frobenius is invertible on the Hyodo-Kato cohomology, we can take the map  in  $\sd_{\phi,N}(C_{\breve{F}})$
 $$
  \wt{\epsilon}^{\hk}_{\st,\xi^0,\pi}:=\phi^{n_{\pi}}{\epsilon}_{\st,\xi^0,\pi}\phi^{-n_{\pi}}: \quad \rg_{\hk}(X_1^0)_{\Q_p}\to \rg_{\crr}(X_1)\wh{\otimes}_{\B^+_{\crr},\iota}\B^+_{\st}.
$$
  Its $\B^+_{\st}$-linearization  in  $\sd_{\phi,N}(C_{\B^+_{\st}})$
      \begin{equation}
     \label{def11}
  {\epsilon}^{\hk}_{\st,\xi^0,\pi}: \quad \rg_{\hk}(X_1^0)_{\Q_p}\wh{\otimes}_{\breve{F},\iota}\B^+_{\st}\to \rg_{\crr}(X_1)_{\Q_p}\wh{\otimes}_{\B^+_{\crr},\iota}\B^+_{\st}
\end{equation}
is the map we want.

\vskip.2cm
\noindent $\bullet$
{\em Independence of the choice of $\pi$ and $\xi^0$.}  
      Fix $\xi^0$.  To show that the map ${\epsilon}^{\hk}_{\st,\xi^0,\pi}$ is  independent of Frobenius twists, that is of the choice\footnote{Up to a contractible set of choices, of course.} of $\pi\in\Psi_{\xi^0}$, we note that, for $m\in\Z_{>0}$, there are natural compatible with Frobenius transition maps $\mu_m:\se_{l_{\pi},\st}\to \se_{l^m_{\pi},\st}, \mu^*_m(t_{a^m})=t^m_a$. Moreover,  $\mu^{m_1}\mu^{m_2}=\mu^{m_1m_2}$ and $\mu^*_m\kappa_{l^m}=\kappa_l$. Then the
  transition map from $\epsilon^{\hk}_{\st,\xi^0,\pi_1}$ to $\epsilon^{\hk}_{\st,\xi^0,\pi_2}$, for $\pi_2\geq \pi_1$, is given by $(\phi^{n})^*$ acting on $\rg_{\crr}(Z^0_1/(S^{\prime,0}_1,E^0_n))$ and $\mu^*_n$, for $n=n_{\pi_2}-n_{\pi_1}$. This suffices since the set $\Psi_{\xi^0}$ is filtered. We set ${\epsilon}^{\hk}_{\st,\xi^0}:={\epsilon}^{\hk}_{\st,\xi^0,\pi}$, for any $\pi\in\Psi_{\xi^0}$. 
  
    To show  that $\epsilon^{\hk}_{\st,\xi^0}$ does not depend on the choice of  $\xi^0\in\Sigma^0$, we use
the above maps $\mu$ to identify $\epsilon^{\hk}_{\st,\xi^0_1}$ and $\epsilon^{\hk}_{\st,\xi^0_1}$, for $\xi^0_1\leq \xi^0_2$. This suffices because
the set  $\Sigma^0$ is filtered. We set ${\epsilon}^{\hk}_{\st}:={\epsilon}^{\hk}_{\st,\xi^0}$, for any $\xi^0\in\Sigma^0$. This map is clearly functorial with respect to $X_1$. 
  
\subsubsection{ Compatibility of the arithmetic and geometric maps  ${\epsilon}^{\hk}$.}~\label{pro2}
It remains to prove  that the map ${\epsilon}^{\hk}_{\st}$ is a quasi-isomorphism and that the last claim of our corollary holds. For that, assume that  $\xi^0$ comes from $\xi=Z_1/K^{\prime}\in\Sigma$. Choose $l\in\sll_1/k^{\prime,*}\subset \overline{\sll}_1/\overline{k}^*$. We get a map of PD-thickenings $(\overline{S}_1,\se_{l,\st,n})\to (S_1,E_n)$ that identifies the $t_a$'s. This yields the base change quasi-isomorphisms  in  $\sd_{\phi,N}(W_n(\overline{k})))$
  $$
  \rg_{\crr}(Z_1/R)_l{\otimes}^{\rm L}_{R}\wh{\A}_{l,\st,n}\stackrel{\sim}{\to}\rg_{\crr}(X_1/(\overline{S}_1,\se_{l,\st,n}))\stackrel{\sim}{\leftarrow}\rg_{\crr}(X_1){\otimes}^{\LL}_{W(\overline{k})}\wh{\A}_{l,\st,n}.
  $$
By applying   $\R\wlim_n$ we remove  $n$. Composing with the $\wh{\A}_{l,\st}$-linear extension of $\iota_l$ from Theorem \ref{section1}   we get the strict quasi-isomorphism  in  $\sd_{\phi,N}(C_{\breve{F}})$
  $$
  \widehat{\epsilon}_{\st,l}: (\rg_{\hk}(X_1^0)\wh{\otimes}^{\LL}_{W(\overline{k})}\wh{\A}_{l,\st})_{\Q}\stackrel{\sim}{\to} (\rg_{\crr}(X_1)\wh{\otimes}^{\LL}_{W(\overline{k})}\wh{\A}_{l,\st})_{\Q}
  $$ Denote by $$ 
  {\epsilon}_{\st,l}^{\hk}:=\kappa_l^{-1}( \widehat{\epsilon}_{\st,l})^{N{\text{-nilp}}}\kappa_l:  \quad \rg_{\hk}(X_1^0)_{\Q_p}\wh{\otimes}_{\breve{F},\iota}\B^+_{\st}\to \rg_{\crr}(X_1)_{\Q_p}\wh{\otimes}_{\B^+_{\crr},\iota}\B^+_{\st}
  $$
  the associated map  in  $\sd_{\phi,N}(C_{\B^+_{\st}})$. It is a strict quasi-isomorphism. 
  Now, choose $m$ large enough so  that the action of $Fr^m$ on $Z_1$ factors as $Z_1\lomapr{F_m} Z^0_1\hookrightarrow Z_1$. Take  $\pi:=(F_m\theta_X,F_{S,m}\theta_1,m)\in\Psi_{\xi^0}$. It is easy to see, using the uniqueness statement from Theorem \ref{section1},  that the associated  map $  {\epsilon}^{\hk}_{\st,\xi^0,\pi}$  equals $  {\epsilon}^{\hk}_{\st,l}$. In particular, the map ${\epsilon}^{\hk}_{\st}$ is a strict quasi-isomorphism, as wanted.
  
  The final claim  of the theorem
 follows since $ {\epsilon}^{\hk}_{\st,\lambda}=  s^*_{\lambda} {\epsilon}_{\st,l}^{\hk}. $

\subsubsection{Comparison between Hyodo-Kato and de Rham cohomologies}
Theorem \ref{HK-crr} implies the following Hyodo-Kato-to-de Rham quasi-isomorphisms:
     \begin{corollary}
     \label{local1}
We  have   natural strict \index{epshk@\epshk}quasi-isomorphisms
\begin{align}
\label{zaby1}
& {\epsilon}^{\hk}_{\dr}:\quad \rg_{\hk}(X^0_1)_{\Q_p}\wh{\otimes}^{\R}_{\breve{F}}C  \stackrel{\sim}{\to}\rg_{\crr}(X_1/\overline{S})_{\Q_p} \quad \mbox{ in } \sd(C_C),\\
& {\epsilon}^{\hk}_{\B^+_{\dr}}:\quad \rg_{\hk}(X^0_1)_{\Q_p}\wh{\otimes}^{\R}_{\breve{F}}\B^+_{\dr} \stackrel{\sim}{\to}\rg_{\crr}(X_1)_{\Q_p}\wh{\otimes}^{\R}_{\B^+_{\crr}}\B^+_{\dr} \quad \mbox{ in } \sd(C_{\B^+_{\dr}}).\notag
\end{align}
They are compatible via the maps $\theta: \B^+_{\dr}\to C$ and $\rg_{\crr}(X_1)\to \rg_{\crr}(X_1/\overline{S})$.
     \end{corollary}
     \begin{proof} From   Theorem \ref{HK-crr} 
        we have  a natural strict quasi-isomorphism in $\sd_{\phi,N}(C_{\breve{F}})$
\begin{equation}
\label{twisted1}
\epsilon^{\hk}_{\st}:\rg_{\hk}(X_1^0)_{\Q_p}\wh{\otimes}_{\breve{F},\iota}\B^+_{\st}\stackrel{\sim}{\to} \rg_{\crr}(X_1)_{\Q_p}\wh{\otimes}_{\B^+_{\crr},\iota}\B^+_{\st}. 
\end{equation}
Take the  map $\B^+_{\st}\to \B^+_{\crr}$ given by sending $\log(\lambda_p)\mapsto 0$. It is not Galois equivariant but this will not be a problem for us. Applying it to the quasi-isomorphism (\ref{twisted1}), which is $\B^+_{\st}$-linear,  we get a strict quasi-isomorphism in $\sd(C_{\B^+_{\crr}})$
\begin{equation}
\label{twisted12}
\epsilon^{\hk}_{\crr}:\quad \rg_{\hk}(X_1^0)_{\Q_p}\wh{\otimes}^{\R}_{\breve{F}}\B^+_{\crr}\stackrel{\sim}{\to} \rg_{\crr}(X_1)_{\Q_p}. 
\end{equation}
We tensor it now over $\B^+_{\crr}$ with C. By Lemma \ref{trick1}, 
 we obtain the  strict quasi-isomorphism in $\sd(C_C)$
 \begin{equation}
\label{twisted13}
\tilde{\epsilon}^{\hk}_{\dr}:\quad \rg_{\hk}(X_1^0)_{\Q_p}\wh{\otimes}^{\R}_{\breve{F}}C\stackrel{\sim}{\to} \rg_{\crr}(X_1)_{\Q_p}\wh{\otimes}^{\LL}_{\B^+_{\crr}}C
\end{equation}
and, composing with the strict quasi-isomorphism in $\sd(C_C)$
$$
 \rg_{\crr}(X_1)_{\Q_p}\wh{\otimes}^{\LL}_{\B^+_{\crr}}C\stackrel{\sim}{\to} \rg_{\crr}(X_1/\overline{S})_{\Q_p},
$$
the quasi-isomorphism $\epsilon^{\hk}_{\dr}$ from our corollary. 
We note that $\epsilon^{\hk}_{\dr}$  is compatible with the Galois action
because $\sigma(\log(\lambda_p))-\log(\lambda_p)\in{\rm Ker}\,\theta$.

   Proceeding as above we get 
   the  strict quasi-isomorphism in $\sd(C_{\B^+_{\crr}})$
 \begin{equation}
\tilde{\epsilon}^{\hk}_{\dr}:\quad \rg_{\hk}(X_1^0)_{\Q_p}\wh{\otimes}^{\R}_{\breve{F}}(\B^+_{\crr}/F^i)\stackrel{\sim}{\to} \rg_{\crr}(X_1)_{\Q_p}\wh{\otimes}^{\LL}_{\B^+_{\crr}}(\B^+_{\crr}/F^i),\quad i\geq 0.
\end{equation}
Taking $\R\wlim_i$ of both sides gives us now the second strict quasi-isomorphism of the theorem. 
      \end{proof}

\section{$\B^+_{\dr}$-cohomology} 
This section is devoted to the definitions of rigid analytic 
and overconvergent $\B^+_{\dr}$-cohomologies $\rg_{\rm dR}(X/\B^+_{\dr})$, for $X\in{\rm Sm}_C$
or $X\in{\rm Sm}^\dagger_C$, and to the study of their basic properties. 
These cohomologies are replacements for $\rg_{\rm dR}(X)\wh{\otimes}_C^\R\B^+_{\dr}$ which does not
exist since there is no continuous ring morphism $C\to \B^+_{\dr}$ although $\ovk$ is naturally
a subring of $\B^+_{\dr}$: if $X$ is defined over $K$, then
$\rg_{\rm dR}(X/\B^+_{\dr})\simeq \rg_{\rm dR}(X)\wh{\otimes}_K^\R\B^+_{\dr}$.
In general, we have the relation  $\rg_{\rm dR}(X/\B^+_{\dr})\wh{\otimes}_{\B^+_{\dr}}^{\R}\C\simeq \rg_{\rm dR}(X)$
(see Proposition~\ref{projection}
and Proposition~\ref{prison} for this comparison and analogous results concerning filtrations).

In the next chapter, using the Hyodo-Kato map, we will prove that, 
if $X\in {\rm Sm}_C$ is partially proper, then the rigid analytic
and overconvergent $\B^+_{\dr}$-cohomologies give the same result:
if $X^\dagger$ is the associated dagger variety,
the natural map $\rg_{\rm dR}(X^\dagger/\B^+_{\dr})\to \rg_{\rm dR}({X}/\B^+_{\dr})$ is a
strict quasi-isomorphism (Corollary~\ref{roznosci1}).

Our rigid analytic $\B^+_{\dr}$-cohomology is  defined by, locally, Hodge-completing absolute crystalline cohomology, but it  gives
the same object (see Proposition~\ref{BMSG}) 
as the  constructions of Bhatt-Morrow-Scholze~\cite{BMS1} and Guo~\cite{Guo}  via
the infinitesimal site.

\subsection{CliffsNotes}
  For a quick reference, we will recall now some results from \cite{CN3} and add few complements. 
  \subsubsection{Review} We start with a review of \cite{CN3}. 
 \begin{proposition}{\rm (Colmez-Nizio\l, \cite[Th. 1.1]{CN3})}
 \label{main00}
 \begin{enumerate}
 \item {\bf Dagger varieties}:
 To any smooth dagger  variety $X$ over $L=K,C$ there are  naturally  \index{rgproet@\rgproet}associated:
 \begin{enumerate}
 \item A pro-\'etale cohomology $\rg_{\proeet}(X,\Q_p(r))\in \sd(C_{\Q_p})$, $r\in \Z$.
 \item For $L=C$, a $\ovk$-valued rigid 
\index{rgrig@\rgrig}cohomology $\rg_{\rig,\ovk}(X)\in \sd(C_{\ovk})$ and a natural strict quasi-isomorphism\footnote{See \cite[Prop. 5.20]{CN3} for the definition of the tensor product.} in $\sd(C_{\ovk})$
 $$
 \rg_{\rig,\ovk}(X)\wh{\otimes}^{\R}_{\ovk}C\simeq \rg_{\dr}(X).
 $$
 This defines a natural $\ovk$-structure on the de Rham cohomology\footnote{By the same procedure one can define a $F^{\nr}$-valued rigid cohomology $\rg_{\rig,F^{\nr}}(X)$ and a natural strict quasi-isomorphism
  $ \rg_{\rig,F^{\nr}}(X)\wh{\otimes}^{\R}_{F^{\nr}}C\simeq \rg_{\dr}(X).
 $}.
\item A Hyodo-Kato \index{rghk@\rghk}cohomology\footnote{To distinguish this overconvergent Hyodo-Kato cohomology -- which was  defined by Grosse-Kl\"onne -- from the Hyodo-Kato cohomology defined later in this paper we will add the subscript $\gk$ to the former. Similarly, we will distinguished the induced  overconvergent syntomic cohomology.}
$\rg^{\gk}_{\hk}(X)\in \sd_{\phi,N}(C_{F_L})$, where $F_L=F$   if $L=K$
and   $F_L=F^{\nr}$if $L$=$C$.
For $L=C$, we have  natural Hyodo-Kato strict \index{iotahk@\iotahk}quasi-isomorphisms\footnote{See \cite[Sec. 5.3.3]{CN3} for the definition of  tensor products.} in, resp.,  $\sd(C_{\ovk}), \sd(C_C)$
$$
 \iota_{\hk}: \rg^{\gk}_{\hk}(X)\wh{\otimes}_{F^{\nr}}\ovk\stackrel{\sim}{\to} \rg_{\rig,\ovk}(X),\quad  \iota_{\hk}: \rg^{\gk}_{\hk}(X)\wh{\otimes}^{\R}_{F^{\nr}}C\stackrel{\sim}{\to} \rg_{\dr}(X).
 $$
  \item For $L=K$, a syntomic cohomology $\R\Gamma^{\gk}_{\synt}(X,\Q_p(r))\in \sd(C_{\Q_p})$, 
\index{rgsyn@\rgsyn}$r\in\N$, that fits into a distinguished triangle
  $$
  \R\Gamma^{\gk}_{\synt}(X,\Q_p(r))\lomapr{} [\rg_{\hk}^{\rm GK}(X)]^{N=0,\phi=p^r}\lomapr{\iota_{\hk}} \rg_{\dr}(X)/F^r, 
  $$
 and  
  a natural period \index{alphar@\alphar}map in $\sd(C_{\Q_p})$
    $$
  \alpha_r: \R\Gamma^{\gk}_{\synt}(X,\Q_p(r))\to \R\Gamma_{\proeet}(X,\Q_p(r)).
    $$
 It is  a strict quasi-isomorphism after truncation $\tau_{\leq r}$.
 \item {\rm (Local-global compatibility)} In the case $X$ has a semistable weak formal model the above constructions are compatible with their analogs defined using the model. 
 \end{enumerate}
 \item  {\bf Rigid analytic varieties}: To any smooth rigid analytic  variety $X$ over $L=K,C$ there are  naturally  associated:
 \begin{enumerate} 
 \item For $L=C$, a $\ovk$-valued convergent cohomology $\rg_{\conv,\ovk}(X)\in \sd(C_{\ovk})$ 
\index{rgconv@\rgconv}and a natural strict quasi-isomorphism in $\sd(C_C)$
  $$
 \rg_{\conv,\ovk}(X)\wh{\otimes}^{\R}_{\ovk}C\simeq \rg_{\dr}(X).
 $$
 This defines a natural $\ovk$-structure on the de Rham cohomology.
 \item A Hyodo-Kato \index{rghk@\rghk}cohomology 
$\rg_{\hk}(X)\in \sd_{\phi,N}(C_{F_L})$. For $L=C$, we have  natural Hyodo-Kato strict 
\index{iotahk@\iotahk}quasi-isomorphisms in, resp., $\sd(C_{\ovk}), \sd(C_C)$
$$
 \iota_{\hk}: \rg_{\hk}(X)\wh{\otimes}_{F^{\nr}}\ovk\stackrel{\sim}{\to} \rg_{\conv,\ovk}(X),\quad  \iota_{\hk}: \rg_{\hk}(X)\wh{\otimes}^{\R}_{F^{\nr}}C\stackrel{\sim}{\to} \rg_{\dr}(X).
 $$
 \item For $L=K,C$, a natural period \index{alphar@\alphar}map in $\sd(C_{\Q_p})$
    $$
  \alpha_r: \R\Gamma_{\synt}(X,\Q_p(r))\to \R\Gamma_{\proeet}(X,\Q_p(r)).
    $$
 It is a  strict quasi-isomorphism after truncation $\tau_{\leq r}$.
 \item {\rm(Local-global compatibility)} In the case $X$ has a semistable  formal model the  constructions in {\rm (a), (b)} are compatible with their analogs defined using the model. This is also the case in {\rm (c)}, for $L=K$.
\end{enumerate}
 \item {\bf Compatibility}: For $L=K,C$, let $X$ be a smooth dagger variety over $L$ and let $\wh{X}$ denote its completion. Then:
 \begin{enumerate}
 \item There exists a natural \index{iotahk@\iotahk}map \cite[Sec.\,3.2.4]{CN3} in $\sd(C_{\Q_p})$
 $$
 \iota_{\proeet}: \rg_{\proeet}(X,\Q_p(r))\to \rg_{\proeet}(\wh{X},\Q_p(r))\quad r\in \Z.
 $$
 It is a strict quasi-isomorphism if $X$ is partially proper. 
 \item There exists  natural maps in, resp.,    $\sd_{\phi,N}(C_{F_L})$, $ \sd\sff(C_{F_L})$
$$
  \rg^{\gk}_{\hk}(X)\to \rg_{\hk}(\wh{X}),\quad  \rg_{\dr}(X)\to \rg_{\dr}(\wh{X}).
$$
Here  $\sd\sff(C_{F_L})$ is the  filtered  $\infty$-category of $\sd(C_{F_L})$\footnote{Recall that, for a stable $\infty$-category $\scc$ having sequential limits, the  filtered $\infty$-category $\sd\sff(\scc)$  was  defined in  \cite[Thm. 2.5]{GP18}. It is a stable $\infty$-category.}.
{\rm If $X$ is partially proper},
 the second map is a strict quasi-isomorphism; the first map is a quasi-isomorphism if $L=K$ or 
if $X$ comes from a dagger variety defined over a finite extension of $K$. 
\item For $L=K$, there is a natural map in $\sd(C_{\Q_p})$
$$
\iota^{\rm GK}: \rg^{\gk}_{\synt}(X,\Q_p(r))\to \rg_{\synt}(\wh{X},\Q_p(r))
$$
and  the following diagram commutes
$$
\xymatrix@R=6mm{
\rg^{\gk}_{\synt}(X,\Q_p(r))\ar[r]^{\alpha_r}\ar[d]^{\iota^{\rm GK}} 
& \rg_{\proeet}(X,\Q_p(r))\ar[d]^{\iota_{\proeet}}\\
\rg_{\synt}(\wh{X},\Q_p(r))\ar[r]^{\alpha_r} & \rg_{\proeet}(\wh{X},\Q_p(r))
}
$$
\end{enumerate}
\end{enumerate}
 \end{proposition}
 \begin{remark}
 (i) Below, in Section \ref{period-dagger-section},  we will define the overconvergent period map in 1d over $C$ and, in Proposition \ref{lwiatko11},  we will remove the condition $L=K$ in  3c. 
To do this we could not use the constructions from \cite{CDN3} and \cite{CN3}: the first one was not functorial enough, the second one, using a ``killing nilpotents'' trick, 
just did not transfer to the geometric setting.  
This depressing state of affairs
made us take a break of more than a year from the project before coming back to it with an
  approach that adapts to the analytic setting
an early construction by  Beilinson of the Hyodo-Kato quasi-isomorphism.
 
 (ii) The local-global compatibility for rigid analytic geometric syntomic cohomology also holds. This will be proved in Proposition \ref{synt-locglob} using local-global compatibility for Hyodo-Kato  and $\B^+_{\dr}$-cohomologies. 
 \end{remark}
 \subsubsection{Complements} \label{sketchy} Now we pass to complementary results. 
 
 (1) {\em $\eta$-\'etale descent.} The following proposition should have been included in \cite{CN3}. 
 \begin{proposition}
 Let $(\sbb,F)$ be a Beilinson base\footnote{Such a base was introduced by Beilinson in \cite[2.1]{BE0}; it is a slightly more general notion than that of a Verdier base which is commonly used.} of an essentially small site  $\sv$. Then: 
 \begin{enumerate}
 \item The functor $F:\sbb\to\sv$ is continuous.
 \item $F$ induces an equivalence of topoi
 $$
 {\rm Sh}(\sbb)\stackrel{\sim}{\to} {\rm Sh}(\sv).
 $$
 \item Let $\sd$ be a presentable  $\infty$-category. Then $F$ induces an equivalence of $\infty$-categories
 $$
 {\rm Sh}^{\rm hyp}(\sbb,\sd)\stackrel{\sim}{\to} {\rm Sh}^{\rm hyp}(\sv,\sd)
 $$
of hypersheaves. 
\end{enumerate}
 \end{proposition}
 \begin{proof}
 Claims (1) and (2) were shown in the proof of \cite[Prop. 2.2]{CN3}. 
 For claim (3), recall that, for a site $\scc$, the $\infty$-category of hypersheaves is defined as
 $$
  {\rm Sh}^{\rm hyp}(\scc,\sd) :={\rm Sh}^{\rm hyp}(\scc,{\rm Ani})\otimes \sd, 
 $$
 where ${\rm Ani}$ is the $\infty$-category of anima and $\otimes$ 
 denotes the tensor product of $\infty$-categories \cite[Sec. 4.8.1]{Lur1}. Hence it suffices to prove claim (3) for the $\infty$-category of anima and in that case it follows easily from (2) and the fact that the Brown-Joyal-Jardine model structure on simplicial presheaves presents the $\infty$-topos of hypercomplete sheaves (see \cite[Prop. 6.5.2.14]{Lur}). 
 \end{proof}
 \begin{remark} To lighten up the terminology, in the rest of the paper we will call "hypersheaves" "sheaves" and a "hypersheafification" a  "sheafification". 
 \end{remark}
\begin{remark}   The example most relevant for this paper is the following: $\sv={\rm Sm}_{C,{\eet}}$, 
the site of smooth rigid analytic varieties over $C$ equipped with the \'etale topology. $\sv$  has a Beilinson base $(\sm, F_{\eta})$, where $\sm$ is the category of basic semistable formal models 
\index{Mss@\Mss}$\sm^{{\rm ss}, b}_C$ or semistable formal models $\sm^{\rm ss}_C$  and $F_{\eta}$ is the forgetful functor $\sx\to \sx_{\eta}$ from formal schemes to their rigid analytic generic fibers (see \cite[Prop. 2.8]{CN3}). We have similar constructions for the site $\sv={\rm Sm}^{\dagger}_{C,\eet}$ of smooth dagger varieties over $C$ with the corresponding categories $\sm^{\dagger, {\rm ss}, b}_C$ and $\sm^{\dagger, \rm ss}_C$ of basic semistable and semistable weak formal models.

 If $\sff\in\sd$, for a presentable $\infty$-category $\sd$,  is a presheaf on a Beilinson base $\sbb$, then the presheaf on $\sv$ defined by 
 $$
 U\mapsto (\sff^a(U):=\LL \colim \sff(V_{\jcdot})),
 $$
 where the colimit is taken over hypercoverings\footnote{Here and below, to simplify notation,  we write $\sff(V_{\jcdot})$ for $\lim\sff(V_{\jcdot})$.} $V_{\cdot}\to U$  from $\sbb$, 
 defines a hypersheaf on $\sv$. In the context of the above example of a Beilinson  base we call it {\em $\eta$-\'etale descent of $\sff$}. 
 \end{remark}
 
 (2)  The following  corollary removes the condition $L=K$ in  3b of Proposition \ref{main00} and could have been included in \cite{CN3}. 
 \begin{corollary}
 \label{ogrod1}
 Let $X$ be a smooth partially proper dagger variety over $C$ and let $\wh{X}$ denote its completion. Let $W$ be a Fr\'echet space over $\breve{F}$. Then the natural map\footnote{See the point (4) below for the reminder on the definition of the tensor products used.} in $\sd(C_{\breve{F}})$
$$
  \rg^{\gk}_{\hk}(X)\wh{\otimes}^{\R}_{F^{\nr}}W\to \rg_{\hk}(\wh{X})\wh{\otimes}^{\R}_{F^{\nr}}W
$$
is a strict quasi-isomorphism. 
 \end{corollary}
 \begin{proof} 
 Find an admissible covering of $X$ by dagger affinoids and then look at the set of their naive interiors (a {\em naive interior} of a smooth dagger  affinoid is a Stein subvariety whose complement is open and quasi-compact\footnote{We have an analogous definition of a naive interior of a rigid analytic affinoid.}). By the definition of partially proper dagger varieties this is an admissible  covering of $X$ as well. The individual varieties in the covering are partially proper and, moreover, are  defined over a finite extension of $K$. The latter fact is true because   the corresponding rigid affinoids are defined over a finite extension of $K$
 by Elkik's theorem \cite[Th. 7, Rem. 2]{Elk} (the finite presentation condition there is satisfied in our case by the finiteness theorems of Grauert-Remmert-Gruson and Gruson-Raynaud \cite[Th. 3.1.17, Th. 3.2.1]{Lit}).  Same can be said about the intersections of a  finite number of them. 
 
 Now, taking the associated \v{C}ech cover and evaluating on it  the morphism from the corollary we get a strict quasi-isomorphism by point 3b of Proposition \ref{main00}. We conclude by rigid analytic descent.
 \end{proof}

    (3) We will recall now a result from \cite{CN3} together with a new proof (since the proof supplied in loc. cit. is a bit  sketchy). This proof will serve us as a template for  proofs of analogous claims.
    \begin{proposition}{\rm(Local-global compatibility, \cite[Prop. 4.23]{CN3})}\label{stara1}
  Let  $\sx\in \sm^{\sem,b}_C$.
The natural map in $\sd(C_{\ovk})$
  \begin{align*}
 \rg_{\conv,\ovk}(\sx_1)\to \rg_{\conv,\ovk}(\sx_C)
  \end{align*}
  is a  strict quasi-isomorphism. 
\end{proposition}
    \begin{proof} It suffices to show that, for any $\eta$-\'etale hypercovering $\su_{\cdot}$ of $\sx$  from $\sm^{\sem, b}_C$, the natural map in $\sd(C_{\ovk})$
$$
 \R\Gamma_{\conv,\ovk}(\sx_1)\to  \R\Gamma_{\conv,\ovk}(\su_{\cdot,1})
$$
is a strict quasi-isomorphism (modulo taking a refinement of $\su_{\cdot}$). We may assume that in every degree of the hypercovering we have a finite number of formal models. Passing to cohomology ($\wt{H}(-)$-cohomology) and then to a truncated hypercovering we can assume that all the formal schemes mod $p$  and maps between them that are involved  are defined over a common field $L$ (we will denote them with subscript $\so_L$), a finite extension of $K$. We may leave that way the category of semistable models but we will still be in the category of log-smooth models (with Cartier type reduction).
We are reduced to showing that the map 
  \begin{equation}
  \label{wreszcie}
 \alpha:  \R\Gamma_{\conv}(\sx_{\so_L,1}/S_L)\to  \R\Gamma_{\conv}(\su_{\cdot,\so_{L}, 1}/S_L)
\end{equation}
is a strict quasi-isomorphism.  

     Tensoring both sides of (\ref{wreszcie}) with $C$ over $L$ we obtain a commutative diagram
$$
\xymatrix@R=6mm{
 \R\Gamma_{\conv}(\sx_{\so_L,1}/S_L)\wh{\otimes}^{\R}_LC\ar[d]^{\wr}\ar[r]^{\alpha_C} &  \R\Gamma_{\conv}(\su_{\cdot,{\so_L},1}/S_L)\wh{\otimes}^{\R}_LC \ar[d]^{\wr}\\
 \R\Gamma_{\dr}(\sx_{C})\ar[r]^{\sim} &  \R\Gamma_{\dr}(\su_{\cdot,C}).
 }
$$
Since the bottom map is a 
strict quasi-isomorphism (by \'etale descent) so is the top map $\alpha_C$. We claim that this, in turn, implies that the map $\alpha$ itself is a strict-quasi-isomorphism. Indeed, passing to fibers of the horizontal arrows in the commutative  diagram
$$
\xymatrix@R=6mm{
 \R\Gamma_{\conv}(\sx_{\so_L,1}/S_L)\ar[d]\ar[r]^{\alpha} &  \R\Gamma_{\conv}(\su_{\cdot,{\so_L},1}/S_L)\ar[d]\\
 \R\Gamma_{\conv}(\sx_{\so_L,1}/S_L)\wh{\otimes}^{\R}_LC\ar[r]^{\alpha_C} &  \R\Gamma_{\conv}(\su_{\cdot,{\so_L},1}/S_L)\wh{\otimes}^{\R}_LC, 
}
$$
 we see that it suffices to prove the following claim: 
\begin{center}
if $A\in \sd(C_L)$ is a complex such that  $A\wh{\otimes}^{\R}_LC$ is strictly acyclic then $A$ is  strictly acyclic as well.
\end{center}
To show this,  write $C\simeq L\oplus W$, for  a Banach space $W\in C_L$, and conclude. 
    \end{proof}
  
   (4) Let $X$ be a smooth rigid analytic  variety. In \cite{CN3}, we have considered a number of nonstandard tensor products. For example, 
   we have defined (see \cite[4.21]{CN3}) in, resp., $\sd(C_{\ovk})$, $\sd(C_C)$:
 \begin{align*}
  & \R\Gamma_{\hk}(X)\wh{\otimes}_{F^{\nr}}{\ovk}:=\LL\colim((\R\Gamma_{\hk}{\otimes}_{F^{\nr}}{\ovk})(\su_{\cdot,1})),\\
 & \R\Gamma_{\hk}(X)\wh{\otimes}^{\R}_{F^{\nr}}C:=\LL\colim((\R\Gamma_{\hk}\wh{\otimes}^{\R}_{F^{\nr}}{C})(\su_{\cdot,1})),\notag
 \end{align*}
 where the homotopy colimit is taken over affinoid $\eta$-\'etale hypercoverings $\su_{\cdot}$ from $\sm^{\sem, b}_C$. These tensor products satisfy local-global compatibility. A fact  that  can be proved as in the following example:
\begin{proposition}\label{product11}{\rm (Local-global compatibility for tensor products)} Let $\sx\in \sm^{\sem, b}_C$. The canonical maps in, resp., $\sd(C_{\ovk})$, $\sd(C_C)$
\begin{align*}
 & \R\Gamma_{\hk}(\sx_1)\wh{\otimes}_{F^{\nr}}{\ovk} \to \R\Gamma_{\hk}(X)\wh{\otimes}_{F^{\nr}}{\ovk},\\
  & \R\Gamma_{\hk}(\sx_1)\wh{\otimes}^{\R}_{F^{\nr}}C \to \R\Gamma_{\hk}(X)\wh{\otimes}^{\R}_{F^{\nr}}C
\end{align*}
are strict quasi-isomorphisms.
\end{proposition}
\begin{proof}For the second morphism, proceeding as in the proof of Proposition \ref{stara1} and using its notation,  we reduce to showing that the canonical map in $\sd(C_C)$
$$
\R\Gamma_{\hk}(\sx_{\so_L,1})\wh{\otimes}^{\R}_{F_L}C  \to \R\Gamma_{\hk}(\su_{\cdot, \so_L,1})\wh{\otimes}^{\R}_{F_L}C
$$
is a strict quasi-isomorphism. But this is clear since, via the Hyodo-Kato morphism, this map is strictly quasi-isomorphic to the map in $\sd(C_C)$
$$
\R\Gamma_{\dr}(\sx_{C}) \to \R\Gamma_{\dr}(\su_{\cdot,C}),
$$
which is a strict quasi-isomorphism. 

 For the first morphism, we proceed in the same way ending up with the  strict quasi-isomorphism in $\sd(L)$ 
 $$
\R\Gamma_{\dr}(\sx_{L}) \to \R\Gamma_{\dr}(\su_{\cdot,L}).
$$
Passing to homotopy colimit over finite extensions of $L$ in $\ovk$, we finish the argument.
\end{proof}
\begin{remark} (1) The local-global compatibility of nonstandard tensor products (see \cite[5.16]{CN3}) also holds for the dagger varieties and the Grosse-Kl\"onne Hyodo-Kato cohomology. The proof  of this fact is a simple analog of the proof of Proposition \ref{product11}.

 (2) In Proposition \ref{product11} we can replace $C$ with any Fr\'echet space $B$ over $\breve{F}$.  This requires just a slight modification of the proof: pass from $\rg_{\hk}(-)\wh{\otimes}^{\R}_{F_L}B$ to 
 $(\rg^{\rm GK}_{\hk}(-)\wh{\otimes}^{\R}_{F_L}C)\wh{\otimes}^{\R}_{F_L}B$, 
 use the Hyodo-Kato morphism to pass to de Rham cohomology $\rg_{\dr}((-)_C)\wh{\otimes}^{\R}_{F_L}B$, use \'etale descent for de Rham cohomology, and go back to $\rg_{\hk}(-)\wh{\otimes}^{\R}_{F_L}B$ via $C\simeq F_L\oplus W$.
\end{remark}
\subsection{Geometric crystalline cohomology} Our rigid analytic $\B^+_{\dr}$-cohomology will be defined locally as a completion of the absolute crystalline cohomology. We will start then by recalling the definition of  the latter. 
\subsubsection{Relative crystalline cohomology}  Let  ${\overline{f}}:X_1\to \overline{S}_1$ be a map of log-schemes, with integral quasi-coherent source.
 Suppose  that $\overline{f}$ is the base change of 
 a fine log-smooth log-scheme  ${f}_L:Z_{1}\to S_{L,1}$, by the natural map $\theta: \overline{S}_1\to S_{L,1}$, for a finite extension $L/K$.
 That is, we have a map $\theta_{L}: X_1\to Z_{1}$ such that the square $(\overline{f},{f}_L,\theta,\theta_{L})$ is Cartesian.
Such data $\Sigma_1:=\{(L,f_L,\theta_{L})\}$  form a filtered set.

     (a) {\em $C$-version.} Let \index{acalcris@\acalcris}$\sa^{\rm rel}_{\crr}$ 
be the $\eta$-\'etale sheafification of the presheaf
  $\sx\mapsto \rg_{\crr}(\sx/\overline{S})_{\Q_p}$ on $\sm_{C}^{\sem,b}$. Note that $\rg_{\crr}(\sx/\overline{S})_{\Q_p}\in \sd\sff(C_C)$.  For $X\in {\rm Sm}_C$, we 
\index{rgcris@\rgcris}set\footnote{Here we think of $X$ as an $\eta$-\'etale sheaf on $\sm_{C}^{\sem,b}$.} in $\sd\sff(C_C)$
$$\rg^{\rm rel}_{\crr}(X):=\rg_{\eet}(X,\sa^{\rm rel}_{\crr}).$$ 
It is a filtered dg $C$-algebra  equipped with a continuous action of
 $\sg_K$ if $X$ is defined over $K$. It is equipped with the topology induced   from the topology of the $\R\Gamma_{\crr}(\sx/\overline{S})$'s. Since the models $\sx$ are log-smooth over $\overline{S}$, we have natural (strict) 
\index{acaldr@\acaldr}quasi-isomorphisms (the first one in the category of sheaves with values in $\sd\sff(C_C)$, the second one -- in $\sd\sff(C_C)$).
 \begin{equation}
 \label{deszcz1}
 \sa^{\rm rel}_{\crr}\simeq \sa_{\dr},\quad \rg^{\rm rel}_{\crr}(X)\simeq\rg_{\dr}(X).
 \end{equation}
 
     (b) {\em $\ovk$-version.} Let $\sa_{\crr,\ovk}$ 
\index{acalcris@\acalcris}be the $\eta$-\'etale sheafification of the presheaf 
$\sx\mapsto \rg_{\crr,\ovk}(\sx)$ on $\sm_{C}^{\sem,b}$, where we 
\index{rgcris@\rgcris}set in $\sd(C_{\ovk})$
  $$ \R\Gamma_{\crr,\ovk}(\sx):= \LL\colim_{\Sigma_1}\R\Gamma_{\crr}(Z_1/S_L)_{\Q_p}.
  $$
For $X\in {\rm Sm}_C$, we set $\rg_{\crr,\ovk}(X):=\rg_{\eet}(X,\sa_{\crr,\ovk})$ in $\sd(C_{\ovk})$. It is a dg $\ovk$-algebra  equipped with a continuous action of
 $\sg_K$ if $X$ is defined over $K$ (this action is smooth if $X$ is quasi-compact). It is equipped  with the topology induced  from the topology of the $\R\Gamma_{\crr}(Z_1/S_L)$'s. 
 There are natural continuous morphisms  (the first one in the category of sheaves with values in $\sd(C_{\ovk})$, the second one -- in $\sd(C_{\ovk})$)
$$
  \sa_{\crr,\ovk}  \to \sa^{\rm rel}_{\crr},\quad \rg_{\crr,\ovk}(X)\to \rg^{\rm rel}_{\crr}(X).
$$
\begin{lemma}\label{etale-descent201}  \begin{enumerate}
\item {\rm ({\em Local-global compatibility})} Let  $\sx\in\sm^{{\rm ss},b}_C$.
The natural map in $\sd(C_{\ovk})$
$$
\rg_{\crr,\ovk}(\sx)\stackrel{\sim}{\to} \rg_{\crr,\ovk}(\sx_C)
$$
is a strict quasi-isomorphism.
\item {\rm ({\em Product formula})} For $X\in {\rm Sm}_K$, the natural map in $\sd(C_{\ovk})$
$$
\rg_{\dr}(X)\wh{\otimes}_K\ovk\to  \rg_{\crr,\ovk}(X_C)
$$
is a  strict quasi-isomorphism.
\end{enumerate}
\end{lemma}
\begin{proof}  Since, for $\sy\in\sm^{{\rm ss},b}_C$,  the natural map in $\sd(C_{\ovk})$
$$
\rg_{\conv,\ovk}(\sy)\to \rg_{\crr,\ovk}(\sy)
$$
is a  strict quasi-isomorphism, it induces a strict quasi-isomorphism  in $\sd(C_{\ovk})$
$$
\rg_{\conv,\ovk}(X)\stackrel{\sim}{\to} \rg_{\crr,\ovk}(X), \quad X\in {\rm Sm}_C.
$$
 Hence 
our lemma  follows from  analogous claims  for convergent $\ovk$-cohomology which are  known (see \cite[the proof of Prop. 4.23]{CN3} or Proposition \ref{stara1}).
\end{proof}
\subsubsection{Absolute crystalline cohomology}  \label{absolute-crr}  
Let $\sx\in \sm^{\rm ss}_C$.
  Recall that we have the absolute crystalline cohomology $ \rg_{\crr}(\sx)_{\Q_p}\in \sd\sff_{\phi}(C_{\B^+_{\crr}})$  equipped with  the Hodge  filtration  $F^r\rg_{\crr}(\sx)_{\Q_p}:=\rg_{\crr}(\sx,\sj^{[r]})_{\Q_p}$, for $r\geq 0$.
 \index{acalcris@\acalcris}Let $\sa_{\crr}$ and $F^r\sa_{\crr}$, $r\geq 0$,  
be the $\eta$-\'etale sheafifications of the presheaves
  $\sx\mapsto \rg_{\crr}(\sx)_{\Q_p}$ and  $\sx\mapsto \rg_{\crr}(\sx,\sj^{[r]})_{\Q_p}$, respectively, on $\sm_{C}^{\sem}$.
    For $X\in {\rm Sm}_C$, we \index{rgcris@\rgcris}set in $\sd\sff_{\phi}(C_{\B^+_{\crr}})$
  $$\rg_{\crr}(X):=\rg_{\eet}(X,\sa_{\crr}),\quad F^r\rg_{\crr}(X):=\rg_{\eet}(X,F^r\sa_{\crr}),\quad r\geq 0.
  $$ It is a dg filtered $\B^+_{\crr}$-algebra  equipped with a continuous action of
 $\sg_K$ if $X$ is defined over $K$.
It is equipped  with the topology induced   from the topology of the $ \rg_{\crr}(\sx,\sj^{[r]})_{\Q_p}$'s.
 
   The local-global comparison requires the Hyodo-Kato quasi-isomorphism and will be proven  in Lemma \ref{abs-crr}  below (just the nonfiltered case). 
\subsection{Rigid analytic $\B^+_{\dr}$-cohomology}  We will define now rigid analytic $\B^+_{\dr}$-cohomology,  list its basic properties, and compare it with already existing definitions.  
 \subsubsection{Definition of rigid analytic $\B^+_{\dr}$-cohomology}   \label{defcr}
Let $\sx\in \sm^{\rm ss}_C$.  To define rigid analytic $\B^+_{\dr}$-cohomology, we start with   the absolute crystalline cohomology $ \rg_{\crr}(\sx)_{\Q_p}$ and complete it with respect to  the Hodge  filtration  $F^r\rg_{\crr}(\sx)_{\Q_p}$, $r\geq 0$: 
 \begin{align*}
 \rg_{\crr}(\sx)_{\Q_p}^{\bwedge}:=\R\wlim_r (\rg_{\crr}(\sx)_{\Q_p}/F^r), \quad  \rg_{\crr}(\sx,\sj^{[r]})_{\Q_p}^{\bwedge}:=\R\wlim_{j \geq r}(\rg_{\crr}(\sx,\sj^{[r]})_{\Q_p}/F^j).
\end{align*}
This is a  dg  filtered $\B^+_{\dr}$-algebra,  hence a complex in $\sd\sff(C_{\B^+_{\dr}})$.
The corresponding $\eta$-\'etale sheafifications on $\sm^{\rm ss}_C$ we will denote \index{acalcris@\acalcris}by 
$F^r\sa^{\bwedge}_{\crr}$, $r\geq 0$.  We have  canonical maps
$\kappa: F^r\sa_{\crr}\to F^r\sa^{\bwedge}_{\crr}$, $r\geq 0$. Moreover, the canonical 
\index{thetavar@\thetavar}map $\vartheta: F^r \rg_{\crr}(\sx)_{\Q_p}\to F^r\rg_{\crr}(\sx/\overline{S})_{\Q_p}$, compatible with the map $\theta: \B^+_{\crr}\to C$,  extends to a map 
$\vartheta: F^r \rg_{\crr}(\sx)_{\Q_p}^{\bwedge}\to F^r\rg_{\crr}(\sx/\overline{S})_{\Q_p}$, which, in turn,  globalizes  to a map 
$\vartheta: F^r\sa_{\crr}^{\bwedge}\to F^r\sa^{\rm rel}_{\crr}$.

  For $X\in {\rm Sm}_C$,  define the 
\index{rgdr@\rgdr}$\B^+_{\dr}$-cohomology in $\sd\sff(C_{\B^+_{\dr}})$:
$$\rg_{\dr}(X/\B^+_{\dr}):=\rg_{\eet}(X,\sa^{\bwedge}_{\crr}),\quad F^r\rg_{\dr}(X/\B^+_{\dr}):=\rg_{\eet}(X,F^r\sa^{\bwedge}_{\crr}), \quad r\geq 0.
$$
This is a  dg  filtered $\B^+_{\dr}$-algebra,  equipped with a continuous action of
 $\sg_K$ if $X$ is defined over $K$. It is equipped  with the topology induced   from the topology of the $\rg_{\crr}(\sx,\sj^{[r]})^{\bwedge}_{\Q_p}$'s.
 
    The local-global comparison requires product formula and  will be proven  in Lemma \ref{etale-descent20}  below.

 We have canonical \index{thetavar@\thetavar}maps (the first one in $\sd\sff(C_{\B^+_{\dr}})$, the second one -- in $\sd(C_{\B^+_{\dr}})$)
\begin{align*}
 \kappa: \quad & \rg_{\crr}(X)\to  \rg_{\dr}(X/\B^+_{\dr}),\\
\vartheta: \quad & F^r\rg_{\dr}(X/\B^+_{\dr})\to F^r\rg_{\dr}(X),\quad r\geq 0.
\end{align*}
 It is  immediate from the definitions that the first map 
yields a  strict quasi-isomorphism in $\sd\sff(C_{\B^+_{\dr}})$
  $$
 \kappa: \quad  \rg_{\crr}(X)^{\bwedge}\stackrel{\sim}{\to} \rg_{\dr}(X/\B^+_{\dr}),
  $$
where we set $ {\rg_{\crr}(X)}^{\bwedge}:=\R\lim_r( \rg_{\crr}(X)/F^r)$ in $\sd\sff(C_{\B^+_{\dr}})$.

    \subsubsection{Comparison results.} 
    (1)  We start with a comparison of $\B^+_{\dr}$- and de Rham cohomologies. 
    
    (i) {\em Projection from $\B^+_{\dr}$-cohomology to de Rham cohomology}.
 \begin{proposition}\label{projection}
 Let $X\in {\rm Sm}_C$. 
 \begin{enumerate}
 \item
We have  a natural  strict quasi-isomorphism in $\sd\sff(C_C)$
 $$
 \vartheta:\quad \rg_{\dr}(X/\B^+_{\dr})\wh{\otimes}^{\R}_{\B^+_{\dr}}C \stackrel{\sim}{\to}\rg_{\dr}(X).
 $$
 \item More generally, for $r\geq 0$, we have a natural distinguished triangle in $\sd(C_{\B^+_{\dr}})$
 \begin{equation}
 \label{detail1}
 F^{r-1}\rg_{\dr}(X/\B^+_{\dr})\lomapr{t} F^r\rg_{\dr}(X/\B^+_{\dr})\lomapr{\vartheta}F^r\rg_{\dr}(X)
 \end{equation}
 \item For $r\geq 0$, we have a natural distinguished 
\index{betax@\BETAXX}triangle in $\sd(C_{\B^+_{\dr}})$
 \begin{equation}
 \label{coffee1}
  F^{r+1}\rg_{\dr}(X/\B^+_{\dr})\to F^r\rg_{\dr}(X/\B^+_{\dr})\stackrel{\beta_X}{\to} \bigoplus_{i\leq r}\rg(X,\Omega^i_X)(r-i)[-i]
 \end{equation}
 \end{enumerate}
 \end{proposition}
  \begin{proof}  In the first claim, the tensor product is simply defined as the  cofiber of the map in $\sd(C_{\B^+_{\dr}})$
  $$
   \rg_{\dr}(X/\B^+_{\dr})\lomapr{t} \rg_{\dr}(X/\B^+_{\dr}).
  $$
  Hence it suffices to show that we have the distinguished triangle in $\sd(C_{\B^+_{\dr}})$
  $$
   \rg_{\crr}(X/\B^+_{\dr})\lomapr{t} \rg_{\crr}(X/\B^+_{\dr}) \lomapr{\vartheta} \rg_{\dr}(X). 
  $$
  \'Etale locally this translates into the  triangle in $\sd(C_{\B^+_{\dr}})$
$$
\rg_{\crr}(\sx)^{\bwedge}_{\Q_p}\lomapr{t} \rg_{\crr}(\sx)^{\bwedge}_{\Q_p}\lomapr{\vartheta} \rg_{\crr}(\sx/\overline{S})_{\Q_p},
$$
where  $\sx=\sx_{L,C}$, for a semistable affine model $\sx_{\so_L}$ over a finite extension $L$ of $K$. 

   This  triangle 
 fits into a commutative diagram:
\begin{equation}
\label{ncis1}
\xymatrix@R=6mm{
\rg_{\crr}(\sx)^{\bwedge}_{\Q_p}\ar[r]^{t} & \rg_{\crr}(\sx)^{\bwedge}_{\Q_p} \ar[r]^{\vartheta} & \rg_{\crr}(\sx/\overline{S})_{\Q_p}\\
\R\Gamma_{\dr}(\sx_L)\wh{\otimes}^{\R}_{L}{\B}^+_{\dr}\ar[r]^{1\otimes t} \ar[u]^{\wr}_{\hat{\iota}_{\rm BK}} &\R\Gamma_{\dr}(\sx_L)\wh{\otimes}^{\R}_{L}{\B}^+_{\dr}\ar[r]^{1\otimes\theta} \ar[u]^{\wr}_{\hat{\iota}_{\rm BK}} &
 \R\Gamma_{\dr}(\sx_L)\wh{\otimes}^{\R}_{L}C \ar[u]^{\wr}.
}
\end{equation}
Here the vertical 
\index{iotabk@\iotabk}maps 
$\hat{\iota}_{\rm BK}:=\R\wlim_r\iota_{{\rm BK},r} $, with  $\iota_{{\rm BK},r}$ defined as the composition
\begin{align}
\label{the-map}
\iota_{{\rm BK},r}:\quad  (\R\Gamma_{\dr}(\sx_L)\wh{\otimes}^{\R}_{L}{\B}^+_{\dr})/F^r& \stackrel{\sim}{\rightarrow}(\R\Gamma_{\crr}(\sx_{\so_L}/S_L)_{\Q_p}\wh{\otimes}^{\R}_L\R\Gamma_{\crr}(\overline{S}/S_L)_{\Q_p})/F^r\\
 &  \xrightarrow[\sim]{\cup} 
 \R\Gamma_{\crr}(\sx/S_L)_{\Q_p}/F^r 
     \stackrel{\sim}{\leftarrow}\R\Gamma_{\crr}(\sx)_{\Q_p}/F^r,\notag
\end{align}
where we set 
\begin{equation}
\label{expression}
 F^r(\R\Gamma_{\dr}(\sx_L)\wh{\otimes}_{L}^{\R}\B^+_{\dr}):=\R\wlim(\so(\sx_L)\wh{\otimes}^{\R}_LF^r\B^+_{\dr}\to \Omega^1(\sx_L)\wh{\otimes}^{\R}_LF^{r-1}\B^+_{\dr}\to\cdots).
 \end{equation}
 The first quasi-isomorphism in \eqref{the-map} follows from the fact that $\sx_{\so_L}$ is log-smooth over $\so_L$ and that, more generally,   derived de Rham complex computes crystalline cohomology for log-syntomic schemes (both Hodge completed) by \cite[1.9.2]{BEp}.  The second quasi-isomorphism is just log-smooth base change for crystalline cohomology (more explicitly, one can proceed as in \cite[Prop. 4.5.4]{Ts}). And the third quasi-isomorphism is a formal scheme version of \cite[Cor. 2.4]{NN} (the proof in {\em loc. cit.} goes through in our setting).

  The bottom row in diagram \eqref{ncis1}  is a distinguished triangle. It follows that the top row in our diagram is a distinguished triangle as well, as wanted. 

 The second claim, \'etale locally, reduces to showing that the triangle
$$
F^{r-1}\rg_{\crr}(\sx)^{\bwedge}_{\Q_p}\lomapr{t} F^r\rg_{\crr}(\sx)^{\bwedge}_{\Q_p}\lomapr{\vartheta} F^r\rg_{\crr}(\sx/\overline{S})_{\Q_p},
$$
where  $\sx=\sx_{L,C}$, for a semistable affine model $\sx_{\so_L}$ over a finite extension $L$ of $K$, is distinguished.
This  triangle 
 fits into a commutative diagram:
$$
\xymatrix@R=6mm{
F^{r-1}\rg_{\crr}(\sx)^{\bwedge}_{\Q_p}\ar[r]^{t} & F^r\rg_{\crr}(\sx)^{\bwedge}_{\Q_p} \ar[r]^{\vartheta} & F^r\rg_{\crr}(\sx/\overline{S})_{\Q_p}\\
F^{r-1}(\R\Gamma_{\dr}(\sx_L)\wh{\otimes}^{\R}_{L}{\B}^+_{\dr})\ar[r]^{1\otimes t} \ar[u]^{\wr}_{\hat{\iota}_{\rm BK}} &F^r(\R\Gamma_{\dr}(\sx_L)\wh{\otimes}^{\R}_{L}{\B}^+_{\dr})\ar[r]^{1\otimes\theta} \ar[u]^{\wr}_{\hat{\iota}_{\rm BK}} &
 F^r(\R\Gamma_{\dr}(\sx_L)\wh{\otimes}^{\R}_{L}C) \ar[u]^{\wr}.
}
$$
The bottom row is a distinguished triangle: use the expression (\ref{expression}) to reduce to showing that, for $r-1\geq i\geq 0$,  
we have a strict quasi-isomorphism
$$ t: \Omega^i(\sx_L)\wh{\otimes}^{\R}_LF^{r-1-i}\B^+_{\dr}\stackrel{\sim}{\to}\Omega^i(\sx_L)\wh{\otimes}^{\R}_LF^{r-i}\B^+_{\dr}
$$
and the triangle
$$
 \Omega^r(\sx_L)\wh{\otimes}^{\R}_L\B^+_{\dr}\lomapr{t}\Omega^r(\sx_L)\wh{\otimes}^{\R}_L\B^+_{\dr}\to \Omega^r(\sx_L)\wh{\otimes}^{\R}_LC
$$
is distinguished. But the first claim is clear and the  second claim was just proved in (1). 

      The third claim, \'etale locally, reduces to identifying the graded term in  the distinguished  triangle
$$
F^{r+1}\rg_{\crr}(\sx)^{\bwedge}_{\Q_p}\to  F^r\rg_{\crr}(\sx)^{\bwedge}_{\Q_p}\to \gr^r\rg_{\crr}(\sx/\overline{S})_{\Q_p},
$$
where  $\sx=\sx_{L,C}$, for a semistable affine model $\sx_{\so_L}$ over a finite extension $L$ of $K$. That is, we want to show that
$$
\gr^r\rg_{\crr}(\sx/\overline{S})_{\Q_p}\simeq \bigoplus_{i\leq r}\Omega^i(\sx_L)\wh{\otimes}^{\R}_Lt^{r-i}C[-i].
$$
The above   triangle 
 fits into a commutative diagram:
$$
\xymatrix@R=6mm@C=6mm{
F^{r+1}\rg_{\crr}(\sx)^{\bwedge}_{\Q_p}\ar[r] & F^r\rg_{\crr}(\sx)^{\bwedge}_{\Q_p} \ar[r] & \gr^r\rg_{\crr}(\sx/\overline{S})_{\Q_p}\\
F^{r+1}(\R\Gamma_{\dr}(\sx_L)\wh{\otimes}^{\R}_{L}{\B}^+_{\dr})\ar[r] \ar[u]^{\wr}_{\hat{\iota}_{\rm BK}} &F^r(\R\Gamma_{\dr}(\sx_L)\wh{\otimes}^{\R}_{L}{\B}^+_{\dr})\ar[r] \ar[u]^{\wr}_{\hat{\iota}_{\rm BK}} &
 \gr^r(\R\Gamma_{\dr}(\sx_L)\wh{\otimes}^{\R}_{L}\B^+_{\dr}) \ar[u]^{\wr}_{f_{\sx}}.
}
$$
Using expression (\ref{expression}) we get  
\begin{align*}
\gr^r(\R\Gamma_{\dr}(\sx_L)\wh{\otimes}_{L}^{\R}\B^+_{\dr}) & \xrightarrow[\sim]{g_{\sx}} \R\wlim(\so(\sx_L)\wh{\otimes}^{\R}_Lt^rC\stackrel{0}{\to} \Omega^1(\sx_L)\wh{\otimes}^{\R}_Lt^{r-1}C\stackrel{0}{\to}\cdots\stackrel{0}{\to} \Omega^r(\sx_L)\wh{\otimes}^{\R}_Lt^{0}C)\\
  & \simeq \bigoplus_{i\leq r}\Omega^i(\sx_L)\wh{\otimes}^{\R}_Lt^{r-i}C[-i],
\end{align*}
as wanted. The global map $\beta_X$ is defined by globalizing the local maps $\beta_{\sx}:=g_{\sx}f^{-1}_{\sx}$.
 \end{proof}
    \begin{remark}\label{gen-paris}
   (1)  The above proof shows that we have a  distinguished triangle
    $$
    \sa^{\bwedge}_{\crr}\lomapr{t}\sa^{\bwedge}_{\crr}\lomapr{\vartheta}\sa^{\rm rel}_{\crr}.
    $$
    
    (2)  The  maps $\hat{\iota}_{\rm BK}$ above can be defined in a more general set-up, where $\sx_{\so_L}$ is assumed to be just  log-syntomic over $S_L$. It is again a strict quasi-isomorphism and the proof of this claim is  not much different than in  the log-smooth case:  The fact that the second map 
in the definition of $\iota_{{\rm BK},r} $ in (\ref{the-map}) is a strict quasi-isomorphism  can be seen by unwinding both sides of the cup product map: one finds a K\"unneth morphism for certain de Rham complexes.  It is an integral  quasi-isomorphism because these complexes are "flat enough" which follows from the fact that the maps $\so_{C,n}\to \so_{L,n}$ and $\sx_{\so_L,n}\to\so_{L,n}$, for $n\geq 0$, 
are log-syntomic (see the proof of \cite[Prop. 4.5.4]{Ts} for a similar argument).  The third map in (\ref{the-map}) is a strict quasi-isomorphism (integrally,  a quasi-isomorphism up to a constant dependent on $L$) by an argument analogous to the one given in the proof of \cite[Cor. 2.4]{NN}. 
    \end{remark}
    
  (ii)    {\em Product formula}. Let $X\in {\rm Sm}_K$. The morphisms $\hat\iota_{{\rm BK}}$ from Remark \ref{gen-paris} 
  induce a 
\index{iotabk@\iotabk}morphism\footnote{See the proof of Lemma \ref{sroka1} for details.} in $\sd\sff(C_{\B^+_{\dr}})$
       $$
\iota_{\rm BK}:  \quad \rg_{\dr}(X)\wh{\otimes}^{\R}_K\B^+_{\dr}\to\rg_{\dr}(X_C/\B^+_{\dr}).
 $$
 \begin{lemma}\label{sroka1} The  morphism $\iota_{\rm BK}$
 is a   strict quasi-isomorphism  in $\sd\sff(C_{\B^+_{\dr}})$. 
 \end{lemma}
\begin{remark}\label{sroka12}
The filtration on $\R\Gamma_{\dr}(X)\wh{\otimes}_K^{\R}\B^+_{\dr}$ is defined by the formula
$$
F^r(\R\Gamma_{\dr}(X)\wh{\otimes}_{K}^{\R}\B^+_{\dr}):=\LL\colim(F^r(\R\Gamma_{\dr}\wh{\otimes}^{\R}_{K}\B^+_{\dr})(U_{\cdot})),
$$
 where the homotopy colimit is taken over \'etale affinoid hypercoverings $U_{\cdot}$ of  $X$ and, for an affinoid $U$, 
 $$
 F^r(\R\Gamma_{\dr}(U)\wh{\otimes}_{K}^{\R}\B^+_{\dr}):=\R\wlim(\so(U)\wh{\otimes}^{\R}_KF^r\B^+_{\dr}\to \Omega^1(U)\wh{\otimes}^{\R}_KF^{r-1}\B^+_{\dr}\to\cdots).
 $$
Since  $\R\Gamma_{\dr}(X)$ satisfies filtered \'etale descent\footnote{Use that the varieties are smooth and we have \'etale descent for the structure sheaf.}, it is easy to see that so does $\R\Gamma_{\dr}(X)\wh{\otimes}_K^{\R}\B^+_{\dr}$. 
  \end{remark}
 \begin{proof} We may argue \'etale locally and assume that $X =\sx_{L}$, for a semistable affine model $\sx_{\so_L}$, for  a finite extension $L$ of $K$. Then $X_C=\sx_L\times_KL$. We need to show that the map
 $$
 \iota_{\rm BK}:\quad \R\Gamma_{\dr}(\sx_{L})\wh{\otimes}^{\R}_K\B^+_{\dr}\to (\R\Gamma_{\crr}(\sx_{\so_L}/S_K)_{\Q_p}\wh{\otimes}^{\R}_K\B^+_{\dr})^{\bwedge}
  \xrightarrow[\sim]{\hat{\iota}_{\rm BK}}  \R\Gamma_{\crr}(\sx_{\so_L}\times_{\so_K}{\so_C})_{\Q_p}^{\bwedge}
$$
is a filtered strict quasi-isomorphism. For that, since the base change map 
$$
\R\Gamma_{\crr}(\sx_{\so_L}/S_K)_{\Q_p}^{\bwedge}\to \R\Gamma_{\crr}(\sx_{\so_L}/S_L)_{\Q_p}^{\bwedge}
$$
is a filtered strict quasi-isomorphism, 
it suffices to show that so is 
the canonical map
 \begin{equation}
 \label{completion20}
  \R\Gamma_{\crr}(\sx_{\so_L}/S_L)_{\Q_p}\wh{\otimes}^{\R}_K\B^+_{\dr} {\to} \R\wlim_r(\R\Gamma_{\crr}(\sx_{\so_L}/S_L)_{\Q_p}\wh{\otimes}^{\R}_K\B^+_{\dr})/F^r.
 \end{equation}

   But we can   write (the differentials are over $L$):
 \begin{align*}
   F^j(\R\Gamma_{\crr}(\sx_{\so_L}/S_L)_{\Q_p}\wh{\otimes}^{\R}_K\B^+_{\dr})& =\R\wlim(\so(\sx_L)\wh{\otimes}^{\R}_KF^j\B^+_{\dr}\to \Omega^1(\sx_L)\wh{\otimes}^{\R}_KF^{j-1}\B^+_{\dr}\to \cdots)\\
   F^j(\R\Gamma_{\crr}(\sx_{\so_L}/S_L)_{\Q_p}\wh{\otimes}^{\R}_K\B^+_{\dr})/F^r & = \R\wlim(\so(\sx_L)\wh{\otimes}^{\R}_K(F^j\B^+_{\dr}/F^r)\to \Omega^1(\sx_L)\wh{\otimes}^{\R}_K(F^{j-1}\B^+_{\dr}/F^{r-1})\to \cdots).
 \end{align*}
And now  we can argue degreewise. But then, for $s\geq 0$, we have 
 \begin{align*}
  \R\wlim_r (\Omega^i(\sx_L)\wh{\otimes}_K^{\R} (F^s\B^+_{\dr}/F^r))
   \simeq \Omega^i(\sx_L)\wh{\otimes}_K^{\R} \R\wlim_r(F^s\B^+_{\dr}/F^r)\simeq   \Omega^i(\sx_L)\wh{\otimes}^{\R}_KF^s\B^+_{\dr}.
 \end{align*}
 The second quasi-isomorphism follows from the fact that  we have $F^s\B^+_{\dr}/F^r\simeq C^{r-s}$ as  topological $K$-vector spaces and the maps $F^s\B^+_{\dr}/F^{r+i}\to F^s\B^+_{\dr}/F^r$, $i\geq 0$,  are surjective.
This finishes the proof.
 
 \end{proof}

  (2) Now we pass to  comparisons between $\B^+_{\dr}$-cohomology and crystalline cohomology.  
\begin{lemma}{\rm (Local-global compatibility)}\label{etale-descent20} 
 Let  $\sx\in\sm^{{\rm ss},b}_C$ and let $r\geq 0$.  The canonical  map in $\sd(C_{\B^+_{\dr}})$
  $$
\kappa:\quad F^r \rg_{\crr}(\sx)_{\Q_p}^{\bwedge}{\to}F^r\rg_{\dr}(\sx_C/{\B^+_{\dr}})
  $$
  is a strict quasi-isomorphism.
\end{lemma}
\begin{proof}  We may argue \'etale locally on $\sx$. Assume thus  that   $\sx\simeq \sx_{\so_L,\so_C}$, for a semistable affine model $\sx_{\so_L}$ over $S_L$, $[L:K]<\infty$. 
From the product quasi-isomorphisms from Lemma \ref{sroka1} and its proof (where we took $L=K$) we
get  the commutative diagram
$$
\xymatrix@R=6mm@C=6mm{
F^r\rg_{\crr}(\sx)_{\Q_p}^{\bwedge}\ar[r]^-{\kappa} &  F^r\rg_{\dr}(\sx_C/\B^+_{\dr})\\
F^r(\rg_{\crr}(\sx_{\so_L})_{\Q_p}\wh{\otimes}^{\R}_{L}\B^+_{\dr})\ar[r]^{\sim} \ar[u]_{\wr}^{\iota_{\rm BK}} & F^r(\rg_{\dr}(\sx_L)\wh{\otimes}^{\R}_{L}\B^+_{\dr})\ar[u]_{\wr}^{\iota_{\rm BK}}
}
$$
and the two vertical strict quasi-isomorphisms. The bottom horizontal 
 strict quasi-isomorphism follows from the local-global property of $F^r(\rg_{\dr}(\sx_L)\wh{\otimes}^{\R}_{L}\B^+_{\dr})$ (see  Remark \ref{sroka12}). 
\end{proof}

      \begin{lemma}\label{kappa-new}The canonical map in $\sd(C_{\B^+_{\dr}})$ 
\begin{align*}
& \kappa\otimes 1: \quad \rg_{\crr}(X)\wh{\otimes}^{\R}_{\B^+_{\crr}}\B^+_{\dr}\stackrel{\sim}{\to} \rg_{\dr}(X/\B^+_{\dr}), \quad r\geq 0,
\end{align*}
is a strict quasi-isomorphism.
\end{lemma}
Here, we set
 \begin{align*}
  \R\Gamma_{\crr}(X)\wh{\otimes}^{\R}_{\B^+_{\crr}}\B^+_{\dr} & :=\LL\colim((\R\Gamma_{\crr}\wh{\otimes}^{\R}_{\B^+_{\crr}}\B^+_{\dr})(\su_{\cdot})),  
 \end{align*}
where the homotopy colimit is taken over $\eta$-\'etale quasi-compact hypercoverings $\su_{\cdot}$ from $\sm^{\sem, b}_C$ (that is, hypercoverings $\su_{\cdot}$ such that every $\su_n$, $n\geq0$,   is quasi-compact).
\begin{proof} It suffices to show that, for an affine $\sx\in \sm^{{\rm ss},b}_C$, the canonical map 
$$
\R\wlim_r (\rg_{\crr}(\sx)_{\Q_p}\wh{\otimes}^{\R}_{\B^+_{\crr}}(\B^+_{\crr}/F^r)){\to} \R\wlim_r(\rg_{\crr}(\sx)_{\Q_p}/F^r)
$$ 
is a strict quasi-isomorphism. Take a log-smooth lifting 
$\sy$ of $\sx$ over $\Spf (\A_{\crr})$. We have 
\begin{align}
(\rg_{\crr}(\sx)_{\Q_p}\wh{\otimes}^{\R}_{\B^+_{\crr}}(\B^+_{\crr}/F^r) & \simeq (\so(\sy)_{\Q_p}\wh{\otimes}^{\R}_{\B^+_{\crr}}(\B^+_{\crr}/F^r)\to \Omega^1_{\sy/\A_{\crr},\Q_p}\wh{\otimes}^{\R}_{\B^+_{\crr}}(\B^+_{\crr}/F^{r})\to\cdots),\\
\rg_{\crr}(\sx)_{\Q_p}/F^r & \simeq (\so(\sy)_{\Q_p}\wh{\otimes}^{\R}_{\B^+_{\crr}}(\B^+_{\crr}/F^r)\to \Omega^1_{\sy/\A_{\crr}, \Q_p}\wh{\otimes}^{\R}_{\B^+_{\crr}}(\B^+_{\crr}/F^{r-1})\to\cdots)\notag
\end{align}
The claim in the lemma is now clear.
\end{proof}

   \subsubsection{History} Let  $X\in {\rm Sm}_C$.

 (i) Recall that Bhatt-Morrow-Scholze in \cite[Sec.\,13]{BMS1}  introduced $\B^+_{\dr}$-cohomology of $X$, which they call\footnote{We take here the \'etale version studied in \cite[Sec.\,6.2]{CK} and not the original analytic version. The two versions are quasi-isomorphic by \cite[Sec.\,6.2]{CK}.} {\em crystalline cohomology of $X$ over $\B^+_{\dr}$}. 
We will denote it \index{rgdr@\rgdr}by $\rg^{\rm BMS}_{\dr}(X/\B^+_{\dr})$ and see as a complex in $\sd(C_{\B^+_{\dr}})$. 
As they mention \cite[Rem. 13.2]{BMS1}, morally speaking, it is the infinitesimal cohomology of $X$ over the embedding given by the map $\theta: \B^+_{\dr}\to C$. It is defined though in such a way that  it is easy to compare it with $\A_{\rm inf}$-cohomology. Similarly here, we have defined $\rg_{\dr}(X/\B^+_{\dr})$ in such a way that it is easy to compare it with crystalline cohomology over $\A_{\crr}$.  

 (ii) The infinitesimal site definition of $\rg^{\rm BMS}_{\dr}(X/\B^+_{\dr})$ was carried out by Guo in \cite[Sec.\,7.2]{Guo} (see also \cite{Guo-Li}). We will denote this version of $\B^+_{\dr}$-chomology \index{rgdr@\rgdr}by 
 $\rg^{\rm Guo}_{\dr}(X/\B^+_{\dr})\in \sd\sff(C_{\B^+_{\dr}})$ ($\rg^{\rm Guo}_{\dr}(X/\B^+_{\dr}):=\rg_{\rm inf}(X/\B^+_{\dr})$). 
\index{rginf@\rginf}It comes equipped with a Hodge filtration (which was ignored in \cite{BMS1}). Moreover, Guo constructed a natural quasi-isomorphism (see \cite[Cor. 1.2.9., Th. 1.2.7]{Guo}) 
 \begin{equation}
 \label{guo1}
 \rg^{\rm Guo}_{\dr}(X/\B^+_{\dr})\simeq \rg^{\rm BMS}_{\dr}(X/\B^+_{\dr}).
 \end{equation}
 
 (iii) Our construction of $\B^+_{\dr}$-cohomology is compatible with the above constructions: 
\begin{proposition}\label{BMSG}
Let $X\in {\rm Sm}_C$. 
\begin{enumerate}
\item There is a natural quasi-isomorphism in $\sd(C_{\B^+_{\dr}})$
$$
\rg_{\dr}(X/\B^+_{\dr})\simeq\rg^{\rm BMS}_{\dr}(X/\B^+_{\dr}).
$$
\item
There is a natural  quasi-isomorphism in $\sd\sff(C_{\B^+_{\dr}})$
$$
\rg_{\dr}(X/\B^+_{\dr})\simeq\rg^{\rm Guo}_{\dr}(X/\B^+_{\dr}).
$$
\end{enumerate}
\end{proposition}
\begin{proof} 
Claim (1) follows from claim (2) and the quasi-isomorphism (\ref{guo1}). 

    To prove claim (2), recall that $\rg_{\dr}(X/\B^+_{\dr})$ is defined by taking, \'etale locally, the Hodge completed absolute crystalline cohomology and then globalizing.  More specifically, let  $\sx\in \sm^{\rm ss}_C$.  We have
  \begin{align*}
& \rg_{\dr}(\sx_C/\B^+_{\dr})\simeq  \rg_{\crr}(\sx)_{\Q_p}^{\bwedge}:=\R\wlim_r (\rg_{\crr}(\sx)_{\Q_p}/F^r),\\
&  F^r\rg_{\dr}(\sx_C/\B^+_{\dr})\simeq \rg_{\crr}(\sx,\sj^{[r]})_{\Q_p}^{\bwedge}:=\R\wlim_{j \geq r}(\rg_{\crr}(\sx,\sj^{[r]})_{\Q_p}/F^j).
\end{align*}
On the other hand $\rg^{\rm Guo}_{\dr}(X/\B^+_{\dr})$ is defined as the infinitesimal cohomology $\rg_{\rm inf}(X/\B^+_{\dr})$ equipped with its natural Hodge filtration.  It satifies \'etale descent.

  This means that,
 if $\sx$ is affine and (exactly and) closely embedded in an affine  formal log-scheme $\sy$, log-smooth over $\A_{\crr}$,  then in $\sd\sff(C_{\B^+_{\dr}})$
 $$
  \rg_{\dr}(\sx_C/\B^+_{\dr})/F^r\simeq \R\wlim((\sd_{\sx}(\sy)/F^r)_{\Q_p}\to ((\sd_{\sx}(\sy)/F^{r-1})\wh{\otimes}_{\so(\sy)}\Omega^{1}_{\sy/\A_{\crr}})_{\Q_p}\to \cdots),
 $$
 where $ \sd_{\sx}(\sy)$ is the PD-envelope of $\sx$ in $\sy$ and the tensor product is $p$-adic. 
 On the other hand, we have in $\sd\sff(C_{\B^+_{\dr}})$
 $$
   \rg_{\rm inf}(\sx_C/\B^+_{\dr})/F^r\simeq \R\wlim(\sd_{\sx_C}(\sy_C)/F^r\to (\sd_{\sx_C}(\sy_C)/F^{r-1})\wh{\otimes}_{\so(\sy_C)} \Omega^{1}_{\sy_C/\B^+_{\dr}}\to\cdots),
   $$
    where $ \sd_{\sx_C}(\sy_{\B})$ is the inf-envelope of $\sx_C$ in $\sy_\B:=\sy_{\B^+_{\dr}}$. Since $\A_{\crr,\Q_p}/F^i\simeq \B^+_{\dr}/F^i$, we have a natural map in $\sd\sff(C_{\B^+_{\dr}})$
    $$
     \rg_{\dr}(\sx_C/\B^+_{\dr})/F^r\to  \rg_{\rm inf}(\sx_C/\B^+_{\dr})/F^r.
    $$

 This can be globalized to a map in $\sd\sff(C_{\B^+_{\dr}})$
 $$
   \rg_{\dr}(X/\B^+_{\dr})/F^r\to  \rg^{\rm Guo}_{\dr}(X/\B^+_{\dr})/F^r.
 $$
 We claim that it is a  strict quasi-isomorphism. Indeed, it suffices to show this locally so we may assume that we have the data of integral models $\sx, \sy$ as above and, moreover, $\sy$ is a lifting of $\sx$. Then 
\begin{align*}
 &  \rg_{\dr}(\sx_C/\B^+_{\dr})/F^r\simeq \R\wlim(\so(\sy_\B)/F^r\to ((\so(\sy)/F^{r-1})\wh{\otimes}_{\so(\sy)}\Omega^{1}_{\sy/\A_{\crr}})_{\Q_p}\to \cdots),\\
 &  \rg_{\dr}^{\rm Guo}(\sx_C/\B^+_{\dr})/F^r\simeq \R\wlim(\so(\sy_\B)/F^r\to (\so(\sy_\B)/F^{r-1})\wh{\otimes}_{\so(\sy_\B)} \Omega^{1}_{\sy_\B/\B^+_{\dr}}\to\cdots).
\end{align*}
But we have the topological  isomorphisms
\begin{align*}
 ((\so(\sy)/F^{i})\wh{\otimes}_{\so(\sy)}\Omega^{j}_{\sy/\A_{\crr}})_{\Q_p} & \simeq (\so(\sy_\B)/F^{i})\wh{\otimes}_{\so(\sy_\B)}\Omega^{j}_{\sy_{\B,i}/\A_{\crr,\Q_p,i}}\\
 (\so(\sy_\B)/F^{i})\wh{\otimes}_{\so(\sy_\B)} \Omega^{j}_{\sy_\B/\B^+_{\dr}} & \simeq (\so(\sy_\B)/F^{i})\wh{\otimes}_{\so(\sy_\B)} \Omega^{j}_{\sy_{\B,i}/\B^+_{\dr,i}},
\end{align*}
where $(-)_i$ denotes moding out by $F^i$. Hence the  strict quasi-isomorphism
 $$
   \rg_{\dr}(\sx_C/\B^+_{\dr})/F^r\stackrel{\sim}{\to}  \rg^{\rm Guo}_{\dr}(\sx_C/\B^+_{\dr})/F^r,
 $$
as wanted. 

  Having the  strict quasi-isomorphism
  $$
   \rg_{\dr}(X/\B^+_{\dr})/F^r\stackrel{\sim}{\to}  \rg^{\rm Guo}_{\dr}(X/\B^+_{\dr})/F^r,
 $$
  we may take $\R\wlim_r$ of both sides to obtain the strict quasi-isomorphism
 $$
  \rg_{\dr}(X/\B^+_{\dr})\stackrel{\sim}{\to} \rg^{\rm Guo}_{\dr}(X/\B^+_{\dr}).
$$
This is because we have
\begin{align*}
 \rg_{\dr}(X/\B^+_{\dr})\stackrel{\sim}{\to} \R\wlim_r( \rg_{\dr}(X/\B^+_{\dr})/F^r),\quad  \rg^{\rm Guo}_{\dr}(X/\B^+_{\dr})\stackrel{\sim}{\to}  \R\wlim_r(\rg^{\rm Guo}_{\dr}(X/\B^+_{\dr})/F^r)
\end{align*}
as can be easily seen by a computation similar to the one used in the proof of Lemma \ref{sroka1}. 
Finally, to obtain the strict quasi-isomorphism
$$
  F^r\rg_{\dr}(X/\B^+_{\dr})\stackrel{\sim}{\to} F^r\rg^{\rm Guo}_{\dr}(X/\B^+_{\dr}),\quad r\geq 0,
  $$
  we use the  distinguished triangles
  $$
  F^r\rg_{\dr}\to\rg_{\dr}\to\rg_{\dr}/F^r
  $$
  for both cohomologies. 
\end{proof}

 \subsection{Overconvergent $\B^+_{\dr}$-cohomology}  We define overconvergent $\B^+_{\dr}$-cohomology via presentations of dagger structures.
  \subsubsection{Definition of overconvergent $\B^+_{\dr}$-cohomology} Let $X$ be a smooth dagger affinoid over $C$. Let ${\rm pres}(X)=\{X_h\}$ be a presentation of $X$ (see \cite[Sec. 3.2.1]{CN3} for relevant definitions). 
\index{rgdr@\rgdr}Define in $\sd\sff(C_{\B^+_{\dr}})$
    \begin{align*}
      F^r\rg^{\dagger}_{\dr}(X/\B^+_{\dr})  := \LL\colim_hF^r\rg_{\dr}(X_h/\B^+_{\dr}),\quad r\geq 0.
   \end{align*}
  For $r\geq 0$,     the  \'etale sheafification\footnote{See \cite[Def. 2.1]{Vez} for the definition of \'etale topology of dagger varieties.}  of    $ F^r\rg^{\dagger}_{\dr}(X/\B^+_{\dr})$  on ${\rm Sm}^{\dagger}_C$ gives us a \index{acalcris@\acalcris}sheaf $F^r\sa^{\bwedge}_{\crr}$.  The filtered $\B^+_{\dr}$-cohomology  in $\sd\sff(C_{\B^+_{\dr}})$ of a  smooth dagger variety $X$ over $C$
   is defined as $$
  F^r \rg_{\dr}(X/\B^+_{\dr}):=\rg_{\eet}(X,F^r\sa^{\bwedge}_{\crr}),\quad r\geq 0. 
   $$
\begin{remark} If $X$ is a smooth dagger affinoid over $C$ the above two definitions of $\B^+_{\dr}$-cohomology $ \rg^{\dagger}_{\dr}(X/\B^+_{\dr})$ and $ \rg_{\dr}(X/\B^+_{\dr})$ agree. This will be shown in Corollary \ref{local-global-kwak} below by reduction to Hyodo-Kato cohomology via the Hyodo-Kato quasi-isomorphism. 
\end{remark}
\subsubsection{Properties of overconvergent $\B^+_{\dr}$-cohomology} We will now prove properties of overconvergent $\B^+_{\dr}$-cohomology that do not require Hyodo-Kato cohomology.

 We have canonical \index{thetavar@\thetavar}\index{iotabk@\iotabk}maps in, resp., $\sd(C_{\B^+_{\dr}})$ and $\sd\sff(C_K)$
\begin{align*}
\vartheta: \quad & F^r\rg_{\dr}(X/\B^+_{\dr})\to F^r\rg_{\dr}(X),\quad X\in{\rm Sm}^{\dagger}_C, r\geq 0,\\
\iota_{\rm BK}:  \quad & \rg_{\dr}(X)\to\rg_{\dr}(X_C/\B^+_{\dr}),\quad X\in{\rm Sm}^{\dagger}_K,
\end{align*}
induced by their rigid analytic analogs. 
 \begin{proposition}\label{prison}
 \begin{enumerate}
 \item  {\rm (Projection)} Let $X\in {\rm Sm}^{\dagger}_C$. 
\begin{enumerate}
\item The map $\vartheta$ defined above yields  a natural  strict quasi-isomorphism in $\sd\sff(C_C)$
 $$
 \vartheta:\quad \rg_{\dr}(X/\B^+_{\dr})\wh{\otimes}^{\R}_{\B^+_{\dr}}C \stackrel{\sim}{\to}\rg_{\dr}(X).
 $$
  \item More generally, for $r\geq 0$, we have a natural distinguished triangle in $\sd(C_{\B^+_{\dr}})$
 \begin{equation}
 \label{detail11}
 F^{r-1}\rg_{\dr}(X/\B^+_{\dr})\lomapr{t} F^r\rg_{\dr}(X/\B^+_{\dr})\lomapr{\vartheta}F^r\rg_{\dr}(X)
 \end{equation}
 \item  For $r\geq 0$, we have a natural distinguished triangle in $\sd(C_{\B^+_{\dr}})$
 \begin{equation}
 \label{niedziela10}
 F^{r+1}\rg_{\dr}(X/\B^+_{\dr})\to  F^r\rg_{\dr}(X/\B^+_{\dr})\to \bigoplus_{i\leq r}\rg(X,\Omega^i_X)(r-i)[-i]
 \end{equation}
\end{enumerate}
 \item  {\rm (Product  formula)} Let $X\in {\rm Sm}^{\dagger}_K$.  The map $\iota_{\rm BK}$ 
defined above yields a natural  quasi-isomorphism  in $\sd\sff(C_{\B^+_{\dr}})$
       $$
\iota_{\rm BK}:  \quad \rg_{\dr}(X)\wh{\otimes}^{\R}_K\B^+_{\dr}\to\rg_{\dr}(X_C/\B^+_{\dr}).
 $$
 See Remark \ref{samolot1} below for the definition of the tensor product. 
 \item {\rm ($t$-completeness)} The canonical map  in $\sd\sff(C_{\B^+_{\dr}})$
 $$
 \rg_{\dr}(X/\B^+_{\dr})\to \R\wlim_r(\rg_{\dr}(X/\B^+_{\dr})\wh{\otimes}^{\R}_{\B^+_{\dr}}(\B^+_{\dr}/F^r))
 $$
 is a  strict quasi-isomorphism. 
\end{enumerate}
\end{proposition}
\begin{remark}\label{samolot1}In  Proposition \ref{prison} (2), 
the filtration on $\R\Gamma_{\dr}(X)\wh{\otimes}_K^{\R}\B^+_{\dr}$ is defined by the formula
$$
F^r(\R\Gamma_{\dr}(X)\wh{\otimes}_{K}^{\R}\B^+_{\dr}):=\LL\colim(F^r(\R\Gamma_{\dr}\wh{\otimes}^{\R}_{K}\B^+_{\dr})(U_{\cdot})),
$$
 where the homotopy colimit is taken over \'etale dagger affinoid hypercoverings $U_{\cdot}$ of  $X$ and, for a smooth dagger  affinoid $U$, 
 $$
 F^r(\R\Gamma_{\dr}(U)\wh{\otimes}_{K}^{\R}\B^+_{\dr}):=\R\wlim(\so(U)\wh{\otimes}^{\R}_KF^r\B^+_{\dr}\to \Omega^1(U)\wh{\otimes}^{\R}_KF^{r-1}\B^+_{\dr}\to\cdots).
 $$
 In particular, if $r=0$, 
 $
 \R\Gamma_{\dr}(U)\wh{\otimes}_{K}^{\R}\B^+_{\dr}
 $ is just the usual projective tensor product.

\end{remark}

\begin{proof}
To prove the first projection formula in (1) it suffices to argue locally for the dagger cohomologies. So we may assume that $X$ is a smooth dagger affinoid with the  presentation $\{X_h\}$. We need to show that
the projection 
$$\vartheta:\quad \rg_{\dr}^{\dagger}(X/\B^+_{\dr})\wh{\otimes}^{\R}_{\B^+_{\dr}}C \stackrel{\sim}{\to}\rg^{\dagger}_{\dr}(X)
$$
is a strict quasi-isomorphism. We can write this projection  more explicitly as the composition
\begin{align*}
\rg_{\dr}^{\dagger}(X/\B^+_{\dr})\wh{\otimes}^{\R}_{\B^+_{\dr}}C & \simeq (\LL\colim_h\rg_{\dr}(X_h/\B^+_{\dr}))\wh{\otimes}^{\R}_{\B^+_{\dr}}C\stackrel{\sim}{\leftarrow}
\LL\colim_h\rg_{\dr}(X_h/\B^+_{\dr})\wh{\otimes}^{\R}_{\B^+_{\dr}}C\\
 & \veryverylomapr{\LL\colim_h\vartheta}\LL\colim_h\rg_{\dr}(X_h)\simeq \rg^{\dagger}_{\dr}(X).
\end{align*}
The second map is a strict quasi-isomorphism because the tensor product is defined as the cone of multiplication by $t$; the third map is a strict quasi-isomorphism by Proposition \ref{projection}.

  To prove the second formula in (1), we argue locally as well. We need to show that, for $r\geq0$, we have a distinguished triangle
 \begin{equation}
 \label{detail13}
 F^{r-1}\rg^{\dagger}_{\dr}(X/\B^+_{\dr})\lomapr{t} F^r\rg^{\dagger}_{\dr}(X/\B^+_{\dr})\lomapr{\vartheta}F^r\rg^{\dagger}_{\dr}(X),
 \end{equation}
where $X$ is a smooth dagger affinoid with the  presentation $\{X_h\}$. But this triangle can be written as:
 $$
  \LL\colim_h F^{r-1}\rg_{\dr}(X_h/\B^+_{\dr})\verylomapr{\LL\colim_h t}  \LL\colim_hF^{r}\rg_{\dr}(X_h/\B^+_{\dr})\verylomapr{\LL\colim_h\vartheta}\LL\colim_h F^r\rg_{\dr}(X_h)
 $$ 
 and then it is clear that it is distinguished by  Proposition \ref{projection}.
 
 To prove the third formula in (1), we again argue locally. We need to show that, for $r\geq 0$, we have a distinguished triangle
 \begin{equation}
 \label{detail131}
 F^{r+1}\rg^{\dagger}_{\dr}(X/\B^+_{\dr})\stackrel{\can}{\to} F^r\rg^{\dagger}_{\dr}(X/\B^+_{\dr})\stackrel{\beta_X}{\to} \bigoplus_{i\leq r}\rg(X,\Omega^i_X)(r-i)[-i],
 \end{equation}
where $X$ is a smooth dagger affinoid with the  presentation $\{X_h\}$. 
But we can define this triangle as:
$$\xymatrix@R=8mm@C=-14mm{
  &\LL\colim_h F^{r+1}\rg_{\dr}(X_h/\B^+_{\dr})\ar[]!<15ex,-2ex>;[dr]!<5ex,1ex>^-{\LL\colim_h \can }\\
\LL\colim_h \bigoplus_{i\leq r}\rg(X_h,\Omega^i_{X_h})(r-i)[-i]\ar[]!<-15ex,-2ex>;[ur]!<-5ex,1ex>^-{[\,1\,]}&&
 \LL\colim_hF^{r}\rg_{\dr}(X_h/\B^+_{\dr})\ar[ll]^-{\LL\colim_h\beta_{X_h}}
} $$ 
 and then it is clear that it is distinguished by Proposition \ref{projection}.

In  the product formula (2), the map $\iota_{\rm BK}$ is defined by globalizing maps $\iota^{\dagger}_{\rm BK}$ for dagger affinoids. To define the latter,
 assume that $X$ is a smooth dagger affinoid with the  presentation $\{X_h\}$ and 
\index{iotabk@\iotabk}set
\begin{align}\label{defnie}
\iota^{\dagger}_{\rm BK}:\quad  \rg^{\dagger}_{\dr}(X)\wh{\otimes}^{\R}_K\B^+_{\dr} & \simeq (\LL\colim_h\rg_{\dr}(X_h))\wh{\otimes}^{\R}_K\B^+_{\dr}\stackrel{\sim}{\leftarrow}
\LL\colim_h\rg_{\dr}(X_h)\wh{\otimes}^{\R}_K\B^+_{\dr}\\
  &  \veryverylomapr{\LL\colim_h \iota_{\rm BK}}\LL\colim_h\rg_{\dr}(X_{h,C}/\B^+_{\dr})
    \simeq 
\rg^{\dagger}_{\dr}(X_C/\B^+_{\dr}).\notag
\end{align}
The third map is a filtered quasi-isomorphism by Lemma \ref{sroka1}. It remains to show that so is the second map,  i.e., that  the map 
\begin{equation}
\label{mapcia1}
\xymatrix@R=5mm{
\LL\colim_h\rg_{\dr}(X_h)\wh{\otimes}^{\R}_K\B^+_{\dr} \ar[r]& \rg^{\dagger}_{\dr}(X)\wh{\otimes}^{\R}_K\B^+_{\dr},
}
\end{equation}
is a  filtered strict quasi-isomorphism. Indeed, look at the cohomology of both sides. On the right hand side, arguing as in \cite[Sec.\,3.2.2]{CDN3},  we get 
$$\wt{H}^i(\rg^{\dagger}_{\dr}(X)\wh{\otimes}^{\R}_K\B^+_{\dr})\simeq H^i_{\dr}(X)\wh{\otimes}_K\B^+_{\dr}.
$$
For the left hand side, we compute 
\begin{align*}
\wt{H}^i(\LL\colim_h\rg_{\dr}(X_h)\wh{\otimes}^{\R}_K\B^+_{\dr}) & \stackrel{\sim}{\to} \wt{H}^i(\LL\colim_h\rg_{\dr}(X^\circ_h)\wh{\otimes}^{\R}_K\B^+_{\dr})\simeq \colim_h(H^i_{\dr}(X^\circ_h)\wh{\otimes}^{\R}_K\B^+_{\dr})\\
& \stackrel{\sim}{\to} {H}^i_{\dr}(X)\wh{\otimes}_K\B^+_{\dr}.
\end{align*}
\begin{remark}\label{version2}
 Here, for a pair of affinoids $X_h\Subset X_{h+1}$ as above, we define, slightly abusively,  
   the (naive) interior $X^{\circ}_{h+1}$  as the connected component of ${\rm Int}(X_{h+1})$ containing $X_h$. See \cite[Appendix]{Vez} for a discussion of (relative) interiors. By \cite[Prop. 2.5.8]{Berk}, this definition is functorial in the pair $X_h\Subset X_{h+1}$. 
   Moreover, $X^{\circ}_{h+1}$ is Stein and its complement in $X_{h+1}$ is open and quasi-compact. 
\end{remark}
The second and the third  isomorphisms above  follow from: 
\begin{enumerate}
\item the fact that the cohomology $H^i_{\dr}(X^\circ_h)$ is a finite rank vector space over $K$ with its canonical topology (by \cite[Th. 3.1]{GKdR});
\item point (1) implies the quasi-isomorphism 
$$
\wt{H}^i(\rg_{\dr}(X^\circ_h)\wh{\otimes}^{\R}_K\B^+_{\dr})\simeq {H}^i_{\dr}(X^\circ_h)\wh{\otimes}_K\B^+_{\dr}
$$
proved as in \cite[Sec.\,3.2.2]{CDN3};
\item the system $\{H^i_{\dr}(X^\circ_h)\}_h$ is essentially constant and isomorphic to $H^i_{\dr}(X)$;
\item  point (3) implies that the system $\{{H}^i_{\dr}(X^\circ_h)\wh{\otimes}_K\B^+_{\dr}\}_h$ is essentially constant and isomorphic to ${H}^i_{\dr}(X)\wh{\otimes}_K\B^+_{\dr}$.
\end{enumerate}
This proves that   the map (\ref{mapcia1}) is a strict quasi-isomorphism. 

We shall need to argue more that it is a filtered strict quasi-isomorphism as well. We argue by induction on $r\geq 0$; the base case of $r=0$ being proved above. For the inductive step ($r-1\Rightarrow r$) consider the following commutative diagram
$$
\xymatrix@R=5mm{
\LL\colim_hF^{r-1}(\rg_{\dr}(X_h)\wh{\otimes}^{\R}_K\B^+_{\dr}) \ar[d]^-{t} \ar[r]^-{\sim}& F^{r-1}(\rg^{\dagger}_{\dr}(X)\wh{\otimes}^{\R}_K\B^+_{\dr})\ar[d]^-{t}\\
\LL\colim_hF^r(\rg_{\dr}(X_h)\wh{\otimes}^{\R}_K\B^+_{\dr}) \ar[r] \ar[d]& F^r(\rg^{\dagger}_{\dr}(X)\wh{\otimes}^{\R}_K\B^+_{\dr})\ar[d]\\
\LL\colim_hF^r(\rg_{\dr}(X_h)\wh{\otimes}^{\R}_KC) \ar[r]^-{\sim}& F^r(\rg^{\dagger}_{\dr}(X)\wh{\otimes}^{\R}_KC)
}
$$
The left and the  right vertical triangles are distinguished by (\ref{detail1}) and (\ref{detail11}), respectively. The bottom map is clearly a strict quasi-isomorphism; the top map is a strict quasi-isomorphism by the inductive assumption. It follows that so is the middle horizontal map, as wanted. 

 We finish the proof of the second claim of our proposition by noting that the map $\R\wlim_h\iota_{\rm BK}$ in (\ref{defnie}) is a strict quasi-isomorphism by Lemma \ref{sroka1}.

 For the third claim of the proposition, it suffices to argue locally for the dagger cohomologies. Hence we can assume that $X\simeq Y_C$ for a smooth dagger affinoid $Y$ defined over $K$. 
 And then the wanted $t$-completeness follows from the second claim of the proposition and the fact that, since our tensor products are  projective and the Mittag-Leffler condition is satisfied,  the canonical map
 $$
\rg^{\dagger}_{\dr}(Y)\wh{\otimes}^{\R}_K\B^+_{\dr}\to  \R\wlim_r(\rg^{\dagger}_{\dr}(Y)\wh{\otimes}^{\R}_K(\B^+_{\dr}/F^r))
 $$
 is a filtered strict quasi-isomorphism.
\end{proof}

\section{Geometric Hyodo-Kato morphisms}\label{hihi1}
This section is devoted to the definition of compatible rigid analytic (for $X\in{\rm Sm}_C$)
and overconvergent (for $X\in{\rm Sm}^\dagger_C$)  Hyodo-Kato cohomologies $\rg_{{\rm HK},\breve{F}}(X)$.
For a general rigid analytic variety, the Hyodo-Kato cohomology is in general quite ugly (not separated
and, locally, infinite dimensional), but for dagger varieties the Hyodo-Kato cohomology
has nice properties (separated and, locally, finite dimensional). On the other hand
(Lemma~\ref{cwir-cwir}),
if $X\in {\rm Sm}_C$ is partially proper, then the rigid analytic
and overconvergent Hyodo-Kato cohomologies give the same result:
if $X^\dagger$ is the associated dagger variety,
the natural map $\rg_{{\rm HK},\breve{F}}(X^\dagger)\to \rg_{{\rm HK},\breve{F}}({X})$ is a
strict quasi-isomorphism (Corollary~\ref{roznosci1}).

We define $\rg_{{\rm HK},\breve{F}}(X)$ for dagger varieties by, locally, going to the limit
over a presentation in the Hyodo-Kato cohomology for rigid analytic varieties, and globalizing.
This definition is much more flexible than Grosse-Kl\"onne's~\cite{GKFr}, and we show 
(Lemma~\ref{herbata2}) that the two definitions give rise to the same cohomology.

The rigid analytic
and overconvergent Hyodo-Kato cohomologies are related (Theorem \ref{rynek2.0} and Theorem~\ref{HK-dagger2})
to the rigid analytic
and overconvergent de Rham and $\B_{\dr}^+$-cohomologies by
the Hyodo-Kato quasi-isomorphisms in, resp., $\sd(C_C)$ and $\sd(C_{\B^+_{\dr}})$: 
$$
\iota_{\hk}:\rg_{\hk,\breve{F}}(X)\wh{\otimes}_{\breve{F}}^{\R}C\stackrel{\sim}{\to}\rg_{\dr}(X),
\quad \iota_{\hk}:\rg_{\hk,\breve{F}}(X)\wh{\otimes}^{\R}_{\breve{F}}\B^+_{\dr}\stackrel{\sim}{\to} \rg_{\dr}(X/\B^+_{\dr})
$$

\subsection{Rigid analytic setting} We start our definitions of Hyodo-Kato morphisms with rigid-analytic varieties. 
 \subsubsection{Completed Hyodo-Kato cohomology}\label{completion1} The completed Hyodo-Kato cohomology $\rg_{\hk}(X^0_1) $ that appeared in the proof of Theorem \ref{HK-crr}  has better topological properties than the classical Hyodo-Kato cohomology $\rg_{\hk}(X_1)$ (being over $p$-complete field $\breve{F}$ instead of  $F^{\nr}$).  Because of this we will often use it. 
  
  Let $X\in {\rm Sm}_C$. Let $\sa^c_{\hk}$ 
\index{acalhk@\acalhk}($c$ stands for "completion") be the $\eta$-\'etale sheafification of the presheaf $\sx\to \rg_{\hk}(\sx^0_1)_{\Q_p}$ on $\sm_C^{{\rm ss},b}$. We 
\index{rghk@\rghk}set in $\sd_{\phi,N}(C_{\breve{F}})$
    $$\rg_{\hk, \breve{F}}(X):=\rg_{\eet}(X,\sa^c_{\hk}).$$ 
It is a dg $\breve{F}$-algebra equipped with a Frobenius, monodromy action, and a continuous action of $\sg_K$, if $X$ is defined over $K$. It is equipped with the topology induced from  the topology of the $\rg_{\hk}(\sx^0_{1})_{\Q_p}$'s.
    
    Unwinding the definitions,  using the base change quasi-isomorphism (\ref{beta}), and globalizing we obtain that the canonical morphism in $\sd_{\phi,N}(C_{\breve{F}})$
     \begin{equation}
     \label{beta22}
     \beta:\quad \rg_{\hk}(X)\wh{\otimes}^{\R}_{F^{\nr}}\breve{F}\to \rg_{\hk,\breve{F}}(X)
     \end{equation}
     is a strict quasi-isomorphism. It implies: 
     \begin{lemma}\label{loc-global1}{\rm ({Local-global compatibility})}
   For  $\sx\in \sm^{{\rm ss}}_C$, the canonical morphism in $\sd_{\phi,N}(C_{\breve{F}})$
    \begin{equation}
    \label{herbata1}
    \rg_{\hk}(\sx^0_{1})_{\Q_p}\to \rg_{\hk,\breve{F}}(\sx_C)
    \end{equation}
    is a strict quasi-isomorphism.
    \end{lemma}
   \begin{proof}  We can pass from $K$ to $\breve{K}:=K\breve{F}$ (which amounts to passing from $F$ to $\breve{F}$ for the absolutely unramified subfields) without changing the cohomologies in (\ref{herbata1}). And then we can simply use local-global compatibility for $(\breve{F})^{\nr}=\breve{F}$-cohomology (see \cite[Prop. 4.23]{CN3}).
   \end{proof}

     \subsubsection{Geometric rigid analytic Hyodo-Kato quasi-isomorphisms}
     We will now use Theorem \ref{HK-crr} to define, both local and global, geometric Hyodo-Kato quasi-isomorphisms. 
     \label{HK-geom}$\quad$\\

     (i) {\em Local setting.} We will define two types of Hyodo-Kato morphisms: Hyodo-Kato-to-de Rham and Hyodo-Kato-to-$\B^+_{\dr}$.
     
    Let $\sx\in \sm_C^{\rm{ss}, b}$.
The Hyodo-Kato-to-de Rham morphism is defined by the \index{iotahk@\iotahk}composition in $\sd(C_C)$: 
\begin{align}
\label{zaby11}
 \iota_{\hk}:\quad \rg_{\hk}(\sx^0_1)_{\Q_p}\wh{\otimes}^{\R}_{\breve{F}}C  \xrightarrow[\sim]{\epsilon_{\dr}^{\hk}}\rg_{\crr}(\sx_1/\overline{S})_{\Q_p}\simeq \rg_{\dr}(\sx_C).
\end{align}
It is a natural strict quasi-isomorphism.

For the Hyodo-Kato-to-$\B^+_{\dr}$ morphism we have:
\begin{corollary}\label{kawa1}
Let $\sx\in \sm_C^{\rm{ss}, b}$. There exists a natural  strict \index{iotahk@\iotahk}quasi-isomorphism in $\sd(C_{\B^+_{\dr}})$
$$
\iota_{\hk}:\quad \rg_{\hk}(\sx_1^0)_{\Q_p}\wh{\otimes}^{\R}_{\breve{F}}\B^+_{\dr}\stackrel{\sim}{\to} \rg_{\dr}(\sx_C/\B^+_{\dr}). 
$$
Moreover, we have the commutative diagram  in $\sd(C_{\B^+_{\dr}})$
$$
\xymatrix@R=6mm{ 
\rg_{\hk}(\sx_1^0)_{\Q_p}\wh{\otimes}^{\R}_{\breve{F}}\B^+_{\dr}\ar[d]^{1\otimes\theta}\ar[r]^-{\iota_{\hk}}_-{\sim} & \rg_{\dr}(\sx_C/\B^+_{\dr})\ar[d]^{\vartheta}\\
\rg_{\hk}(\sx^0_1)_{\Q_p}\wh{\otimes}^{\R}_{\breve{F}}C \ar[r]^-{\iota_{\hk}}_-{\sim} & \rg_{\dr}(\sx_C).
} $$
\end{corollary}
\begin{proof}
To define $\iota_{\hk}$, we use the natural  strict quasi-isomorphism  in $\sd(C_{\B^+_{\dr}})$
   $$
\epsilon^{\hk}_{\B^+_{\dr}}:\quad \rg_{\hk}(\sx_1^0)_{\Q_p}\wh{\otimes}_{\breve{F}}^{\R}\B^+_{\dr}\stackrel{\sim}{\to} \rg_{\crr}(\sx_1)_{\Q_p}\wh{\otimes}^{\R}_{\B^+_{\crr}}\B^+_{\dr}
$$ from Lemma \ref{local1} and compose it with the strict quasi-isomorphism  in $\sd(C_{\B^+_{\dr}})$
$$
\kappa: \quad \rg_{\crr}(\sx_1)_{\Q_p}\wh{\otimes}^{\R}_{\B^+_{\crr}}\B^+_{\dr}\stackrel{\sim}{\to} \rg_{\crr}(\sx)_{\Q_p}^{\bwedge}
$$ 
from the proof of Lemma \ref{kappa-new}. 

  Commutativity of the diagram follows from Lemma \ref{local1}.
     \end{proof}
     (ii) {\em Global setting.} We can now state the main theorem of this chapter:
     \begin{theorem}{\rm (Geometric Hyodo-Kato isomorphisms)} \label{rynek2.0}
Let $X\in {\rm Sm}_C$.  We have  the natural Hyodo-Kato strict 
\index{iotahk@\iotahk}quasi-isomorphisms in, resp., $\sd(C_C)$ and $\sd(C_{\B^+_{\dr}})$
       \begin{align}
       \label{rynek2}
       \iota_{\hk}:\quad  \rg_{\hk,\breve{F}}(X)\wh{\otimes}_{\breve{F}}^{\R}C\stackrel{\sim}{\to}\rg_{\dr}(X),\quad \iota_{\hk}:\quad \rg_{\hk,\breve{F}}(X)\wh{\otimes}^{\R}_{\breve{F}}\B^+_{\dr}\stackrel{\sim}{\to} \rg_{\dr}(X/\B^+_{\dr})
       \end{align}
       that are compatible via the maps $\theta$ and $\vartheta$.  
       \end{theorem}
       \begin{proof}
       Globalize the  local strict quasi-isomorphisms  from  Corollary \ref{local1} and Corollary \ref{kawa1}.
       \end{proof}
       
         (iii) {\em Complements.} In a similar fashion,   the  local strict quasi-isomorphism 
$ \epsilon^{\rm HK}_{\st}$  from  Theorem  \ref{HK-crr}  induces the natural  strict 
\index{epshk@\epshk}quasi-isomorphism in $\sd_{\phi,N}(C_{\B^+_{\st}})$
       \begin{equation}
       \label{epsilon1}
      \epsilon^{\rm HK}_{\st}:\quad  \rg_{\hk,\breve{F}}(X)\wh{\otimes}_{\breve{F}}\B^+_{\st}\stackrel{\sim}{\to}\rg_{\crr}(X)\wh{\otimes}_{\B^+_{\crr}}\B^+_{\st},
  \end{equation}
where we set in $\sd_{\phi,N}(C_{\B^+_{\st}})$
 \begin{align}
 \label{HK-referee}
     \R\Gamma_{\hk,\breve{F}}(X)\wh{\otimes}_{\breve{F}}\B^+_{\st} & :=\LL\colim((\R\Gamma_{\hk,\breve{F}}\wh{\otimes}_{\breve{F},\iota}\B^+_{\st})(\su_{\cdot,1})),  \\
  \R\Gamma_{\hk,\breve{F}}(X)\wh{\otimes}_{\B^+_{\crr}}\B^+_{\st} & :=\LL\colim((\R\Gamma_{\hk,\breve{F}}\wh{\otimes}_{\B^+_{\crr},\iota}\B^+_{\st})(\su_{\cdot,1})),  \notag
 \end{align}
with the homotopy colimit is taken over $\eta$-\'etale quasi-compact hypercoverings $\su_{\cdot}$ from $\sm^{\sem, b}_C$.  
Applying the map $\B^+_{\st}\to \B^+_{\crr}$ given by sending $\log (\lambda_p)\mapsto 0$ to the morphism (\ref{epsilon1}) we obtain
the  strict 
\index{epshk@\epshk}quasi-isomorphism in  $\sd_{\phi}(C_{\B^+_{\crr}})$
       \begin{equation}
       \label{epsilon12}
      \epsilon^{\rm HK}_{\crr}:\quad  \rg_{\hk,\breve{F}}(X)\wh{\otimes}_{\breve{F}}\B^+_{\crr}\stackrel{\sim}{\to}\rg_{\crr}(X).
  \end{equation}
    \subsection{The overconvergent setting} We are now ready to define the overconvergent geometric Hyodo-Kato morphism. We do it locally by using, via presentations, the rigid-analytic geometric Hyodo-Kato morphism constructed in the previous section and then we glue. The advantage of this approach is that, by construction,  the overconvergent and the rigid analytic geometric Hyodo-Kato morphisms are compatible. This is in contrast to \cite{CDN3}, \cite{CN3}, where a lot of effort was devoted to proving compatibility between the overconvergent construction due to Grosse-Kl\"onne, and the rigid-analytic construction due to Hyodo-Kato\footnote{Recently, Ertl-Yamada in \cite{EY} have introduced a particularly simple definition of overconvergent Hyodo-Kato cohomology for weak-formal semistable schemes and equally simple definition of the relevant Hyodo-Kato map. Their construction is compatible with the crystalline Hyodo-Kato analog when the scheme is proper. It is likely that their construction can be extended to the set-up needed in this paper.}.

  \subsubsection{Overconvergent Hyodo-Kato cohomology via presentations of dagger structures}\label{referee1} In this section we introduce a definition of overconvergent Hyodo-Kato cohomology using presentations of dagger structures (see \cite[Appendix]{Vez}, \cite[Sec.\,6.3]{CN3}). We show that the so defined Hyodo-Kato cohomology, a priori different from the one defined by Grosse-Kl\"onne, is, in fact,  strictly quasi-isomorphic to it.

\vskip.2cm
    (i) {\em Local definition}. Let $X$ be a dagger affinoid over $L=K,C$. Let ${\rm pres}(X):=\{X_h\}_{h\in\N}$ be a presentation of dagger structures. 
\index{rghk@\rghk}Define in  $\sd_{\phi,N}(C_{\breve{F}})$:
  $$
  \rg^{\dagger}_{\hk}(X):=\LL\colim_h\rg_{\hk}(X_h)
  $$
  and equip it with the induced Frobenius and monodromy. 
  We have a natural \index{alphahk@\alphahk}map 
  \begin{equation}\label{map11}
  \alpha^{\dagger}_{\hk}:  \rg^{\dagger}_{\hk}(X) \to \rg^{\gk}_{\hk}(X) 
  \end{equation}
  defined as the composition
  \begin{align}\label{map12}
  \rg^{\dagger}_{\hk}(X)  & =\LL\colim_h\rg_{\hk}(X_h)\stackrel{\sim}{\to} \LL\colim_h\rg_{\hk}(X^{\circ}_h)\\
  & \stackrel{\sim}{\leftarrow}\LL\colim_h\rg^{\gk}_{\hk}(X^{{\circ},\dagger}_h)\to
   \rg^{\gk}_{\hk}(X). \notag
  \end{align}
 The third map  is a strict quasi-isomorphism by Corollary \ref{ogrod1}: this is   because the 
\index{Xn@\Xn}interior $X^{\circ}_h$ is Stein. Note that the   proof of the cited corollary relies on a nontrivial  comparison result between the rigid analytic and Grosse-Kl\"onne's overconvergent Hyodo-Kato morphisms. 
 The map $  \alpha^{\dagger}_{\hk}$ is functorial  (see Remark \ref{version2}).
  
\vskip.2cm
  (ii) {\em Globalization}. For a general smooth dagger variety $X$ over $L$, using the natural equivalence of analytic topoi 
  \begin{equation}
  \label{equiv1}
  {\rm Sh}({\rm SmAff}^{\dagger}_{L,\eet})\stackrel{\sim}{\to} {\rm Sh}({\rm Sm}^{\dagger}_{L,\eet})\
  \end{equation}
  we define the sheaf $\sa_{\hk}^{\dagger}$ on $X_{\eet}$ as the sheaf associated to the presheaf defined by $U\mapsto \rg_{\hk}^{\dagger}(U), U\in  {\rm SmAff}^{\dagger}_{L}$, $U\to X$  an \'etale map. We define  in $\sd_{\phi,N}(C_{F_L})$
  $$
  \rg_{\hk}(X):=\rg_{\eet}(X,\sa^{\dagger}_{\hk}).
  $$
If $ L =K$,  it is a dg $F$-algebra.    If $L=C$,   it is a dg ${F}^{\nr}$-algebra equipped with a Frobenius, monodromy action, and a continuous action of $\sg_K$ if   $X$ is defined over $K$. Its  topology is induced  from the topology of the $ \rg^{\dagger}_{\hk}(X)$'s.


  Globalizing the map $\alpha_{\hk}^{\dagger}$ from (\ref{map11}) we obtain a natural 
\index{alphahk@\alphahk}map  in $\sd_{\phi,N}(C_{F_L})$
  $$
  \alpha_{\hk}: \rg_{\hk}(X)\to \rg^{\gk}_{\hk}(X).
  $$
  \begin{lemma}\label{herbata2}Let $L=K,C$. 
  \begin{enumerate} 
  \item The above map $\alpha_{\hk}$ is a strict quasi-isomorphism.
  \item  {\rm (Local-global compatibility)} If $X$ is  a smooth dagger affinoid the natural map  in $\sd_{\phi,N}(C_{F_L})$
  $$\rg^{\dagger}_{\hk}(X)\to \rg_{\hk}(X)$$ is a strict quasi-isomorphism. 
  \end{enumerate}
  \end{lemma}
  \begin{proof}
  For the first claim, by \'etale descent, 
  we may assume that $X$ comes from   a smooth dagger affinoid.    Looking at the composition (\ref{map12})  defining the map $\alpha^{\dagger}_{\hk}$ we see that it suffices to show that the natural map
  \begin{equation}
  \label{dagger1}
  \LL\colim_h\rg_{\hk}^{\gk}(X^{{\circ},\dagger}_h)\to \rg^{\gk}_{\hk}(X)
  \end{equation}
  is a strict quasi-isomorphism. But this was shown in the proof of Proposition 6.17 in \cite{CN3}.  We note that that proof uses the Hyodo-Kato quasi-isomorphism of Grosse-Kl\"{o}nne to pass to the de Rham cohomology where the analog of (\ref{dagger1})  is obvious. 
  
    For the second claim, consider the commutative local-global diagram  in $\sd_{\phi,N}(C_{F_L})$
    $$
    \xymatrix{
    \rg^{\dagger}_{\hk}(X)\ar[r]\ar[d]_{\alpha^{\dagger}_{\hk}} &  \rg_{\hk}(X)\ar[dl]_{\sim}^{\alpha_{\hk}}\\
    \rg^{\rm GK}_{\hk}(X) 
    }
    $$
    The slanted  arrow is a strict quasi-isomorphism by the first claim of the lemma. It suffices to show that the left vertical arrow is a strict quasi-isomorphism as well. For that, it suffices to show that the map 
    $$
     \LL\colim_h\rg_{\hk}^{\gk}(X^{{\circ},\dagger}_h)\to \rg^{\gk}_{\hk}(X)
     $$ appearing in the definition (\ref{map12}) of the map $\alpha^{\dagger}_{\hk}$ is a strict quasi-isomorphism but this was just shown above. 
\end{proof}

  (iii) {\em Completed overconvergent Hyodo-Kato cohomology.}  
   We can define the {\em completed overconvergent Hyodo-Kato cohomology} by a similar procedure to the one used above. It will have better topological properties than its classical version. Let $X$ be a smooth dagger affinoid over $C$. 
   Let ${\rm pres}(X)=\{X_h\}_{h\in\N}$. Define in  $\sd_{\phi,N}(C_{\breve{F}})$
  $$
  \rg^{\dagger}_{\hk,\breve{F}}(X):=\LL\colim_h\rg_{\hk,\breve{F}}(X_h). 
  $$
For a general smooth dagger variety over $C$, we can globalize the above definition and obtain the 
\index{acalhk@\acalhk}sheaf $\sa^{c}_{\hk}$  for the $\eta$-\'etale topology on $ \sm^{\dagger,{\rm ss}}_C$ and 
   \index{rghk@\rghk}cohomology
    $$\rg_{\hk, \breve{F}}(X):=\rg_{\eet}(X,\sa^{c}_{\hk})\in   \sd_{\phi,N}(C_{\breve{F}}).
    $$ It is a dg $\breve{F}$-algebra equipped with a Frobenius, monodromy action, and a continuous action of $\sg_K$, if $X$ is defined over $K$. It is equipped with  the topology induced from the topology of the $ \rg^{\dagger}_{\hk,\breve{F}}(X)$'s.
     
    We have  the local-global compatibility  by  Lemma \ref{herbata2} 
(replace, without loss of information, $F$ by~$\breve{F}$). 
      
\vskip.2cm
      (iv) {\em Completed overconvergent Hyodo-Kato cohomology ala Grosse-Kl\"onne.} But we can also define the {completed overconvergent Hyodo-Kato cohomology} as in the rigid analytic case, by modifying the definition of the overconvergent Hyodo-Kato cohomology of Grosse-Kl\"onne. That is, we can 
\index{rghk@\rghk}set $\rg_{\hk,\breve{F}}^{\rm GK}(\sx_1):=\rg_{\hk}(\sx_0)$, 
for  $\sx\in \sm^{\dagger,{\rm ss}}_C$, where $\sx_0:=\sx_{\overline{k}}$,  and globalize.  We will denote by $\rg_{\hk,\breve{F}}^{\rm GK}(X), X\in {\rm Sm}^{\dagger}_C$, the so obtained cohomology in  $\sd_{\phi,N}(C_{\breve{F}})$. 

   We easily check that we have strict quasi-isomorphisms in  $\sd_{\phi,N}(C_{\breve{F}})$: 
\begin{align}\label{niedziela1}
\rg_{\hk,\breve{F}}^{\rm GK}(\sx_1) & \stackrel{\sim}{\leftarrow} \rg_{\hk}^{\rm GK}(\sx_1)\wh{\otimes}^{\R}_{F^{\nr}}\breve{F}, \quad \sx\in \sm^{\dagger,{\rm ss}},\\
\rg_{\hk,\breve{F}}^{\rm GK}(X) & \stackrel{\sim}{\leftarrow} \rg_{\hk}^{\rm GK}(X)\wh{\otimes}^{\R}_{F^{\nr}}\breve{F},\quad X\in {\rm Sm}_C^{\dagger}.\notag
\end{align}
We also have    local-global compatibility: pass from $F$ to $\breve{F}$ as in the proof of Lemma \ref{loc-global1}. This reduces the problem to the local-global compatibility for the usual Hyodo-Kato cohomology
of Grosse-Kl\"onne and this we know is true.

   The two definitions of completed overconvergent Hyodo-Kato cohomology give the same objects:
 \begin{lemma}\label{cwir-cwir}
 Let $X\in {\rm Sm}_C^{\dagger}$. There exists a natural strict 
\index{alphahk@\alphahk}quasi-isomorphism in  $\sd_{\phi,N}(C_{\breve{F}})$
 $$
 \alpha_{\hk,\breve{F}}:\quad \rg_{\hk,\breve{F}}(X) \stackrel{\sim}{\to} \rg_{\hk,\breve{F}}^{\rm GK}(X). 
 $$
 \end{lemma}
 \begin{proof}
 Pass from $F$ to $\breve{F}$ and use Lemma \ref{herbata2}. 
 \end{proof}
 (v) {\em Tensor products.}  The following lemma will allow us to pass between tensor products involving the two definitions of overconvergent Hyodo-Kato cohomology. 
  \begin{lemma}\label{passage}
  Let $W$ be a Banach space\footnote{In applications, $W$ will be most often a period rings.} over $\breve{F}$. 
  \begin{enumerate}
  \item {\rm (Local-global compatibility)} Let $X$ be a smooth dagger affinoid over $C$. The canonical map  in  $\sd(C_{\breve{F}})$
  $$
 \rg^{\dagger}_{\hk}(X)\wh{\otimes}^{\R}_{F^{\nr}}W\to \rg_{\hk}(X)\wh{\otimes}^{\R}_{F^{\nr}}W
$$
is a strict quasi-isomorphism.
  \item   Let $X\in {\rm Sm}^{\dagger}_C$. There exists a following commutative diagram in  $\sd(C_{\breve{F}})$
 $$
 \xymatrix@R=6mm{
  \rg_{\hk,\breve{F}}(X)\wh{\otimes}^{\R}_{\breve{F}}W\ar[r]^{\alpha_{\hk,\breve{F}}(W)}_{\sim} & \rg_{\hk,\breve{F}}^{\rm GK}(X)\wh{\otimes}^{\R}_{\breve{F}}W \\
    \rg_{\hk}(X)\wh{\otimes}^{\R}_{F^{\nr}}W\ar[u]^{\wr} \ar[r]^{\alpha_{\hk}(W)}_{\sim}& \rg_{\hk}^{\rm GK}(X)\wh{\otimes}^{\R}_{F^{\nr}}W\ar[u]^{\wr}.
 }
 $$
  \end{enumerate}
  \end{lemma}
  \begin{remark}\label{wrzask1}
  (1) The tensor product $ \rg^{\dagger}_{\hk}(X)\wh{\otimes}^{\R}_{F^{\nr}}W$ is 
\index{tensor@\tensor}defined in  $\sd(C_{\breve{F}})$ as
  $$
  \rg^{\dagger}_{\hk}(X)\wh{\otimes}^{\R}_{F^{\nr}}W:=\LL\colim_h(\rg_{\hk}(X_h)\wh{\otimes}^{\R}_{F^{\nr}}W),
  $$
  where $\{X_h\}$ is the presentation of $X$. 
  
   (2) {\bf Warning}:  One has to be careful with  tensor products  as in (1) (because we chose projective tensor products hence we lost the  commutation with general inductive limits). For example, when $F^{\nr}=\breve{F}$, the tensor product 
    $
  \rg^{\dagger}_{\hk}(X)\wh{\otimes}^{\R}_{\breve{F}}W$ is already defined. Luckily, in this case, the two definitions give the same tensor product. To see this, note that we have
   $\rg^{\dagger}_{\hk}(X)\wh{\otimes}^{\R}_{\breve{F}}W=(\LL\colim_h\rg_{\hk}(X_h))\wh{\otimes}^{\R}_{\breve{F}}W$.  Hence  the canonical map
   $$
  \LL\colim_h\rg_{\hk}(X_h)\wh{\otimes}^{\R}_{F^{\nr}}W \to \rg^{\dagger}_{\hk}(X)\wh{\otimes}^{\R}_{\breve{F}}W
   $$
   induces a map $\rg^{\dagger}_{\hk}(X)\wh{\otimes}^{\R}_{F^{\nr}}W \to \rg^{\dagger}_{\hk}(X)\wh{\otimes}^{\R}_{\breve{F}}W. $ In the proof of Lemma \ref{passage} below we will show that this is a strict quasi-isomorphism. 
  
  (3) For any smooth dagger variety $X$, the tensor product $\rg_{\hk}(X)\wh{\otimes}^{\R}_{F^{\nr}}W$ is defined by globalizing the tensor product from (1). 
  \end{remark}
  \begin{proof} For (1), we start with the case $W=\breve{F}$. Consider the commutative diagram
  \begin{equation}
  \label{diagram-kwak}
  \xymatrix@R=6mm{
   \rg^{\dagger}_{\hk}(X)\wh{\otimes}^{\R}_{F^{\nr}}\breve{F}\ar[r] \ar[d]^{\wr}&  \rg_{\hk}(X)\wh{\otimes}^{\R}_{F^{\nr}}\breve{F}\ar[d]^{\wr}\\
   \rg^{\dagger}_{\hk,\breve{F}}(X)\ar[r]^{\sim} &  \rg_{\hk,\breve{F}}(X).
}
  \end{equation}
  The bottom map is a strict quasi-isomorphism by Lemma \ref{herbata2} (replace $F$ by $\breve{F}$). The left vertical map is a strict quasi-isomorphism by definition and (\ref{beta22}); the right vertical map is the globalization of the left vertical map hence a strict quasi-isomorphism as well. It follows that the top map is also a strict quasi-isomorphism, as wanted. 
  
   Now, for a general $W$, we take the top map in the diagram (\ref{diagram-kwak}) and tensor it with $W$ over $\breve{F}$ to obtain the strict quasi-isomorphism in the top of the commutative diagram
   $$
   \xymatrix@R=6mm{
     (\rg^{\dagger}_{\hk}(X)\wh{\otimes}^{\R}_{F^{\nr}}\breve{F})\wh{\otimes}^{\R}_{\breve{F}}W\ar[r]^{\sim}  & (\rg_{\hk}(X)\wh{\otimes}^{\R}_{F^{\nr}}\breve{F})\wh{\otimes}^{\R}_{\breve{F}}W\\
          \rg^{\dagger}_{\hk}(X)\wh{\otimes}^{\R}_{F^{\nr}}W\ar[r] \ar[u]& \rg_{\hk}(X)\wh{\otimes}^{\R}_{F^{\nr}}W\ar[u]
  }
   $$
 It remains to show that the left vertical map in the diagram is a strict quasi-isomorphism because then so is the right vertical map (being the globalization of the left vertical map) and then the bottom map as well, as wanted. 
   \begin{remark}
   The tensor product in the top row is the usual projective tensor product. Hence the vertical maps are not identities and the statement that they are strict quasi-isomorphisms is not trivial even for $\breve{F}$. 
   \end{remark}
 It is clear that the left vertical map  is a strict quasi-isomorphism if we drop the dagger and replace $X$ with $X_h$ for the presentation $\{X_h\}$ of $X$. 
   It suffices thus to show that    the map
   \begin{equation}
   \label{rzut1}
 \LL\colim_h(\rg_{\hk}(X_h)\wh{\otimes}^{\R}_{F^{\nr}}\breve{F}\wh{\otimes}^{\R}_{\breve{F}}W) \to    (\LL\colim_h\rg_{\hk}(X_h)\wh{\otimes}^{\R}_{F^{\nr}}\breve{F})\wh{\otimes}^{\R}_{\breve{F}}W
   \end{equation}
   is a strict quasi-isomorphism.   Applying the Hyodo-Kato morphism we pass to the canonical map
   $$
      \LL\colim_h\rg_{\dr}(X_h)\wh{\otimes}^{\R}_{\breve{F}}W\to  (\LL\colim_h\rg_{\dr}(X_h))\wh{\otimes}^{\R}_{\breve{F}}W,
   $$
 which is a strict quasi-isomorphism by \cite[2.1.2]{CDN3}. Now we go back to the map (\ref{rzut1}) by a projection $C\to \breve{F}$.
 
   We pass now to the second claim of the lemma. Assume first that $X$ is a smooth dagger affinoid. Then we define the map 
  $$
  \alpha_{\hk}(W):\quad  \rg^{\dagger}_{\hk}(X)\wh{\otimes}^{\R}_{F^{\nr}}W\to \rg_{\hk}^{\rm GK}(X)\wh{\otimes}^{\R}_{F^{\nr}}W
   $$
   as the composition
    \begin{align}\label{map120}
  \rg^{\dagger}_{\hk}(X)\wh{\otimes}^{\R}_{F^{\nr}}W & =\LL\colim_h\rg_{\hk}(X_h)\wh{\otimes}^{\R}_{F^{\nr}}W\stackrel{\sim}{\to} \LL\colim_h\rg_{\hk}(X^{\circ}_h)\wh{\otimes}^{\R}_{F^{\nr}}W\\
  & \stackrel{\sim}{\leftarrow}\LL\colim_h\rg^{\gk}_{\hk}(X^{{\circ},\dagger}_h)\wh{\otimes}^{\R}_{F^{\nr}}W\to
   \rg^{\gk}_{\hk}(X)\wh{\otimes}^{\R}_{F^{\nr}}W. \notag
  \end{align}
  The third map  is a strict quasi-isomorphism by Corollary \ref{ogrod1}: this is   because the interior $X^{\rm o}_h$ is partially proper.  
  
  For a general $X$, we obtain the map $\alpha_{\hk}(W)$ by globalizing the above definition. Changing $F$ into $\breve{F}$ in the definition of $\alpha_{\hk}(W)$, we get the map $\alpha_{\hk,\breve{F}}(W)$ compatible with the map $\alpha_{\hk}(W)$. This gives us the commutative diagram we wanted. Moreover, it is clear from the definitions that the right vertical map in the diagram is a strict quasi-isomorphism. The top map is a strict quasi-isomorphism by Lemma \ref{cwir-cwir}. 
The left vertical map is a strict quasi-isomorphism because we can check it locally where  claim (1) reduces us to the dagger cohomology of an affinoid and there this is clear from the definitions. It follows then that the bottom map is a strict quasi-isomorphism as well, as wanted. 
    \end{proof}

  (vi) {\em Properties of overconvergent Hyodo-Kato cohomology.} Let $X$ be a smooth dagger variety over $C$.  Recall that (see \cite[Prop. 4.38]{CN3})  the Hyodo-Kato cohomology $\wt{H}^*_{\hk}(X)$ is classical. 
If $X$ is quasi-compact it  is a finite dimensional $F^{\nr}$-vector space with its natural topology. For a general $X$, 
it is a limit in $C_F$ of finite dimensional $F^{\nr}$-vector spaces. 
The endomorphism $\phi$ on $H^*_{\hk}(X)$ is  a  homeomorphism.

  We will need the following computation later on:
  \begin{proposition}\label{tluczenie}Let $X$ be a smooth dagger variety over $C$. Let $W$ be a Banach  space with an $\breve{F}$-module structure. 
\begin{enumerate}
\item If $X$ is quasi-compact then the cohomology of the complex $\R\Gamma_{\hk}(X)\wh{\otimes}_{F^{\nr}}^{\R}W$ is classical and we have an $\breve{F}$-linear topological isomorphism
$$
\wt{H}^i(\rg_{\hk}(X)\wh{\otimes}^{\R}_{F^{\nr}}W)\simeq H^i_{\hk}(X)\wh{\otimes}_{F^{\nr}}W,\quad i\geq 0.
$$
\item Take an  increasing admissible covering  $\{U_n\}_{n\in\N}$ of $X$ by quasi-compact dagger varieties $U_n$. Then we have a natural strict quasi-isomorphism in $\sd(C_{\breve{F}})$
$$
\rg_{\hk}(X)\wh{\otimes}^{\R}_{F^{\nr}}W\stackrel{\sim}{\to} \R\wlim_n(\rg_{\hk}(U_{n,C})\wh{\otimes}^{\R}_{F^{\nr}}W).
$$
The cohomology of $\rg_{\hk}(X)\wh{\otimes}^{\R}_{F^{\nr}}W$ is classical and we have, for $i\geq 0$, an $\breve{F}$-linear topological isomorphism
$$
\wt{H}^i(\rg_{\hk}(X)\wh{\otimes}^{\R}_{F^{\nr}}W)\simeq H^i_{\hk}(X)\wh{\otimes}^{\R}_{F^{\nr}}W:=\wlim_n(H^i_{\hk}(U_n){\otimes}_{F^{\nr}}W).
$$
In particular, it is a Fr\'echet space\footnote{We note that $H^i_{\hk}(U_n)$ is a finite rank vector space over $F^{\nr}$ equipped with the canonical topology.}.
\end{enumerate}
\end{proposition}
\begin{proof} By Lemma \ref{passage}, we may replace $\rg_{\hk}(-)$ with  Grosse-Kl\"onne's version $\rg_{\hk}^{\rm GK}(-)$. Let $X$ be quasi-compact. Consider an \'etale hypercovering $\su_{\cdot}$ of $X$ built from quasi-compact models from $\sm^{\dagger,\sem,b}_C$. By \cite[Ex. 3.16]{CDN3}, claim  (1) is true for every $\su_{i,C}$. Hence we have the spectral sequence
$$
E_2^{i,j}= H^{{\rm GK},i}_{\hk}(\su_{j,C}){\otimes}_{F^{\nr}}W\Rightarrow \wt{H}^{i+j}(\R\Gamma^{\rm GK}_{\hk}(X)\wh{\otimes}_{F^{\nr}}^{\R}W).
$$
The terms of the spectral sequence are Banach spaces and the differentials in the spectral sequence are $W$-linear. Since the Hyodo-Kato cohomology groups
$H^{{\rm GK},i}_{\hk}(\su_{j,C})$ are of finite rank, claim (1) follows. 

 Having (1), claim (2) follows just as in the proof of \cite[3.26]{CDN3} (note that the system $\{H^i_{\hk}(U_n){\otimes}_{F^{\nr}}W\}_{n\in\N}$ satisfies the Mittag-Leffler condition).
 \end{proof}

  \subsubsection{Overconvergent geometric Hyodo-Kato morphism via presentations of dagger structures} In this section we introduce a definition of  overconvergent geometric Hyodo-Kato morphism using presentations of dagger structures.
  
   (i) {\em Local definition}. Let $X$ be a dagger affinoid over $C$. Let ${\rm pres}(X)=\{X_h\}$. 
      
   \index{iotahk@\iotahk}Define natural   Hyodo-Kato morphisms in $\sd(C_{\breve{F}})$
   \begin{align}
   \label{map13}
  \iota^{\dagger}_{\hk}:\quad   \rg^{\dagger}_{\hk,\breve{F}}(X)\to \rg^{\dagger}_{\dr}(X),\quad 
   \iota^{\dagger}_{\hk}:\quad   \rg^{\dagger}_{\hk,\breve{F}}(X)\to \rg^{\dagger}_{\dr}(X/\B^+_{\dr})
  \end{align}
  as the compositions
  \begin{align*}
   &  \rg^{\dagger}_{\hk,\breve{F}}(X)=\LL\colim_h\rg_{\hk,\breve{F}}(X_h)\veryverylomapr{\LL\colim_h(\iota_{\hk})}\LL\colim_h\rg_{\dr}(X_h)=\rg^{\dagger}_{\dr}(X),\\
  & \rg^{\dagger}_{\hk,\breve{F}}(X)=\LL\colim_h\rg_{\hk,\breve{F}}(X_h)\veryverylomapr{\LL\colim_h(\iota_{\hk})}\LL\colim_h\rg_{\dr}(X_h/\B^+_{\dr})=\rg^{\dagger}_{\dr}(X/\B^+_{\dr}).
  \end{align*}
  They are compatible via the map $\theta:\B^+_{\dr}\to C$.
   \begin{proposition} \label{ogrod2}The linearizations of the Hyodo-Kato morphisms in (\ref{map13}) yield  compatible natural  strict Hyodo-Kato quasi-isomorphisms in, resp., $\sd(C_{\breve{F}})$ and $\sd(C_{\B^+_{\dr}})$
    \begin{align*}
      \iota^{\dagger}_{\hk}:\quad   \rg^{\dagger}_{\hk,\breve{F}}(X)\wh{\otimes}^{\R}_{\breve{F}}C\stackrel{\sim}{\to} \rg^{\dagger}_{\dr}(X),\quad 
   \iota^{\dagger}_{\hk}:\quad   \rg^{\dagger}_{\hk,\breve{F}}(X)\wh{\otimes}^{\R}_{\breve{F}}\B^+_{\dr}\stackrel{\sim}{\to} \rg^{\dagger}_{\dr}(X/\B^+_{\dr}).
\end{align*}
   \end{proposition}
   \begin{proof}
 For the first map, we need to show that  the map 
 $$
 (\LL\colim_h\rg_{\hk,\breve{F}}(X_h))\wh{\otimes}_{\breve{F}}^{\R}C\veryverylomapr{\LL\colim_h(\iota_{\hk})}\LL\colim_h\rg_{\dr}(X_h)
 $$
 is a  strict quasi-isomorphism. But this map fits into a commutative diagram
 $$
 \xymatrix{
   \LL\colim_h\rg_{\hk,\breve{F}}(X_h)\wh{\otimes}_{\breve{F}}^{\R}C \ar[r] \ar[d]^{\LL\colim_h(\iota_{\hk})}_{\wr} & (\LL\colim_h\rg_{\hk,\breve{F}}(X_h))\wh{\otimes}_{\breve{F}}^{\R}C\ar[dl]^{\LL\colim_h(\iota_{\hk})}\\
  \LL\colim_h\rg_{\dr}(X_h)
 }
 $$
 The bottom term is just the overconvergent de Rham cohomology and its cohomology is classical and  a finite rank vector space over $C$ with its canonical topology. Via the vertical strict quasi-isomorphism the same is true of the upper left term. The upper right term is strictly quasi-isomorphic to $\rg_{\hk,\breve{F}}(X)\wh{\otimes}^{\R}_{\breve{F}}C$, which, by Lemma \ref{passage} and Proposition \ref{tluczenie}, also has classical cohomology that is  finite rank over $C$. Hence, looking at the above diagram one cohomology degree at a time, we obtain a commutative diagram of finite rank vector spaces over $C$. These ranks are, in fact, equal: this is clear for the bottom and the upper left term; for the upper right term consider the maps: 
 $$
 \rg_{\hk,\breve{F}}(X)\wh{\otimes}^{\R}_{\breve{F}}C\simeq \rg^{\rm GK}_{\hk}(X)\wh{\otimes}^{\R}_{{F}^{\nr}}C\verylomapr{\iota^{\rm GK}_{\hk}}\rg_{\dr}(X).
 $$
The first map is a strict quasi-isomorphism by Lemma \ref{passage}. The second map is the Grosse-Kl\"onne Hyodo-Kato morphism and it is a strict quasi-isomorphism by  \cite[5.15]{CN3}. Hence the rank in question is the same as that of the corresponding de Rham cohomology, as wanted.

  For the second map in our proposition, we argue in a similar fashion. We need to show that  the map 
 $$
 (\LL\colim_h\rg_{\hk,\breve{F}}(X_h))\wh{\otimes}_{\breve{F}}^{\R}\B^+_{\dr}\veryverylomapr{\LL\colim_h(\iota_{\hk})}\LL\colim_h\rg_{\dr}(X_h/\B^+_{\dr})
 $$
 is a  strict quasi-isomorphism. But this map fits into a commutative diagram
 $$
 \xymatrix{
   \LL\colim_h\rg_{\hk,\breve{F}}(X_h)\wh{\otimes}_{\breve{F}}^{\R}\B^+_{\dr} \ar[r] \ar[d]^{\LL\colim_h(\iota_{\hk})}_{\wr} & (\LL\colim_h\rg_{\hk,\breve{F}}(X_h))\wh{\otimes}_{\breve{F}}^{\R}\B^+_{\dr}\ar[dl]^{\LL\colim_h(\iota_{\hk})}\\
  \LL\colim_h\rg_{\dr}(X_h/\B^+_{\dr})
 }
 $$
The vertical map is a strict quasi-isomorphism by (\ref{rynek2}). The horizontal map can be shown to be a strict quasi-isomorphism by an argument analogous to the one used in the proof of Proposition \ref{prison}. 
 It follows that so is the slanted map, as wanted. 
  \end{proof}
    (ii) {\em Globalization}. For a general smooth dagger variety $X$ over $C$, 
     globalizing the maps  $\iota^{\dagger}_{\hk}$  from (\ref{map13}), we obtain  compatible 
natural \index{iotahk@\iotahk}maps  in $\sd(C_{\breve{F}})$
  $$
  \iota_{\hk}: \rg_{\hk,\breve{F}}(X)\to \rg_{\dr}(X),\quad   \iota_{\hk}: \rg_{\hk,\breve{F}}(X)\to \rg_{\dr}(X/\B^+_{\dr}).
  $$
\begin{theorem}{\rm (Overconvergent Hyodo-Kato isomorphisms)}
\label{HK-dagger2}
   The linearizations of the above Hyodo-Kato morphisms yields compatible natural strict quasi-isomorphisms in, resp., $\sd(C_C)$ and $\sd(C_{\B^+_{\dr}})$
  \begin{align*}
  &  \iota_{\hk}: \rg_{\hk,\breve{F}}(X)\wh{\otimes}^{\R}_{\breve{F}}C\stackrel{\sim}{\to }\rg_{\dr}(X),\\
   &  \iota_{\hk}: \rg_{\hk,\breve{F}}(X)\wh{\otimes}^{\R}_{\breve{F}}\B^+_{\dr}\stackrel{\sim}{\to} \rg_{\dr}(X/\B^+_{\dr}).
 \end{align*}
  \end{theorem}
  \begin{proof} Looking at $\eta$-\'etale hypercoverings and using that our tensor products commute with products,  we may assume $X$ to be a dagger  affinoid and then the result is known by Proposition \ref{ogrod2}.
\end{proof}

   (iii) {\em Application.} As an immediate application of the overconvergent Hyodo-Kato quasi-isomorphisms we get the local-global compatibility for $\B^+_{\dr}$-cohomology: .
\begin{corollary}\label{local-global-kwak} 
{\rm ({Local-global compatibility})} Let $X$ be a smooth dagger affinoid over $C$.   The canonical morphism in $\sd\sff(C_{\B^+_{\dr}})$
\begin{align}
\label{mumu1}
  \rg^{\dagger}_{\dr}(X/\B^+_{\dr})\to \rg_{\dr}(X/\B^+_{\dr})
\end{align}
is a    strict quasi-isomorphism. 
\end{corollary}
\begin{proof}
Consider the following commutative diagram
$$
\xymatrix@R=6mm{\rg^{\dagger}_{\hk,\breve{F}}(X)\wh{\otimes}^{\R}_{\breve{F}}\B^+_{\dr}\ar[d]^{\iota^{\dagger}_{\hk}}_{\wr} \ar[r]^{\sim} & \rg_{\hk,\breve{F}}(X)\wh{\otimes}^{\R}_{\breve{F}}\B^+_{\dr}\ar[d]^{\iota_{\hk}}_{\wr}  \\
  \rg^{\dagger}_{\dr}(X/\B^+_{\dr})\ar[r] &  \rg_{\dr}(X/\B^+_{\dr})
}
$$
The vertical arrows are strict quasi-isomorphisms by Proposition \ref{ogrod2} and Theorem \ref{HK-dagger2}. The top arrow is a strict quasi-isomorphism by the local-global compatibility for completed overconvergent Hyodo-Kato cohomology. 
 It follows that so is the bottom horizontal arrow, proving that the map (\ref{mumu1}) is a strict quasi-isomorphism.
 
  To show that this map is a filtered strict quasi-isomorphism, we will argue by induction on $r\geq 0$. The inductive step uses 
  the following commutative diagram
  $$
  \xymatrix@=5mm{
   F^{r-1}\rg^{\dagger}_{\dr}(X/\B^+_{\dr})\ar[r]^{t} \ar[d]^{\wr}& F^r\rg^{\dagger}_{\dr}(X/\B^+_{\dr})\ar[d]\ar[r]^{\vartheta} & F^r\rg^{\dagger}_{\dr}(X)\ar[d]^{\wr},\\
 F^{r-1}\rg_{\dr}(X/\B^+_{\dr})\ar[r]^{t} &  F^r\rg_{\dr}(X/\B^+_{\dr})\ar[r]^{\vartheta} & F^r\rg_{\dr}(X),
  }
  $$
  in which the rows are distinguished triangles by Proposition \ref{prison} and its proof. The first and the third vertical maps are strict quasi-isomorphism by the inductive hypothesis and by the local-global property for filtered de Rham cohomology (see \cite[Sec. 5.1]{CN3}), respectively. It follows that the middle vertical map is a strict quasi-isomorphism as well, as wanted.
\end{proof}
\subsubsection{Comparison with the rigid analytic constructions}Let $X$ be a smooth dagger variety over $L=K,C$. Let $\wh{X}$ be its completion. 
 \begin{lemma}
 \begin{enumerate}
 \item There is a natural morphism in $\sd_{\phi,N}(C_{F^{\rm nr}})$
 \begin{equation}           
 \label{wrzesien1}
 \rg_{\hk}(X)\to\rg_{\hk}(\wh{X}). 
 \end{equation}
 \item Let $L=C$. There are compatible natural morphisms in, resp., $\sd_{\phi,N}(C_{\breve{F}})$ and $\sd\sff(C_{\B^+_{\dr}})$
 \begin{align*}
 \rg_{\hk,\breve{F}}(X)\to\rg_{\hk,\breve{F}}(\wh{X}),\quad 
   \rg_{\dr}(X/\B^+_{\dr})\to\rg_{\dr}(\wh{X}/\B^+_{\dr}). 
 \end{align*}
 They are  compatible with the map (\ref{wrzesien1}).
 \item The morphism in (2) are compatible with the Hyodo-Kato morphisms, i.e., we have the commutative diagrams in $\sd(C_{\breve{F}})$
 $$
 \xymatrix@R=6mm{
   \rg_{\hk,\breve{F}}(X)\ar[r]\ar[d]^{\iota_{\hk}} &  \rg_{\hk,\breve{F}}(\wh{X})\ar[d]^{\iota_{\hk}}\\
   \rg_{\dr}(X)\ar[r] & \rg_{\dr}(\wh{X}),
 }\quad 
 \xymatrix@R=6mm{
  \rg_{\hk,\breve{F}}(X)\ar[r]\ar[d]^{\iota_{\hk}} &  \rg_{\hk,\breve{F}}(\wh{X})\ar[d]^{\iota_{\hk}}\\
   \rg_{\dr}(X/\B^+_{\dr})\ar[r] & \rg_{\dr}(\wh{X}/\B^+_{\dr})
 }
 $$
 \end{enumerate}
 \end{lemma}
 \begin{proof}
 Let $X$ be a smooth dagger affinoid over $L$ with the presentation $\{X_h\}$.  Using the compatible maps $\wh{X}\to X_h$, we define the map 
 \begin{align*}
 \rg^{\dagger}_{\hk}(X)=\LL\colim_h\rg_{\hk}(X_h)\to \LL\colim_h\rg_{\hk}(\wh{X})=\rg_{\hk}(\wh{X}).
 \end{align*}
 It globalizes to give the map in (\ref{wrzesien1}). 
 
  We proceed in a similar way for the other two cohomologies. The stated compatibilities follow  easily from the definitions.
 \end{proof}
 
  The  Hyodo-Kato quasi-isomorphisms imply the following:  
\begin{corollary}\label{roznosci1}
     Let $X\in {\rm Sm}^{\dagger}_C$. 
   If $X$ is partially proper, then the canonical  morphisms  in, resp., $\sd_{\phi,N}(C_{\breve{F}})$ and $\sd\sff(C_{\B^+_{\dr}})$
    \begin{equation}
    \label{cieplo1}
  \rg_{\hk,\breve{F}}(X)\to \rg_{\hk,\breve{F}}(\wh{X}),\quad  \rg_{\dr}(X/\B^+_{\dr})\to \rg_{\dr}(\wh{X}/\B^+_{\dr})
    \end{equation}
are    strict quasi-isomorphisms. 
     \end{corollary}
     \begin{proof}     
For the first map, consider the commutative diagram
   $$
   \xymatrix@R=6mm{  
   \rg_{\hk,\breve{F}}(X)\wh{\otimes}^{\R}_{\breve{F}}C\ar[r]\ar[d]^{\iota_{\hk}}_{\wr} &  \rg_{\hk,\breve{F}}(\wh{X})\wh{\otimes}^{\R}_{\breve{F}}C\ar[d]^{\iota_{\hk}}_{\wr}\\
   \rg_{\dr}(X)\ar[r]^{\sim} & \rg_{\dr}(\wh{X})
   }
   $$
 It implies that the top arrow is a strict quasi-isomorphism. Splitting off $\breve{F}$ from $C$ we obtain
 the claim of the corollary.  
 
 For the second map,  in the unfiltered case, consider the commutative diagram 
$$ \xymatrix@R=6mm{
  \rg_{\hk,\breve{F}}(X)\wh{\otimes}^{\R}_{\breve{F}}\B^+_{\dr}\ar[r]^{\sim} \ar[d]^{\iota_{\hk}}_{\wr} &  \rg_{\hk,\breve{F}}(\wh{X})\wh{\otimes}^{\R}_{\breve{F}}\B^+_{\dr}\ar[d]^{\iota_{\hk}}_{\wr}\\
   \rg_{\dr}(X/\B^+_{\dr})\ar[r] & \rg_{\dr}(\wh{X}/\B^+_{\dr})
 }
 $$
The top arrow is a strict quasi-isomorphism by what was just proved.  It implies that the bottom arrow is a strict quasi-isomorphism, as wanted.

  To treat filtrations, we proceed by  induction on the filtration level $r$ (the base case of $r=0$ just proved). The inductive step ($r-1 \Rightarrow r$) uses the commutative diagram 
  $$
  \xymatrix@R=5mm{
   F^{r-1}\rg_{\dr}(X/\B^+_{\dr})\ar[r]^{t} \ar[d]^{\wr}& F^r\rg_{\dr}(X/\B^+_{\dr})\ar[d]\ar[r]^{\vartheta} & F^r\rg_{\dr}(X)\ar[d]^{\wr},\\
 F^{r-1}\rg_{\dr}(\wh{X}/\B^+_{\dr})\ar[r]^{t} &  F^r\rg_{\dr}(\wh{X}/\B^+_{\dr})\ar[r]^{\vartheta} & F^r\rg_{\dr}(\wh{X}),
  }
  $$
  in which the rows are distinguished triangles by  Proposition \ref{prison} and Proposition \ref{projection}. The first vertical map is a  strict quasi-isomorphism by the inductive hypothesis. 
It follows that the middle vertical map is a strict quasi-isomorphism as well, as wanted.
 \end{proof}

\section{Overconvergent geometric syntomic cohomology} In this section we will define overconvergent geometric syntomic cohomology and  prove a comparison theorem for smooth dagger affinoids and Stein varieties over $C$.   
   \subsection{Local-global compatibility for rigid analytic geometric syntomic cohomology} 
   Recall that in \cite[Sec. 4.1]{CN3} the syntomic cohomology 
\index{rgsyn@\rgsyn}$\rg_{\synt}(X,\Q_p(r))\in\sd(C_{\Q_p})$ of a rigid analytic variety $X$ is defined by  $\eta$-\'etale descent  from the crystalline syntomic cohomology of Fontaine-Messing.
The latter is defined as   the fiber  ($\sx$ is a semistable formal scheme over $\so_C$ equipped with its canonical 
\index{rgsyn@\rgsyn}log-structure)
  $$\rg_{\synt}(\sx,\Q_p(r)):=[F^r\rg_{\crr}(\sx)\lomapr{\phi-p^r}\rg_{\crr}(\sx)], 
  $$ where the (logarithmic) crystalline cohomology is absolute (i.e., over $\Z_p$).
 By definition, it fits into the distinguished triangle in $\sd(C_{\Q_p})$
  \begin{equation}
  \label{triangle222}
  \rg_{\synt}(X,\Q_p(1))\to [\rg_{\crr}(X)]^{\phi=p^r}\to \rg_{\crr}(X)/F^r
  \end{equation}
 
   We were not able to prove the local-global compatibility for this syntomic cohomology in \cite{CN3}: the usual technique is to pass from the second term of (\ref{triangle222}) to Hyodo-Kato cohomology and from the third term -- to filtered de Rham cohomology; then one passes, via the Hyodo-Kato quasi-isomorphism,  from Hyodo-Kato cohomology to de Rham cohomology and we do have local-global compatibility for filtered de Rham cohomology. The problem was: we did not have then the Hyodo-Kato morphism. But we have it now
thanks to Theorem~\ref{rynek2.0}, so in this section we
 will prove the local-global compatibility for rigid analytic geometric syntomic cohomology that we will need. 
 
   We start with stating such a compatibility for absolute crystalline cohomology.
\begin{lemma}{\rm (Crystalline local-global compatibility)} \label{abs-crr}Let $\sx\in \sm^{{\rm ss},b}_C$. 
The canonical map in $\sd_{\phi}(C_{\B^+_{\crr}})$
$$\rg_{\crr}(\sx)_{\Q_p}\to \rg_{\crr}(\sx_C)$$
is a strict  quasi-isomorphism.
\end{lemma}
\begin{proof}
We have the commutative diagram in $\sd_{\phi}(C_{\B^+_{\crr}})$
$$
\xymatrix@R=6mm{
\rg_{\crr}(\sx)_{\Q_p}\ar[r] &  \rg_{\crr}(\sx_C)\\
\rg_{\hk,\breve{F}}(\sx^0_1)_{\Q_p}\wh{\otimes}^{\R}_{\breve{F}}\B^+_{\crr}\ar[u]^{\wr}_{\epsilon_{\crr}^{\hk}} \ar[r]^{\beta\otimes\id}  & \rg_{\hk,\breve{F}}(\sx_{C})\wh{\otimes}^{\R}_{\breve{F}}\B^+_{\crr} \ar[u]^{\wr}_{\epsilon^{\hk}_{\crr}}
}
$$
The bottom map is a strict quasi-isomorphism by Lemma \ref{loc-global1}. Hence so is the top map, as wanted.
\end{proof}

\begin{proposition}{\rm (Syntomic local-global compatibility)} \label{synt-locglob}
Let $\sx\in \sm^{{\rm ss},b}_C$. Let $r\geq 0$. 
The canonical map in $\sd(C_{\Q_p})$
$$\rg_{\synt}(\sx,\Z_p(r))_{\Q_p}\to \rg_{\synt}(\sx_C,\Q_p(r))$$
is a strict  quasi-isomorphism. Here $\rg_{\synt}(\sx,\Z_p(r))$ is the syntomic cohomology of Fontaine-Messing \cite{FM} {\rm (see also \cite{BEp})}. 
\end{proposition}
\begin{proof} Set $X:=\sx_C$. 
First, we define a natural strict quasi-isomorphism in $\sd(C_{\Q_p})$:
  \begin{align*}
   \iota_2: \quad & 
    \big[ [\rg_{\hk}({X})\wh{\otimes}_{F^{\nr}}{\B}^+_{\st}]^{N=0,\phi=p^r}\lomapr{\iota_{\hk}\otimes\iota} \rg_{\dr}({X}/\B^+_{\dr})/F^r\big]\\
    & \to  [[\R\Gamma_{\crr}({X})]^{\phi=p^r}\lomapr{\can} \R\Gamma_{\crr}({X})/F^r]= \rg_{\synt}({X},\Q_p(r)),
  \end{align*}
  where we set
  $$[\rg_{\hk}({X})\wh{\otimes}_{F^{\nr}}{\B}^+_{\st}]^{N=0,\phi=p^r}:= \left[
\begin{aligned}\xymatrix@=40pt{
 \rg_{\hk}({X})\wh{\otimes}_{F^{\nr}}{\B}^+_{\st}\ar[d]^{N}\ar[r]^-{1-\phi_r} &
\rg_{\hk}({X})\wh{\otimes}_{F^{\nr}}{\B}^+_{\st}\ar[d]^{N}\\
 \rg_{\hk}({X})\wh{\otimes}_{F^{\nr}}{\B}^+_{\st}\ar[r]^-{1-\phi_{r-1}} & \rg_{\hk}({X})\wh{\otimes}_{F^{\nr}}{\B}^+_{\st}
 }\end{aligned}\right]
$$
    For that, it suffices to define the maps $\iota_{\rm BK}^1$ and $\iota_{\rm BK}^2$  in the following diagram, with the first map being Frobenius equivariant,  and to show that this diagram commutes in $\sd(C_{\Q_p})$:
    \begin{equation}
    \label{film1}
\xymatrix{
[\R\Gamma_{\hk}({X})\wh{\otimes}_{F^{\nr}}{\B}^+_{\st}]^{N=0}\ar[d]^{\epsilon^{\hk}_{\st}}_{\wr} \ar@/_80pt/[dd]^{\wr}_{\iota_{\rm BK}^1}
\ar[r]^-{\iota_{\hk}\otimes \iota}   &   
   \rg_{\crr}({X}/\B^+_{\dr})/F^r
 \ar@/^80pt/[dd]^{\iota_{\rm BK}^2}_{\wr}\\
     [\R\Gamma_{\crr}({X})\wh{\otimes}_{\B^+_{\crr}}{\B}^+_{\st}]^{N=0}\ar[ru]^{\kappa} \\
         \R\Gamma_{\crr}({X})\ar[u]^{\wr}_{\beta}\ar[r]^{\can}  &        \R\Gamma_{\crr}({X})/F^r\ar[uu]^{\kappa}_{\wr}.
}
\end{equation}  
 Here the map $\epsilon^{\hk}_{\st}$ is the one from \eqref{epsilon1}. It is Frobenius equivariant.  We set:  $\iota_{\rm BK}^1:=\beta^{-1}\epsilon^{\hk}_{\st}$ and $\iota_{\rm BK}^2:=\kappa^{-1}$. Since the map $\beta$ is Frobenius equivariant so is the map $\iota_{\rm BK}^1$.  By definition,   all the pieces of  the diagram commute and the maps  $\iota_{\rm BK}^1, \iota_{\rm BK}^2$ are  strict quasi-isomorphisms. 
  
    The morphism $\iota_2$ has a compatible local version. Now, the wanted local-global compatibility, via the strict quasi-isomorphisms $\iota_2$,  follows from local-global compatibility for Hyodo-Kato cohomology  and filtered $\B^+_{\dr}$-cohomology, proved in Proposition \ref{main00} and Lemma \ref{etale-descent20}, respectively.
\end{proof}

   The proof of Proposition  \ref{synt-locglob} actually shows the following:
   \begin{corollary}
   \label{presentation}
   Let $X\in{\rm Sm}_C$ and $r\geq 0$. There exist a natural strict quasi-isomorphism in $\sd(C_{\Q_p})$
   \begin{equation}
   \label{pres11}
   \rg_{\synt}(X,\Q_p(r))\simeq \big[[\rg_{\hk}(X)\wh{\otimes}_{F^{\nr}}\B^+_{\st}]^{N=0,\phi=p^r}\lomapr{\iota_{\hk}\otimes\iota}\rg_{\dr}(X/\B^+_{\dr})/F^r\big].
   \end{equation}
   \end{corollary}
We like to call the expression on the right {\em the Bloch-Kato syntomic cohomology} because it resembles the definition of Bloch-Kato Selmer groups in \cite{BK}.

 \subsection{Twisted Hyodo-Kato cohomology}\label{HK11}
Let $X$ be a smooth dagger variety over $C$.  In this section we will study the twisted Hyodo-Kato 
\index{hkr@\hkr}cohomology  in $\sd(C_{\Q_p})$
\begin{equation}
  \label{lwiatko1}
  {\rm HK}(X,r):=[\R\Gamma_{\hk}(X)\wh{\otimes}_{F^{\nr}}^{\R}\wh{\B}^+_{\st}]^{N=0,\phi=p^r},\quad r\geq 0,
  \end{equation}
  where  $\R\Gamma_{\hk}(X)$ is the geometric Hyodo-Kato cohomology defined in \cite[Sec.\,4.3.1]{CN3} and
 we set in $\sd_{\phi,N}(C_{\wh{\B}^+_{\st}})$
 \begin{align*}
  \R\Gamma_{\hk}(X)\wh{\otimes}^{\R}_{{F}^{\nr}}\wh{\B}^+_{\st}:=\LL\colim((\R\Gamma_{\hk}\wh{\otimes}^{\R}_{F^{\nr}}\wh{\B}^+_{\st})(U_{\cdot})), 
 \end{align*}
 where the homotopy colimit is taken over \'etale affinoid hypercoverings  $U_{\cdot}$ from ${\rm Sm}^{\dagger}_C$.  
  We wrote $[\R\Gamma_{\hk}(X)\wh{\otimes}_{F^{\nr}}^{\R}\wh{\B}^+_{\st}]^{N=0,\phi=p^r}$ for the homotopy limit of the commutative diagram  in $\sd(C_{\Q_p})$
  $$
\xymatrix@R=6mm{
\R\Gamma_{\hk}(X)\wh{\otimes}_{F^{\nr}}^{\R}\wh{\B}^+_{\st} \ar[r]^{\phi-p^r} \ar[d]^N& \R\Gamma_{\hk}(X)\wh{\otimes}_{F^{\nr}}^{\R}\wh{\B}^+_{\st}\ar[d]^N\\
\R\Gamma_{\hk}(X)\wh{\otimes}_{F^{\nr}}^{\R}\wh{\B}^+_{\st}  \ar[r]^{p\phi-p^r} & \R\Gamma_{\hk}(X)\wh{\otimes}_{F^{\nr}}^{\R}\wh{\B}^+_{\st}.
}
$$

  The following proposition generalizes the computations from \cite[Sec.\,3.2.2]{CDN3} done in the case when $X$ has a semistable integral model over a finite extension of $K$. 
\begin{proposition}\label{ciemno0}Let $X$ be a smooth dagger variety over $C$. Let $r\geq 0$.
\begin{enumerate}
\item If $X$ is quasi-compact then the cohomology of the complex $\R\Gamma_{\hk}(X)\wh{\otimes}_{F^{\nr}}^{\R}\wh{\B}^+_{\st}$ is classical and we have an isomorphism of $(\phi,N)$-modules over $F^{\rm nr}$ 
$$
\wt{H}^i(\rg_{\hk}(X)\wh{\otimes}^{\R}_{F^{\nr}}\wh{\B}^+_{\st})\simeq H^i_{\hk}(X)\wh{\otimes}_{F^{\nr}}\wh{\B}^+_{\st},\quad i\geq 0.
$$
\item  If $X$  is quasi-compact there is a natural isomorphism
$$
\wt{H}^i({\rm HK}(X,r))\simeq (H^i_{\hk}(X)\wh{\otimes}_{F^{\nr}}\wh{\B}^+_{\st})^{N=0,\phi=p^r},\quad i\geq 0,
$$
of Banach spaces. In particular, $\wt{H}^i({\rm HK}(X,r))$ is classical. 
\item Take an  increasing admissible covering  $\{U_n\}_{n\in\N}$ of $X$ by quasi-compact dagger varieties $U_n$. Then we have a natural strict quasi-isomorphism in 
$\sd(C_{\wh{\B}^+_{\st}})$
$$
\rg_{\hk}(X)\wh{\otimes}^{\R}_{F^{\nr}}\wh{\B}^+_{\st}\stackrel{\sim}{\to} \R\wlim_n(\rg_{\hk}(U_{n,C})\wh{\otimes}^{\R}_{F^{\nr}}\wh{\B}^+_{\st}).
$$
The cohomology of $\rg_{\hk}(X)\wh{\otimes}^{\R}_{F^{\nr}}\wh{\B}^+_{\st}$ is classical and we have, for $i\geq 0$, an isomorphism of $(\phi,N)$-modules over $\wh{\B}^+_{\st}$ 
$$
\wt{H}^i(\rg_{\hk}(X)\wh{\otimes}^{\R}_{F^{\nr}}\wh{\B}^+_{\st})\simeq H^i_{\hk}(X)\wh{\otimes}^{\R}_{F^{\nr}}\wh{\B}^+_{\st}:=\wlim_n(H^i_{\hk}(U_n){\otimes}_{F^{\nr}}\wh{\B}^+_{\st}).
$$
In particular, it is a Fr\'echet space\footnote{We note that $H^i_{\hk}(U_n)$ is a finite rank vector space over $F^{\nr}$ equipped with the canonical topology.}.
\item The cohomology $\wt{H}^i([\rg_{\hk}(X)\wh{\otimes}^{\R}_{F^{\nr}}\wh{\B}^+_{\st}]^{N=0,\phi=p^r})$, $i\geq 0$,  is classical and we have natural isomorphisms in $\sd(C_{\Q_p})$
$$
{H}^i([\rg_{\hk}(X)\wh{\otimes}^{\R}_{F^{\nr}}\wh{\B}^+_{\st}]^{N=0,\phi=p^r})\simeq (H^i_{\hk}(X)\wh{\otimes}^{\R}_{F^{\nr}}\wh{\B}^+_{\st})^{N=0,\phi=p^r},\quad i\geq 0.
$$
In particular, the space ${H}^i([\rg_{\hk}(X)\wh{\otimes}^{\R}_{F^{\nr}}\wh{\B}^+_{\st}]^{N=0,\phi=p^r})$ is Fr\'echet. Moreover,
$$
{H}^i([\rg_{\hk}(X)\wh{\otimes}^{\R}_{F^{\nr}}\wh{\B}^+_{\st}]^{N=0})\simeq  (H^i_{\hk}(X)\wh{\otimes}^{\R}_{F^{\nr}}\wh{\B}^+_{\st})^{N=0}\simeq H^i_{\hk}(X)\wh{\otimes}^{\R}_{F^{\nr}}{\B}^+_{\crr},
$$
where the last isomorphism is not, in general, Galois equivariant (in the case $X$ comes from $X_K$ over $K$).
\end{enumerate}
\end{proposition}
\begin{proof}Since $\wh{\B}^+_{\st}$ is a Banach space over $\breve{F}$, claims (1) and (3) are a special case of Proposition \ref{tluczenie}. 
 Claim (2)  follows from (1) just as in the proof of \cite[Lemma 3.20]{CDN3}. Finally, claim (4) follows from (3) is proved as in \cite[Lemma 3.28]{CDN3}.
 \end{proof}
 \subsubsection{A variant of the twisted Hyodo-Kato cohomology} 
There is a variant of the twisted Hyodo-Kato 
\index{hkr@\hkr}cohomology  in $\sd(C_{\Q_p})$
$$
 \wt{\rm HK}(X,r):=[\R\Gamma_{\hk}(X)\wh{\otimes}_{F^{\nr}}^{\R}{\B}^+_{\st}]^{N=0,\phi=p^r},\quad r\geq 0,
$$
  that we will often use. 
 Here  we set in $\sd_{\phi,N}(C_{\B^+_{\st}})$
 \begin{align*}
  \R\Gamma_{\hk}(X)\wh{\otimes}^{\R}_{{F}^{\nr}}{\B}^+_{\st}:=\LL\colim((\R\Gamma_{\hk}\wh{\otimes}_{F^{\nr},\iota}{\B}^+_{\st})(U_{\cdot})), 
 \end{align*}
 where the homotopy colimit is taken over \'etale affinoid hypercoverings  $U_{\cdot}$ from ${\rm Sm}^{\dagger}_C$. We have  in $\sd_{\phi,N}(C_{\B^+_{\st}})$
  $$
 \rg_{\hk}(X)\wh{\otimes}_{F^{\nr},\iota}\B^+_{\st}\simeq \LL\colim_h(\rg_{\hk}(X_h)\wh{\otimes}^{\R}_{F^{\nr},\iota}\B^+_{\st}),
  $$
  where $\{X_h\}$ is the presentation of $X$. It is easy to check that this tensor product satisfies local-global compatibility.
  \begin{lemma} \label{cicho1}Let $X\in {\rm Sm}^{\dagger}_C$. The canonical morphism in $\sd_{\phi,N}(C_{F^{\rm nr}})$
  $$
   \wt{\rm HK}(X,r)\to  {\rm HK}(X,r),\quad r\geq 0.
  $$
  is a strict quasi-isomorphism. 
  \end{lemma}
  \begin{proof} It suffices to show that the canonical morphism
  $$
  [\R\Gamma_{\hk}(X)\wh{\otimes}_{F^{\nr}}^{\R}{\B}^+_{\st}]^{N=0}\to [\R\Gamma_{\hk}(X)\wh{\otimes}_{F^{\nr}}^{\R}\wh{\B}^+_{\st}]^{N=0}
  $$
  is a strict quasi-isomorphism. For that, from the definitions of both sides, we can assume that $X$ is a dagger affinoid. Then this map can be rewritten as 
  $$
    [\R\Gamma_{\hk}(X)\wh{\otimes}_{F^{\nr},\iota}{\B}^+_{\st}]^{N=0}\to [\R\Gamma_{\hk}(X)\wh{\otimes}_{F^{\nr}}\wh{\B}^+_{\st}]^{N=0},
$$
which, by Lemma \ref{passage}, can be written as 
$$
\LL\colim_h([\R\Gamma_{\hk}(X_h)\wh{\otimes}_{F^{\nr},\iota}{\B}^+_{\st}]^{N=0})\to \LL\colim_h([\R\Gamma_{\hk}(X_h)\wh{\otimes}_{F^{\nr}}\wt{\B}^+_{\st}]^{N=0}),
$$
for the presentation $\{X_h\}$ of $X$. But this map  is a strict quasi-isomorphism because so is the canonical map
$$
[\R\Gamma_{\hk}(X_h)\wh{\otimes}_{F^{\nr},\iota}{\B}^+_{\st}]^{N=0}\to [\R\Gamma_{\hk}(X_h)\wh{\otimes}_{F^{\nr}}\wt{\B}^+_{\st}]^{N=0},
$$
by the same argument as the one used to show   (\ref{Paris-hot}).
  \end{proof}
\subsection{$\B^+_{\dr}$-cohomology}   Let $X$ be a smooth dagger variety over $C$.  In this section we will study the filtered $\B^+_{\dr}$-cohomology $\rg_{\dr}(X/\B^+_{\dr})$ and its 
\index{drr@\drr}quotients 
$${\rm DR}(X,r):=\rg_{\dr}(X/\B^+_{\dr})/F^r,\quad r\geq 0.
$$
We note that, immediately from the distinguished triangle (\ref{detail11}), we obtain
\begin{lemma}\label{prison-break11}Let $X$ be a smooth dagger variety over $C$. 
Let $r\geq 0$. We have a distinguished triangle in $\sd\sff(C_{\B^+_{\dr}})$
$$
{\rm DR}(X,r-1)\lomapr{t} {\rm DR}(X,r)\lomapr{\vartheta} \rg_{\dr}(X)/F^r
$$
\end{lemma}

    By Theorem \ref{HK-dagger2},  we have the strict quasi-isomorphism in $\sd\sff(C_{\B^+_{\dr}})$
 \begin{equation}
 \label{comp1}
 \iota_{\hk}:\quad \rg_{\hk,\breve{F}}(X)\wh{\otimes}^{\R}_{\breve{F}}\B^+_{\dr}\stackrel{\sim}{\to} \rg_{\dr}(X/\B^+_{\dr}).
 \end{equation}
It yields the following computation: 
\begin{proposition}\label{tluczenie1}
Let $X$ be a smooth dagger variety over $C$. 
\begin{enumerate}
\item If $X$ is quasi-compact then the cohomology of the complex $\R\Gamma_{\dr}(X/\B^+_{\dr})$ is classical and we have $$
\wt{H}^i(\R\Gamma_{\dr}(X/\B^+_{\dr}))\simeq H^i_{\hk,\breve{F}}(X)\wh{\otimes}_{\breve{F}}\B^+_{\dr},\quad i\geq 0.
$$
\item Take an  increasing admissible covering  $\{U_n\}_{n\in\N}$ of $X$ by quasi-compact dagger varieties $U_n$. Then we have a natural strict quasi-isomorphism  in $\sd\sff(C_{\B^+_{\dr}})$
$$
\R\Gamma_{\dr}(X/\B^+_{\dr})\stackrel{\sim}{\to} \R\wlim_n\R\Gamma_{\dr}(U_n/\B^+_{\dr}).
$$
The cohomology of $\R\Gamma_{\dr}(X/\B^+_{\dr})$ is classical and we have, for $i\geq 0$, 
$$
\wt{H}^i(\R\Gamma_{\dr}(X/\B^+_{\dr}))\simeq H^i_{\hk,\breve{F}}(X)\wh{\otimes}^{\R}_{\breve{F}}\B^+_{\dr}\simeq \wlim_n(H^i_{\hk,\breve{F}}(U_n){\otimes}_{\breve{F}}\B^+_{\dr}).
$$
In particular, it is a Fr\'echet space\footnote{Recall that $H^i_{\hk,\breve{F}}(U_n)$ is a finite rank vector space over $\breve{F}$ equipped with the canonical topology.}.
\end{enumerate}
\end{proposition}
\begin{proof}
Using the Hyodo-Kato morphism (\ref{comp1}), we may pass to the computation of the cohomology of the complex $\rg_{\hk,\breve{F}}(X)\wh{\otimes}^{\R}_{\breve{F}}\B^+_{\dr}$. 
Since $\B^+_{\dr}\simeq\prod_{k\geq 0}Ct^k$ in $C_{\breve{F}}$, we have
$$\rg_{\hk,\breve{F}}(X)\wh{\otimes}^{\R}_{\breve{F}}\B^+_{\dr}\simeq \prod_{k\geq 0}(\rg_{\hk,\breve{F}}(X)\wh{\otimes}^{\R}_{\breve{F}}Ct^k)$$ and we can use Lemma \ref{passage} to pass to  $\rg^{\rm GK}_{\hk,\breve{F}}(X)\wh{\otimes}^{\R}_{\breve{F}}\B^+_{\dr}$. Then the proof of Proposition \ref{tluczenie} goes through.
\end{proof}
\subsubsection{Varieties over  $K$}  \label{sroka11} Before studying filtrations on $\B^+_{\dr}$-cohomology we will look more carefully  at the example of varieties defined over $K$. 

  Recall that (see \cite[Sec.\,5.1]{CN3}), for a smooth dagger variety $X$ over $L$, $L=K,C$,   the de Rham  cohomology $\wt{H}^i_{\dr}(X)$ is classical. 
If $X$ is quasi-compact it  is a finite dimensional $L$-vector space with its natural topology. For a general $X$, 
it is a limit in $C_{\Q_p}$ of finite dimensional $L$-vector spaces (hence a Fr\'echet space).

  Let $X\in {\rm Sm}^{\dagger}_K$. By Proposition \ref{prison}, we have the  strict quasi-isomorphisms  in $\sd\sff(C_{\B^+_{\dr}})$
\begin{align*}
\iota_{\rm BK}:\quad \R\Gamma_{\dr}(X)\wh{\otimes}_{K}^{\R}{\B}^+_{\dr} & \stackrel{\sim}{\to} \rg_{\dr}(X_C/\B^+_{\dr}),\\
  {\rm DR}(X_C,r) & \simeq (\R\Gamma_{\dr}(X)\wh{\otimes}_{K}^{\R}\B^+_{\dr})/F^r.
\end{align*}
 (i) {\em Example: Stein varieties over $K$.}
  Assume that $X$  is Stein. We easily see  that in $\sd(C_K)$
    \begin{align}
    \label{derham11}
     F^r(\R\Gamma_{\dr}(X)  \wh{\otimes}^{\R}_{K}\B^+_{\dr}) &  \simeq F^r(\Omega\kr({X})\wh{\otimes}_{K}\B^+_{\dr})
\\\notag & =(\so({X})\wh{\otimes}_{K} F^r\B^+_{\dr}\to\Omega^1({X})\wh{\otimes}_{K} F^{r-1}\B^+_{\dr}\to \cdots)\\\notag
      {\rm DR}(X_C,r)    = (\R\Gamma_{\dr}(X) & \wh{\otimes}^{\R}_{K}  \B^+_{\dr})/F^r  \simeq (\Omega\kr({X})\wh{\otimes}_{K}\B^+_{\dr})/F^r\\  
   = (\so({X}) & \wh{\otimes}_K (\B^+_{\dr}/F^r)\to
      \Omega^1({X})\wh{\otimes}_K(\B^+_{\dr}/F^{r-1})\to\cdots 
        \to  \Omega^{r-1}({X})\wh{\otimes}_K(\B^+_{\dr}/F^{1})).\notag
  \end{align}
   In low degrees we have
\begin{align*}
{\rm DR}(X_C,0) &=0,\quad 
      {\rm DR}(X_C,1)\simeq \so({X})\wh{\otimes}_K C,\\
      {\rm DR}(X_C,2) & \simeq (\so({X})\wh{\otimes}_K(\B^+_{\dr}/F^2)\to
      \Omega^1({X})\wh{\otimes}_K C).
      \end{align*}
 Recall that, because $X$ is Stein,  the de Rham complex is built from Fr\'echet spaces and it has strict differentials.   Arguing just as in \cite[Ex. 3.30]{CDN3} it follows that: 
 \begin{enumerate}
 \item the complexes  $   F^r(\R\Gamma_{\dr}(X)  \wh{\otimes}^{\R}_{K}\B^+_{\dr}) $ and  ${\rm DR}(X_C,r)$ are  built from Fr\'echet spaces;
 \item  their  differentials are  strict;
 \item  and 
 the cohomologies  $   \wt{H}^iF^r(\R\Gamma_{\dr}(X)  \wh{\otimes}^{\R}_{K}\B^+_{\dr}) $ and $ \wt{H}^i{\rm DR}(X_C,r)$ are classical and Fr\'echet. 
\end{enumerate}
 (ii) {\em Example: Affinoids over $K$.}  Assume now that  $X$ is an affinoid.   Then the computation is a bit more complicated because the spaces $\Omega^i(X)$ and $\B^+_{\dr}$ (an $LB$-space and a Fr\'echet space, respectively) do not work together well with tensor products. However, if we use the fact that $\B^+_{\dr}\simeq\prod_{k\geq 0}Ct^k$ in  $\sd(C_K)$, we get the strict quasi-isomorphisms 
 $$
 \Omega^i({X})\wh{\otimes}_{K}\B^+_{\dr}\stackrel{\sim}{\to} \Omega^i({X})\wh{\otimes}^{\R}_{K}\B^+_{\dr},
 $$
 which implies the strict quasi-isomorphisms 
 from (\ref{derham11}). 
 
       Then, arguing just as in \cite[Ex. 3.30]{CDN3}, 
one shows that  the cohomology  $ \wt{H}^i{\rm DR}(X_C,r)$ is classical and that  it is an ${LB}$-space. Also, we easily see that the differentials in the complex 
 $   F^r(\R\Gamma_{\dr}(X)\wh{\otimes}^{\R}_{K}\B^+_{\dr}) $  are  strict; hence 
 the cohomology  $   \wt{H}^iF^r(\R\Gamma_{\dr}(X)  \wh{\otimes}^{\R}_{K}\B^+_{\dr}) $ is  classical.

  (iii) {\em General varieties over $K$.} The following computation can be done in the same way as the computation in Proposition \ref{tluczenie1}. 
\begin{proposition}Let $X$ be a smooth dagger variety over $K$.
\begin{enumerate}
\item If $X$ is quasi-compact then the cohomology of the complex $\R\Gamma_{\dr}(X)\wh{\otimes}_{K}^{\R}{\B}^+_{\dr}$ is classical and we have 
\begin{equation}
\label{new1}
\wt{H}^i(\rg_{\dr}(X)\wh{\otimes}^{\R}_{K}{\B}^+_{\dr})\simeq H^i_{\dr}(X)\wh{\otimes}_{K}{\B}^+_{\dr},\quad i\geq 0.
\end{equation}
\item Take an  increasing admissible covering  $\{U_n\}_{n\in\N}$ of $X$ by quasi-compact dagger varieties $U_n$. Then we have a natural strict quasi-isomorphism  in $\sd\sff(C_{\B^+_{\dr}})$
$$
\rg_{\dr}(X)\wh{\otimes}^{\R}_{K}{\B}^+_{\dr}\stackrel{\sim}{\to} \R\wlim_n(\rg_{\dr}(U_{n})\wh{\otimes}^{\R}_{K}{\B}^+_{\dr}).
$$
The cohomology of $\rg_{\dr}(X)\wh{\otimes}^{\R}_{K}{\B}^+_{\dr}$ is classical and we have, for $i\geq 0$, 
$$
\wt{H}^i(\rg_{\dr}(X)\wh{\otimes}^{\R}_{K}{\B}^+_{\dr})\simeq H^i_{\dr}(X)\wh{\otimes}^{\R}_{K}{\B}^+_{\dr}\simeq \wlim_n(H^i_{\dr}(U_n){\otimes}_{K}{\B}^+_{\dr}).
$$
In particular, it is a Fr\'echet space\footnote{We note that $H^i_{\dr}(U_n)$ is a finite rank vector space over $K$ equipped with the canonical topology.}.
\end{enumerate}
\end{proposition}
\subsubsection{Stein varieties and affinoid   over $C$} If $X$ is a smooth dagger affinoid over $C$ then it is  defined over a finite field extension of $K$ and its de Rham type cohomologies have properties  listed in  Section \ref{sroka11}. 

  In the case of Stein varieties we need to argue a bit more. 
\begin{proposition} \label{ciemno1}Let $X\in {\rm Sm}^{\dagger}_C$ be Stein and $r\geq 0$.
Then
\begin{enumerate}
\item concerning the complex ${\rm DR}(X,r)$, we have: 
\begin{enumerate}
\item 
 the cohomology  $ \wt{H}^i{\rm DR}(X,r)$ is classical and Fr\'echet. 
 \item we have a strictly exact sequence
 $$ 
0\to   \Omega^{i}(X)/\im d \to {H}^{i}{\rm DR}(X,r)\to H^{i}_{\dr}(X/\B^+_{\dr})/t^{r-i-1}\to 0
 $$
\end{enumerate}
\item the cohomology $\wt{H}^iF^r\rg_{\dr}(X/\B^+_{\dr})$ is classical and Fr\'echet.
\end{enumerate}
\end{proposition}
\begin{proof}
Concerning claim (1), cover $X$ with a Stein covering by affinoids $\{U_n\}$, $n\in\N$. Since every affinoid $U_n$ is defined over a finite extension of $K$, we have the strict exact sequences from \cite[Ex. 3.30]{CDN3}
$$
0\to   \Omega^{i}(U_n)/\im d \to {H}^{i}{\rm DR}(U_n,r)\to H^{i}_{\dr}(U_n/\B^+_{\dr})/t^{r-i-1}\to 0
$$
All the terms are classical and Hausdorff. 
We claim that, taking their $\lim_n$, we obtain 
\begin{align}
\label{dzisiaj1}
 & 0\to   \lim_n(\Omega^{i}(U_n)/\im d) \to \lim_n{H}^{i}{\rm DR}(U_n,r)\to \lim_n(H^{i}_{\dr}(U_n/\B^+_{\dr})/t^{r-i-1})\to 0\\
& \R^1\lim_n{H}^{i}{\rm DR}(U_n,r)\simeq \R^1\lim_nH^{i}_{\dr}(U_n/\B^+_{\dr})/t^{r-i-1}=0\notag
\end{align}
Indeed, the sequence is strictly exact   since $ \R^1\lim_n \Omega^{i}(U_n)=0$. For the same reason we have the isomorphism between $\R^1\lim$'s. Since we have Hyodo-Kato isomorphisms 
$H^{i}_{\dr}(U_n/\B^+_{\dr})\simeq H^i_{\hk,\breve{F}}(U_n)\wh{\otimes}_{\breve{F}}\B^+_{\dr}$ and the Hyodo-Kato cohomology $H^i_{\hk,\breve{F}}(U_n)$ is of finite rank, these $\R^1\lim_n$ vanish. From (\ref{dzisiaj1}) we obtain the strictly exact sequence
$$
  0\to   \Omega^{i}(X)/\im d \to \wt{H}^{i}{\rm DR}(X,r)\to H^{i}_{\dr}(X/\B^+_{\dr})/t^{r-i-1}\to 0
$$
  Hence, $\wt{H}^{i}{\rm DR}(X,r) $ is classical (as an extension of two classical objects). It is also an extension of two Fr\'echet spaces; which implies that it is, in particular,  Hausdorff. It is also a quotient of two Fr\'echet spaces by construction, 
which implies that it is a Fr\'echet space itself, as wanted.

    For claim (2), since we have the Hyodo-Kato strict quasi-isomorphism (from Theorem \ref{HK-dagger2})
   $$
   \iota_{\rm HK}: \rg_{\hk,\breve{F}}(X)\wh{\otimes}^{\R}_{\breve{F}}\B^+_{\dr}\stackrel{\sim}{\to}\rg_{\dr}(X/\B^+_{\dr})
   $$
   and the cohomology $$\wt{H}^i(\rg_{\hk,\breve{F}}(X)\wh{\otimes}^{\R}_{\breve{F}}\B^+_{\dr})\simeq H^i_{\hk,\breve{F}}(X)\wh{\otimes}_{\breve{F}}\B^+_{\dr}$$
   is classical, we get that the cohomology $\wt{H}^i_{\dr}(X/\B^+_{\dr})$ is also classical and Fr\'echet. 
   For $i>r$, we have an isomorphism $\wt{H}^i(F^r\rg_{\dr}(X/\B^+_{\dr}))\stackrel{\sim}{\to}\wt{H}^i_{\dr}(X/\B^+_{\dr})$ (take an exhaustive affinoid  covering and use the fact that affinoids are defined over a finite extension of $K$); hence this cohomology is also classical and Fr\'echet. 
   
    For $i\leq r$,  we argue by induction on $r$, the base case of $r=0$ being shown above. For the inductive step ($r-1 \Rightarrow r$), take 
     the distinguished triangle (\ref{detail11}) and consider the induced  long exact sequence
  $$
0 \to  {H}^i(F^{r-1}\rg_{\dr}(X/\B^+_{\dr}))\lomapr{t}  \wt{H}^i(F^{r}\rg_{\dr}(X/\B^+_{\dr})) \lomapr{\vartheta} {H}^i(F^{r}\rg_{\dr}(X)) \lomapr{\partial}{H}^{i+1}(F^{r-1}\rg_{\dr}(X/\B^+_{\dr}))
  $$
The injection on the left follows from the fact that $ {H}^{i-1}(F^{r}\rg_{\dr}(X))=0$; the terms involving $F^{r-1}$ filtration  are classical by the inductive hypothesis.
 
   $\bullet$  If $i<r$, then this yields an isomorphism
  $$
   {H}^i(F^{r-1}\rg_{\dr}(X/\B^+_{\dr}))\xrightarrow[\sim]{t}  \wt{H}^i(F^{r}\rg_{\dr}(X/\B^+_{\dr})),
  $$
  showing that $\wt{H}^i(F^{r}\rg_{\dr}(X/\B^+_{\dr}))$ is classical and Fr\'echet. 
  
    $\bullet$ For $i=r$, we get a short exact sequence
$$
0 \to  {H}^i(F^{r-1}\rg_{\dr}(X/\B^+_{\dr}))\lomapr{t}  \wt{H}^i(F^{r}\rg_{\dr}(X/\B^+_{\dr})) \lomapr{\vartheta} \ker{\partial}\to 0
$$
Hence, $ \wt{H}^i(F^{r}\rg_{\dr}(X/\B^+_{\dr})) $ is classical and a Fr\'echet space by the argument we have used in the case of $\wt{H}^{i}{\rm DR}(X,r) $ in the proof of claim (1). 
\end{proof}
\subsection{Overconvergent geometric syntomic cohomology} We are now ready to define  overconvergent geometric syntomic cohomology and  prove a comparison theorem for smooth dagger affinoids and Stein varieties. 

 Let $X$ be a smooth dagger  variety over  $C$.
 Take $r\geq 0$. 
We define the {\em geometric syntomic cohomology} of $X$ as the following mapping fiber 
\index{rgsyn@\rgsyn}(taken in $\sd(C_{\Q_p})$)
\begin{align}
\label{part1}
\R\Gamma_{\synt}(X,\Q_p(r))& :=[[\R\Gamma_{\hk}(X)\wh{\otimes}_{F^{\nr}}^{\R}\wh{\B}^+_{\st}]^{N=0,\phi=p^r}
\verylomapr{\iota_{\hk}\otimes\iota}\rg_{\crr}(X/\B^+_{\dr})/F^r]\\
  & =[{\rm HK}(X,r)\verylomapr{\iota_{\hk}\otimes\iota}{\rm DR}(X,r)].\notag
\end{align}
This is a generalization  of the geometric syntomic cohomology introduced  in \cite[Sec.\,3.2.2]{CDN3} in the case $X$ comes from a semistable model over $\so_K$. We will define below in Section \ref{here1} overconvergent geometric syntomic cohomology via presentations of dagger structures from rigid-analytic geometric syntomic cohomology and show in Proposition \ref{zamek3} that the two definitions give strictly quasi-isomorphic cohomologies.

   The following proposition generalizes \cite[Prop. 3.36]{CDN3}. 
\begin{remark}
We will often use an equivalent definition of overconvergent geometric syntomic cohomology:
\begin{align*}
\R\Gamma_{\synt}(X,\Q_p(r)) :=[[\R\Gamma_{\hk}(X)\wh{\otimes}_{F^{\nr}}^{\R}{\B}^+_{\st}]^{N=0,\phi=p^r}
\verylomapr{\iota_{\hk}\otimes\iota}\rg_{\crr}(X/\B^+_{\dr})/F^r].
\end{align*}
See Lemma \ref{cicho1} for why the two definitions give the same object (up to a canonical strict quasi-isomorphism). 
\end{remark}
\begin{proposition} 
\label{fd11}Let $X$ be a smooth dagger  affinoid or   a smooth dagger Stein variety over $C$.
Let  $r\geq 0$. There is a natural  map of strictly exact sequences
  $$
\xymatrix@C=.6cm@R=.5cm{
 0\ar[r] & \Omega^{r-1}(X)/\kker d\ar[r]^-{\partial}\ar@{=}[d] & H^r_{\synt}(X,\Q_p(r))\ar[d]^{\beta} \ar[r] & (H^r_{\hk}(X)\wh{\otimes}^{\R}_{F^{\nr}}\wh{\B}^+_{\st})^{N=0,\phi=p^r}
\ar[r]\ar[d]^{\iota_{\hk}\otimes\theta} & 0\\
 0\ar[r] & \Omega^{r-1}(X)/\kker d \ar[r]^-d & \Omega^r(X)^{d=0} \ar[r] & H^r_{\dr}(X)\ar[r] & 0
}
$$
Moreover, $\kker(\iota_{\hk}\otimes\theta)\simeq (H^r_{\hk}(X)\wh{\otimes}^{\R}_{F^{\nr}}\wh{\B}^+_{\st})^{N=0,\phi=p^{r-1}}
$, $H^r_{\synt}(X,\Q_p(r))$ is $LB$ or Fr\'echet, respectively, and the maps $\beta$, $\iota_{\hk}\otimes\theta$ are strict and have closed images.
  \end{proposition}
\begin{proof} 
The diagram in the proposition arises  from  the commutative  diagram:
$$
\xymatrix{
\R\Gamma_{\synt}(X,\Q_p(1))\ar[r]\ar@/_60pt/[dd]^-{\beta}\ar@{.>}[d]^-{\tilde{\beta}} & [\R\Gamma_{\hk}(X)\wh{\otimes}^{\R}_{F^{\nr}}\wh{\B}^+_{\st}]^{\phi=p,N=0}
\ar[d]^{\iota_{\hk}\otimes\iota}\ar[r]^-{\iota_{\hk}\otimes\iota} & \R\Gamma_{\dr}(X/\B^+_{\dr})/F^r\ar@{=}[d]\\
F^r(\R\Gamma_{\dr}(X/\B^+_{\dr})\ar[r]\ar[d]^-{\vartheta}&  \R\Gamma_{\dr}(X/\B^+_{\dr})\ar[d]^{\vartheta}\ar[r] & \R\Gamma_{\dr}(X/\B^+_{\dr})/F^r\ar[d]^{\vartheta}\\
\Omega^{\geq r}({X}) \ar[r]& \Omega\kr(X)\ar[r] & \Omega^{\leq r-1}({X})
}
$$
The map $\tilde{\beta}$ is the map on mapping fibers induced by the commutative right square.  We  set $\beta:=\vartheta\tilde{\beta}$. 
The map $\Omega^{r-1}(X)\to \Omega^r(X)$ induced from the bottom row of the above diagram is easily checked to be equal to $d$.

 Applying cohomology to the above diagram we obtain a commutative diagram
 $$
 \xymatrix@R=.6cm@C=.5cm{
  (H^{r-1}_{\hk}(X)\wh{\otimes}^{\R}_{F^{\nr}}\wh{\B}^+_{\st})^{\phi=p,N=0}\ar[r] \ar[d]^-{\iota_{\hk}\otimes\iota}& 
\Omega^{r-1}({X})/\im d\ar@{=}[d] \ar[r]^-{\partial}  & \wt{H}^r_{\synt}(X,\Q_p(1))\ar[r]\ar[d]^-{\beta}& (H^r_{\hk}(X)\wh{\otimes}^{\R}_{F^{\nr}}\wh{\B}^+_{\st})^{\phi=p,N=0}
\ar[d]^{\iota_{\hk}\otimes\iota} \\
0\to H^{r-1}_{\dr}(X)\ar[r] & \Omega^{r-1}({X})/\im d\ar[r]^-d & \Omega^{r}({X})^{d=0} \ar[r] & H^r_{\dr}(X)
}
$$
We have used here Proposition \ref{ciemno0} and Proposition \ref{ciemno1}. 
 We can now use the proof of Proposition 3.36 in \cite{CDN3} as soon as we know that, for a quasi-compact smooth dagger variety $Y$ over $C$, the slopes of Frobenius on $H^i_{\hk}(Y)$ are $\leq i$. But this is true when $Y=\sy_{K,C}$ for a semistable model over $\so_K$ (by the weight spectral sequence) and it follows for a general  $Y$ by taking \'etale hypercoverings built from semistable basic models, quasi-compact in every degree. 
\end{proof}

\section{Two comparison morphisms} 
In this section  we define two comparison morphisms: from geometric syntomic cohomology of a smooth dagger variety to geometric syntomic cohomology of its completion and between geometric syntomic cohomology of a smooth dagger variety and its pro-\'etale cohomology.
 We also prove that the first morphism is a quasi-isomorphism 
for partially proper varieties (Theorem~\ref{zamek})
and the second morphism is a quasi-isomorphism in a stable range (Theorem~\ref{period15}). 

\subsection{From overconvergent to rigid analytic geometric syntomic cohomology}
We start with  a morphism from geometric syntomic cohomology of a smooth dagger variety to geometric syntomic cohomology of its completion. 

  \subsubsection{Construction of the comparison morphism.} \label{constr1} 
    Let $X$ be a smooth dagger variety over $C$. We will  construct a natural map in $\sd(C_{\Q_p})$
    \begin{equation}
    \label{map-comp}
    \iota: \quad  \rg_{\synt}(X,\Q_p(r))\to \rg_{\synt}(\wh{X},\Q_p(r))
   \end{equation}
    from the syntomic cohomology of $X$ to the syntomic cohomology of its completion $\wh{X}$.
    
     (i) {\em The map $\iota_1$.}  
    First, we note that we have a canonical natural morphism  in $\sd(C_{\Q_p})$
    \begin{align*}
  \iota_1:  \rg_{\synt}(X,\Q_p(r)) & =  \big[[\rg_{\hk}(X)\wh{\otimes}^{\R}_{F^{\nr}}\wh{\B}^+_{\st}]^{N=0,\phi=p^r}\verylomapr{\iota_{\hk}\otimes\iota}\rg_{\dr}(X/\B^+_{\dr})/F^r\big]\\
   & \rightarrow\big[[\rg_{\hk}(\wh{X})\wh{\otimes}^{\R}_{F^{\nr}}\wh{\B}^+_{\st}]^{N=0,\phi=p^r}\verylomapr{\iota_{\hk}\otimes\iota}\rg_{\dr}(\wh{X}/\B^+_{\dr})/F^r\big]\\
   & \xleftarrow{\sim} \big[[\rg_{\hk}(\wh{X})\wh{\otimes}_{F^{\nr}}{\B}^+_{\st}]^{N=0,\phi=p^r}\lomapr{\iota_{\hk}\otimes\iota} \rg_{\dr}(\wh{X}/\B^+_{\dr})/F^r\big].
    \end{align*}
Indeed, for that it suffices to show that the canonical map
$$
[\rg_{\hk}(\wh{X})\wh{\otimes}_{F^{\nr}}{\B}^+_{\st}]^{N=0}\to [\rg_{\hk}(\wh{X})\wh{\otimes}^{\R}_{F^{\nr}}\wh{\B}^+_{\st}]^{N=0}
$$
is a strict quasi-isomorphism. We may assume that $X$ has a semistable weak formal model $\sx$defined over $\so_{K^{\prime}}$. Then the above map is equal to the map
$$
[\rg_{\hk}(\sx^0_1)\wh{\otimes}_{K^{\prime},\iota}{\B}^+_{\st}]^{N=0}\to [\rg_{\hk}(\sx^1_0)\wh{\otimes}^{\R}_{K^{\prime}}\wh{\B}^+_{\st}]^{N=0}.
$$
But this is a special case of the strict quasi-isomorphism in (\ref{Paris-hot}).

  (ii) {\em The map $\iota_2$.}  
  Next, we use the strict quasi-isomorphism  in $\sd(C_{\Q_p})$
  \begin{align*}
   \iota_2: \quad & 
    \big[ [\rg_{\hk}(\wh{X})_{F}\wh{\otimes}_{F^{\nr}}{\B}^+_{\st}]^{N=0,\phi=p^r}\lomapr{\iota_{\hk}\otimes\iota} \rg_{\dr}(\wh{X}/\B^+_{\dr})/F^r\big]\\
    & \to  [[\R\Gamma_{\crr}(\wh{X})]^{\phi=p^r}\lomapr{\can} \R\Gamma_{\crr}(\wh{X})/F^r]= \rg_{\synt}(\wh{X},\Q_p(r))
  \end{align*}
  from the proof of Proposition \ref{synt-locglob}. 
   
     (iii)    Finally, we define the map  in $\sd(C_{\Q_p})$
   $$\iota: \rg_{\synt}({X},\Q_p(r))\to\rg_{\synt}(\wh{X},\Q_p(r))
   $$ as  $\iota:=\iota_2\iota_1$. 
 \subsubsection{A comparison theorem}We are now ready to prove our  comparison theorem: 
  \begin{theorem} 
  \label{zamek}
  Let $X$ be a partially proper smooth dagger variety over $C$.
    The map  in $\sd(C_{\Q_p})$
    $$
    \iota: \quad \rg_{\synt}(X,\Q_p(r))\to \rg_{\synt}(\wh{X},\Q_p(r))
    $$
    is a strict quasi-isomorphism. 
    \end{theorem}
    \begin{proof}We have $\iota=\iota_2\iota_1$ by definition and as we have seen  the map $\iota_2$ is a strict quasi-isomorphism. Hence it remains to show that so is the map $\iota_1$.  For that,  it suffices to show that the following canonical  maps 
    \begin{equation}
    \label{kwaku2}
  \rg_{\hk}(X)\wh{\otimes}^{\R}_{F^{\nr}}\wh{\B}^+_{\st}\to \rg_{\hk}(\wh{X})\wh{\otimes}^{\R}_{F^{\nr}}\wh{\B}^+_{\st},\quad \rg_{\dr}(X/\B^+_{\dr})/F^r\to \rg_{\dr}(\wh{X}/\B^+_{\dr})/F^r
    \end{equation}
    are  strict quasi-isomorphisms.  For the second map this follows from Corollary \ref{roznosci1}. 
    For the first map, by Lemma \ref{passage}, it suffices to show that the canonical map 
$$
 \rg_{\hk,\breve{F}}(X)\wh{\otimes}^{\R}_{\breve{F}}\wh{\B}^+_{\st}\to \rg_{\hk,\breve{F}}(\wh{X})\wh{\otimes}^{\R}_{\breve{F}}\wh{\B}^+_{\st}
$$
is a strict quasi-isomorphism. 
But this holds because, by Corollary \ref{roznosci1},  the canonical  map $ \rg_{\hk,\breve{F}}(X)\to \rg_{\hk,\breve{F}}(\wh{X})$ is a strict quasi-isomorphism. 
           \end{proof}      
            \subsection{From overconvergent syntomic cohomology to pro-\'etale cohomology} We will construct  now  a comparison morphism between geometric syntomic cohomology of a smooth dagger variety and its pro-\'etale cohomology. We will prove that it is a strict quasi-isomorphism in a stable range. 
 \subsubsection{Overconvergent geometric syntomic cohomology via presentations of dagger structures} \label{here1}We start with showing  that the  overconvergent geometric syntomic cohomology defined as  in \cite[Sec.\,6.3]{CN3} using presentations of dagger structures, a priori different from the overconvergent geometric syntomic cohomology  defined as in \cite[Sec.\,5.4]{CN3}, is strictly quasi-isomorphic to it. This was shown in \cite[Prop. 6.17]{CN3} in the arithmetic case, where the key ingredient of the proof is the comparison theorem between arithmetic overconvergent and rigid analytic syntomic cohomology of partially proper dagger spaces. We had to wait for the geometric version of the later  comparison theorem (our Theorem \ref{zamek}) to state the geometric analog of 
   \cite[Prop. 6.17]{CN3}.
   
      (i) {\em Local definition.} Let $X$ be a dagger  affinoid over $C$. Let ${\rm pres}(X)=\{X_h\}$. Recall that we have defined the syntomic cohomology
      $$\rg_{\synt}^{\dagger}(X,\Q_p(r)):=\LL\colim_h \rg_{\synt}(X_{h},\Q_p(r)),\quad  r\in\N. $$
   We have a natural \index{iotahk@\iotahk}map in $\sd(C_{\Q_p})$
    \begin{equation}
    \label{def1}
     \iota^{\dagger}_{\synt}\colon \rg_{\synt}^{\dagger}(X,\Q_p(r)) \to \rg_{\synt}(X,\Q_p(r))
    \end{equation}
defined as the composition
\begin{align}
\label{composition1}
 \rg_{\synt}^{\dagger}(X,\Q_p(r)) & =\LL\colim_h \rg_{\synt}(X_{h},\Q_p(r))\stackrel{\sim}{\to}\LL\colim_h \rg_{\synt}(X^{\circ}_{h},\Q_p(r))\\
  & \stackrel{\sim}{\leftarrow}
 \LL\colim_h \rg_{\synt}(X^{{\circ},\dagger}_{h},\Q_p(r))\to \rg_{\synt}(X,\Q_p(r)).\notag
\end{align}
The third quasi-isomorphism holds by Theorem \ref{zamek} because $X^{\circ}_h$ is partially proper.
       
     (ii) {\em Globalization.}  For a general smooth dagger variety $X$ over $C$, 
using the  natural equivalence of analytic topoi
$$
{\rm Sh}({\rm SmAff}^{\dagger}_{C,\eet})\stackrel{\sim}{\to} {\rm Sh}({\rm Sm}^{\dagger}_{C,\eet}),
$$
we define the sheaf $\sa_{\synt}^{\dagger}(r)$, $r\in\N$,  on $X_{\eet}$ as the sheaf associated to the presheaf defined by:  $U\mapsto \rg_{\synt}^{\dagger}(U,\Q_p(r))$, $U\in {\rm SmAff}^{\dagger}_{C}$, $U\to X$ an \'etale map.  We define\footnote{We will show below (see Proposition \ref{zamek3}) that this definition of $\rg_{\synt}^{\dagger}(X,\Q_p(r))$, for a smooth dagger affinoid $X$,  gives an object naturally strictly quasi-isomorphic to the one defined above.} in $\sd(C_{\Q_p})$
$$
\rg_{\synt}^{\dagger}(X,\Q_p(r)):=\rg_{\eet}(X, \sa_{\synt}^{\dagger}(r)),\quad r\in\N.
$$
Globalizing  the  map $\iota^{\dagger}_{\synt}$  from (\ref{def1})
     we obtain a natural map 
    $$
    \iota^{\dagger}_{\synt}:  \rg_{\synt}^{\dagger}(X,\Q_p(r)) \to \rg_{\synt}(X,\Q_p(r)).
    $$
    
     (iii) {\em A comparison quasi-isomorphism.}

    \begin{proposition}    \label{zamek3}
    The  above map $\iota^{\dagger}_{\synt}$ is a strict quasi-isomorphism.
    \end{proposition}
    \begin{proof}By \'etale descent, we may assume that $X$ is a smooth dagger affinoid. 
Looking at the composition (\ref{composition1}) defining the map $\iota^{\dagger}_{\synt}$  we see that it suffices to show that the natural map
    $$
    \LL\colim_h \rg_{\synt}(X^{{\rm o},\dagger}_h,\Q_p(r))\to \rg_{\synt}(X,\Q_p(r))
    $$
    is a strict quasi-isomorphism. Or, from the definitions of both sides, that
  we have strict  quasi-isomorphisms in, resp., $\sd_{\phi,N}(C_{\wh{\B}^+_{\st}})$ and $\sd\sff(C_{\B^+_{\dr}})$
     \begin{align*}
      \rg_{\hk}(X)\wh{\otimes}_{F^{\nr}}\wh{\B}^+_{\st} & \stackrel{\sim}{\leftarrow}\LL\colim_h (\rg_{\hk}(X^{{\rm o},\dagger}_h)\wh{\otimes}_{F^{\nr}}\wh{\B}^+_{\st}),\\
       \rg_{\dr}(X/\B^+_{\dr}) & \stackrel{\sim}{\leftarrow}\LL\colim_h \rg_{\dr}(X^{{\rm o},\dagger}_h/\B^+_{\dr}).
     \end{align*}
We may assume that $X$ is defined over a finite field extension $L$ of $K$, i.e., there exists $X_L$ such that $X\simeq X_{L,C}$
 Then    the above  maps  factor as
    \begin{align*}
       &  \LL\colim_h (\rg_{\hk}(X^{{\rm o},\dagger}_h)\wh{\otimes}_{F^{\nr}}\wh{\B}^+_{\st})\stackrel{\sim}{\to}(\LL\colim_h (\rg_{\hk}(X_{h}))\wh{\otimes}_{F^{\nr}}\wh{\B}^+_{\st})\stackrel{\sim}{\to} 
       \rg^{\dagger}_{\hk}(X)\wh{\otimes}_{F^{\nr}}\wh{\B}^+_{\st}\stackrel{\sim}{\to}\rg_{\hk}(X)\wh{\otimes}_{F^{\nr}}\wh{\B}^+_{\st},\\
 & \LL\colim_h (\rg_{\dr}(X^{{\rm o},\dagger}_{L,h})\wh{\otimes}^{\R}_L\B^+_{\dr})\stackrel{\sim}{\to}(\LL\colim_h \rg_{\dr}(X_{L,h}))\wh{\otimes}^{\R}_L\B^+_{\dr}\stackrel{\sim}{\to} \rg^{\dagger}_{\dr}(X/\B^+_{\dr})\stackrel{\sim}{\to}\rg_{\dr}(X/\B^+_{\dr}).
    \end{align*}
       In the Hyodo-Kato case, the first map is a strict quasi-isomorphism by definition of the dagger tensor product. In the de Rham case, the first map is a strict filtered quasi-isomorphism by the computation (\ref{defnie}). 

\end{proof}
\subsubsection{The geometric overconvergent period map and a comparison result}   \label{period-dagger-section}       We are now ready to define and study the overconvergent period map.    
     Let $X\in {\rm Sm}^{\dagger}_C$, $r\geq 0$.  Define the period 
\index{alphar@\alphar}map in $\sd(C_{\Q_p})$
   \begin{equation}
   \label{period-dagger}
   \alpha_r: \rg_{\synt}(X,\Q_p(r))\to \rg_{\proeet}(X,\Q_p(r))
   \end{equation}
   as the composition
   $$
    \rg_{\synt}(X,\Q_p(r))\stackrel{\sim}{\leftarrow}  \rg_{\synt}^{\dagger}(X,\Q_p(r))\lomapr{\alpha_r^{\dagger}} \rg_{\proeet}(X,\Q_p(r)),
   $$
   where the first map is the map $\iota^{\dagger}_{\synt}$ from Proposition   \ref{zamek3} and the second map is defined by globalizing the following map defined for  a dagger affinoid $X$ with  the presentation $\{X_h\}$:
   $$
    \rg_{\synt}^{\dagger}(X,\Q_p(r))=\LL\colim_h\rg_{\synt}(X_h,\Q_p(r))\lomapr{\alpha_r}\LL\colim_h\rg_{\proeet}(X_h,\Q_p(r))\simeq 
     \rg_{\proeet}(X,\Q_p(r)).
   $$
   Here $\alpha_r$ is the rigid analytic period map (see Proposition \ref{main00}).
   
   We have the following compatibility with the rigid analytic period map: 
   \begin{proposition}\label{lwiatko11}
{\rm ({Dagger-rigid analytic compatibility})} Let  $X\in {\rm Sm}^{\dagger}_C$ and $r\geq 0$. 
   \begin{enumerate}
   \item The following diagram
   $$
\xymatrix@R=6mm{
\rg_{\synt}(X,\Q_p(r))\ar[r]^{\alpha_r}\ar[d]^{\iota} & \rg_{\proeet}(X,\Q_p(r))\ar[d]^{\iota_{\proeet}}\\
\rg_{\synt}(\wh{X},\Q_p(r))\ar[r]^{\wh{\alpha}_r} & \rg_{\proeet}(\wh{X},\Q_p(r))
}
$$
commutes. 
\item If $X$ is partially proper then the maps $\iota$ and $\iota_{\proeet}$ in the above diagram are strict quasi-isomorphisms. 
\end{enumerate}
\end{proposition}
Here, the period map ${\alpha}_r$ is the one defined above.  
\index{alphar@\alphar}We put hat above its  rigid analytic analog  to distinguish it from the dagger period map. 
   \begin{proof}
    For the first claim, it suffices to show that this diagram naturally commutes \'etale locally. So we may assume that $X$ is a smooth dagger affinoid. Then checking commutativity is straightforward from the definitions. 
    
       For the second claim, note that the map $\iota$ is a strict quasi-isomorphism by Theorem \ref{zamek} and  the map $\iota_{\proeet}$ is a strict quasi-isomorphism by Proposition \ref{main00}, point 3a.   
\end{proof}

  The following comparison result follows almost immediately from its rigid analytic analog (see Proposition \ref{main00}, point 2c):
\begin{theorem}
   \label{period15}For $X\in {\rm Sm}^{\dagger}_C$ and $r\geq 0$, 
    the period map  in $\sd(C_{\Q_p})$
    \begin{equation}
    \label{period15.1}
    \alpha_r: \R\Gamma_{\synt}(X,\Q_p(r))\to \R\Gamma_{\proeet}(X,\Q_p(r))
    \end{equation}
    is a  strict quasi-isomorphism after truncation $\tau_{\leq r}$.
\end{theorem}
\begin{proof}We may localize and assume that $X$ is a dagger affinoid. 
Proposition \ref{main00} yields immediately the strict quasi-isomorphism after truncation $\tau_{\leq r-1}$ (since $L^i\colim$ vanishes for $i >1$). It remains to show that the map $\alpha_r$ induces an isomorphism on cohomology  in degree $r$. For that, consider the following commutative diagram
$$
\xymatrix{
\wt{H}^r_{\synt}(X,\Q_p(r))\ar[r]^{\alpha_r} \ar[d]^{\wr}_t& \wt{H}^r_{\proeet}(X,\Q_p(r))\ar[d]^{\wr}_{\zeta}\\
\wt{H}^r_{\synt}(X,\Q_p(r+1))\ar[r]^{\alpha_{r+1}}_{\sim} & \wt{H}^r_{\proeet}(X,\Q_p(r+1)).
}
$$
The right vertical arrow is a multiplication by $p$-adic root of unity. The bottom arrow is an isomorphism by the above argument. The left vertical arrow is an isomorphism by the diagram in Proposition \ref{fd11} and a chase of the diagram in \cite[Rem. 4.5]{CDN3} (we note here that we do not need comparison with pro-\'etale cohomology for this chase).
It follows that the top horizontal map is an isomorphism, as wanted. 
\end{proof}
The above theorem it implies the following result (which will be the starting point of our study of  $C_{\st}$-conjecture for smooth analytic varieties in \cite{CN5}):
\begin{corollary}
For $X\in {\rm Sm}^{\dagger}_C$,  $r\geq 0$, and $i\leq r$, we have the long exact sequence
$$
  \cdots\to \wt{H}^{i-1}(\rg_{\dr}(X/\B^+_{\dr})/F^r) \to \wt{H}^i_{\proeet}(X,\Q_p(r))\to (H^i_{\hk}(X)\wh{\otimes}_{F^{\nr}}\B^+_{\st})^{N=0,\phi=p^r}\lomapr{\iota_{\hk}\otimes\iota} \wt{H}^i(\rg_{\dr}(X/\B^+_{\dr})/F^r)
$$
\end{corollary}
Here we set 
$$
H^i_{\hk}(X)\wh{\otimes}_{F^{\nr}}\B^+_{\st}:=\lim_n(H^i_{\hk}(U_n)\wh{\otimes}_{F^{\nr}}\B^+_{\st}),
$$
for an increasing covering $\{U_n\}_n$ of $X$ by quasi-compact open  (note that the groups $H^i_{\hk}(U_n)$ are of finite rank). 
\begin{proof}
Use Theorem \ref{period15} and the obvious fact that the canonical map $[H^i_{\hk}(U_n)\wh{\otimes}_{F^{\nr}}\B^+_{\st}]^{N=0}\to [H^i_{\hk}(U_n)\wh{\otimes}_{F^{\nr}}\widehat{\B}^+_{\st}]^{N=0}$ is an isomorphism.
\end{proof}

\section{Geometrization of period morphisms}\label{Geometrization} 
The purpose of this section is to geometrize syntomic cohomology (and the related Hyodo-Kato and de Rham cohomologies), pro-\'etale cohomology, and the associated period morphisms both in the rigid analytic and the overconvergent set-ups. By "geometrization" we mean putting  a topological VS-structure. The key computation is the one showing that the rigid analytic version of Fontaine-Messing period morphism is a shadow of a VS-morphism. Both sides of the period morphism, crystalline syntomic cohomology and pro-\'etale cohomology,  have natural VS structures. However it it not immediately clear that the period map navigates well between these two VS-structures. To show that, in fact, it does so we use the presentation of the period map via $(\phi,\Gamma)$-modules introduced  in \cite{CN1}, \cite{SG}. 
\subsection{Geometrization}
In this section we explain how to geometrize the cohomologies and the period morphism (th.\,\ref{lifted-period}).
In the next sections we prove Theorem~\ref{lifted-period}, first in the lifted case, then in the general case.

\subsubsection{Vector Spaces}
A VS, resp.~a VS$^+$,  is a functor from perfectoid affinoids  $(\Lambda,\Lambda^+)$ over $(C,\so_C)$
(denoted by $\Lambda=(\Lambda,\Lambda^+)$ in what follows) to 
 $\Q_p$-modules, resp.~$\Z_p$-modules. If ${\mathbb W}$ is a VS$^+$, then $\Q_p\otimes_{\Z_p}{\mathbb W}$
is a VS. VS's form an abelian category. 
 Trivial examples of VS's are:

$\bullet$ 
finite dimensional $\Q_p$-vector spaces $V$,
with associated functor $\Lambda\mapsto V$ for all $\Lambda$,

$\bullet$ ${\mathbb V}^d$, for $d\in\N$, with ${\mathbb V}^d(\Lambda)=\Lambda^d$, for
all $\Lambda$.

\noindent
More interesting examples are provided by Fontaine's rings~\cite{bures,CB}:

$\bullet$ $\Bcris^+$, $\Bst^+$, $\Bdr^+$, $\Bcris$, $\Bst$, $\Bdr$ are naturally VS's (and even Rings).

$\bullet$ If $m\geq 1$, then ${\mathbb B}_m:=\Bdr^+/t^m\Bdr^+$ is a VS (and also a Ring).

$\bullet$ Let $h\geq 1$ and $d\in\Z$. Then ${\mathbb U}_{h,d}=(\Bcris^+)^{\varphi^h=p^d}$ if $d\geq 0$,
and ${\mathbb U}_{h,d}={\mathbb B}_d /\Q_{p^h}$ if $d<0$, are VS's.

\noindent Exemples of VS$^+$'s include $\Ainf$, $\Acris$, or $\bbA^{[u,v]}$ if $0<u\leq v$;
the last example being the functor sending $\Lambda$ to the $p$-adic completion of $\Ainf(\Lambda)[\frac{p}{[\alpha]},\frac{[\beta]}{p}]$, where $\alpha,\beta\in \so_C^\flat$ with $v(\alpha)=\frac{1}{v}$
and $v(\beta)=\frac{1}{u}$.
 By \cite[Prop. 3.2]{SG}, 
 we have  $\Acris(\Lambda)\subset \bbA^{[u,v]}({\Lambda})$,  for $u\geq \frac{1}{p-1}$.
If $v\geq 1$, we have $\bbA^{[u,v]}(\Lambda)\subset\Bdr^+(\Lambda)$ and this inclusion
induces a filtration on $\bbA^{[u,v]}(\Lambda)$.

$\bullet$    The semistable period rings  can be also lifted to VS's. We set 
 \begin{align*}
&  \wh{{\mathbb B}}^+_{\st}:=\wh{{\mathbb B}}^+_{p,\st}
:=(\Acris<t_p[\wt{p}]^{-1}-1>)^{\bwedge}[\tfrac{1}{p}],
\quad {\mathbb B}^+_{\st}:={\mathbb B}^+_{p,\st}:={\mathbb B}^+_{\crr}[\log([\wt{p}])],\\
& \kappa: {\mathbb B}^+_{\st}\to\wh{{\mathbb B}}^+_{\st}, \,\log([\wt{p}])\mapsto -\log(t_p[\wt{p}]^{-1}),\quad  \iota: {\mathbb B}^+_{\st}\to{\mathbb B}^+_{\dr},\, \log([\wt{p}])\mapsto -\log(p[\wt{p}]^{-1}),\\
&  \iota: \wh{{\mathbb B}}^+_{\st}\to{\mathbb B}^+_{\dr},\, t_p[\wt{p}]^{-1}\mapsto p[\wt{p}]^{-1}.
\end{align*}

If ${\mathbb W}$ is one of the above Rings, we denote by ${\bf W}$ the ring ${\mathbb W}(C)$: for 
example $\A^{[u,v]}=\bbA^{[u,v]}(C)$ (for the other Rings $\Ainf$, $\Acris$, $\Bcris^+$, $\Bdr^+$, etc., 
one gets back the rings already defined).

\begin{remark}
The above definition gives presheaves on ${\rm Perf}_C$. Passing to the associated sheaves
gives a natural viewpoint on VS's and VS$^+$'s; this was put to use by Le Bras in his thesis~\cite{lebras}.
\end{remark}

\subsubsection{Pro-\'etale cohomology}
 Let $X$ be a smooth rigid analytic variety over $C$. 
If $\Lambda$ is a perfectoid $C$-Banach algebra, let $X_\Lambda$ be the scalar extension $X\times_C\Lambda$.
The  functor 
$\Lambda\mapsto H^i_{\proeet}(X_{ \Lambda},\Q_p)$ defines a VS.
That is, there exists a VS
${\mathbb H}^i_{{\proeet}}(X,\Q_p)$ such that 
${\mathbb H}^i_{{\proeet}}(X,\Q_p)(\Lambda)=H^i_{\proeet}(X_{\Lambda},\Q_p)$,
for all  perfectoid $C$-Banach algebras.
In particular, $H^i_{\proeet}(X,\Q_p)$ is the space of  $C$-points
of  ${\mathbb H}^i_{{\proeet}}(X,\Q_p)$;  we have  put in this way a geometric structure on  
 $H^i_{\proeet}(X,\Q_p)$.
 
  We will use a bit more general\footnote{That is, presheaves on ${\rm Perf}_C$ with values in different categories than that of $\Q_p$-modules.} VS's:
  \begin{enumerate}
  \item  the cohomology complex: the functor $${\mathbb R}_{\proeet}(X_{\Lambda},\Q_p): \quad \Lambda\mapsto \rg_{\proeet}(X_{\Lambda},\Q_p)$$ defines a VS with values in $\sd(C_{\Q_p})$;
  \item  its cohomology groups $\wt{{\mathbb H}}_{\proeet}(X_{\Lambda},\Q_p)$ form a VS with values in $LH(C_{\Q_p})$;
  \item  its algebraic cohomology groups ${{\mathbb H}}_{\proeet}(X_{\Lambda},\Q_p)$ form the VS described above. We have a natural map $\wt{{\mathbb H}}_{\proeet}(X_{\Lambda},\Q_p)\to {{\mathbb H}}_{\proeet}(X_{\Lambda},\Q_p)$.
  \end{enumerate}
  \subsubsection{Crystalline syntomic cohomology}
To geometrize (filtered) absolute crystalline cohomology, we define the functor 
\begin{equation}
\label{bla1}
\mathbb{F}^r{\mathbb R}_{\crr}(X):\quad \Lambda\mapsto F^r\rg_{\crr}(X)
\wh{\otimes}^{R}_{\B^+_{\crr}}\Bcris^+(\Lambda),\quad r\geq 0,
\end{equation}
  that lifts the absolute crystalline cohomology $\rg_{\crr}(X)$ from Section \ref{absolute-crr}.  The tensor product used in (\ref{bla1}) needs to be defined. We do it in the following way. We set 
$$F^r\rg_{\crr}(X)\wh{\otimes}^{R}_{\B^+_{\crr}}\Bcris^+(\Lambda):=\rg_{\eet}(X,F^r\sa_{\crr,\Lambda}),
$$
 where $F^r\sa_{\crr,\Lambda}$ is the $\eta$-\'etale sheafification\footnote{We do not discuss local-global compatibilities. As far as we can tell this does not cause problems.} 
on $\sm^{\rm ss}_C$ of the presheaf 
$\sx\mapsto (F^r\rg_{\crr}(\sx/\A_{\crr})\wh{\otimes}^{\LL}_{\A_{\crr}}\Acris(\Lambda))_{\Q_p}$. 
 
 We proceed similarly for rigid $\B^+_{\dr}$-cohomology 
(from Section \ref{defcr}): we define the functor 
$${\mathbb F}^r{\mathbb R}_{\dr}(X/\Bdr^+):\quad 
 \Lambda\mapsto F^r\rg_{\dr}(X/\B^+_{\dr})\wh{\otimes}^{\R}_{\B^+_{\dr}}\Bdr^+(\Lambda),\quad  r\geq 0,
 $$
  that lifts the filtered $\B^+_{\dr}$-cohomology $F^r\rg_{\dr}(X/\B^+_{\dr})$. Here 
 $$F^r\rg_{\dr}(X/\B^+_{\dr})\wh{\otimes}^{\R}_{\B^+_{\dr}}\Bdr^+(\Lambda)
:=\rg_{\eet}(X,F^r\sa^{\bwedge}_{{\crr},\Lambda}),
 $$ where $F^r\sa^{\bwedge}_{{\crr},\Lambda}$ is the $\eta$-\'etale sheafification on $\sm^{\rm ss}_C$ of the presheaf 
 $$\sx\mapsto \R\wlim_{i\geq r}((\rg_{\crr}(\sx,\sj^{[r]})/F^i)
\wh{\otimes}^{\LL}_{(\A_{\crr}/F^i)}(\Acris(\Lambda)/F^i))_{\Q_p}.
 $$

   Finally, we lift crystalline syntomic cohomology by setting 
   $$
  \mathbb{R}_{\synt}(X, \Q_p(r)):\quad  \Lambda\mapsto 
[{\mathbb F}^r{\mathbb R}_{\crr}(X)(\Lambda)_{\Q_p}\lomapr{\phi-p^r}{\mathbb R}_{\dr}(X/\Bdr^+)(\Lambda)/F^r].
   $$
   \subsubsection{Rigid analytic varieties, period morphism} 
We move now to the geometrization of rigid analytic period morphisms. 
We will prove the following theorem.
   \begin{theorem}\label{lifted-period}
   For $X\in {\rm Sm}_C$ and $r\geq 0$, the functorial period map in $\sd(C_{\Q_p})$
   $$
   \alpha_r:\quad \rg_{\synt}(X,\Q_p(r))\to \rg_{\proeet}(X,\Q_p(r))
   $$
   lifts to a functorial map of VS's (with values  in $\sd(C_{\Q_p})$):
   $$
     {\bbalpha}_r:\quad {\mathbb R}_{\synt}(X,\Q_p(r))\to {\mathbb R}_{\proeet}(X,\Q_p(r)),
   $$
   which is a strict quasi-isomorphism after truncation $\tau_{\leq r}$.
   \end{theorem}
The next two sections are devoted to the proof of this theorem.

\subsection{Local period morphism, lifted case}  \label{local-models}  
We start by  defining $ {\bbalpha}_r(\Lambda)$ locally,  in the simplest of cases. 

\subsubsection{Coordinates} 
Consider a frame 
\begin{equation}
\label{frame1}
R^+_{\Box}:=\so_C\{X,\tfrac{1}{X_0\ldots X_a},\tfrac{\varpi}{X_{a+1}\ldots X_{d}}\},\quad
R_{\Box}=R^+_{\Box}[\tfrac{1}{p}],
\end{equation}
where  $X=(X_0,\ldots, X_d)$ and $\varpi\in\so_C\setminus \so_C^*$, and a formal scheme $\sx=\Spf R^+$, for  an algebra $R^+$, which is the $p$-adic completion of an \'etale algebra  over $R^+_{\Box}$. 
We equip $\Spf (R^+_{\Box})$ and $\Spf (R^+)$ with the logarithmic structure\footnote{Note that we do not allow horizontal divisors at $\infty$.} induced by the special fiber.

 For $m\geq 0$, define 
   $$
R^{m,+}_{\Box}:=\so_C\{X^{{1}/{p^m}},\tfrac{1}{(X_0\ldots X_a)^{{1}/{p^m}}},\tfrac{\varpi^{{1}/{p^m}}}{(X_{a+1}\ldots X_{d})^{{1}/{p^m}}}\},\quad R^m_{\Box}=R^{m,+}_{\Box}[\tfrac{1}{p}],
$$
   and set $R^{\infty,+}_{\Box}$ equal to the $p$-adic completion of $\colim_m R^{m,+}_{\Box}$. 
Let $$R^{m,+}:=R^{m,+}_{\Box}\wh{\otimes}_{R^{+}_{\Box}}R^+, 
\quad R^{\infty,+}:= R^{\infty,+}_{\Box}\wh{\otimes}_{R^{+}_{\Box}}R^+,
\quad R^m=R^{m,+}[\tfrac{1}{p}],\quad R:=R^0,\quad R^\infty=R^{\infty,+}[\tfrac{1}{p}],$$ 
so that $R^\infty$ is a perfectoid Banach algebra.
Define $\Gamma_R:=\Gal(R^{\infty}/R)$. We have $\Gamma_R\simeq \Z_p^d$.

Choose $\varpi^\flat\in\so_C^\flat$ with $\theta([\varpi^\flat])=\varpi$.  We define 
  $$
  R^+_{\inf,\Box}:=\A_{\inf}\{X,\tfrac{1}{X_0\ldots X_a},\tfrac{[{\varpi}^{\flat}]}{X_{a+1}\ldots X_{d}}\}
  $$
 and  lift the map $R^+_{\Box}\to R^+$ to an \'etale map $R^+_{\inf,\Box}\to R^+_{\inf}$.   
Set 
$$R^+_{\crr}:=R^+_{\inf}\wh{\otimes}_{\A_{\inf}}\A_{\crr},  \quad
R^{[u,v]}:=R^+_{\inf}\wh{\otimes}_{\A_{\inf}}\A^{[u,v]}.$$
Endow everything with the log-structure coming from $\A_{\crr}$ and $\Spf(R^+)$.
 This gives us  the commutative diagram (with cartesian squares)
  $$
  \xymatrix@R=.4cm{
  \Spf(R^+) \ar@{^(->}[r] \ar[d] & \Spf (R^+_{\crr})\ar[d]\\
    \Spf(R^+_{\Box}) \ar@{^(->}[r] \ar[d] & \Spf (R^+_{\crr,\Box})\ar[d]\\
    \Spf (\so_C) \ar@{^(->}[r]  & \Spf (\A_{\crr})
  }
  $$
  We equip $ \Spf (R^+_{\crr})$ with the (unique) lift $\phi$ of the canonical Frobenius on $ \Spf (R^+_{\crr,\Box})$ 
(induced by $\varphi$ on $\A_{\crr}$ and by $X_i\mapsto X_i^p$, $0\leq i\leq d$).

 We define the filtrations $F^{\jcdot}R^+_{\crr}$ on $R^+_{\crr}$ 
and $F^{\jcdot}R^{[u,v]}$ on $R^{[u,v]}$ by inducing them from $\A_{\crr}$
and $\A^{[u,v]}$. We have the corresponding filtered de Rham complex
 $$
 F^r\Omega^{\jcdot}_{R^+_{\crr}/\A_{\crr}}:=F^rR^+_{\crr}\to F^{r-1}R^+_{\crr}\wh{\otimes}_{R^+_{\crr}}\Omega^1_{R^+_{\crr}/\A_{\crr}}\to F^{r-2}R^+_{\crr}\wh{\otimes}_{R^+_{\crr}}\Omega^2_{R^+_{\crr}/\A_{\crr}}\cdots
 $$
 The crystalline syntomic cohomology $\rg_{\synt}(\sx,\Q_p(r))$ is computed by the complex
 $
 {\rm Syn}(R^+,r)_{\Q_p},$
 where
 $$ {\rm Syn}(R^+,r):=[F^r\Omega^{\jcdot}_{R^+_{\crr}/\A_{\crr}}\lomapr{p^r-\phi}\Omega^{\jcdot}_{R^+_{\crr}/\A_{\crr}}].
 $$
 This follows from the fact that the (filtered) absolute crystalline cohomology $F^i\rg_{\crr}(\sx)\simeq F^i\rg_{\crr}(\sx/\A_{\crr})$.
   
\subsubsection{Period rings}  
    Let $\overline{R}^+$ be the maximal extension of $R^+$ that is \'etale in characteristic $0$, i.e., $\overline{R}^+$ is  the integral closure  of $R^+$ in a maximal ind-\'etale extension $\overline{R}$ of $R[\frac{1}{p}]$ inside a fixed algebraic closure of ${\rm Frac} R$. We have  
$\overline{R}=\overline{R}^+[\tfrac{1}{p}]$.
 Set $G_R:=\Gal(\overline{R}/R)$.
  For $0\leq i\leq d$, choose $X_i^{\flat}=(X_i,X_i^{1/p},\cdots)$ in $\overline{R}^{\flat}$ and define an embedding of $R^+_{\inf,\Box}$ 
in $\Ainf({\overline{R}})$ by sending $X_i\mapsto [X_i^{\flat}]$.
 This extends, for $0<u\leq v$ and $v\geq 1$, to embeddings\footnote{In what follows, all these objects will 
depend on a variable perfectoid algebra $\Lambda$; to distinguish what depends on $\Lambda$ from what does not, we often
allow ourself to write ${\mathbf W}_{R^{\rm deco}}$ instead of ${\mathbb W}(R^{\rm deco})$ to indicate
that $R^{\rm deco}$ does not depend on $\Lambda$.}:
    \begin{equation}
    \label{defremind}
    R^+_{\inf}\hookrightarrow \Ainf({\overline{R}}),\quad R^+_{\crr}\hookrightarrow 
\Acris(\overline{R}),\quad \varepsilon: R^{[u,v]}\hookrightarrow\A^{[u,v]}_{R^\infty}
\subset\A_{\overline{R}}^{[u,v]}.
    \end{equation}
    
 \subsubsection{Local period morphism $\alpha^{R^+}_{r,n}$}
\ 
\vskip.1cm
  ($\bullet$) {\em Over $C$.} Consider the following commutative diagram:
\begin{equation}
\label{lifted}
\xymatrix@R=.5cm{
 & \Spf(\E^{\rm PD}_{\overline{R}})\ar[dr]\\
 \Spf(\overline{R}^+) \ar@{^(->}[ru]\ar[d]\ar@{^(->}[rr]& & \Spf(\Acris(\overline{R})\wh{\otimes}_{\A_{\crr}} R^+_{\crr})\ar[d]\\
 \Spf(R^+) \ar[d]\ar@{^(->}[rr]&& \Spf(R^+_{\crr})\ar[d]^{\pi}\\
 \Spf(\so_C)\ar@{^(->}[rr]&&\Spf(\A_{\crr})
}
\end{equation}
Here $\E^{\rm PD}_{\overline{R}}$ is the ${\rm PD}$-envelope  of the closed embedding $\Spf(\overline{R}^+) \hookrightarrow \Spf(\Acris(\overline{R})\wh{\otimes}_{\A_{\crr}} R^+_{\crr})$. 
\begin{remark}
 (a) We take partial divided powers of level $s$, i.e., $x^{[k]}=\frac{x^k}{\lfloor k/p^s\rfloor\,!}$, 
where $s=0$ for $p\neq 2$ and $s=1$ for $p=2$. 

 (b) We induce the filtration on $\E^{\rm PD}_{\overline{R}}$ from the filtration on $R^+_{\crr}$ and $\Acris(\overline{R})$. See \cite[Sec.\,2.6.1]{CN1} for details. 
\end{remark}

   Set $\Omega^i_{\E^{\rm PD}_{\overline{R}}}:=\E^{\rm PD}_{\overline{R}}\wh{\otimes}_{R^+_{\crr}}\Omega^i_{R^+_{\crr}/\A_{\crr}}$. For $r\in \N$, we filter the de Rham complex $\Omega_{\E^{\rm PD}_{\overline{R}}}^{\jcdot}$ by subcomplexes
$$
F^r\Omega_{\E^{\rm PD}_{\overline{R}}}\kr:=F^r\E^{\rm PD}_{\overline{R}}\to F^{r-1}\E^{\rm PD}_{\overline{R}}\wh{\otimes}_{R^+_{\crr}}\Omega^1_{R^+_{\crr}/\A_{\crr}}\to 
F^{r-2}\E^{\rm PD}_{\overline{R}}\wh{\otimes}_{R^+_{\crr}}\Omega^2_{R^+_{\crr}/\A_{\crr}}\to \cdots
$$
For a continuous $G_{R}$-module $M$, let  $C(G_{R},M)$ denote the complex of continuous cochains of $G_{R}$ with values in $M$. 
The Fontaine-Messing period map\footnote{We note that ${\rm Spa}(R)$ is a $K(\pi,1)$-space hence $C(G_{R},\Z/p^n(r)^{\prime})\simeq \R\Gamma_{\proeet}({\rm Spa}(R),\Z/p^n(r)^{\prime})$.}  
\begin{equation}
\label{FM11}
 {\alpha}^{R^+}_{r,n}: \quad {\rm Syn}(R^+,r)_n \to  C(G_{R},\Z/p^n(r)^{\prime}),
 \end{equation}
  where $\Z/p^n(r)^{\prime}:=\frac{1}{p^{a(r)}}\Z/p^n(r)$, for $r=(p-1)a(r)+b(r), 0\leq b(r)\leq p-1$, 
 is defined 
as the composition 
\begin{equation}
\label{kaczka1}
\xymatrix@R=.4cm{
 {\rm Syn}(R^+,r)_n=
[F^r\Omega_{R^{+}_{\crr,n}/\A_{\crr,n}}\kr\verylomapr{\phi-p^r}\Omega_{R^{+}_{\crr,n}/\A_{\crr,n}}^{\jcdot}]
\ar@<13mm>[d]^{\sim, \tau_{\leq r}}_{\omega}\\
{\phantom{XXXX}C(G_{R},[F^r\Omega_{\E^{\rm PD}_{\overline{R},n}}\kr\verylomapr{\phi-p^r}\Omega_{\E^{\rm PD}_{\overline{R},n}}^{\jcdot}])}\\
{\phantom{XXXX}C(G_{R},[F^r \Acris(\overline{R})_n\verylomapr{\phi-p^r} \Acris(\overline{R})_n])\ar@<-13mm>[u]_-{\wr}}\\
{\phantom{XXXXXX}C(G_{R},\Z/p^n(r)^{\prime})\ar@<-13mm>[u]_-{\wr}}\\
}
\end{equation}
It is a $p^{cr}$-quasi-isomorphism\footnote{We call a morphism $f:A\to B$ in a derived category a $N$-quasi-isomorphism if the induced morphism on cohomology has kernel and cokernel annihilated by $N$.}, for a universal constant $c$, after truncation $\tau_{\leq r}$. The second quasi-isomorphism above follows from the filtered Poincar\'e Lemma, i.e., from the $p$-quasi-isomorphism
\begin{equation}
\label{pepitka1}
F^r\Acris(\overline{R})_n\stackrel{\sim}{\to}F^r\Omega_{\E^{\rm PD}_{\overline{R},n}}^{\jcdot},\quad r\geq 0,
\end{equation}
proved in \cite[Prop. 7.3]{SG}, \cite[Lemma 2.37]{CN1}. The third quasi-isomorphism follows 
from the fundamental $p^r$-exact sequence
$$
0\to \Z/p^n(r)^{\prime}\to F^r\Acris(\overline{R})_n\lomapr{\phi-p^r}\Acris(\overline{R})_n\to 0. 
$$
The first truncated quasi-isomorphism is a theorem of Tsuji \cite{Ts}.

   ($\bullet$) {\em Over a perfectoid $C$-algebra.} 
 Let $\Lambda=(\Lambda,\Lambda^+)$ be a perfectoid affinoid over $(C,\so_C)$.
 We refer the reader for a study of the basic properties of $\Ainf(\Lambda)$ to \cite{bures} or \cite[Sec.\,3]{BMS1}. 
  The following lemma is proved by the same argument as \cite[Lemma\,5.3]{SG}:
  \begin{lemma}\label{filtration}
  Let $0<u\leq v$ and $\frac{v}{p} < 1 < v$. Multiplication by $t^r$ induces $p^{3r}$-isomorphisms\footnote{A morphism of abelian groups $f:S\to T$ is called an {\em $N$-isomorphism} if its kernel and cokernel are annihilated by~$N$.}
  $$
  {\mathbb A}^{[u,v]}({\Lambda})\stackrel{\sim}{\to} F^r{\mathbb A}^{[u,v]}({\Lambda}),
\quad {\mathbb A}^{[u,v/p]}({\Lambda})\stackrel{\sim}{\to}{\mathbb A}^{[u,v/p]}({\Lambda}).
  $$
  \end{lemma}

We set
$$(R_{\Lambda},R^+_{\Lambda}):=(R,R^+)\wh{\otimes}_{(C,\so_C)}(\Lambda,\Lambda^+)$$
(by \cite[Prop. 6.18]{Sch0} this is a perfectoid algebra).
Let $\overline{R}^+_{\Lambda}$ be 
the completion of the maximal extension of $R_\Lambda$ \'etale in characteristic~$0$ and
$$\overline{R}_{\Lambda}:=\overline{R}^+_{\Lambda}[\tfrac{1}{p}],\quad
G_{R_\Lambda}={\rm Aut}(\overline{R}_{\Lambda}/R_\Lambda)$$
  For $0<u\leq v$ and $v\geq 1$, set 
$$
R^{[u,v]}_{\Lambda}  :=R^+_{\inf}\wh{\otimes}_{\A_{\inf}}{\mathbb A}^{[u,v]}({\Lambda}) $$
  equipped with the filtration induced from the one on ${\mathbb A}^{[u,v]}({\Lambda})$.

The Fontaine-Messing  morphism $ {\alpha}^{R^+}_{r,n}$ from \eqref{FM11}  
lifts to $\Lambda$. 
To show this we will
   use the commutative diagram:
\begin{equation}
\label{lambda-version}
\xymatrix@R=.5cm{
 & \Spf(\E^{\rm PD}_{\overline{R}_{\Lambda}})\ar[dr]\\
 \Spf(\overline{R}^+_{\Lambda}) \ar@{^(->}[ru]\ar[d]\ar@{^(->}[rr]& & \Spf(\Acris(\overline{R}_{\Lambda})\wh{\otimes}_{\A_{\crr}} R^+_{\crr})\ar[d]\\
 \Spf(R^+) \ar[d]\ar@{^(->}[rr]&& \Spf(R^+_{\crr})\ar[d]^{\pi}\\
 \Spf(\so_C)\ar@{^(->}[rr]&&\Spf(\A_{\crr})
}
\end{equation}
Here   $\Spf(\E^{\rm PD}_{\overline{R}_{\Lambda}})$ is the {\rm PD}-envelope of the closed embedding $ \Spf(\overline{R}^+_{\Lambda}) \hookrightarrow \Spf(\Acris(\overline{R}_{\Lambda})\wh{\otimes}_{\A_{\crr}} R^+_{\crr})$. 
 The period morphism\footnote{We note that ${\rm Spa}(R_{\Lambda})$ is a $K(\pi,1)$-space hence 
$C(G_{R_{\Lambda}},\Z/p^n(r)^{\prime})\simeq \R\Gamma_{\proeet}({\rm Spa}(R),\Z/p^n(r)^{\prime})$.} 
 $$ {\alpha}^{R^+}_{r,n}(\Lambda): \quad {\rm Syn}(R^+,r)_n(\Lambda) \to  C(G_{R_{\Lambda}},\Z/p^n(r)^{\prime}) $$ 
 is defined as the composition
$$\xymatrix@C=.2cm@R=.4cm{
{\rm Syn}(R^+,r)_n(\Lambda)\ar@{=}[r]
&[F^r\Omega\kr_{R^{+}_{\crr,n}/\A_{\crr,n}}{\otimes}_{\A_{\crr,n}}\Acris(\Lambda)_n\verylomapr{\phi-p^r}\Omega_{R^{+}_{\crr,n}/\A_{\crr,n}}^{\jcdot}\wh{\otimes}_{\A_{\crr,n}}\Acris(\Lambda)_n]\ar@<3mm>[d]_{\omega}\\
&{C(G_{R_{\Lambda}},[F^r\Omega_{\E^{\rm PD}_{\overline{R}_{\Lambda}},n}\kr\verylomapr{\phi-p^r}\Omega_{\E^{\rm PD}_{\overline{R}_{\Lambda}},n}^{\jcdot}])\phantom{XXX}}\\
&{C(G_{R_{\Lambda}},[F^r \Acris(\overline{R}_{\Lambda})_n\verylomapr{\phi-p^r}
   \Acris(\overline{R}_{\Lambda})_n])\ar@<-3mm>[u]_-{\wr}\phantom{XXX}}\\
&{\phantom{XX}C(G_{R_{\Lambda}},\Z/p^n(r)^{\prime})\ar@<-3mm>[u]_-{\wr}}
}$$
Here, the second $p$-quasi-isomorphism follows from the filtered Poincar\'e Lemma
  \begin{equation}
  \label{berkeley1}
  F^r \Acris(\overline{R}_{\Lambda})_n\stackrel{\sim}{\to} F^r\Omega_{\E^{\rm PD}_{\overline{R}_{\Lambda}},n}\kr,
  \end{equation}
  which can be proved by arguments analogous to the ones used  in the proof of \cite[Prop. 7.3]{SG}.  The third quasi-isomorphism follows from the fundamental $p^r$-exact sequence
$$
0\to \Z/p^n(r)^{\prime}\to F^r\Acris(\overline{R}_{\Lambda})_n\lomapr{\phi-p^r}\Acris(\overline{R}_{\Lambda})_n\to 0
$$

\subsubsection{Proof that ${\alpha}^{R^+}_{r,n}(\Lambda)$ is a quasi-isomorphism}
\begin{proposition}\label{poniedzialek1}
The period morphism
$$
{\alpha}^{R^+}_{r,n}(\Lambda):  \quad {\rm Syn}(R^+,r)_n(\Lambda) \to C(G_{R_{\Lambda}},\Z/p^n(r)^{\prime})
$$
is a $p^{cr}$-quasi-isomorphism, for a universal constant\footnote{In particular, independent of $R, \Lambda, n$, and $r$.} $c$, after truncation $\tau_{\leq r}$. 
\end{proposition}
\begin{proof}
It suffices to show that the morphism
$$\xymatrix@R=.4cm{
[F^r\Omega\kr_{R^{+}_{\crr,n}/\A_{\crr,n}}\wh{\otimes}_{\A_{\crr,n}}\Acris(\Lambda)_n\verylomapr{\phi-p^r}\Omega_{R^{+}_{\crr,n}/\A_{\crr,n}}^{\jcdot}\wh{\otimes}_{\A_{\crr,n}}\Acris(\Lambda)_n]\ar[d]_{\omega}\\
{  C(G_{R_{\Lambda}},[F^r\Omega_{\E^{\rm PD}_{\overline{R}_{\Lambda}},n}\kr\verylomapr{\phi-p^r}\Omega_{\E^{\rm PD}_{\overline{R}_{\Lambda}},n}^{\jcdot}])\phantom{XXXX}}
}$$
is a $p^{cr}$-quasi-isomorphism.
 We will do it by writing the Fontaine-Messing period morphism as a sequence of morphisms inspired  by the theory of  $(\phi,\Gamma)$-modules as in \cite[Th. 4.16]{CN1}, \cite[Th. 7.5]{SG} and then showing that all these morphisms are $p^{cr}$-quasi-isomorphisms after truncation $\tau_{\leq r}$.  This is done by  commutative diagram \eqref{wies1} below, where we denoted by $\omega$ the map that we want to show to be  a $p^{cr}$-quasi-isomorphism (after truncation 
     $\tau_{\leq r}$, which we indicate on the diagram but will often skip 
in the discussion to lighten up the notation). 
The key ingredient to turn the complex coming from the fundamental exact sequence 
(i.e., $C_G(K_{\phi}(F^r\Acris(\overline{R}_{\Lambda})))$ in the upper right corner)
into something which behaves
like a complex of VS's, as a functor in $\Lambda$ 
($\Lambda\mapsto \overline{R}_{\Lambda}$ is not functorial enough),
is the almost \'etale descent (map $\mu_H$).

We set 
    $u= (p-1)/p, v=p-1$ if $p\geq 3$, and $u=3/4,v=3/2$ if $p= 2$. 
   
 \vskip.1cm
  ($\bullet$) {\em Over $C$.} We will first treat the case $\Lambda=C$. 
The following commutative diagram is a simplified version\footnote{To see that note that the zig-zag in the left-bottom corner of that diagram is homotopic (via an explicit Poincar\'e Lemma homotopy) to the identity map.} of the diagram in \cite[proof of Th. 7.5]{SG}; the diagram in loc. cit. is glued from several diagrams commuting on the nose yielding the commuting homotopy for diagram \eqref{wies0}. The top row represents the Fontaine-Messing period morphism.
The diagram shows that the truncation $\tau_{\leq r}$ of the map $\omega$ is a $p^{cr}$-quasi-isomorphism.
 \begin{equation}
 \label{wies0}
\xymatrix@R=.5cm@C=.6cm{
K_{\partial,\phi}(F^rR^+_{\crr}) \ar[d]^{\wr}_{\tau_{\leq r}}  \ar[r]^-{\omega}  & C_G(K_{\partial, \phi}(F^r\E^{\rm PD}_{\overline{R}}))  & C_G(K_{\phi}(F^r\Acris(\overline{R})))\ar[d]^{\wr}\ar^{\sim}_-{\rm PL}[l] \\
K_{\partial,\phi}(F^rR^{[u,v]})  \ar[d]^-{\wr}_{t^{\jcdot},\tau_{\leq r}} &  & C_G(K_{\phi}(F^r\A^{[u,v]}_{\overline{R}})) \\
K_{\Lie \Gamma,\phi}(F^r R^{[u,v]}) & & C_{\Gamma}(K_{\phi}(F^r\A^{[u,v]}_{{R}^{\infty}}))\ar[u]_{\wr}^{\mu_H} \\
 K_{\Gamma,\phi}(F^r R^{[u,v]})\ar[r]^-{\sim}\ar[u]_{\wr}^{\mathcal Laz} &  C_{\Gamma}(K_{\phi}(F^r R^{[u,v]}))  \ar[ru]^{\sim}_{\varepsilon} 
}
\end{equation}
Here,      all the quasi-morphisms are $p^{cr}$-quasi-isomorphisms (after truncation $\tau_{\leq r}$). 
Moreover:
\begin{itemize}
\item $G$ and $\Gamma$ are $G_{R}$ and $\Gamma_R$;
\item $C_G, C_{\Gamma}$ denote the complexes of continuous cochains on the groups $G, \Gamma$, respectively;

\item[$\bullet$] $K$ denotes a complex of Koszul type:

\begin{itemize}
\item[---] the indices indicate the operators involved in the complex: 
\begin{itemize}

\item[$\diamond$] $\partial$
is a shorthand for $\big(X_1\frac{\partial}{\partial X_1},\dots, X_d\frac{\partial}{\partial X_d}\big)$, 

\item[$\diamond$] $\Gamma$ is a shorthand for $(\gamma_1-1,\dots,\gamma_d-1)$, where the $\gamma_i$'s are our
chosen topological generators of~$\Gamma$,

\item[$\diamond$] $\Lie\Gamma$ is a shorthand for $\big(\nabla_1,\dots,\nabla_d)$, where $\nabla_i=\log\gamma_i$,
so that the $\nabla_i$'s are a basis of $\Lie\Gamma$ over $\Z_p$,

\item[$\diamond$]  $\phi$ is a shorthand for $\phi-p^r$.

\end{itemize}

\item[---] only the first term of the complex is indicated: the rest is implicit and obtained from the first term
so that the maps involved make sense: $\varphi$ does not respect filtration or annulus of convergence,
and $\partial$  decrease the degrees of filtration by $1$.
\end{itemize}
   For example, choosing a basis of $\Omega_{R^+_{\crr}/\A_{\crr}}$ transforms complexes involving differentials
into complexes of Koszul type: $K_{\partial,\varphi}(F^rS)$ if $S=R^+_{{\crr}}$
or $S=R^{[u,v]}$.
\end{itemize}

  Let us now turn our attention to the maps between rows:

\begin{itemize}

\item[$\bullet$] Going from the first row to the second row just uses the injections $R^+_{\crr}\subset R^{[u,v]}$, etc.

\item[$\bullet$] Going from the third row to the second row: the map $\mu_H$  is the inflation map from $\Gamma_R$ to $G_{R_{\Lambda}}$, using the injection
  $R^\infty\subset \overline R$ (we use almost \'etale descent -- i.e., Faltings' almost purity
theorem or its extension by Scholze or Kedlaya-Liu -- to prove that it is a quasi-isomorphism); the other map is a "change of Lie algebra map"    $t^{\jcdot}$  appearing
in the proof of \cite[Lemma 5.7]{SG} (multiplication by suitable powers of $t$).

\item[$\bullet$] Going from the fourth row to the third row:   
uses the injection of $R^{[u,v]}\hookrightarrow \A^{[u,v]}_{R^\infty}$ from \eqref{defremind}; 
the map $\laz$ is  defined as in \cite[Lemma 5.8]{SG}. 
\end{itemize}

 Let us now describe the maps between columns:

\begin{itemize}
\item[$\bullet$] The bottom map from the first column to the second one is  the map connecting
continuous cohomology of $\Gamma_R$ to Koszul complex.

\item[$\bullet$] The PL-map from the third column to the second is also induced
by the canonical  injection of rings; it is a  $p^{cr}$-quasi-isomorphisms by 
\cite[Prop. 7.3 ]{SG}.
\end{itemize}

   ($\bullet$) {\em Over a perfectoid $C$-algebra.} Let $\Lambda$ be a perfectoid $C$-algebra. The relevant commutative diagram now takes the following form. Again, it shows that the truncation $\tau_{\leq r}$ of the map $\omega$ is a $p^{cr}$-quasi-isomorphism.

 \begin{equation}
 \label{wies1}
\xymatrix@R=.5cm@C=.6cm{
K_{\partial,\phi}(F^rR^+_{\crr}\wh{\otimes}\Acris(\Lambda)) \ar@{.>}[d]^{\wr}_{\tau_{\leq r}\beta}  \ar[r]^-{\omega}  & C_G(K_{\partial, \phi}(F^r\E^{\rm PD}_{\overline{R}_{\Lambda}}))  & C_G(K_{\phi}(F^r\Acris(\overline{R}_{\Lambda})))\ar[d]^{\wr}\ar^{\sim}_-{\rm PL}[l] \\
K_{\partial,\phi}(F^rR^{[u,v]}\wh{\otimes}{\mathbb A}^{[u,v]}({\Lambda}))  \ar[d]_-{\wr}^{t^{\jcdot},\tau_{\leq r}} &  & C_G(K_{\phi}(F^r{\mathbb A}^{[u,v]}({\overline{R}_{\Lambda}}))) \\
K_{\Lie \Gamma,\phi}(F^r R^{[u,v]}\wh{\otimes} \bbA^{[u,v]}(\Lambda)) & 
& C_{\Gamma}(K_{\phi}(F^r{\mathbb A}^{[u,v]}({{R}^{\infty}_{\Lambda}})))\ar[u]_{\wr}^{\mu_H} \\
 K_{\Gamma,\phi}(F^r R^{[u,v]}\wh{\otimes} \bbA^{[u,v]}(\Lambda))\ar[r]^-{\sim}\ar[u]^{\wr}_{\mathcal Laz} &  C_{\Gamma}(K_{\phi}(F^r R^{[u,v]}\wh{\otimes} \bbA^{[u,v]}(\Lambda)))  \ar[r]^{\sim}_{\varepsilon} & C_{\Gamma}(K_{\phi}(F^r\A^{[u,v]}_{{R}^{\infty}}\wh{\otimes} \bbA^{[u,v]}(\Lambda)))\ar@{.>}[u]^{\delta}_{\wr} 
}
\end{equation}
 Here,  all the quasi-morphisms are $p^{cr}$-quasi-isomorphisms (after truncation $\tau_{\leq r}$).  
{\em The arrow is plain if it is very similar to the one appearing in diagram \eqref{wies0} and dotted if it requires
additional arguments.}
Moreover (we indicate only differences with diagram \eqref{wies0}):
\begin{itemize}
\item tensor products with $\Acris(\Lambda)$ (resp~$\bbA^{[u,v]}(\Lambda)$) are over
$\A_{\crr}$ (resp.~$\A^{[u,v]}$);
\item $(R^{\infty}_{\Lambda}, R^{\infty,+}_{\Lambda})=
(R^{\infty},R^{\infty,+})\wh{\otimes}_{(C,\so_C)}(\Lambda,\Lambda^+)$; it is a perfectoid affinoid by \cite[Prop. 6.18]{Sch0};
\item $G$ and $\Gamma$ are $G_{R_{\Lambda}}$ and $\Gamma_R$;

\end{itemize}

  Let us now turn our attention to the maps between rows:

\begin{itemize}

\item[$\bullet$] The plain arrows are induced by the analogous maps in diagram \eqref{wies0} They are $p^{cr}$-quasi-isomorphisms by the same argument as in loc. cit. since tensoring with ${\mathbb A}^{[u,v]}({\Lambda})$ 
can be done outside the quasi-isomorphic complexes. We note that the tensor products 
$\wh{\otimes}_{\A^{[u,v]}}{\mathbb A}^{[u,v]}({\Lambda})$ are completed but not, a priori,  derived. This does not cause problems because
 the $\A_{\inf}$-module 
$\Ainf( \Lambda)$  is flat: $\so_C^{\flat}$ is a valuation ring hence the $\so_C^{\flat}$-module $\Lambda^{+,\flat}$, being torsion free, is flat.

\item[$\bullet$] Going from the third row to the second row: the map $\mu_H$  is the inflation map from $\Gamma_R$ to $G_{R_{\Lambda}}$, using the injection
  $R^\infty_{\Lambda}\subset \overline R_{\Lambda}$.  
We use almost \'etale descent (i.e., Faltings' almost purity
theorem or its extension by Scholze or Kedlaya-Liu) to prove that it is a quasi-isomorphism. The map $t^{\jcdot}$  is the multiplication by suitable powers of $t$ (we use here Lemma \ref{filtration}).

\end{itemize}

    Finally, the maps $\beta, \delta$ in the diagram are treated by Lemma~\ref{truc0} below.
\end{proof}
\begin{lemma}\label{truc0} 
\begin{enumerate}
\item The canonical morphism 
$$\tau_{\leq r}\beta: \tau_{\leq r}K_{\partial,\phi}(F^rR^+_{\crr}\wh{\otimes}_{\A_{\crr}}\Acris(\Lambda))_n{\to} \tau_{\leq r}K_{\partial,\phi}(F^rR^{[u,v]}\wh{\otimes}_{\A^{[u,v]}}{\mathbb A}^{[u,v]}({\Lambda}))_n $$
is a $p^{cr}$-quasi-isomorphism. 
\item The canonical morphisms
$$
\delta: F^r\A^{[u,v]}_{R^{\infty}}\wh{\otimes}_{\A^{[u,v]}}{\mathbb A}^{[u,v]}({\Lambda})
\to F^r{\mathbb A}^{[u,v]}({R^{\infty}_{\Lambda}})
$$
are isomorphisms. 
\end{enumerate}
\end{lemma}
\begin{proof} For the first claim, 
set $R^+_{\crr,\Lambda}:=R^+_{\crr}\wh{\otimes}_{\A_{\crr}}\Acris(\Lambda)$. 
This ring has the same form as $R^+_{\crr}$ (see Section \ref{local-models}) but with $\A_{\crr}$ 
replaced by $\Acris(\Lambda)$. The above morphism can be written as 
$$
\tau_{\leq r}K_{\partial,\phi}(F^rR^+_{\crr,\Lambda})_n{\to} \tau_{\leq r}K_{\partial,\phi}(F^rR^+_{\crr,\Lambda}\wh{\otimes}_{\A_{\crr}}\A^{[u,v]})_n.
$$
Now, the proof in \cite[Sec.\,4.1]{SG} goes through verbatim by changing $\A_{\crr}$ to $\Acris(\Lambda)$. 

 For the second claim of the lemma, by Lemma \ref{filtration}, we can replace the filtration by the one given by powers of $t$. Hence,  it is enough to show that the canonical map
\begin{equation} \label{A-iso}
\Ainf(R^{\infty})\wh{\otimes}_{\A_{\inf}}\Ainf(\Lambda)\to \Ainf(R^{\infty}_{\Lambda})
\end{equation}
is an isomorphism (the passage to $[u,v]$-version is obtained by taking the completed tensor product of (\ref{A-iso}) with $\A^{[u,v]}$) . Or, since both sides are $p$-adically  derived complete, 
that so is  its reduction modulo~$p$:
  $$
   R^{\infty,+,\flat}\wh{\otimes}_{\so_C^{\flat}}\Lambda^{+,\flat}\to R^{\infty,+,\flat}_{\Lambda}.
  $$
  But this can be checked  modulo $p^{\flat}$. That is, we want   the canonical map
  $$
  (R^{\infty,+,\flat}/p^{\flat})\otimes_{\so_C^{\flat}/p^{\flat}}(\Lambda^{+,\flat}/p^{\flat}) \to R^{\infty,+,\flat}_{\Lambda}/p^{\flat}
  $$
  to be  an isomorphism.   
  
    Now,  this map identifies with the canonical map
  $$
   (R^{\infty,+}/p)\otimes_{\so_C/p}(\Lambda^{+}/p) \to R^{\infty,+}_{\Lambda}/p.
$$ It suffices thus to show that the canonical map
  $$
    R^{\infty,+}\wh{\otimes}_{\so_C}\Lambda^{+} \to R^{\infty,+}_{\Lambda}
$$
is an isomorphism. But this is clear since both sides are isomorphic to the completion of the same \'etale extension of the tower
$$
\Lambda^+\{X^{{1}/{p^{\infty}}},\tfrac{1}{(X_1\ldots X_a)^{{1}/{p^{\infty}}}},\tfrac{\varpi^{{1}/{p^{\infty}}}}{(X_{a+1}\ldots X_{d})^{{1}/{p^{\infty}}}}\}.
$$
\end{proof}

\subsubsection{Modification of the period morphism ${\alpha}^{R^+}_{r,n}(\Lambda)$} 
The Fontaine-Messing period morphism 
$$
{\alpha}^{R^+}_{r,n}(\Lambda):  \quad {\rm Syn}(R^+,r)_n(\Lambda) \to C(G_{R_{\Lambda}},\Z/p^n(r)^{\prime})
$$
constructed above is neither  functorial  in $R^+$ nor in $\Lambda^+$.  By diagram \eqref{wies1}, we can replace it by the map that traces that diagram down-bottom right-up and replaces $C_G(K_{\phi}(F^r{\mathbb A}_{\crr}(\overline{R}_{\Lambda}))$ with 
$ \R\Gamma_{\proeet}({\rm Spa}(R_{\Lambda}),K_{\phi}(F^r{\mathbb A}_{\crr}))$: 

\begin{align*}
{\mathbb A}^{R^+}_{r,n}(\Lambda): {\rm Syn}(R^+,r)_n(\Lambda) & \simeq K_{\partial,\phi}(F^rR^+_{\crr}\wh{\otimes}\Acris(\Lambda))_n   \xrightarrow[\sim]{\tau_{\leq r}\beta} K_{\partial,\phi}(F^rR^{[u,v]}\wh{\otimes}\bbA^{[u,v]}(\Lambda))_n \\
  &  \xrightarrow[\sim]{t^{\jcdot},\tau_{\leq r}} K_{\Lie \Gamma,\phi}(F^r R^{[u,v]}\wh{\otimes} \bbA^{[u,v]}(\Lambda))_n 
   \xleftarrow[\sim]{\cal Laz} K_{\Gamma,\phi}(F^r R^{[u,v]}\wh{\otimes} \bbA^{[u,v]}(\Lambda))_n \\
 & \xrightarrow[\sim]{} C_{\Gamma}(K_{\phi}(F^r R^{[u,v]}\wh{\otimes} \bbA^{[u,v]}(\Lambda))_n)
  \xrightarrow[\sim]{\varepsilon} C_{\Gamma}(K_{\phi}(F^r\A^{[u,v]}_{{R}^{\infty}}\wh{\otimes} \bbA^{[u,v]}(\Lambda))_n)\\
 &  \xrightarrow[\sim]{\delta} C_{\Gamma}(K_{\phi}(F^r{\mathbb A}^{[u,v]}({{R}^{\infty}_{\Lambda}}))_n) 
   \xrightarrow[\sim]{\nu} \R\Gamma_{\proeet}({\rm Spa}(R_{\Lambda}),K_{\phi}(F^r{\mathbb A}_{\crr})_n) \\
    & \xleftarrow[\sim]{\rm FES} \R\Gamma_{\proeet}({\rm Spa}(R_{\Lambda}),\Z/p^n(r)^{\prime}).
\end{align*}
Here ${\rm FES}$ stands for "fundamental exact sequence". 
This map is   functorial in $\Lambda$ and is a $p^{cr}$-quasi-isomorphism, for a universal constant $c$, after truncation $\tau_{\leq r}$ (by Proposition \ref{poniedzialek1}). In the next section we will modify it to make it  functorial in $R^+$ as well. 

\subsection{Local period morphism, general case} 
 \subsubsection{Over $C$}
  Consider now the same local situation as above: a formal scheme $\sx=\Spf R^+$, for  an algebra $R^+$, which is the $p$-adic completion of an \'etale algebra  over a ring $R^+_{\Box}$ from \eqref{frame1}. 
We equip $\Spf (R^+_{\Box})$ and $\Spf (R^+)$ with the logarithmic structure induced by the special fiber. But now we will allow larger coordinate rings, i.e., we assume that we have
 the following commutative diagram, a relaxed version of diagram (\ref{lifted}):
\begin{equation}
\label{nonlifted}
\xymatrix@R=.5cm{
 & \Spf(\E^{\rm PD}_{\overline{R},\kappa})\ar[dd]\ar[dr]\\
 \Spf(\overline{R}^+) \ar@{^(->}[ru]\ar[dd]\ar@{^(->}[rr]^{\kappa}& & \Spf(\Acris(\overline{R})\wh{\otimes}_{\A_{\crr}} R^+_{\crr})\ar[dd]\\
 &  \sd^+_{\crr}\ar[dr]\\
 \Spf(R^+) \ar[d]\ar@{^(->}[ur]\ar@{^(->}[rr]^{\iota}&& \Spf(R^+_{\crr})\ar[d]^{\pi}\\
 \Spf(\so_C)\ar@{^(->}[rr]&&\Spf(\A_{\crr})
}
\end{equation}
The map $\pi$ is log-smooth and the map $\iota$ is a  closed immersion (and the bottom square is not necessarily cartesian as it was in diagram (\ref{lifted})). This extra degree of freedom will allow us to globalize the period map (the added variables disappear thanks to pro-\'etale techniques, see Lemma \ref{almost-there1}). We assume that $\Spf(R^+_{\crr})$ is equipped with 
a lift $\phi$ of the Frobenius on $\Spf(\A_{\crr})$. $\sd^+_{\crr}$ 
 and  $ \Spf(\E^{\rm PD}_{\overline{R},\kappa}) $ are the  (log)-{\rm PD}-envelopes of $\iota$ and
 $\kappa$, respectively.
\smallskip
 Let  $r\in \N$.  We define the filtered  de Rham complex $\Omega_{\E^{\rm PD}_{\overline{R}}}^{\jcdot}$   as in the case of lifted coordinates.  Let 
 $$
 F^r\Omega^{\jcdot}_{\sd^+_{\crr}/\A_{\crr}}:=F^r\sd^+_{\crr}\to F^{r-1}\sd^+_{\crr}\otimes_{R^+_{\crr}}\Omega^1_{R^+_{\crr}/\A_{\crr}}\to F^{r-2}\sd^+_{\crr}\otimes_{R^+_{\crr}}\Omega^2_{R^+_{\crr}/\A_{\crr}}\to \cdots
 $$
 The crystalline syntomic cohomology $\rg_{\synt}(\sx,r)$ is computed by the complex
 $$
 {\rm Syn}(R^+_{\crr},r):=[F^r\Omega^{\jcdot}_{\sd^+_{\crr}/\A_{\crr}}\lomapr{\phi-p^r}\Omega^{\jcdot}_{\sd^+_{\crr}/\A_{\crr}}]
 $$
  The Fontaine-Messing period map  \begin{equation}
 \label{zimnob0}
   {\alpha}^{R^+_{\crr}}_{r,n}: \quad {\rm Syn}(R^+_{\crr},r)_n \to  C(G_{R},\Z/p^n(r)^{\prime})
   \end{equation}
is computed by  the composition
\begin{equation}
\label{zimnob1}
\xymatrix@C=2mm@R=5mm{
 {\rm Syn}(R^+_{\crr},r)_n \ar@{=}[r] & [F^r\Omega_{\sd^{+}_{\crr,n}/\A_{\crr,n}}\kr\verylomapr{\phi-p^r}\Omega_{\sd^{+}_{\crr,n}/\A_{\crr,n}}^{\jcdot}]\ar[d]_{\omega}\\
& {C(G_{R},[F^r\Omega_{\E^{\rm PD}_{\overline{R},\kappa,n}}\kr\verylomapr{\phi-p^r}\Omega_{\E^{\rm PD}_{\overline{R},\kappa, n}}^{\jcdot}])\phantom{XXX}}\\
   &{C(G_{R},[F^r \Acris(\overline{R})_n\verylomapr{\phi-p^r} \Acris(\overline{R})_n])\ar[u]\phantom{XXX}}\\
&C(G_{R},\Z/p^n(r)^{\prime}).\ar[u]_-{\wr} }
\end{equation}

\begin{lemma}
\label{addedb1}
The period map \eqref{zimnob0} is a $p^{cr}$-quasi-isomorphism, for a universal constant~$c$, after truncation $\tau_{\leq r}$.
\end{lemma}
\begin{proof} It suffices to show that the first and the second map in the composition \eqref{zimnob1}  are  $p^{cr}$-quasi-isomorphisms, for  universal constants~$c$, after truncation $\tau_{\leq r}$.  Consider the product
$\Spf(R^+_{\crr} \wh{\otimes}_{\A_{\crr}}\wt{R}^+_{\crr})$ (we put $\widetilde{\phantom{X}}$ to distinguish diagram \eqref{lifted} from  diagram (\ref{nonlifted})) and the canonical closed immersion $\iota_1: \Spf(R^+)\hookrightarrow \Spf(R^+_{\crr} \wh{\otimes}_{\A_{\crr}}\wt{R}^+_{\crr})$. Let $\sd^+$ be the PD-envelope of $\iota_1$ and let $\E^{\rm PD}_{\overline{R}}$ be as in  diagram \eqref{nonlifted} for $\iota_1$ in place of $\iota$. Consider  the compatible maps 
\begin{align}
\label{passage10}
& p_1: \Spf(\sd^+)\to \Spf(\sd^+_{\crr}),\quad p_2: \Spf(\sd^+)\to\Spf(\wt{R}^+_{\crr}),\\
& p_1: \Spf({\E}^{\rm PD}_{\overline{R}})\to \Spf({\E}^{\rm PD}_{\overline{R},\kappa}),\quad p_2: \Spf({\E}^{\rm PD}_{\overline{R}})\to \Spf(\wt{\E}^{\rm PD}_{\overline{R}})\notag
\end{align} induced by the  two projections from 
$\Spf(\R^+_{\crr} \wh{\otimes}_{\A_{\crr}}\wt{R}^+_{\crr})$ to $\Spf(R^+_{\crr})$ and $\Spf(\wt{R}^+_{\crr})$, respectively. 
These maps are also compatible with the other maps in diagrams (\ref{nonlifted}) and   (\ref{lifted}). 
They induce  compatible maps
\begin{align}
\label{passage1}
& p_2^*: F^r\Omega_{\wt{R}^{+}_{\crr,n}/\A_{\crr,n}}\kr \stackrel{\sim}{\to}  F^r\Omega_{\sd^{+}_{n}/\A_{\crr,n}}\kr,\quad
p^*_2: F^r\Omega_{\wt{\E}^{\rm PD}_{\overline{R},n}}\kr   \stackrel{\sim}{\to} F^r\Omega_{\E^{\rm PD}_{\overline{R}, n}}\kr,\\
& p_1^*:  F^r\Omega_{\sd^{+}_{\crr,n}/\A_{\crr,n}}\kr  \stackrel{\sim}{\to}  F^r\Omega_{\sd^{+}_{n}/\A_{\crr,n}}\kr,\quad
p^*_1: F^r\Omega_{{\E}^{\rm PD}_{\overline{R},\kappa, n}}\kr   \stackrel{\sim}{\to}  F^r\Omega_{\E^{\rm PD}_{\overline{R}, n}}\kr.\notag
\end{align}
Moreover, the $\E$--maps are also compatible with the canonical maps from 
$F^r \Acris(\overline{R})_n\ $.

 The maps in \eqref{passage1}   are quasi-isomorphisms since both terms in the left maps  compute absolute crystalline cohomology of $\Spf(R^+)$ and both terms in the right map --  crystalline cohomology of 
 $\Spf(\overline{R}^+)$ over $\Acris(\overline{R})$. 
 
  Now, since the maps in \eqref{passage10} are compatible with Frobenius, the maps from \eqref{passage1} allow us to replace the maps in the composition \eqref{zimnob1} for $\sd^+_{\crr}$, first, with the ones for $\sd^+$ and, then, with the ones for 
$\wt{R}^+_{\crr}$, which we know to be    $p^{cr}$-quasi-isomorphisms, for a universal constant~$c$, after truncation $\tau_{\leq r}$
   \end{proof}
 
\subsubsection{Over a perfectoid $C$-algebra} 
Let $\Lambda$ be a perfectoid affinoid over $C$. To show that the Fontaine-Messing period map, lifted to $\Lambda$, can be globalized we will use the following commutative diagram:
    \begin{equation}
    \label{nonlifted-Lambda}
\xymatrix@R=.5cm{
 & \Spf(\E^{\rm PD}_{\overline{R}_{\Lambda,\kappa}})\ar[dd]\ar[dr]\\
 \Spf(\overline{R}^+_{\Lambda}) \ar@{^(->}[ru]\ar[dd]\ar@{^(->}[rr]^{\kappa}& & \Spf(\Acris(\overline{R}_{\Lambda})\wh{\otimes}_{\A_{\crr}} R^+_{\crr})\ar[dd]\\
 &  \sd^+_{\crr}\ar[dr]\\
 \Spf(R^+) \ar[d]\ar@{^(->}[ur]\ar@{^(->}[rr]^{\iota}&& \Spf(R^+_{\crr})\ar[d]^{\pi}\\
 \Spf(\so_C)\ar@{^(->}[rr]&&\Spf(\A_{\crr}) }
\end{equation}
The Fontaine-Messing period map  
\begin{equation}
\label{defFM}
 {\alpha}^{R^+_{\crr}}_{r,n}(\Lambda): \quad {\rm Syn}(R^+_{\crr},r)_n(\Lambda) \to  C(G_{R_{\Lambda}},\Z/p^n(r)^{\prime})
 \end{equation}
   can be defined 
by  the composition
\begin{equation}
\label{zimnob11}
\xymatrix@C=2mm@R=5mm{
 {\rm Syn}(R^+_{\crr},r)_n(\Lambda)\ar@{=}[r] &[F^r\Omega_{\sd^{+}_{\crr,n}/\A_{\crr,n}}\kr\wh{\otimes}_{\A_{\crr,n}}\Acris(\Lambda)_n\verylomapr{\phi-p^r}\Omega_{\sd^{+}_{\crr,n}/\A_{\crr,n}}^{\jcdot}\wh{\otimes}_{\A_{\crr,n}}\Acris(\Lambda)_n]\ar@<2mm>[d]_{\omega}\\
 & {C(G_{R_{\Lambda}},[F^r\Omega_{\E^{\rm PD}_{\overline{R}_{\Lambda},\kappa,n}}\kr\verylomapr{\phi-p^r}\Omega_{\E^{\rm PD}_{\overline{R}_{\Lambda},\kappa, n}}^{\jcdot}])\phantom{XXX}}\\
 &  {C(G_{R_{\Lambda}},[F^r \Acris(\overline{R}_{\Lambda})_n\verylomapr{\phi-p^r} \Acris(\overline{R}_{\Lambda})_n])\ar@<-2mm>[u]_-{\wr}\phantom{XXX}}\\
&{C(G_{R_{\Lambda}},\Z/p^n(r)^{\prime})\ar@<-2mm>[u]_-{\wr}\phantom{X}}
}
\end{equation}

\begin{lemma}
\label{addedb2}
The period map \eqref{defFM} is a $p^{cr}$-quasi-isomorphism, for a universal constant~$c$, after truncation $\tau_{\leq r}$.
\end{lemma}
\begin{proof} It suffices to show that the first and second maps in the composition \eqref{zimnob11}  are $p^{cr}$-quasi-isomorphisms, for a universal constant~$c$, after truncation $\tau_{\leq r}$.
But since the map $\wt{\pi}$ in diagram \eqref{lambda-version} is log-smooth 
(we put $\widetilde{\phantom{X}}$ to distinguish diagram \eqref{lambda-version} from  diagram (\ref{nonlifted-Lambda}))  and $\sd^+_{\crr} $ in diagram (\ref{nonlifted-Lambda}) is $I$-adically complete, for the defining PD-ideal $I$,
we have maps from $f: \Spf(\sd^+_{\crr} )$ to $\Spf(\wt{R}^+_{\crr})$ and from $\Spf(\E^{\rm PD}_{\overline{R}_{\Lambda},\kappa})$ to $\Spf(\wt{\E}^{\rm PD}_{\overline{R}_{\Lambda}})$ that are also compatible with with other maps in  diagrams (\ref{nonlifted}) and   (\ref{lambda-version}). These maps induces two compatible maps
\begin{align*}
 F^r\Omega_{\wt{R}^{+}_{\crr,n}/\A_{\crr,n}}\kr\wh{\otimes}_{\A_{\crr,n}}\Acris(\Lambda)_n  & \to  F^r\Omega_{\sd^{+}_{\crr,n}/\A_{\crr,n}}\kr\wh{\otimes}_{\A_{\crr,n}}\Acris(\Lambda)_n,\\
F^r\Omega_{\wt{\E}^{\rm PD}_{\overline{R}_{\Lambda},n}}\kr   & \to F^r\Omega_{\E^{\rm PD}_{\overline{R}_{\Lambda},\kappa,n}}\kr.
\end{align*}
 These maps are quasi-isomorphisms:   the first one by the first quasi-isomorphism from (\ref{passage1}) and flatness of $\Acris(\Lambda)$ over $\A_{\crr}$;  the second one, via the filtered Poincar\'e Lemma (note that both the domain and the target compute crystalline cohomology of $\overline{R}^+_{\Lambda,n}$ over $\Acris(\overline{R}_{\Lambda})_n$), can be identified with the identity map
 $$  \id: F^r\Acris(\overline{R}_{\Lambda})_n   \to F^r\Acris(\overline{R}_{\Lambda})_n.
 $$

 Assume now that the map $f$ is compatible with Frobenius. Then in our lemma we may take 
 $R^+_{\crr}=\wt{R}^+_{\crr}$ in which case we can use Proposition \ref{poniedzialek1}. In general, map $f$ will not be compatible with Frobenius and then we have to argue via a zig-zag of such maps as in the proof of Lemma \ref{addedb1}.
 \end{proof}
 
\subsubsection{Modification of the period morphism ${\alpha}^{R^+_{\crr}}_{r,n}(\Lambda)$}
 As in Section \ref{local-models}, we can replace the Fontaine-Messing period morphism 
$$
{\alpha}^{R^+_{\crr}}_{r,n}(\Lambda):  \quad {\rm Syn}(R^+_{\crr},r)_n(\Lambda) \to C(G_{R_{\Lambda}},\Z/p^n(r)^{\prime})
$$
constructed above (which is neither functorial  in $R^+_{\crr}$ nor in $\Lambda^+$) by a better behaved morphism. But before doing this we need to make a special choice for our coordinate system (a simpler variant of  the one used\footnote{There are notable differences: we did not separate the torus data and we allowed  coordinates  $R^{\Box}_{\delta}$ which do not satisfy point (2) below.} in \cite[Sec. 5.17]{CK}). 

 Assume that each pair of irreducible components of the special fiber of $\sx$ has nontrivial intersection (in particular, $\sx$ is connected) and that we have a closed immersion
\begin{equation}
\label{closed1}
\iota_{\Delta}: \sx=\Spf(R^+)\hookrightarrow \prod_{\delta\in\Delta}\Spf R^{\Box}_{\delta},
\end{equation}
such that
\begin{enumerate}
\item $ R^{\Box}_{\delta}:=\so_C\{X^{\pm 1}_{\delta,0},\ldots,X^{\pm 1}_{\delta,a_{\delta}}, X_{\delta,{a_{\delta}+1}},\ldots ,X_{d_{\delta}}\}/(X_{\delta,{a_{\delta}+1}}\cdots X_{d_{\delta}}-p^{b_\delta})$, where $\delta\in\Delta$, for a finite set $\Delta$, 
 and $b_{\delta}\in \Q_{\geq 0}$;
\item There exists a $\delta_0\in\Delta$ such that the morphism $\Spf(R^+)\to \Spf(R^{\Box}_{\delta_0})$ is \'etale.
\end{enumerate}
We set $$\Spf(R^{\Box}_{\Delta}):=\prod_{\delta\in\Delta}\Spf R^{\Box}_{\delta}.
$$ The formal schemes  $\Spf R^{\Box}_{\delta}$s and $\Spf(R^{\Box}_{\Delta})$ are endowed with the log-structures coming from the special fiber. 

 Let $R^{\Box,\infty}_{\delta}$ be the $p$-adic completion of the ring
$$
 \colim_n\so_C\{X^{\pm \frac{1}{p^n}}_{\delta,0}\cdots X^{\pm \frac{1}{p^n}}_{\delta,a_{\delta}},X^{ \frac{1}{p^n}}_{\delta,{a_{\delta}+1}},\ldots ,X^{ \frac{1}{p^n}}_{d_{\delta}}\}/(X^{ \frac{1}{p^n}}_{\delta,{a_{\delta}+1}}\cdots X^{ \frac{1}{p^n}}_{d_{\delta}}-p^{ \frac{b_\delta}{p^n}}).
 $$
 We denote by $R^{\Box,\infty}_{\Delta}$ the completed tensor product of the above rings and set $R^{+,\infty}_{\Delta}:=R^{\Box,\infty}_{\Delta}\wh{\otimes}_{R^{\Box}_{\Delta}}R^+$, $R^{\infty}_{\Delta}:=R^{+,\infty}_{\Delta}[1/p]$.
We consider the groups
 $$
 \Gamma_{\delta}:=\Gal(R^{\Box,\infty}_{\delta}[\tfrac{1}{p}]/R^{\Box}_{\delta}[\tfrac{1}{p}])\simeq \Z_p^{\oplus d_{\delta}},\quad 
 \Gamma_{\Delta}:=\prod_{\delta\in\Delta}\Gamma_{\delta}.
 $$
 If  $(\gamma_{\delta,i})_{0\leq i <d_{\delta}}$ are the topological generators of $\Gamma_{\delta}$, 
  the action of  $\Gamma_{\delta}$ on $R^{\Box}_{\delta}$ is given by:
  \begin{align*}
   \gamma_{\delta,i}(X_{\delta,i}) & =[\varepsilon]X_{\delta,i}\text{ and }  \gamma_{\delta,i}(X_{\delta, j})=X_{\delta,j} \text{ for } i\neq j; i,  j \leq a_{\delta};\\
    \gamma_{\delta,i}(X_{\delta,i}) & =[\varepsilon]X_{\delta,i}\text{ and }  \gamma_{\delta,i}(X_{\delta, j})=X_{\delta,j},  \gamma_{\delta,i}(X_{\delta,d_{\delta}})  =[\varepsilon]^{-1}X_{\delta,d_{\delta}} \text{ for } i\neq j, a_{\delta} <j < d_{\delta}.
 \end{align*}
 We get the  induced action of $\Gamma_{\Delta}$ on $R^{\rm PD}_{\Delta}$ (the divided power envelope of the closed immersion \eqref{closed1}).
 We note that ${\rm Spa}(R^{\Box,\infty}_{\delta}[\tfrac{1}{p}])$ is an affinoid perfectoid  pro-\'etale $\Gamma_{\delta}$-cover\footnote{We skip the $^+$-structure from the notation to lighten up the writing.} of ${\rm Spa}(R^{\Box}_{\delta}[\tfrac{1}{p}])$; similarly,  ${\rm Spa}(R^{\Box,\infty}_{\Delta}[\tfrac{1}{p}])$ is an affinoid perfectoid  pro-\'etale 
 $\Gamma_{\Delta}$-cover of ${\rm Spa}(R^{\Box}_{\Delta}[\tfrac{1}{p}])$. Its base change ${\rm Spa}(R^{\infty}_{\Delta})$  is an affinoid perfectoid pro-\'etale $\Gamma_{\Delta}$-cover of ${\rm Spa}(R)$ (by almost purity since 
 ${\rm Spa}(R^{\infty}_{\Delta})$ contains ${\rm Spa}(R^{\infty}_{\delta_0})$ as a subcover). 

   Set 
 \begin{align*}
 \A_{\crr}(R^{\Box}_{\delta})& :=\A_{\crr}\{X^{\pm 1}_{\delta,0},\ldots,X^{\pm 1}_{\delta,a_{\delta}}, X_{\delta,{a_{\delta}+1}},\ldots ,X_{d_{\delta}}\}/(X_{\delta,{a_{\delta}+1}}\cdots X_{d_{\delta}}-[(p^{\flat})^{b_\delta}]),\\
  \A_{\crr}(R^{\Box}_{\delta}) & :=\wh{\otimes}_{\delta\in\Delta} \A_{\crr}(R^{\Box}_{\delta}).
 \end{align*}
 The diagram \eqref{nonlifted-Lambda}   takes now  the following shape ($R^+_{\crr}$ changed to  $\A(R^{\Box}_{\Delta})$):
\begin{equation}
\label{nonlifted1}
\xymatrix@R=.5cm{
 & \Spf(\E^{\rm PD}_{\overline{R},\kappa})\ar[dd]\ar[dr]\\
 \Spf(\overline{R}^+_{\Lambda}) \ar@{^(->}[ru]\ar[dd]\ar@{^(->}[rr]^{\kappa}& & \Spf(\Acris(\overline{R}_{\Lambda})\wh{\otimes}_{\A_{\crr}} \A_{\crr}(R^{\Box}_{\Delta}))\ar[dd]\\
 & R^{\rm PD}_{\Delta}\ar[dr]\\
 \Spf(R^+) \ar[d]\ar@{^(->}[ur]\ar@{^(->}[rr]^{\iota}&& \Spf(\A_{\crr}(R^{\Box}_{\Delta}))\ar[d]^{\pi}\\
 \Spf(\so_C)\ar@{^(->}[rr]&&\Spf(\A_{\crr})
}
\end{equation}
The Frobenius $\phi$ on $\Spf(\A_{\crr}(R^{\Box}_{\Delta}))$ is induced by the Frobenius on $\Spf(\A_{\crr})$ and by raising to $p$'th power the coordinates. 
And, we have in this setting the following  analog  of  diagram \eqref{wies1}:
 \begin{equation}
 \label{wies1b}
\xymatrix@R=.5cm@C=.6cm{
K_{\partial,\phi}(F^r R^{\rm PD}_{\Delta}\wh{\otimes}\Acris(\Lambda)) \ar[d]^{\wr}_{\tau_{\leq r}\beta}  \ar[r]^-{\tau_{\leq r}\omega} _-{\sim} & C_G(K_{\partial, \phi}(F^r\E^{\rm PD}_{\overline{R}_{\Lambda,\kappa}}))  & C_G(K_{\phi}(F^r\Acris(\overline{R}_{\Lambda})))\ar[d]^{\wr}\ar^{\sim}_-{PL}[l] \\
K_{\partial,\phi}(F^r R_{\Delta}^{[u,v]}\wh{\otimes}\bbA^{[u,v]}(\Lambda)) \ar[d]_-{t^{\jcdot},\tau_{\leq r}}    &  & C_G(K_{\phi}(F^r{\mathbb A}^{[u,v]}({\overline{R}_{\Lambda}}))) \\
 K_{\Lie \Gamma_{\Delta},\phi}(F^r R_{\Delta}^{[u,v]}\wh{\otimes} \bbA^{[u,v]}(\Lambda)) & & C_{\Gamma_{\Delta}}(K_{\phi}(F^r{\mathbb A}^{[u,v]}({{R}^{\infty}_{\Delta, \Lambda}})))\ar[u]_{\wr}^{\mu_H}\\
  K_{\Gamma_{\Delta},\phi}(F^r R_{\Delta}^{[u,v]}\wh{\otimes} \bbA^{[u,v]}(\Lambda)) \ar[r]_-{\sim}\ar[u]_-{\wr}^-{\mathcal Laz} &  C_{\Gamma_{\Delta}}(K_{\phi}(F^r R_{\Delta}^{[u,v]}\wh{\otimes} \bbA^{[u,v]}(\Lambda)))\ar[r]^{\varepsilon}_{\sim} & C_{\Gamma_{\Delta}}(K_{\phi}(F^r\A^{[u,v]}_{R_{\Delta}^{\infty}}\wh{\otimes} \bbA^{[u,v]}(\Lambda)))\ar[u]^{\delta}_{\wr}
}
\end{equation}
  Here,  everything  is taken modulo $p^n$ and all the quasi-isomorphisms are $p^{cr}$-quasi-isomorphisms (after truncation $\tau_{\leq r}$). Indeed, for the map $\omega$ this follows by comparison with the diagram \eqref{wies1}; for the map $\beta$ -- by the same argument as the one used in the proof of Lemma \ref{truc0}. 
  The map $\varepsilon$ is induced by the maps
  \begin{align*}
  \varepsilon_0: R^{\Box}_{\Delta}\to \A_{\crr}({R^{\infty}}),\quad  \varepsilon_1: R^{\rm PD}_{\Delta}\to  \A_{\crr}({R^{\infty}}),
  \end{align*}
where the first map is defined by choosing $p$-towers of coordinates as in \eqref{defremind} and the second map is the unique extension of the first one (by the universal property of logarithmic divided power envelopes). We treat it and the map  $\delta$ with  the following lemma: 
\begin{lemma}\label{almost-there1}
The maps
\begin{align*}
& \varepsilon: C_{\Gamma_{\Delta}}(K_{\phi}(F^r R_{\Delta}^{[u,v]}\wh{\otimes} \bbA^{[u,v]}(\Lambda)))\to C_{\Gamma_{\Delta}}(K_{\phi}(F^r\A^{[u,v]}_{R_{\Delta}^{\infty}}\wh{\otimes} \bbA^{[u,v]}(\Lambda))),\\
&  \delta: C_{\Gamma_{\Delta}}(K_{\phi}(F^r\A^{[u,v]}_{R_{\Delta}^{\infty}}\wh{\otimes} \bbA^{[u,v]}(\Lambda)))\to C_{\Gamma_{\Delta}}(K_{\phi}(F^r\bbA^{[u,v]}({{R}^{\infty}_{\Delta,\Lambda}})))
\end{align*}
are $p^{cr}$-quasi-isomorphisms after truncation $\tau_{\leq r}$.
\end{lemma}
\begin{proof} We will pass to the frame $R^{\Box}_{\delta_0}$, where the statement of the lemma is known. For the first map,  arguing as in the proof of Lemma \ref{addedb2}, we find a map $f: \Spf(R^{\rm PD}_{\Delta})\to \Spf(\wt{R}_{\crr})$ compatible with the diagram \eqref{nonlifted1} and the analog of the diagram \eqref{lambda-version} for the frame $R^{\Box}_{\delta_0}$. If this map is compatible with Frobenius then it induces the vertical arrows in the commutative diagram:
$$
\xymatrix{
C_{\Gamma_{\Delta}}(K_{\phi}(F^r R_{\Delta}^{[u,v]}\wh{\otimes} \bbA^{[u,v]}(\Lambda)))\ar[r]^{\varepsilon} & C_{\Gamma_{\Delta}}(K_{\phi}(F^r\A^{[u,v]}_{R_{\Delta}^{\infty}}\wh{\otimes} \bbA^{[u,v]}(\Lambda)))\\
C_{\Gamma_{\delta_0}}(K_{\phi}(F^r \wt{R}_{\delta_0}^{[u,v]}\wh{\otimes} \bbA^{[u,v]}(\Lambda)))\ar[r]^{\varepsilon_{\delta_0}}_-{\sim}  \ar[u]^{\tau_{\leq r}}_{\wr}&
 C_{\Gamma_{\delta_0}}(K_{\phi}(F^r\A^{[u,v]}_{R_{\delta_0}^{\infty}}\wh{\otimes} \bbA^{[u,v]}(\Lambda))).\ar[u]
}
$$
The map $\varepsilon_{\delta_0}$ is a $p^{cr}$-quasi-isomorphism by diagram \eqref{wies1}, the left vertical map is a $p^{cr}$-quasi-isomorphism (after truncation $\tau_{\leq r}$) 
 because both the domain and the target compute ${\rm Syn}(R^+,r)_n(\Lambda)$, and the right vertical map is an almost quasi-isomorphism by almost \'etale descent.
It follows that  the map $\varepsilon$ from our lemma is a  $p^{cr}$-quasi-isomorphism after truncation $\tau_{\geq r}$, as wanted.  In general, the map $f$ is not compatible with Frobenius and we have to proceed by a zig-zag as in the proof of Lemma \ref{addedb1}.

For the map $\delta$, consider the commutative diagram
$$
\xymatrix{
 C_{\Gamma_{\Delta}}(K_{\phi}(F^r\A^{[u,v]}_{R_{\Delta}^{\infty}}\wh{\otimes} \bbA^{[u,v]}(\Lambda)))\ar[r]^-{\mu_H\delta} &  C_G(K_{\phi}(F^r\bbA^{[u,v]}({\overline{R}_{\Lambda}}))) \\
  C_{\Gamma_{\delta_0}}(K_{\phi}(F^r\A^{[u,v]}_{R_{\delta_0}^{\infty}}\wh{\otimes} \bbA^{[u,v]}(\Lambda)))\ar[u]^{\wr}\ar[ru]^-{\mu_H\delta}_-{\sim}
}
$$
The diagonal map is a $p^{cr}$-quasi-isomorphism by the diagram \eqref{wies1}. It follows that so is the horizontal map and then the map $\delta$, as wanted. 
\end{proof}
\begin{remark}The reader will probably notice that for what follows we did not need to prove Lemma \ref{almost-there1}: it will suffice to know that the composition $\mu_H\delta\varepsilon$ is a, truncated at $r$, 
$p^{cr}$-quasi-isomorphism and this we know since the diagram \eqref{wies1b} commutes and the top horizontal maps are truncated at $r$ $p^{cr}$-quasi-isomorphisms. 
\end{remark}

 Diagram \eqref{wies1b}  allows us to replace ${\alpha}^{R^+_{\crr}}_{r,n}(\Lambda)$ with the  map $ {\bbalpha}^{R^{\Box}_{\Delta}}_{r,n}(\Lambda)$ that traces that diagram down-bottom right-up and replaces $C_G(K_{\phi}(F^r{\mathbb A}_{\crr}(\overline{R}_{\Lambda}))$ with 
$ \R\Gamma_{\proeet}({\rm Spa}(R_{\Lambda}),K_{\phi}(F^r{\mathbb A}_{\crr}))$: 

\begin{align*}
 {\bbalpha}^{R^{\Box}_{\Delta}}_{r,n}(\Lambda): {\rm Syn}({R^{\Box}_{\Delta}},r)_n(\Lambda) \to \R\Gamma_{\proeet}({\rm Spa}(R_{\Lambda}),\Z/p^n(r)^{\prime}).
\end{align*}

\begin{align*}
 {\bbalpha}^{R^{\Box}_{\Delta}}_{r,n}(\Lambda): {\rm Syn}({R^{\Box}_{\Delta}},r)_n(\Lambda) & \simeq K_{\partial,\phi}(F^rR^{\rm PD}_{\Delta}\wh{\otimes}\Acris(\Lambda))_n   \xrightarrow[\sim]{\tau_{\leq r}\beta} K_{\partial,\phi}(F^rR^{[u,v]}_{\Delta}\wh{\otimes}\bbA^{[u,v]}(\Lambda))_n \\
  &  \xrightarrow[\sim]{t^{\jcdot},\tau_{\leq r}} K_{\Lie \Gamma_{\Delta},\phi}(F^r R^{[u,v]}_{\Delta}\wh{\otimes} \bbA^{[u,v]}(\Lambda))_n 
   \xleftarrow[\sim]{\cal Laz} K_{\Gamma_{\Delta},\phi}(F^r R^{[u,v]}_{\Delta}\wh{\otimes} \bbA^{[u,v]}(\Lambda))_n \\
 & \xrightarrow[\sim]{} C_{\Gamma_{\Delta}}(K_{\phi}(F^r R^{[u,v]}_{\Delta}\wh{\otimes} \bbA^{[u,v]}(\Lambda))_n)
  \xrightarrow[\sim]{\varepsilon} C_{\Gamma_{\Delta}}(K_{\phi}(F^r\A^{[u,v]}_{{R}^{\infty}_{\Delta}}\wh{\otimes} \bbA^{[u,v]}(\Lambda))_n)\\
 &  \xrightarrow[\sim]{\delta} C_{\Gamma_{\Delta}}(K_{\phi}(F^r\bbA^{[u,v]}({{R}^{\infty}_{\Delta,\Lambda}}))_n) 
   \xrightarrow[\sim]{\nu} \R\Gamma_{\proeet}({\rm Spa}(R_{\Lambda}),K_{\phi}(F^r{\mathbb A}_{\crr})_n) \\
    & \xleftarrow[\sim]{\rm FES} \R\Gamma_{\proeet}({\rm Spa}(R_{\Lambda}),\Z/p^n(r)^{\prime}).
\end{align*}
Here ${\rm FES}$ stands for "fundamental exact sequence". 
This map is   functorial in $\Lambda$ and is a $p^{cr}$-quasi-isomorphism, for a universal constant $c$, after truncation $\tau_{\leq r}$ (by the discussion below diagram \ref{wies1b}).  Hence its rational version
$$
 {\bbalpha}^{R^{\Box}_{\Delta}}_{r}(\Lambda): {\rm Syn}({R^{\Box}_{\Delta}},r)_{\Q_p}(\Lambda) \to \R\Gamma_{\proeet}({\rm Spa}(R_{\Lambda}),\Q_p(r))
$$
is functorial in $\Lambda$ and a strict quasi-isomorphism after truncation $\tau_{\leq r}$.

 The map $ {\bbalpha}^{R^{\Box}_{\Delta}}_{r}(\Lambda)$ is functorial in the triples $(R^+,{R^{\Box}_{\Delta}},\iota_{\Delta})$ and taking the  colimit over the filtered system of such embeddings with fixed $R^+$ we get a map in $\sd(C_{\Q_p})$
 \begin{equation}
 \label{defFM1}
 {\bbalpha}^{R^+}_{r}(\Lambda): \R\Gamma_{\synt}({\rm Spf}(R^+),r)_{\Q_p}(\Lambda) \to \R\Gamma_{\proeet}({\rm Spa}(R_{\Lambda}),\Q_p(r)),
\end{equation}
which is functorial with respect to $R^+$ and $\Lambda$ and a strict quasi-isomorphism after truncation $\tau_{\leq n}$. 

  \subsubsection{Proof of Theorem \ref{lifted-period}}
  Let now $X\in {\rm Sm}_C$. The definition (\ref{defFM1}) of the period map globalizes  using $\eta$-\'etale sheafification to a period map in $\sd(C_{\Q_p})$
  $$
      {\bbalpha}_r(\Lambda):\quad {\mathbb R}_{\synt}(X,\Q_p(r))(\Lambda)\to {\mathbb R}_{\proeet}(X,\Q_p(r))(\Lambda),\quad r\geq 0,
  $$
  which for $\Lambda=C$ recovers the period map of Fontaine-Messing. 
This is a strict quasi-isomorphism after truncation $\tau_{\leq r}$. Now, varying $\Lambda$ we get the map we wanted:
 $$
      {\bbalpha}_r:\quad {\mathbb R}_{\synt}(X,\Q_p(r))\to {\mathbb R}_{\proeet}(X,\Q_p(r)),\quad r\geq 0.
  $$

 \subsection{Dagger varieties} We will now geometrize cohomologies and period morphisms associated to dagger varieties.
  \subsubsection{Cohomologies}
 Let $X$ be a dagger  affinoid over $C$, and $\{X_h\}$ be a presentation. Define the VS:
  $$
 {\mathbb R}^{\dagger}_{\proeet}(X,\Q_p):=\LL\colim_n{\mathbb R}_{\proeet}(X_h,\Q_p).
   $$
  For a smooth dagger variety $X$ over $C$, this globalizes, via \'etale sheafification,  to the VS
  $
 {\mathbb R}_{\proeet}(X,\Q_p). 
 $
 We set $$
 \wt{\mathbb H}^i_{\proeet}(X,\Q_p),  {\mathbb H}^i_{\proeet}(X,\Q_p):\quad  \Lambda\mapsto \wt{H}^i( {\mathbb R}_{\proeet}(X,\Q_p)(\Lambda)), {H}^i( {\mathbb R}_{\proeet}(X,\Q_p)(\Lambda)).
 $$
We define similarly the Hyodo-Kato cohomology, the $\Bdr^+$-cohomology, and the syntomic cohomology:
   \begin{align*}
   & {\mathbb R}_{\hk}(X),\quad \wt{\mathbb H}^i_{\hk}(X);\quad 
   {\mathbb R}_{\dr}(X/\Bdr^+),\quad  \wt{\mathbb H}^i_{\dr}(X/\Bdr^+);\quad
   {\mathbb R}^{\dagger}_{\synt}(X,\Q_p(r)),\quad  \wt{\mathbb H}^{\dagger, i}_{\synt}(X,\Q_p(r));\\
   &  {\mathbb R}_{\synt}(X,\Q_p(r)),\quad  \wt{\mathbb H}^{i}_{\synt}(X,\Q_p(r)).
   \end{align*}
   We note that the Hyodo-Kato cohomology is the constant functor equal to $\rg_{\hk}(X), \wt{H}^i_{\hk}(X)$.

   \subsubsection{Period maps}   Let $X$ be a dagger  affinoid over $C$, and $\{X_h\}$ be a presentation. 
Let $r\geq 0$. 
     The local period morphisms of VS's (with values in $\sd(C_{\Q_p})$)
   \begin{align*}
    {\bbalpha}^{\dagger}_r:\quad  {\mathbb R}^{\dagger}_{\synt}(X,\Q_p(r))\to {\mathbb R}^{\dagger}_{\proeet}(X,\Q_p(r))
   \end{align*}
   are defined as
   $$
      {\mathbb R}^{\dagger}_{\synt}(X,\Q_p(r))=\LL\colim_h    {\mathbb R}_{\synt}(X_h,\Q_p(r))\verylomapr{\LL\colim_h  {\bbalpha}_{r,h}}\LL\colim_h{\mathbb R}_{\proeet}(X_h,\Q_p(r))={\mathbb R}^{\dagger}_{\proeet}(X,\Q_p(r)).
   $$
   For a smooth dagger variety $X$, this globalizes to  period morphisms
   $$
       {\bbalpha}^{\dagger}_r:\quad     {\mathbb R}^{\dagger}_{\synt}(X,\Q_p(r))\to {\mathbb R}_{\proeet}(X,\Q_p(r)).
$$
These are strict quasi-isomorphisms after truncation $\tau_{\leq r}$ because so are the rigid analytic period morphisms $   {\bbalpha}_{r,h}$ by Theorem \ref{lifted-period}.

  Recall now that, for $X\in{\rm Sm}^{\dagger}_C$,   the period morphisms in $\sd(C_{\Q_p})$
  $$
  \alpha_r:\quad \rg_{\synt}(X,\Q_p(r))\to \rg_{\proeet}(X,\Q_p(r))
  $$
  are defined as the compositions
  $$
   \rg_{\synt}(X,\Q_p(r)) \xleftarrow[\sim]{ \iota^{\dagger}_{\synt}}\rg^{\dagger}_{\synt}(X,\Q_p(r))\stackrel{\alpha^{\dagger}_r}{\to} \rg_{\proeet}(X,\Q_p(r)).
  $$
  These morphisms lift to VS. Indeed, it remains to show that we can lift the map $\iota^{\dagger}_{\synt}$ to a map
  $$
  {\bbiota}^{\dagger}_{\synt}: \quad {\mathbb R}^{\dagger}_{\synt}(X,\Q_p(r)) \to {\mathbb R}_{\synt}(X,\Q_p(r)),
  $$
 and that this map  is a strict quasi-isomorphism. We define the map $   {\bbiota}^{\dagger}_{\synt}$ by \'etale sheafifying   the following composition ($X$ is  a smooth dagger affinoid over $C$),
$$\xymatrix@R=4mm@C=4mm{
   {\mathbb R}_{\synt}^{\dagger}(X,\Q_p(r)) =\LL\colim_h {\mathbb R}_{\synt}(X_{h},\Q_p(r))
\ar[r]^-{\sim}&\LL\colim_h {\mathbb R}_{\synt}(X^{\rm o}_{h},\Q_p(r))\\
{\phantom{XXXXXXXX}{\mathbb R}_{\synt}(X,\Q_p(r))}
  & \LL\colim_h {\mathbb R}_{\synt}(X^{{\rm o},\dagger}_{h},\Q_p(r))\ar[l]_-{\sim}\ar[u]_-{\bbiota}^-{\wr}
}$$
 Here, the  morphism $ {\bbiota}$  needs to be defined and both it and  
the bottom morphism need to be shown to be strict quasi-isomorphisms. 
 
 \begin{proposition} {\rm (Definition of the map $ {\bbiota}$)} 
Let $X\in {\rm Sm}^\dagger_C$. We have a natural map  of VS's (with values in $\sd(C_{\Q_p})$)
 $$
 {\bbiota}:\quad {\mathbb R}_{\synt}(X,\Q_p(r))\to {\mathbb R}_{\synt}(\wh{X},\Q_p(r))
 $$
 It is a strict quasi-isomorphism for $X$ partially proper. 
 \end{proposition}
 \begin{proof}
 We will set $ {\bbiota}:= {\bbiota}_2 {\bbiota}_1$, with the maps $ {\bbiota}_1,  {\bbiota}_2$ defined as follows. 
 
 (i) {\em The map $ {\bbiota}_1$.}  The map $ {\bbiota}_1$ is defined as the following composition:
$$\xymatrix@R=3mm@C=2mm{
 {\mathbb R}_{\synt}(X,\Q_p(r))\ar@{=}[r] & \big[[{\mathbb R}_{\hk}(X)\wh{\otimes}^{\R}_{F^{\nr}}\wh{\mathbb B}^+_{\st}]^{N=0,\phi=p^r}\verylomapr{\iota_{\hk}\otimes\iota}{\mathbb R}_{\dr}(X/\Bdr^+)/F^r\big]\ar@<6mm>[d]\\
   & \big[[{\mathbb R}_{\hk}(\wh{X})\wh{\otimes}_{F^{\nr}}{\wh{\mathbb B}}^+_{\st}]^{N=0,\phi=p^r}\verylomapr{\iota_{\hk}\otimes\iota}{\mathbb R}_{\dr}(\wh{X}/\Bdr^+)/F^r)\big]\\
   & \big[[{\mathbb R}_{\hk}(\wh{X})_{F}\wh{\otimes}_{F^{\nr}}{\mathbb B}^+_{\st}]^{N=0,\phi=p^r}\verylomapr{\iota_{\hk}\otimes\iota}{\mathbb R}_{\dr}(\wh{X}/\Bdr^+)/F^r)\big]\ar@<-6mm>[u]_{\wr}
}$$
    It is a strict quasi-isomorphism. 
Indeed, for that it suffices to show that the canonical map
$$
[{\mathbb R}_{\hk}(\wh{X})\wh{\otimes}_{F^{\nr}}{\mathbb B}^+_{\st}]^{N=0}\to [{\mathbb R}_{\hk}(\wh{X})\wh{\otimes}^{\R}_{F^{\nr}}\wh{\mathbb B}^+_{\st}]^{N=0}
$$
is a  strict quasi-isomorphism. But this can be shown exactly as in Section \ref{constr1}.  

  (ii) {\em The map $ {\bbiota}_2$.}  
Now, we define a natural strict quasi-isomorphism $ {\bbiota}_2$ by
$$\xymatrix@R=3mm@C=-2mm{
    \big[ [{\mathbb R}_{\hk}(\wh{X})_{F}\wh{\otimes}_{F^{\nr}}{\mathbb B}^+_{\st}]^{N=0,\phi=p^r}\verylomapr{\iota_{\hk}\otimes\iota} {\mathbb R}_{\dr}(\wh{X}/\Bdr^+)/F^r\big]\ar@<8mm>[d]\\
{\phantom{XXXX}[[{\mathbb R}_{\crr}(\wh{X})]^{\phi=p^r}\verylomapr{\can} {\mathbb R}_{\crr}(\wh{X})/F^r]}
\ar@{=}[r]
& {\mathbb R}_{\synt}(\wh{X},\Q_p(r)).
}$$
    For that, it suffices to define the maps $\iota_{\rm BK}^1$ and $\iota_{\rm BK}^2$ in the following diagram and to show that this diagram commutes:
    \begin{equation}
    \label{film11}
\xymatrix{
[{\mathbb R}_{\hk}(\wh{X})\wh{\otimes}_{F^{\nr}}{\mathbb B}^+_{\st}]^{N=0}\ar[d]^{\epsilon^{\hk}_{\st}\otimes\id}_{\wr} \ar@/_80pt/[dd]^{\wr}_{\iota_{\rm BK}^1}
\ar[r]^-{\iota_{\hk}\otimes \iota}   &   
   ({\mathbb R}_{\dr}(\wh{X}/\Bdr^+)\wh{\otimes}^{\R}_{\B^+_{\dr}}{\mathbb B}^+_{\dr})/F^r
 \ar@/^80pt/[dd]^{\iota_{\rm BK}^2}_{\wr}\\
     [{\mathbb R}_{\crr}(\wh{X})\wh{\otimes}_{\B^+_{\crr}}{\mathbb B}^+_{\st}]^{N=0}\ar[ru]^{\kappa\otimes\iota} \\
         {\mathbb R}_{\crr}(\wh{X})\wh{\otimes}_{\B^+_{\crr}}{\mathbb B}^+_{\crr}\ar[u]^{\wr}\ar[r]^-{\can}  &        ({\mathbb R}_{\crr}(\wh{X})\wh{\otimes}_{\B^+_{\crr}}{\mathbb B}^+_{\crr})/F^r\ar[uu]^{\kappa\otimes\can}_{\wr}.
}
\end{equation}  
  We define the maps $\iota_{\rm BK}^1$ and $\iota_{\rm BK}^2$ to make the left  and the right triangles in the diagram commute. They are strict quasi-isomorphisms.  The remaining pieces of the diagram commute by definition.
  \end{proof}
 Let $X\in {\rm Sm}^{\dagger}_C$. We define the global period morphism  of VS's (with values in $\sd(C_{\Q_p})$)
 $$
  {\bbalpha}_r: {\mathbb R}_{\synt}(X,\Q_p(r))\to {\mathbb R}_{\proeet}(X,\Q_p(r))
 $$
 as the composition $ {\bbalpha}^{\dagger}_r({\bbiota}^{\dagger}_{\synt})^{-1} $. From what we have shown above, it follows that:
 \begin{corollary}
 The natural map  of VS's (with values in $\sd(C_{\Q_p})$)
 $$
 \tau_{\leq r}  {\bbalpha}_r: \tau_{\leq r}{\mathbb R}_{\synt}(X,\Q_p(r))\to \tau_{\leq r}{\mathbb R}_{\proeet}(X,\Q_p(r))
 $$
 is a strict quasi-isomorphism.
 \end{corollary}

\printindex

\end{document}